%% file: spt.tex
\documentclass{article}
\usepackage {amsmath, amsfonts, amssymb, mathrsfs, amsthm,stmaryrd, sectsty,hyphenat,enumitem,calc,url}
\usepackage{palatino}
\usepackage[all]{xy}
\usepackage[nottoc]{tocbibind}
\input{macros.tex}

\newtheorem{thm}{Theorem}[subsection]
\newtheorem{lem}[thm]{Lemma}
\newtheorem{pro}[thm]{Proposition}
\newtheorem{cor}[thm]{Corollary}
\newtheorem{dfn}[thm]{Definition}
\newtheorem{fct}[thm]{Fact}

\sectionfont{\center\sc\normalsize}
\subsectionfont{\bf\normalsize}

\numberwithin{equation}{section}

\newlength{\sumcorr}
\def\ssum#1{\setlength{\sumcorr}{(\widthof{$\displaystyle\sum_{#1}$}-\widthof{$\displaystyle\sum$})/2} \hspace{-\sumcorr}\sum_{#1}\hspace{-\sumcorr} }
\def\sprod#1{\setlength{\sumcorr}{(\widthof{$\displaystyle\prod_{#1}$}-\widthof{$\displaystyle\prod$})/2} \hspace{-\sumcorr}\prod_{#1}\hspace{-\sumcorr} }

\begin{document}

\begin{mytitle} Regular supercuspidal representations \end{mytitle}
\begin{center} Tasho Kaletha \end{center}

\begin{abstract}
We show that, in good residual characteristic, most supercuspidal representations of a tamely ramified reductive $p$-adic group $G$ arise from pairs $(S,\theta)$, where $S$ is a tame elliptic maximal torus of $G$, and $\theta$ is a character of $S$ satisfying a simple root-theoretic property. We then give a new expression for the roots of unity that appear in the Adler-DeBacker-Spice character formula for these supercuspidal representations and use it to show that this formula bears a striking resemblance to the character formula for discrete series representations of real reductive groups. Led by this, we explicitly construct the local Langlands correspondence for these supercuspidal representations and prove stability and endoscopic transfer in the case of toral representations. In large residual characteristic this gives a construction of the local Langlands correspondence for almost all supercuspidal representations of reductive $p$-adic groups.
\end{abstract}
{\let\thefootnote\relax\footnotetext{AMS 2010 Mathematics subject classification: 22E50; 11S37; 11F70.}}
{\let\thefootnote\relax\footnotetext{This research is supported in part by NSF grant DMS-1161489 and a Sloan Fellowship.}}

\tableofcontents

\section{Introduction}

This paper pursues multiple interconnected goals, all of which are related to Yu's construction of supercuspidal representations of reductive $p$-adic groups \cite{Yu01}, which generalizes Adler's earlier construction \cite{Ad98}. Recall briefly that if $G$ is a connected reductive group over a $p$-adic field $F$ that splits over a tamely ramified extension of $F$, a supercuspidal representation of $G(F)$ can be constructed by giving the following data: a tower $G^0 \subset \dots \subset G^d = G$ of connected reductive subgroups that become Levi subgroups of $G$ over some tame Galois extension of $F$, a sequence of characters $\phi_i : G^i(F) \to \C^\times$ for all $i \geq 0$ satisfying a certain genericity condition, and a depth-zero supercuspidal representation $\pi_{-1}$ of $G^0(F)$, which we may call the socle of the Yu-datum. Representations obtained from this construction are customarily called tame, even though they can have arbitrary depth (in the case of $G=\tx{GL}_N$, these representations are called essentially tame in the work of Bushnell and Henniart; when $p \nmid N$ all supercuspidal representations are essentially tame). Hakim and Murnaghan \cite{HM08} have shown that different Yu-data can lead to the same representation and have made a precise study of when this happens. This leads to the natural question of whether one can use simpler data to parameterize the supercuspidal representations resulting from Yu's construction. Ideally, such data would consist simply of a maximal torus $S \subset G$ and a character $\theta : S(F) \to \C^\times$, in analogy with the classification of discrete series representations of real reductive groups, as well as that of supercuspidal representations of $\tx{GL}_N$ when $p\nmid N$. There is an immediate obstruction to this: Most reductive groups over finite fields (but not $\tx{GL}_N$) have cuspidal representations that are not immediately parameterizable by such pairs (for example cuspidal unipotent representations), and this obstruction propagates to depth-zero supercuspidal representations of reductive groups over $F$. We therefore restrict our attention to Yu-data that satisfy a slight regularity condition, which is automatically satisfied for $G=\tx{GL}_N$, and whose main part is that the socle $\pi_{-1}$ (when it is non-trivial) corresponds to a Deligne-Lusztig representation (of the reductive quotient of a parahoric subgroup of $G$) that is associated to a character in general position. Let us call supercuspidal representations arising from such Yu-data regular. The first main goal of this paper is to give an explicit parameterization of regular supercuspidal representations in terms of $G(F)$-conjugacy classes of pairs $(S,\theta)$. Partial results towards this were obtained earlier by Murnaghan in \cite{Mur11}, where a further technical restriction is imposed on $\pi_{-1}$ and an injective map is constructed from the set of equivalence classes of regular supercuspidal representations satisfying this additional technical restriction to the set of $G(F)$-conjugacy classes of pairs $(S,\theta)$ consisting of an elliptic maximal torus and a character of it. No effective description of the image of this map was known. For many purposes it is important to have a map in the opposite direction -- from pairs $(S,\theta)$ to representations. In the current paper we introduce the notion of a tame regular elliptic pair $(S,\theta)$. This notion is defined in simple and explicit root-theoretic terms. We show that in the case of $\tx{GL}_N$ it specializes to the classical notion of an admissible character. We give an explicit algorithm that, starting from a tame regular elliptic pair, produces a Yu-datum for a regular supercuspidal representation. This algorithm can be seen as a generalization to arbitrary reductive groups of the Howe factorization lemma (\cite[Lemma 11 and Corollary]{Howe77}) that plays an important role in the construction of supercuspidal representations of $\tx{GL}_N$. Just as in the case of $\tx{GL}_N$, the factorization we obtain is not unique, but we show that two possible factorizations are related to each other by a process already introduced by Hakim and Murnaghan, called refactorization. Their work implies that the resulting supercuspidal representation is unaffected by this ambiguity, and may thus be called $\pi_{(S,\theta)}$. We then show that two such representations are isomorphic if and only if the pairs giving rise to them are $G(F)$-conjugate. It is then straightforward to check that the map $(S,\theta) \mapsto \pi_{(S,\theta)}$ is the inverse to Murnaghan's injection (after removing the additional technical restriction on $\pi_{-1}$ imposed in \cite{Mur11}). This implies that the image of Murnaghan's map is precisely the set of $G(F)$-conjugcy classes of tame regular elliptic pairs. In this way, we obtain explicit mutually inverse bijections between the set of $G(F)$-conjugacy classes of tame regular elliptic pairs and the set of isomorphism classes of regular supercuspidal representations. This classification result includes as a special basic case the regular supercuspidal representations of depth zero. In fact, this special case is needed as the basis of our argument. When $G$ splits over an unramified extension, regular depth-zero supercuspidal representations were studied by DeBacker and Reeder \cite{DR09}. As a preparation for the study of regular supercuspidal representations of general depth, we extend their classification results to the case of tamely ramified groups $G$.

When the residual characteristic of $F$ is not too small for $G$ the work of Kim \cite{Kim07} shows that all supercuspidal representations of $G(F)$ arise from Yu's construction. Most of these are regular and thus of the form $\pi_{(S,\theta)}$. For $G=\tx{GL}_N$ with $p \nmid N$ all supercuspidal representations are regular, but for other groups non-regular supercuspidal representations do exist, as the example of the four exceptional supercuspidal representations of $\tx{SL}_2$ shows. We believe that our work can be used to reduce the description of general supercuspidal representations in terms of elliptic (but not necessarily regular) pairs $(S,\theta)$ to the description of cuspidal representations of finite groups of Lie type in terms of Deligne-Lusztig virtual characters. It would be interesting to pursue this question.

In the second part of the paper we study the Harish-Chandra character of supercuspidal representations, and in particular of the representations $\pi_{(S,\theta)}$. A formula for the character of a supercuspidal representation $\pi$ arising from Yu's construction has been given by
Adler and Spice \cite{AS09} and subsequently refined by DeBacker and Spice \cite{DS}. At the moment this formula is only valid under the assumption that $G^{d-1}(F)/Z(G)(F)$ is compact, but in private communication the authors have assured me that this assumption will soon be removed. In the mean time we have proved in this paper a technical result which removes this condition in a certain special case that still allows us to draw conclusions from it. The character formula \cite[Theorem 4.6.2]{DS} has the following form. Recall first that Yu's construction produces not just a supercuspidal representation $\pi$ of $G(F)$, but in fact a supercuspidal representation $\pi_i$ of $G^i(F)$ for each $i$. Let $r$ be the depth of $\pi_{d-1}$. Given a regular semi-simple element $\gamma \in G(F)$ admitting a decomposition (or approximation) $\gamma = \gamma_{<r}\cdot \gamma_{\geq r}$ in the sense of \cite{AS08}, the value at $\gamma$ of the normalized Harish-Chandra character function of $\pi=\pi_d$ is
\begin{equation} \label{eq:ichar} \Phi_{\pi_d}(\gamma) = \sum_g \epsilon_{s,r}(\gamma_{<r}^g)\epsilon^r(\gamma_{<r}^g)\tilde e(\gamma_{<r}^g)\Phi_{\pi_{d-1}}(\gamma_{<r}^g)\hat\mu_{^gX^*}(\log(\gamma_{\geq r})), \end{equation}
where the sum is over certain elements $g \in G(F)$. The term $\hat\mu$ is the Fourier-transform of an orbital integral (on the Lie-algebra of the connected centralizer of $\gamma_{<r}$), and $\epsilon_{s,r}$, $\epsilon^r$, and $\tilde e$, are roots of unity. This formula mirrors the inductiveness of Yu's construction by expressing the character of $\pi=\pi_d$ in terms of that of $\pi_{d-1}$. The function $\epsilon^r$ is a character of $S(F)$ and this allows it to be handled easily. In fact, as is pointed out in \cite{DS}, this function might be an artifact resulting from certain choices inherent in Yu's construction of supercuspidal representations, and a modification of this construction suppresses it. On the other hand, the two functions $\epsilon_{s,r}$ and $\tilde e$ are not characters of $S(F)$. Their definition is quite subtle and involves the fine structure of the $p$-adic group $G(F)$, in particular its Bruhat-Tits building and associated Moy-Prasad filtrations. This makes the analysis of these two functions with respect to stable conjugacy and related questions difficult. The second main result of this paper gives a new expression for the product $\epsilon_{s,r}(\gamma_{<r}) \cdot \tilde e(\gamma_{<r})$. This expression is a quotient of two terms of the form
\begin{equation} \label{eq:i1} e(G)e(J)\epsilon_L(X^*(T_G)_\C-X^*(T_J)_\C,\Lambda)\Delta_{II}^\tx{abs}[a,\chi](\gamma_{<r}). \end{equation}
Here $J$ is the connected centralizer of $\gamma_{<r}$ and the terms $e(G)$ and $e(J)$ are the Kottwitz signs \cite{Kot83} of the connected reductive groups $G$ and $J$. The tori $T_G$ and $T_J$ are the minimal Levi subgroups in the quasi-split inner forms of $G$ and $J$, and $\epsilon_L$ is the $\epsilon$-factor at $s=1/2$ of the given virtual Galois representation. Finally, the term $\Delta_{II}^\tx{abs}$ is an absolute version of the corresponding term of the Langlands-Shelstad endoscopic transfer factor \cite[\S3.3]{LS87}. What we mean here by the word ``absolute'' is that while the term $\Delta_{II}$ of Langlands-Shelstad is associated to a group $G$ and an endoscopic group $G^\mf{e}$, the term $\Delta_{II}^\tx{abs}$ is only associated to the group $G$, and moreover one obtains the Langlands-Shelstad term $\Delta_{II}$ as a quotient of the terms $\Delta_{II}^\tx{abs}$, with the one for $G$ in the numerator and the one for $G^\mf{e}$ in the denominator.

Before we discuss the main implication of \eqref{eq:i1}, let us consider some if its features. First, none of the terms in \eqref{eq:i1} involve Bruhat-Tits theory in their construction. Rather, they come from Lie-theory and basic $p$-adic arithmetic and are thus more elementary (the reader might argue that $\epsilon$-factors of non-abelian local Galois representations are not elementary, but a result of Kottwitz computes the particular $\epsilon$-factor we are dealing with in elementary terms, see \cite[\S3.5]{KalEpi}). Note also that the first three terms in \eqref{eq:i1} depend only on the stable conjugacy class of $\gamma_{<r}$, and it is just the term $\Delta_{II}^\tx{abs}$ that depends on the full triple $(S,\theta,\gamma_{<r})$.

Another interesting feature of the character formula \eqref{eq:i2} is that it provides an interpretation of most of the Langlands-Shelstad endoscopic transfer factor in terms of the characters of supercuspidal representations. Recall that the Langlands-Shelstad transfer factor is given as a product
\[ \epsilon_L \cdot \Delta_I \cdot \Delta_{II} \cdot \Delta_{III_1} \cdot \Delta_{III_2} \cdot \Delta_{IV}. \]
In view of \eqref{eq:i2}, each of the factors $\epsilon_L$, $\Delta_{II}$, $\Delta_{III_2}$, and $\Delta_{IV}$ has an interpretation as the quotient of a piece of the character formula for $G$ by the corresponding piece for $G^\mf{e}$. The factors $\epsilon_L$, $\Delta_{II}$, and $\Delta_{IV}$, are directly visible in \eqref{eq:i2}, and so is also the factor $\Delta_{III_2}$, being the quotient of the character $\theta$ by the corresponding character on the side of $G^\mf{e}$. The factor $\Delta_I$ also appears, albeit in a more subtle way: It measures which representation in the $L$-packet is generic, as we will discuss in Subsection \ref{sub:chargen}. We believe that in this way the character formula \eqref{eq:i2} sheds a different light on the Langlands-Shelstad transfer factor and shows that almost all parts of it are actually not of strictly endoscopic nature (except for the term $\Delta_{III_1}$, which is indeed purely endoscopic and not a relative term in the sense discussed here). We hope that this point of view will be fruitful in the study of more general functoriality beyond the endoscopic case, and might in particular help with the study of the transfer factors occurring there \cite{Lan13}.

The main implication of \eqref{eq:i1} stems from the following observation. If we apply the formula \eqref{eq:ichar} to a regular supercuspidal representation $\pi_{(S,\theta)}$ and a regular semi-simple element $\gamma
\in S(F)$ that is very far from the identity (this is the special case in which we have been able to remove the compactness hypothesis), we obtain as a consequence of \eqref{eq:i1} the following formula for the un-normalized Harish-Chandra character $\Theta_{\pi_{(S,\theta)}}(\gamma)$
\begin{equation} \label{eq:i2}
\frac{e(G)\epsilon_L(X^*(T_G)_\C-X^*(S)_\C)}{|D(\gamma)|^{\frac{1}{2}}} \sum_{w \in N(S,G)(F)/S(F)}
\Delta_{II}^\tx{abs}[a,\chi](\gamma^w) \theta'(\gamma^w).
\end{equation}
Here we have set $\theta'=\epsilon^r \cdot \theta$ using the fact that $\epsilon^r$ is a character of $S(F)$. What is striking about this formula is that each term has an interpretation for groups $G$ defined over an arbitrary local field, not just a $p$-adic field, and when $F=\R$ \eqref{eq:i2} becomes the formula for the character of the discrete series representation of the real group $G(\R)$ associated to the elliptic maximal torus $S$ and the character $\theta' : S(\R) \to \C^\times$, i.e. the well-known formula
\[ (-1)^{q(G)}\sum_{w \in N(S,G)(\R)/S(\R)} \frac{\theta'(w^{-1}\gamma)}{\prod\limits_{\alpha >0} (1-\alpha(w^{-1}\gamma)^{-1})}. \]
This can be seen as an instance of Harish-Chandra's Lefschetz principle, which suggests a mysterious analogy between the behaviors of real and $p$-adic reductive groups. In fact, if we consider the full character formula \eqref{eq:ichar}, we see that it combines two extreme behaviors -- the behavior at elements near the identity ($\gamma=\gamma_{\geq r}$), which is controlled by $\hat\mu$, and the behaviour at elements far from the identity ($\gamma=\gamma_{<r}$), to which all the roots of unity contribute. The Fourier-transform of the orbital integral $\hat\mu$ appears to belong to the world of finite groups of Lie type. For example, when $\pi$ has depth zero, the term $\hat\mu$ is a lift \cite[Lemma 12.4.3]{DR09} of a Green function, expressing the character of a cuspidal representation of a finite group of Lie type at a unipotent element. The roots of unity on the other hand seem to belong to the world of real reductive groups. This suggests that the behavior of $p$-adic groups is an interpolation between the behavior of finite groups of Lie type and the behavior of real reductive groups.

The close parallel between the characters of regular supercuspidal representations at shallow elements and the characters of real discrete series, besides being alluring in its own right, also has practical value, which brings us to the third main goal of this paper -- the construction and study of $L$-packets of regular supercuspidal representations and their matching with Langlands parameters. The original approach \cite{Lan89} of Langlands to the construction of $L$-packets of discrete series representations for real reductive groups was to first extract from the Langlands parameter a character of the elliptic maximal torus and to then use this character to write down the Harish-Chandra characters of the constituents of the $L$-packet. The recent explicit constructions of $L$-packets for $p$-adic groups \cite{DR09}, \cite{Ree08}, \cite{KalEpi} have followed this procedure to the extent that they extract form the Langlands parameter a character $\chi$ of an elliptic maximal torus, but then they determine the constituents of the $L$-packet not via their Harish-Chandra characters, but by plugging in some modification of the character $\chi$ into Adler's construction in the case of $r>0$, or into the construction of \cite[\S4.4]{DR09} in the case of $r=0$. The works of Adler-Spice \cite{AS09}, DeBacker-Reeder \cite{DR09}, and DeBacker-Spice \cite{DS} on the character formula for supercuspidal representations and our reinterpretation of it from the first part of this paper allow us to implement a much closer analog of Langlands' construction and use it to construct the $L$-packets that consist of regular supercuspidal representations and associate them to Langlands parameters.

The class of parameters we consider in this paper  contains as special cases those considered in the above mentioned papers, but is much larger. More precisely, it consists of those discrete Langlands parameters $\varphi : W_F \to {^LG}$ for which $\varphi(P_F)$ is contained in a maximal torus of $\hat G$ and $\tx{Cent}(\varphi(I_F),\hat G)$ is abelian (as well as a small amount of slightly more complicated parameters that we need in order to obtain a balanced theory). Guided by the character formula \eqref{eq:i2} we assemble the $L$-packet corresponding to a given parameter in the same way as Langlands constructs the packets of real discrete series representations -- by writing down (a piece) of the Harish-Chandra character of each constituent of the $L$-packet. Our construction is nonetheless completely explicit: Given a parameter, we explicitly give the inducing data for each constituent of the $L$-packet. Conversely, one can also explicitly recover the $L$-parameter from this inducing data. Important for this is the fact that the notion of a tame regular elliptic pair $(S,\theta)$ has a direct interpretation in terms of ${^LG}$.

Let us now describe the construction of $L$-packets in more detail. Initially it follows the framework laid out in \cite{KalEpi}. A parameter $\varphi$ satisfying the above conditions determines an algebraic torus $S$. We use $\chi$-data to produce an embedding of ${^Lj_\chi} : {^LS} \rw {^LG}$ whose image contains the image of $\varphi$, and hence leads to a factorization $\varphi = {^Lj_\chi} \circ \varphi_{S,\chi}$, after which the parameter $\varphi_{S,\chi} : W_F \rw {^LS}$ leads to a character $\theta_\chi : S(F) \rw \C^\times$ via the Langlands correspondence for tori. The torus $S$ comes equipped with a stable class of embeddings into $G$ (an in fact into any inner form of $G$). For any embedding $j : S \rw G$ belonging to this stable class, we obtain an elliptic maximal torus $jS \subset G$ with a character $j\theta_\chi$ of it. It is at this point that the construction of the current paper diverges from the previous constructions. We write down the formula
\[ e(G)\epsilon_L(X^*(T_G)_\C-X^*(S)_\C,\Lambda) \sum_{w \in N(G,jS)(F)/jS(F)} \Delta_{II}^\tx{abs}[a,\chi](\gamma^w)j\theta_\chi(\gamma^w), \]
and demand that $\pi_j$ be the regular supercuspidal representation whose normalized character at shallow elements $\gamma\in jS(F)$ is given by this formula. In practice we ensure that this demand is met by explicitly providing the pair $(S,\theta)$ that parameterizes the regular supercuspidal representation, but we feel that the difference in point of view is essential. At this point, a remark is in order about the choice of $\chi$-data involved. In \cite[\S4.2]{KalEpi} we spent a lot of effort to choose the correct $\chi$-data so that the character $\theta_\chi$ of $S$ we obtain would be the right one for Adler's construction. From the current point of view, the choice of $\chi$-data is irrelevant. This is because both $\theta_\chi$ and $\Delta_{II}^\tx{abs}[a,\chi]$ depend on this choice in a parallel way and the dependence cancels in the product. However, $\Delta_{II}^\tx{abs}[a,\chi]$ also depends on $a$-data, and there is no other object in the character formula with this dependence. This means that the burden is now on choosing the $a$-data correctly. It turns out that this choice is given by a simple formula \eqref{eq:adatahowe} that is uniform for real and $p$-adic groups. The only difference in the $p$-adic case is that one needs to pay attention to the first upper numbering filtration subgroup of inertia whose image under $\varphi$ is detected by a given root of $\hat G$. This is reminiscent of the study of the jumps of an admissible character in the work of Bushnell and Henniart \cite{BH05a}, \cite{BH05b}. In fact, our work here might be seen as a generalization to arbitrary tamely ramified $p$-adic groups of the work of Bushnell-Henniart, insofar as both have the goal of giving an explicit realization of the local Langlands correspondence.

Once the representations $\pi_j$ are determined, the $L$-packet is defined to be the set $\{\pi_j\}$ where $j$ runs over all rational classes of embeddings of $S$ into $G$. The internal parameterization of this $L$-packet is again done as in \cite[\S4.3]{KalEpi}, the only difference being that now we are using the cohomology functor $H^1(u \rw W,-)$ introduced in \cite{KalRI} instead of the set $B(G)_\tx{bas}$ used in \cite{KalEpi}. This allows us to uniformly treat all connected reductive groups, without conditions on the center. A reader interested in having a parameterization in terms of $B(G)_\tx{bas}$, say for the purpose of studying Rapoport-Zink spaces, can either replace in the construction all occurrences of $H^1(u \rw W,-)$ with $B(-)_\tx{bas}$, or appeal to the general results of \cite{KalRIBG}.

We now give a brief overview of the contents of this paper. Section \ref{sec:regsc} contains the study of regular supercuspidal representations. In Subsection \ref{sub:tori} we collect some basic facts about $p$-adic tori, and in particular extend Yu's theorem \cite[Theorem 7.10]{Yu09} that the local Langlands correspondence for tamely ramified tori preserves depth from the case of positive depth to the case of characters vanishing on the Iwahori subgroup and on the maximal bounded subgroup of a torus. In Subsection \ref{sub:rdz} we classify the regular depth-zero supercuspidal representations of tamely ramified groups. This is based on the notion of a maximally unramified maximal torus of a tamely ramified group, that generalizes the notion of an unramified maximal torus of a group that splits over an unramified extension. This notion, suggested by Dick Gross, already appears in \cite{Roe11}, where the regular depth-zero supercuspidal representations of ramified unitary groups are studied. We then extend results of DeBacker \cite{Deb06} on the parameterization of unramified tori to the maximally unramified setting, focusing on the case of elliptic tori that will be needed later. Using these results, we classify the regular depth-zero supercuspidal representations of tamely ramified groups, extending results of DeBacker-Reeder \cite{DR09}. The main hurdle in the construction of regular depth-zero supercuspidal representations is that if $S$ is a maximally unramified maximal torus of the connected reductive group $G$, then the equality $S(F)=S(F)_0 \cdot Z(G)(F)$ is not always true, as was pointed out to us by Cheng-Chiang Tsai. This equality holds in the unramified case, as well as in the case of ramified unitary groups, and makes the passage from a cuspidal representation of a parahoric subgroup of $G(F)$ to a supercuspidal representation of $G(F)$ straightforward. This equality is also equivalent to the technical condition imposed on $\pi_{-1}$ in \cite{Mur11}. We deal with the additional difficulty in the general ramified case in Subsubsection \ref{subsub:ext} by exploiting the fact that the Deligne-Lusztig variety $\tilde X(\dot w)$ associated to $\dot w \in N(T^*)$ (notation as in \cite[\S1.8]{DL76}) admits an action of $\mathbf{T}_\tx{ad}(w)^F$ by conjugation.

The rest of Section \ref{sec:regsc} is devoted to the study of the positive-depth case. In Subsection \ref{sub:rs} we define the notion of a tame regular elliptic pair and show that it specializes in the case of $\tx{GL}_N$ to the classical notion of an admissible character. We also show that a regular Yu-datum gives rise to a tame regular elliptic pair. In Subsection \ref{sub:howe} we introduce the process of Howe factorization, which produces conversely from a tame regular elliptic pair a regular Yu-datum. In Subsection \ref{sub:regclass} we show that these two processes are indeed inverse to each other and that the choices involved in producing a Howe factorization do not influence the resulting supercuspidal representations. The result of this subsection is a bijection from the set of $G(F)$-conjugacy classes of tame regular elliptic pairs to the set of isomorphism classes of regular supercuspidal representations. The two Subsections \ref{sub:howe} and \ref{sub:regclass} have the assumption that $p$ does not divide the order of the fundamental group of the the derived subgroup of $G$. This assumption is removed in Subsection \ref{sub:nsc}.

Before moving on to Section 4 we mention here a recent draft \cite{Hak16} that was sent to us after this paper was written, in which Hakim reinterprets Yu's construction and gives a parameterization of the resulting representations in terms of a different kind of data. His interpretation has the advantage that the refactorization process studied in \cite{HM08} becomes unnecessary. The goal and results of this draft are quite disjoint from ours. It would be interesting to see if the two approaches can be combined.

Section \ref{sec:char} is devoted to our reinterpretation of the Adler-DeBacker-Spice character formula. The technical heart of this section is Subsection \ref{sub:ordx}, in which we give a formula for a certain subset $\tx{ord}_x(\alpha) \subset \R$ associated by \cite[Definition 3.1.3]{DS} to a tame maximal torus $T$ of $G$, a root $\alpha$ of $T$, and a point $x$ in the Bruhat-Tits building of $T$ seen as embedded into the building of $G$. This set plays a fundamental role in the character formula, because all roots of unity occurring in the formula are defined based on it. According to \cite[Corollary 3.1.9]{DS} there are only two possibilities for this set and Proposition \ref{pro:ordx} shows that these possibilities are distinguished by the toral invariant introduced in \cite[\S3]{KalEpi}. After giving the definition of the term $\Delta_{II}^\tx{abs}$ in Subsection \ref{sub:delta2} we are in a position to rewrite the character formula. We need however to pay attention to the technical assumption that $G^{d-1}(F)/Z(G)(F)$ is compact, under which the character formula of \cite{AS09} and \cite{DS} is valid. For toral supercuspidal representations this assumption is automatically satisfied and we can write the full character formula in this case, which is done in Subsection \ref{sub:toral}. For general regular supercuspidal representations $\pi_{(S,\theta)}$ we are able to show in Subsection \ref{sub:charshallow} that this assumption can be dropped provided we consider sufficiently shallow elements belonging to the torus $S$. We use this fact, together with a computation in the depth-zero case done in Subsection \ref{sub:depthzero}, to prove \eqref{eq:i2} in Subsection \ref{sub:shallow}. We conclude with Subsection \ref{sub:realchar}, where we compare \eqref{eq:i2} with  the character formula for real discrete series representations.

Section \ref{sec:pack} contains the construction of regular supercuspidal $L$-packets. We also give a description of the internal structure of each $L$-packet $\Pi_\varphi$ by showing that it has a simply transitive action of the abelian group $\pi_0(S_\varphi^+)^D$. In order to convert this into a bijection, we need to know that the choice of a Whittaker datum for the quasi-split group $G$ determines a base point in the compound $L$-packet $\Pi_\varphi$, in accordance with the strong form of Shahidi's tempered $L$-packet conjecture \cite[\S9]{Sha90}. Due to the technical compactness assumption under which at the moment the Adler-DeBacker-Spice character formula is valid for elements close to the identity, we are not in a position to do so for general regular supercuspidal $L$-packets. For the same reason, we can only prove stability or endoscopic transfer for these packets for shallow elements, but not for general regular semi-simple elements. Both of these points will be addressed in forthcoming joint work with DeBacker and Spice, based on onging work of Spice on removing the compactness assumption from \cite{DS}. 

We are however able to prove these statements for toral $L$-packets, which are the topic of Section \ref{sec:toral}, where we specialize the construction of $L$-packets to the case of toral supercuspidal representations. These are the representations obtained from a Yu-datum for which the twisted Levi sequence is of the form $S=G^0 \subset G^1=G$, where $S$ is an elliptic maximal torus of $G$. For these representations the compactness assumption is satisfied and thus the Adler-DeBacker-Spice character formula is valid for general elements, rather than just for shallow elements. Using it, we are able to prove the existence and uniqueness of generic constituent in each compound $L$-packet as well as the stability and endoscopic transfer of these $L$-packets (the stability of toral $L$-packets under the additional assumption that $S$ is unramified was already shown in \cite{DS}). We expect the same arguments to apply to the case of the general regular supercuspidal $L$-packets of Section \ref{sec:pack}, once the compactness assumption on the Adler-DeBacker-Spice character formula has been removed.

\tb{Acknowledgements:} The initial spark for this paper came from a remark of Robert Kottwitz that the pieces $\epsilon_L$ and $\Delta_{II}$ of the Langlands-Shelstad transfer factor can be directly observed in the character formulas for supercuspidal representations of $\tx{SL}_2$ due to Sally and Shalika. We are grateful to Kottwitz for drawing our attention to this. We were also influenced by Moshe Adrian's thesis \cite{Adrian10} and subsequent paper \cite{Adrian13}, which carries out in the case of $\tx{GL}_n$ (for prime $n$) the main idea we employ here -- the description of the local Langlands correspondence in terms of Harish-Chandra characters. We further thank Cheng-Chiang Tsai for multiple helpful discussions about Bruhat-Tits theory and in particular for a useful counterexample that he provided, and Jeffrey Hakim for his careful reading.

\section{Notation and assumptions}

\subsection{Assumptions on the ground field}
Throughout most of the paper, $F$ denotes a non-archimedean local field of zero or positive characteristic. Exceptions to this are \S\ref{sub:reviewtrans}, where $F$ can by any field and \S\ref{sub:realchar}, where $F=\R$.

For convenience, we collect here the assumptions on $F$ placed in different parts of the paper. In \S\ref{sub:tori} and \S\ref{sub:rdz} there are no further assumptions on $F$. Starting with \S\ref{sub:reviewhm} we assume that the residual characteristic of $F$ is odd and this assumption is kept throughout. In \S\ref{sub:rs}, \S\ref{sub:howe} and \S\ref{sub:regclass}, we assume furtherthat the residual characteristic is not a bad prime for $G$ and does not divide the order of $\pi_1(G_\tx{der})$. The last of these assumptions is only for technical convenience and is removed in \S\ref{sub:nsc}.

We recall from \cite[I,\S4]{SS70} the bad primes: For type $A_n$ there are no bad primes, for types $B_n$, $C_n$, or $D_n$ the only bad prime is $2$, for types $E_6$, $E_7$, $F_4$, or $G_2$ the bad primes are $2$ and $3$, and for type $E_8$ the bad primes are $2$, $3$, and $5$.

In \S\ref{sub:ordx}, \S\ref{sub:delta2} and \S\ref{sub:signs} the only assumption on the local field $F$ is that its residual characteristic is odd. For the rest of \S\ref{sec:char} we assume further that the residual characteristic is not a bad prime for $G$.

In \S\ref{sec:pack} and \S\ref{sec:toral} we assume that the residual characteristic is odd, not a bad prime for $G$, and does not divide $|\pi_0(Z(G))|$. The latter condition is implied by the former unless $G$ has components of type $A_n$. If $G$ has a component of type $A_n$, a sufficient condition would be $p \nmid (n+1)$. We are moreover forced to assume that the characteristic of $F$ is zero due to the usage of \cite{KalRI}, where this assumption is made. We believe that the results of \cite{KalRI} are valid without this assumption, but have not checked this carefully. Finally, in \S\ref{sub:toraltrans} we must assume that $F$ has characteristic zero and large residual characteristic. Among other things, this allows us to use the exponential map for all topologically nilpotent elements of the Lie algebra of $G$.

In some parts of the paper we appeal to papers such as \cite{LS87} or \cite{KS99}, where a blanket assumption is made that the ground field is of characteristic zero. It is however easy to check that for the results we use this assumption is not needed.

\subsection{Further notation}

We denote the ring of integers of $F$ by $O_F$, its maximal ideal by $\mf{p}_F$, and its residue field by $k_F$, of cardinality $q$. We fix a separable closure $F^s$ of $F$ and let $\Gamma=\Gamma_F$ be the Galois group of $F^s/F$, $W=W_F$ the Weil-group, and $I=I_F$ the inertia group. If $E/F$ is a finite separable extension, which we will assume to be contained in $F^s$, we will use the subscript $E$ to denote the analogous objects relative to $E$ instead of $F$. Moreover, we will denote the relative Galois group of $E/F$ by $\Gamma_{E/F}$ and the relative Weil group by $W_{E/F}$. We will write $F^u$ for the maximal unramified extension of $F$ within $F^s$ and by $\tx{Fr}$ the element of $\Gamma_{F^u/F}$ that induces the automorphism $x \mapsto x^q$ on the residue field $\ol{k_F}$.

Given a connected reductive group $G$ defined over $F$, we denote by $G_\tx{der}$ its derived subgroup, by $G_\tx{sc}$ and $G_\tx{ad}$ the simply connected cover and adjoint quotient of $G_\tx{der}$, and by $\mf{g}$ the Lie-algebra of $G$. For an element $g \in G$ we will write $\tx{Ad}(g)$ for the conjugation action of $g$ on $G$ as well as for the adjoint action of $g$ on $\mf{g}$. We will write $\mf{g}^*$ for the dual space of $\mf{g}$ and $\tx{Ad}^*(g)$ for the coadjoint action of $g$.

Given a maximal torus $S \subset G$, we will always assume that it is defined over $F$, unless explicitly stated otherwise. We will write $N(S,G)$ for the normalizer of $S$ in $G$ and $\Omega(S,G)=N(S,G)/S$ for the absolute Weyl group, a finite algebraic group defined over $F$. We write $R(S,G)$ for the corresponding set of roots. This set has an action of $\Gamma$ and for any $\alpha \in R(S,G)$ we will write $\Gamma_\alpha$ and $\Gamma_{\pm\alpha}$ for the stabilizers of the subsets $\{\alpha\}$ and $\{\alpha,-\alpha\}$ respectively, and $F_\alpha$ and $F_{\pm\alpha}$ for the corresponding fixed subfield of $F^s$. Then $F_\alpha/F_{\pm\alpha}$ is an extension of degree at most 2. Following \cite{LS87} we will call $\alpha$ symmetric if the degree of this extension is $2$, and asymmetric if the degree is $1$. Moreover, following \cite{AS09} we will call $\alpha$ ramified or unramified if the extension $F_\alpha/F_{\pm\alpha}$ is such. Note that $\alpha$ is symmetric and ramified if and only if it is inertially symmetric in the sense if \cite{KalEpi}. For each $\alpha \in R(S,G)$ we have the 1-dimensional root subspace $\mf{g}_\alpha \subset \mf{g}$, which is defined over $F_\alpha$.

We will write $\mc{B}(G,F)$ for the reduced Bruhat-Tits building of $G$ and $\mc{A}(T,F)$ for the apartment associated to any maximal torus of $G$ which is maximally split (this notation is slightly different than the one used by other authors, who prefer to write $\mc{A}(A_T,F)$, where $A_T$ is the maximal split subtorus of $T$). For any $x \in \mc{B}(G,F)$ the corresponding parahoric subgroup of $G(F)$ will be denoted by $G(F)_{x,0}$, and more generally the Moy-Prasad filtration subgroups \cite{MP94,MP96} will be denoted by $G(F)_{x,r}$ for any $r \in \R_{\geq 0}$. On the Lie-algebra we have the analogous filtration lattices $\mf{g}(F)_{x,r}$ for any $r \in \R$. It is sometimes convenient to use the shorthand notation $G(F)_{x,r:s}=G(F)_{x,r}/G(F)_{x,s}$ for $r<s$, as well as $G(F)_{x,r+}=\bigcup_{s>r}G(F)_{x,s}$.

In the special case $G=\tx{Res}_{E/F}\mb{G}_m$, the Moy-Prasad filtration can be described simply as $E^\times_0=O_E^\times$ and $E^\times_r = 1+\mf{p}_E^{\lceil er \rceil}$ for $r>0$, where $e$ is the ramification index of $E/F$. We shall also use the filtration of the additive group $E$ given by $E_0=O_E$ and $E_r=\mf{p}_E^{\lceil er \rceil}$ for $r>0$.

\section{Regular supercuspidal representations} \label{sec:regsc}

\subsection{Basics on $p$-adic tori} \label{sub:tori}

Let $S$ be a torus defined over $F$. The topological group $S(F)$ has a unique maximal bounded subgroup $S(F)_b$ (which is also the unique maximal compact subgroup, as $F$ is locally compact) and this group is equipped with a decreasing filtration $S(F)_r$ indexed by the non-negative real numbers, namely the Moy-Prasad filtration. When the splitting field of $S$ is wildly ramified over $F$ it is known that this filtration exhibits some pathologies, which are not present when for some tamely ramified extension $E/F$ the torus $S \times E$ becomes induced, see \cite[\S4]{Yu03}. In particular, the pathologies are not present when the splitting field of $S$ is tamely ramified over $F$. We will call such $S$ \emph{tame} for short.

We recall the definition of $S(F)_r$. For $r=0$ there are two ways to define the subgroup $S(F)_0$. The torus $S$ possesses an lft-Neron model $\mf{S}^\tx{lft}$ by \cite[\S10]{BLR90}. This is a smooth group scheme over $O_F$ satisfying a certain universal property. It is locally of finite type and the maximal subgroup-scheme of finite type is called the ft-Neron model $\mf{S}^\tx{ft}$. Both models share the same neutral connected component, called the connected Neron model $\mf{S}^\circ$. Then $S(F)_0=\mf{S}^\circ(O_F)$. One also has $S(F)_b=\mf{S}^\tx{ft}(O_F)$.

A second way to define $S(F)_0$ is via the Kottwitz homomorphism \cite[\S7]{Kot97}. This is a functorial surjective homomorphism $S(F) \to X_*(S)_I^\tx{Fr}$. The kernel of this homomorphism is $S(F)_0$, and the preimage of the torsion subgroup $[X_*(S)_I^\tx{Fr}]_\tx{tor}$ is $S(F)_b$. See the first note at the end of \cite{Rap05}.

From this description it is obvious that when $S$ is unramified the equality $S(F)_b=S(F)_0$ holds. In fact, it holds slightly more generally. We shall call $S$ \emph{inertially induced} if $S \times F^u$ is an induced torus. This is equivalent to demanding that $X_*(S)$, or equivalently $X^*(S)$, has a basis invariant under the action of $I_F$. In that case $X_*(S)_I$ is torsion-free and thus $S(F)_b=S(F)_0$.

For $r>0$, the definition of $S(F)_r$ is
\[ S(F)_r = \{s \in S(F)_0| \forall \chi \in X^*(S),\ \tx{ord}(\chi(s)-1) \geq r \}, \]
see \cite[\S3.2]{MP96} and \cite[\S4.2]{Yu03}. Denoting by $S(F)_{r+}$ the union of $S(F)_s$ over $s>r$ we see that $S(F)_{0+}$ is precisely the pro-$p$-Sylow subgroup of $S(F)_0$.

We shall call $S$ \emph{wildly induced} if $S \times F^\tx{tr}$ is an induced torus, where $F^\tx{tr}/F$ is the maximal tamely ramified extension. This condition was called ``Condition (T)'' in \cite[\S4.7.1]{Yu03}. It is implied by the stronger condition of $S$ being tame. It is also equivalent to demanding that $X_*(S)$, or equivalently $X^*(S)$, has a basis invariant under the action of $P_F$. In that case $X_*(S)_I=[X_*(S)_P]_{I/P}$ has no $p$-torsion. This implies that for $r>0$ we have the simpler description
\begin{equation} \label{eq:mpr} S(F)_r = \{s \in S(F)| \forall \chi \in X^*(S),\ \tx{ord}(\chi(s)-1) \geq r \}. \end{equation}
Indeed, the right hand side lies in $S(F)_b$ and is pro-$p$, therefore must lie in $S(F)_0$. For an argument that does not involve the Kottwitz homomorphism, see \cite[4.7.2]{Yu03}.

\begin{lem} \label{lem:mpex1} If $1 \rw A \rw B \rw C \rw 1$ is an exact sequence of tame tori and $r>0$, then
\[ 1 \rw A(F)_r \rw B(F)_r \rw C(F)_r \rw 1 \]
is also exact. If $A \to B$ is an isogeny of tame tori whose kernel has order prime to $p$ and $r>0$, then $A(F)_r \to B(F)_r$ is a bijection.
\end{lem}
\begin{proof} Let $E/F$ be a tame finite Galois extension splitting the tori. We have $A(E)_r = X_*(A) \otimes E^\times_r$ and $A(F)_r=A(E)_r^{\Gamma_{E/F}}$, according to \eqref{eq:mpr} . Applying the functor $X_*$ to the exact sequence of tori produces an exact sequence of finite rank free $\Z$-modules with $\Gamma_{E/F}$ action, which remains exact after $\otimes_\Z E^\times_r$, leading to
\[ 1 \rw A(E)_r \rw B(E)_r \rw C(E)_r \rw 1.\]
Taking $\Gamma_{E/F}$-invariants and applying \cite[Proposition 2.2]{Yu01} finishes the proof.

For the second point, the functor $X_*(-)\otimes_\Z E^\times_r$ turns the isogeny $A \to B$ into a bijection, which remains bijective after taking $\Gamma_{E/F}$-fixed points.
\end{proof}

\begin{lem} \label{lem:mpex2} If $1 \to A \to B \to C \to 1$ is an exact sequence of tori and $A$ is inertially induced, then for $r=0$ and $r=0+$
\[ 1 \rw A(F)_r \rw B(F)_r \rw C(F)_r \rw 1 \]
is also exact.
\end{lem}
\begin{proof}
The case of $r=0+$ follows from the case of $r=0$ because $(-)_{0+}$ is the pro-$p$-Sylow subgroup of $(-)_0$ as remarked above. For the case $r=0$ we apply again $X_*$ to the exact sequence of tori to obtain an exact sequence of finite-rank free $\Z$-modules with $\Gamma$-action. We claim that after taking inertial co-invariants the sequence remains exact. The only issue would be the injectivity of $X_*(A)_I \to X_*(B)_I$. We may of course replace $I$ by a suitable finite quotient through which it acts. The kernel of this map is the image of the connecting homomorphism $H_1(I,X_*(C)) \to X_*(A)_I$. But $H_1(I,X_*(C))$ is finite, while by assumption $X_*(A)_I$ is torsion-free, so this connecting homomorphism is zero. We thus obtain the commutative diagram with exact rows
\[ \xymatrix{
1\ar[r]&A(F^u)\ar[r]\ar[d]&B(F^u)\ar[r]\ar[d]&C(F^u)\ar[r]\ar[d]&1\\
0\ar[r]&X_*(A)_I\ar[r]&X_*(B)_I\ar[r]&X_*(C)_I\ar[r]&0
}\]
The exactness of the top row on the right follows from $H^1(I,A(F^u))=0$ due to Steinberg's theorem \cite[Theorem 1.9]{Ste65reg}. The vertical maps are surjective. The kernel-cokernel lemma implies that the sequence
\[ 1\to A(F^u)_0 \to B(F^u)_0 \to C(F^u)_0 \to 1 \]
is exact. It is well known that $H^1(\tx{Fr},A(F^u)_0)=0$, see e.g. \cite[Lemma 2.3.1]{DR09}. Taking Frobenius-invariants finishes the proof.
\end{proof}

Every torus $S$ defined over $F$ has a maximal unramified subtorus $S' \to S$, characterized by $X_*(S')=X_*(S)^{I_F}$, as well as a maximal unramified quotient $S \to S''$, characterized by $X^*(S'')=X^*(S)^{I_F}$. One has $X^*(S')=X^*(S)_{I_F,\tx{free}}$ and $X_*(S'')=X_*(S)_{I_F,\tx{free}}$, i.e. the torsion-free quotient of the inertial coinvariants of $X^*(S)$ or $X_*(S)$ respectively.

\begin{lem} \label{lem:sred} Let $S$ be a tame torus defined over $F$ and let $S' \subset S$ be the maximal unramified subtorus. The natural map
\[ S'(F)_0/S'(F)_{0+} \to S(F)_0/S(F)_{0+} \]
is an isomorphism.
\end{lem}
\begin{proof}
The injectivity of this map follows from Lemma \ref{lem:mpex1}. Surjectivity follows from Lemmas \ref{lem:mpex1} and \ref{lem:mpex2} and the fact that $[S/S'](F)_0=[S/S'](F)_{0+}$. The latter is proved in \cite[Theorem 1.3]{NX91} when $F$ has characteristic zero, or in Propositions 6 and 8 and Corollary to Theorem 4 of \cite{Popov03} in general.
\end{proof}

\begin{lem} \label{lem:h2van} Let $S$ be a tame torus. Then we have
\[ H^2(\Gamma_F/I_F, S(F^u)_b) = H^2(\Gamma_F/I_F,S(F^u)_0) = 0. \]
\end{lem}
\begin{proof}
We shall use \cite[Ch. XIII, \S1,Prop. 2]{SerLF}, according to which for any torsion $\Gamma_F/I_F$-module $A$ we have $H^2(\Gamma_F/I_F,A)=0$ . This applies in particular when $A$ is the set of $\ol{k_F}$-rational points of a commutative linear algebraic group defined over $k_F$. Kottwitz's homomorphism leads to the exact sequence of $\Gamma_F/I_F$-modules
\[ 1 \to S(F^u)_0 \to S(F^u)_b \to [X_*(S)_I]_\tx{tor} \to 0. \]
From $H^2(\Gamma_F/I_F,[X_*(S)_I]_\tx{tor})=0$ we see that $H^2(\Gamma_F/I_F,S(F^u)_b)=0$ would follow from $H^2(\Gamma_F/I_F,S(F^u)_0)=0$. To see the latter, we use the fact that $S(F^u)_0/S(F^u)_{0+}=\mf{S}^\circ(\ol{k_F})$ and hence $H^2(\Gamma_F/I_F,S(F^u)_0/S(F^u)_{0+})=0$. On the other hand, we have $H^2(\Gamma_F/I_F,S(F^u)_{0+})=\varinjlim H^2(\Gamma_{F'/F},S(F')_{0+})$, where the colimit runs over the finite unramified extensions of $F$. The group $S(F')_{0+}$ is equal to $\varprojlim_r S(F')_{0+}/S(F')_r$, where $S(F')_r$ is the $r$-th Moy-Prasad filtration subgroup. The steps of the filtration are discrete and the quotients are the $k_{F'}$-points of an abelian unipotent algebraic group $\ms{U}_r$ defined over $k_F$. From the inflation restriction sequence
\[ H^1(\Gamma_{k_{F'}},\ms{U}(\ol{k_F})) \to  H^2(\Gamma_{k_{F'}/k_F},\ms{U}(k_{F'})) \to H^2(\Gamma_{k_F},\ms{U}(\ol{k_F})) \]
and the vanishing of the two outer terms we see that the middle term vanishes. From \cite[Chap. XII, \S3, Lem 3]{SerLF} we see that $H^2(\Gamma_{F'/F},S(F')_{0+})$ vanishes, and hence that $H^2(\Gamma_F/I_F,S(F^u)_{0+})$ vanishes.
\end{proof}

Consider now a tame torus $S$ defined over $F$ and its complex dual torus $\hat S$. The local Langlands correspondence provides an isomorphism of abelian groups $H^1_\tx{cts}(W_F,\hat S) \to \tx{Hom}_\tx{cts}(S(F),\C^\times)$. This bijection is functorial in $S$ and is characterized uniquely by a short list of properties \cite{Yu09}. According to \cite[Theorem 7.10]{Yu09}, if $\varphi \in H^1(W_F,\hat S)$ corresponds to $\theta : S(F) \to \C^\times$, then for any $r>0$ the restriction $\theta|_{S(F)_r}$ is zero if and only if the restriction $\varphi|_{I^r}$ is zero. The following two lemmas extend this result to the restrictions of $\theta$ to $S(F)_0$ and $S(F)_b$.

\begin{lem} \label{lem:llcres1} The restriction $\theta|_{S(F)_0}$ is trivial if and only if $\varphi$ lies in the kernel of the restriction map $H^1(W_F,\hat S) \to H^1(I_F,\hat S)$, or, equivalently, belongs to the image of the inflation map $H^1(W_F/I_F,\hat S^{I_F}) \to H^1(W_F,\hat S)$.
\end{lem}
\begin{proof}
Assume first that $S$ is split. The claim reduces immediately to the case $S=\mb{G}_m$, where it follows from the fact that the Artin reciprocity map $W_F \to F^\times$ carries $I_F$ surjectively onto $O_F^\times$. Assume next that $S$ is unramified. Let $E/F$ be the splitting field of $S$ and $R=\tx{Res}_{E/F}(S \times E)$. The kernel of the norm map $R \to S$ is an unramified torus. Applying Lemma \ref{lem:mpex2} to the resulting sequence of unramified tori we obtain a surjection $S(E)_0 = R(F)_0 \to S(F)_0$. Thus $\theta|_{S(F)_0}$ is trivial if and only if $[\theta\circ N]|_{S(E)_0}$ is trivial. But $\theta \circ N$ is a character of the split torus $S(E)$ whose parameter is equal to $\varphi|_{W_E}$. According to the split case, $\theta \circ N$ is trivial on $S(E)_0$ if and only if $\varphi|_{W_E}$ has trivial restriction to $I_E$. But $I_E=I_F$ and the unramified case is complete.

Assume now that $S$ is tamely ramified. Let $S' \subset S$ be the maximal unramified subtorus. According to Lemma \ref{lem:sred}, $\theta|_{S(F)_0}$ is trivial if and only if $\theta|_{S(F)_{0+}}$ and $\theta|_{S'(F)_0}$ are trivial. The parameter of $\theta|_{S'(F)}$ is the image of $\varphi$ in $H^1(W_F,\hat S_{I_F})$. If $\varphi$ has trivial image in $H^1(I_F,\hat S)$, then it has trivial images in $H^1(I_F^{0+},\hat S)$ and $H^1(I_F,\hat S_{I_F})$, so we conclude that $\theta|_{S(F)_{0+}}$ and $\theta|_{S'(F)_0}$ are trivial. Conversely, if $\theta|_{S(F)_{0+}}$ is trivial, then the image of $\varphi$ in $H^1(I_F^{0+},\hat S)$ is trivial, so $\varphi$ is inflated from $H^1(W_F/I_F^{0+},\hat S)$ and its restriction to $I_F$ is inflated from $H^1(I_F/I_F^{0+},\hat S)$. The group $I_F/I_F^{0+}$ is pro-cyclic, let $x$ be a pro-generator and let $\bar x$ be the finite-order automorphism of $\hat S$ through which $x$ acts. We have $\hat S_{I_F}=\hat S/(1-\bar x)\hat S$. If $\theta|_{S'(F)_0}$ is also trivial, then the image of $\varphi|_{I_F}$ in $H^1(I_F/I_F^{0+},\hat S_{I_F})$ is zero and hence $\varphi|_{I_F}$ comes from an element of $H^1(I_F/I_F^{0+},(1-\bar x)\hat S)$. But we claim that this cohomology group is zero. Indeed, let $N_{\bar x} : \hat S \to \hat S$ be the norm map for the action of $\bar x$. Evaluating 1-cocycles at the pro-generator $x$ provides an isomorphism from $H^1(I_F/I_F^{0+},(1-\bar x)\hat S)$ to the quotient of $\tx{ker}(N_{\bar x}|_{(1-\bar x)\hat S})$ by $(1-\bar x)(1-\bar x)\hat S$. But $N_{\bar x}$ is zero on $(1-\bar x)\hat S$, so the numerator of this quotient is equal to $(1-\bar x)\hat S$. We claim that the denominator is also equal to that. This follows from the fact that the map $(1-\bar x) : (1-\bar x)\hat S \to (1-\bar x)\hat S$ is an isogeny. Indeed, its kernel consists of those elements of $(1-\bar x)\hat S$ that are fixed by $\bar x$ and is thus equal to the intersection $(1-\bar x)\hat S \cap \hat S^{\bar x}$. This intersection is contained in the kernel of the restriction of $N_{\bar x}$ to $\hat S^{\bar x}$. But that restriction is just the $\tx{ord}(\bar x)$-power map and its kernel is finite.
\end{proof}

Note that the abelian group $\hat S^{I_F}$ might be disconnected. In fact, its group of connected components satisfies
\[ X^*(\hat S^{I_F}/\hat S^{I_F,\circ}) = X_*(S)_{I_F,\tx{tor}} \]
which means that the disconnectedness of $\hat S^{I_F}$ mirrors exactly the disconnectedness of the ft-Neron model of $S$. This motivates the following.

\begin{lem} \label{lem:llcres2}
The restriction $\theta|_{S(F)_b}$ is trivial if and only if $\varphi$ belongs to the image of the inflation map $H^1(W_F/I_F,\hat S^{I_F,\circ}) \to H^1(W_F,\hat S)$.
\end{lem}
\begin{proof}
Let $S \to S''$	be the maximal unramified quotient of $S$ and let $S_1 \subset S$ be the kernel of this quotient. Thus $X_*(S_1)$ is the kernel of the norm map $X_*(S) \to X_*(S)$ for the action of inertia. Note that $X_*(S_1)^{I_F}=\{0\}$, which means that $S_1$ is inertially anisotropic and in particular $S_1(F)$ is compact.

We claim that $S(F)_b$ is the preimage of $S''(F)_0$. Indeed, the image of $S(F)_b$ in $S''(F)$ is compact and hence belongs to $S''(F)_b=S''(F)_0$. Thus the preimage $X \subset S(F)$ of $S''(F)_0$ contains $S(F)_b$ so it is enough to show that it is compact, which follows from the fact that
\[ 1 \to S_1(F) \to X \to S''(F)_0 \to 1 \]
is an exact sequence of Hausdorff topological groups the outer terms of which are compact. We have used here the open mapping theorem to conclude that $X \to S(F)_0$ is open and hence a quotient map.

We conclude that the natural map $S(F)/S(F)_b \to S''(F)/S''(F)_0$ is injective. Since its cokernel is finite, the characters of $S(F)$ that are trivial on $S(F)_b$ are precisely those obtained from characters of $S''(F)$ that are trivial on $S''(F)_0$ by composing them with the natural map $S(F) \to S''(F)$. But the dual of the natural map $S(F) \to S''(F)$ is the map $\hat S^{I_F,\circ} \to \hat S$ and the statement follows from Lemma \ref{lem:llcres1}.
\end{proof}

\subsection{Review of stable conjugacy of tori} \label{sub:reviewtrans}

In this subsection we review the standard notions and results concerning stable conjugacy and transfer of tori between inner forms, mainly in order to have a convient reference that does not impose conditions on the ground field. We work over an arbitrary ground field $F$ with a fixed separable closure $F^s$ and let $\Gamma=\tx{Gal}(F^s/F)$.

Let $G$ and $G'$ be connected reductive groups defined over $F$. Recall that an inner twist $\psi : G \to G'$ is an isomorphism $\psi : G \times F^s \to G' \times F^s$ such that $\psi^{-1}\sigma(\psi)$ is an inner automorphism of $G$ for all $\sigma \in \Gamma$. Let $T \subset G$ be a maximal torus. Recall that $T$ is said to transfer to $G'$ if there exists $g \in G(F^s)$ such that $\psi\circ\tx{Ad}(g)|_T : T \to G'$ is defined over $F$, i.e. invariant under $\Gamma$. The image $T' \subset G$ of $\psi\circ\tx{Ad}(g)|_T$ is a maximal torus of $G'$ and one says that $T$ and $T'$ are stably conjugate. In the special case where $G=G'$ and $\xi=\tx{id}$ this recovers the usual notion of stable conjugacy of maximal tori of $G$. Note that since every torus splits over $F^s$, for any two tori $T,T' \subset G$ defined over $F$ there exists $g \in G(F^s)$ such that $\tx{Ad}(g)T=T'$. However, usually the homomorphism $\tx{Ad}(g) : T \to T'$ will not be defined over $F$.

Fix a maximal torus $T \subset G$. Given any other maximal torus $T' \subset G$ choose $g \in G(F^s)$ such that $\tx{Ad}(g)T=T'$. Then $\sigma \mapsto g^{-1}\sigma(g)$ is an element of $Z^1(\Gamma_F,N(T,G))$ whose cohomology class $\tx{cls}(T')$ is independent of the choice of $g$. Two tori $T'$ and $T''$ are conjugate in $G(F)$ if and only if $\tx{cls}(T')=\tx{cls}(T'')$, and are stably conjugate if and only if $\tx{cls}(T')$ and $\tx{cls}(T'')$ have the same image in $H^1(\Gamma_F,\Omega(T,G))$.

This criterion can be extended across inner forms, at least when the groups in question are adjoint, which we now assume. Let $\psi_i : G \to G_i$ for $i=1,2$ be inner twists and let $T_i \subset G_i$ be maximal tori. Replace each $\psi_i$ by $\psi_i \circ g_i$ for $g_i \in G(F^s)$ to achieve that $\psi_i(T) = T_i$. Then the class $\tx{cls}(T_i)$ of $\psi_i^{-1}\sigma(\psi_i) \in Z^1(\Gamma,N(T,G))$ is independent of the choice of $g_i$. Note that the image of $\tx{cls}(T_i)$ in $H^1(\Gamma,G)$ is the class of $\psi_i$. The tori $T_1$ and $T_2$ are rationally conjugate if and only if $\tx{cls}(T_1)=\tx{cls}(T_2)$. This implies in particular that the classes of $\psi_1$ and $\psi_2$ in $H^1(F,G)$ are equal. Furthermore, $T_1$ and $T_2$ are stably conjugate if and only if the images of $\tx{cls}(T_i)$ in $H^1(\Gamma,\Omega(T,G))$ are equal.

Note that for the purpose of checking stable conjugacy of tori we can always replace $G$ by its adjoint group. This is not true for the purposes of checking rational conjugacy. The above discussion can be extended to rational conjugacy of not necessarily adjoint groups by replacing $H^1(\Gamma,-)$ with the cohomology sets $H^1(u \to W,Z \to G)$ or $H^1(P \to \mc{E},Z \to G)$ of \cite{KalRI} or \cite{KalGRI}, in the case of a local and global fields of characteristic zero.

The following result is well-known, e.g. \cite[\S10]{Kot86}, but we have not been able to find a reference that allows positive characteristic.

\begin{lem} \label{lem:elltrans} Assume that $F$ is local. If $T$ is elliptic then it transfers to $G'$.
\end{lem}
\begin{proof}
As above we may assume without loss of generality that $G$ is adjoint. One checks that $T$ transfers to $G'$ if and only if the class of $\psi$ in $H^1(F,G)$ lies in the image of $H^1(F,T)$. Let $G_\tx{sc}$ be the simply connected cover of $G$ and $Z \subset G_\tx{sc}$ its center. Let $T_\tx{sc}$ be the preimage of $T$ in $G_\tx{sc}$
The exact sequences of algebraic groups
\[ 1 \to Z \to G_\tx{sc} \to G \to 1 \qquad 1 \to Z \to T_\tx{sc} \to T \to 1 \]
give exact sequences of sheaves on $\tx{Spec}(F)$ for the fpqc topology. All groups above are smooth except possibly $Z$. For them, the first fpqc-cohomology group coincides with the first etale cohomology group by \cite[exp XXIV, Proposition 8.1]{SGA3}. Since $T_\tx{sc}$ is anisotropic Tate-Nakayama duality implies that $H^2(F,T_\tx{sc})=0$. On the other hand, Kneser's theorem \cite[\S4.7]{BT3} implies that all inner twists of $G_\tx{sc}$ have vanishing first cohomology. This leads to the commutative diagram of pointed sets
\[ \xymatrix{
H^1(F,T)\ar@{->>}[d]\ar[r]&H^1(F,G)\ar@{_(->}[d]\\
H^2_\tx{fpqc}(F,Z)\ar@{=}[r]&H^2_\tx{fpqc}(F,Z)
}\]
from which we conclude that the top map must be surjective.
\end{proof}

\begin{lem} \label{lem:qstrans} Assume that $F$ is local. If $G$ is quasi-split then any maximal torus $T' \subset G'$ transfers to $G$.
\end{lem}
\begin{proof}
When $F$ is perfect this follows easily from Steinberg's work \cite{Ste65reg} on rational elements in conjugacy classes. The following alternative argument, based on work of Langlands and Shelstad and Kneser's vanishing theorem, works for arbitrary local fields.

Fix an $F$-pinning $(T,B,\{X_\alpha\})$ of $G$. As above we conjugate $\psi$ so that $\psi(T)=T'$ and obtain $\tx{cls}(T') \in H^1(F,N(T,G))$, after assuming that $G$ is adjoint. Our job is to construct a maximal torus $T'' \subset G$ for which the image of $\tx{cls}(T'')$ in $H^1(F,\Omega(T,G))$ coincides with that of $T'$.

Recall \cite[\S2.2]{LS87} that $a$-data for $T'$ is a collection $a_\alpha \in F^{s,\times}$ for each $\alpha \in R(T',G')$ such that $a_{\sigma\alpha}=\sigma(a_\alpha)$ for all $\sigma \in \Gamma$ and $a_{-\alpha}=-a_\alpha$. Such $a$-data always exists: When $F$ has characteristic $2$ one can take $a_\alpha=1$; otherwise any quadratic extension of fields of odd characteristic is generated by an element whose trace is zero. Transporting this $a$-data to $T$ via $\psi$ the construction of \cite[\S2.2]{LS87} gives a lift $x(\sigma_T)n(\omega_T(\sigma))$ to $H^1(\Gamma,N(T_\tx{sc},G_\tx{sc}))$ of $\tx{cls}(T')$. Knesers't theorem $H^1(\Gamma,G_\tx{sc})=1$ implies that a 1-cocycle representing this lift is of the form $\sigma \mapsto g^{-1}\sigma(g)$ for some $g \in G_\tx{sc}(F^s)$. Then $T''=\tx{Ad}(g)T$ is the desired maximal torus of $G$.
\end{proof}

\subsection{Short remarks about parahoric subgroups}

Recall that Borovoi has defined in \cite{Brv98} the algebraic fundamental group $\pi_1(G)$ of $G$. The assignment $G \mapsto \pi_1(G)$ is a functor from the category of connected reductive groups defined over $F$ to the category of finitely generated abelian groups with $\Gamma$-action. Let $L$ denote the completion of the maximal unramified extension of $F$. In \cite[\S7]{Kot97} Kottwitz has constructed a surjective homomorphism $\kappa_G : G(L) \to \pi_1(G)_I$. It is a natural transformation from the identity functor to the functor $\pi_1(-)_I$. Note that in loc. cit. Kottwitz uses $X^*(Z(\hat G)^I)$ instead of $\pi_1(G)_I$. These two abelian groups are equal, and just as in \cite{RR96} we prefer to use $\pi_1(G)_I$ because it is obviously a functor. In \cite[Appendix]{PR08}, Haines and Rapoport prove that for any $x \in \mc{B}(G,F)$ one has
\[ G(F)_{x,0} = G(F)_x \cap \tx{ker}(\kappa_G). \]

\begin{cor} \label{cor:par} Let $f : H \to G$ be a homomorphism of connected reductive groups defined over $F$, $x \in \mc{B}(H,F)$ and $y \in \mc{B}(G,F)$. Then
\[ f(H(F)_{x,0}) \cap G(F)_x \subset G(F)_{x,0}. \]
In particular, if $T \subset G$ is a maximal torus, then
\[ T(F)_0 \cap G(F)_x \subset G(F)_{x,0}.\]
\end{cor}

\begin{lem} \label{lem:par} Let $S \subset G$ be an inertially induced central torus and let $\bar G=G/S$. For any $x \in \mc{B}(G,F)$ and $r=0$ or $r=0+$ we have the exact sequence
\[ 1 \to S(F)_r \to G(F)_{x,r} \to \bar G(F)_{x,r} \to 1. \]
\end{lem}
\begin{proof}
The argument for $r=0$ is essentially the same as for Lemma \ref{lem:mpex2}, where now one uses $\pi_1(-)_I$ instead of $X_*(-)_I$. For $r=0+$ it follows from Lemma \ref{lem:mpex2} applied to the sequence $1 \to S \to T \to \bar T \to 1$, where $T \subset G$ is a maximally split maximal torus whose apartment contains $x$, and \cite[Theorem 8.3]{Yu03}.
\end{proof}

\subsection{Regular supercuspidal representations of depth zero} \label{sub:rdz}

Let $G$ be a connected reductive group defined over $F$ and splitting over a tamely ramified extension of $F$. In this subsection we will define and classify regular depth-zero supercuspidal representations of $G$. As a preparation we will first need to study maximally unramified elliptic maximal tori of $G$ and extend to this setting some results of DeBacker.

\subsubsection{Maximally unramified elliptic maximal tori} \label{subsub:relur}

\begin{fct} \label{fct:relur} Let $S \subset G$ be a maximal torus and $S' \subset S$ be the maximal unramified subtorus. The following statements are equivalent.
\begin{enumerate}
	\item $S'$ is of maximal dimension among the unramified subtori of $G$.
	\item $S'$ is not properly contained in an unramified subtorus of $G$.
	\item $S$ is the centralizer of $S'$.
	\item $S \times F^u$ is a minimal Levi subgroup of $G \times F^u$.
	\item The action of $I_F$ on $R(S,G)$ preserves a set of positive roots.
\end{enumerate}
\end{fct}
\begin{proof}
This follows from the fact that $G \times F^u$ is quasi-split.
\end{proof}

A maximal torus $S \subset G$ will be called \emph{maximally unramified} if it satisfies the above equivalent conditions. Note that when $G$ splits over $F^u$, i.e. when it is an inner form of an unramified group, then $S$ is unramified. Therefore, this notion generalizes the notion of an unramified maximal torus to the case of ramified groups (in \cite[Definition 3.1.1]{Roe11}, such tori were called ``unramified'', but we prefer the term maximally unramified because it emphasizes that the splitting field of $S$ need not be an unramified extension of $F$).

In \cite{Deb06}, DeBacker classifies the unramified maximal tori of a given reductive $p$-adic group (that necessarily splits over $F^u$). We will extend some of his results to the setting of maximally unramified maximal tori. We will only deal with those maximally unramified tori that are also elliptic, although we believe that our arguments apply to non-elliptic tori as well.

Let $S \subset G$ be a maximally unramified elliptic maximal torus. Then $S$ splits over a tamely ramified extension, because we are assuming that $G$ does. Let $x \in \mc{B}^\tx{red}(G,F)$ be the point associated to $S$ by \cite{Pr01}. That is, $x$ is the unique Frobenius-fixed point in the apartment $\mc{A}^\tx{red}(S,F^u) \subset \mc{B}^\tx{red}(G,F^u)$. Recall here our convention that we write $\mc{A}^\tx{red}(S,F^u)$ for the apartment in $\mc{B}^\tx{red}(G,F^u)$ associated to the maximal split subtorus of $S \times F^u$. Since $S$ is maximally unramified, this maximally split subtorus is $S' \times F^u$, where $S' \subset S$ is the maximal unramified subtorus.

\begin{lem} \label{lem:vertex}
The point $x$ is a vertex of $\mc{B}^\tx{red}(G,F)$.
\end{lem}
\begin{proof}
We have $\mc{B}^\tx{red}(G,F)=\mc{B}^\tx{red}(G,F^u)^\tx{Fr}$ and the simplicial structure on this building is defined so that the facets of $\mc{B}^\tx{red}(G,F)$ are the non-empty intersections with $\mc{B}^\tx{red}(G,F)$ of the facets of $\mc{B}^\tx{red}(G,F^u)$, see \cite[5.1.28]{BT2}.

Let $X \subset \mc{B}^\tx{red}(G,F^u)$ be the unique facet of smallest dimension containing $x$. It belongs to $\mc{A}^\tx{red}(S,F^u)$ and is Frobenius-invariant. The set of its Frobenius-fixed points is then a facet of $\mc{B}^\tx{red}(G,F)$ that belongs to $\mc{A}^\tx{red}(S,F^u)$. This means that this set contains a single point, namely $x$, and it follows that $x$ is a facet of $\mc{B}^\tx{red}(G,F)$, i.e. a vertex.
\end{proof}

As shown in \cite[\S5]{BT2}, the vertex $x$ specifies  a smooth connected $O_F$-group scheme $\mf{G}_x^\circ$ with $\mf{G}_x^\circ(F)=G(F)$, $\mf{G}_x^\circ(O_F)=G(F)_{x,0}$. We shall write $\ms{G}^\circ_x$ for the reductive quotient of the special fiber of $\mf{G}_x^\circ$. Then $\ms{G}_x^\circ(k_F)=G(F)_{x,0:0+}$. We further have the $O_F$-group scheme $\mf{\hat G}_x$ with $\mf{\hat G}_x(F)=G(F)$ and $\mf{\hat G}_x(O_F)=\tx{Stab}(x,G(F)^1)$. We shall write $\ms{G}_x$ for the quotient of the special fiber of this group scheme by its maximal connected normal unipotent subgroup. Then $\ms{G}_x^\circ$ coincides with the neutral connected component of $\ms{G}_x$. In particular, $\ms{G}_x$ is a usually disconnected algebraic group over $k_F$ with reductive neutral connected component.

\begin{lem} \label{lem:rut1}
\begin{enumerate}
\item The special fiber of the (automatically connected) ft-Neron model of $S'$ embeds canonically as an elliptic maximal torus $\ms{S}'$ of the reductive group $\ms{G}_x^\circ$. Explicitly, $\ms{S}'(k_F) \subset \ms{G}_x^\circ(k_F)$ is the image in $G(F)_{x,0:0+}$ of $S(F) \cap G(F)_{x,0}$, or equivalently of $S'(F) \cap G(F)_{x,0}$.
\item Every elliptic maximal torus of $\ms{G}_x^\circ$ arises in this way.
\end{enumerate}
\end{lem}

\begin{proof}
For the first point, apply \cite[Proposition 4.6.4(i)]{BT2} to $K=F^u$ to see that the special fiber of $S'$ is a maximal torus of $\ms{G}_x^\circ$. From the explicit description of $\ms{S}'(\ol{k_F})$ as the projection of $S'(F^u) \cap G(F^u)_{x,0}$ onto $\ms{G}_x^\circ(\ol{k_F})$ it is clear that this maximal torus is invariant under Frobenius. To discuss its ellipticity we may assume that $G$ is adjoint. We have $X_*(\ms{S}')=X_*(S')=X_*(S)_{I,\tx{free}}$ and the latter is identified via the norm map for the action of inertia with a sublattice (of finite index) of $X_*(S)^I$. This identification is Frobenius-invariant and we see that $X_*(\ms{S}'')^\tx{Fr} \subset X_*(S)^\Gamma=\{0\}$.

For the second point, let $\ms{S}_1' \subset \ms{G}_x^\circ$ be an elliptic maximal torus. As a first step, we shall lift the canonical embedding $\ms{S}_1' \to \ms{G}_x^\circ$ to an embedding $\ms{S}_1' \to \mf{G}_x^\circ \times k_F$. For this, we consider the extension
\[ 1 \to R \to \mf{G}_x^\circ \times k_F \to \ms{G}_x^\circ \to 1 \]
of affine algebraic groups over $k_F$, where $R$ is the unipotent radical of $\mf{G}_x^\circ$ and apply \cite[16.7 Proposition]{Bor91} to obtain a maximal torus of $\mf{G}_x^\circ$ defined over $k_F$ and mapping to $\ms{S}_1'$. Since $R$ is unipotent, this map of tori is an isomorphism and we obtain the desired embedding $\ms{S}_1' \to \mf{G}_x^\circ \times k_F$. Next, we let $\mf{S}_1'$ be the torus over $O_F$ whose cocharacter module is the Frobenius module $X_*(\ms{S}_1')$. The functor $\ul{\tx{Hom}}(\mf{S}_1',\mf{G}_x^\circ)$ that assigns to each $O_F$-scheme $X$ the set of morphisms $\tx{Hom}_{X-\tx{grp}}(\mf{S}_1' \times X,\mf{G}_x^\circ \times X)$ of algebraic groups over $X$ is represented by a smooth $O_F$-scheme \cite[exp XI, Theorem 4.1]{SGA3}. The embedding $\ms{S}_1' \to \mf{G}_x^\circ$ is a $k_F$-point of $\ul{\tx{Hom}}(\mf{S}_1',\mf{G}_x^\circ)$, which by smoothness \cite[\S2.3, Proposition 5]{BLR90} lifts to an $O_F$-point, i.e. to a morphism $\mf{S}_1' \to \mf{G}_x^\circ$ of $O_F$-group schemes. According to \cite[exp IX, Corollaries 2.5 and 6.6]{SGA3}, this morphism is a closed embedding, and so is its generic fiber $S_1' \to G$, where $S_1' = \mf{S}_1' \times F$. The torus $S_1'$ splits over $F^u$ and its rank is equal to that of $\ms{S}_1'$, which in turn equals to that of $\ms{S}'$, which in turn is equal to that of $S'$. Thus, the image of $S_1'$ is a maximal split torus in $G \times F^u$. Let $S_1$ be the centralizer of $S_1'$ in $G$. It is defined over $F$, because $S_1'$ is, and is a minimal Levi of $G \times F^u$. Thus, $S_1$ is a maximally unramified maximal torus. The argument for the first part of the proof shows that $S_1$ is elliptic.
\end{proof}

\begin{lem} \label{lem:rut2}
Let $S_1,S_2 \subset G$ be two maximally unramified elliptic maximal tori. Assume that their points in $\mc{B}^\tx{red}(G,F)$ coincide, call them $x$. Assume furthermore that $S_1(F^u) \cap G(F^u)_{x,0}$ and $S_2(F^u) \cap G(F^u)_{x,0}$ have the same projection to $\ms{G}_x^\circ(\ol{k_F})$. Then $S_1$ and $S_2$ are $G(F)_{x,0+}$-conjugate.
\end{lem}
\begin{proof}
Since $G(F^u)_{x,0}$ acts transitively on the apartments in $\mc{B}^\tx{red}(G,F^u)$ containing $x$ \cite[Proposition 4.6.28]{BT2}, we may choose $g \in G(F^u)_{x,0}$ such that $\tx{Ad}(g)S_1'=S_2'$, where $S_i' \subset S_i$ is the maximal unramified subtorus. Since $S_1' \times F^u$ is a maximally split torus in the quasi-split group $G \times F^u$ contained in the maximal torus $S_1 \times F^u$ and the same is true for $S_2$, we have $\tx{Ad}(g)S_1=S_2$. By assumption, the image of $g$ in $\ms{G}_x^\circ(\ol{k_F})$ belongs to the normalizer of $\ms{S}_1'(\ol{k_F})=\ms{S}_2'(\ol{k_F})$. By \cite[Cor. 4.6.13]{BT2}, the restriction of the surjection $\mf{G}_x^\circ(O_{F^u}) \rw \ms{G}_x^\circ(\ol{k_F})$ to $N(S_1,G)(F^u)_{x,0}$ is a surjection onto the Weyl group $\Omega(\ms{S}_1',\ms{G}_x^\circ)$. Thus we may modify $g$ so that its image in $\ms{G}_x^\circ(\ol{k_F})$ belongs to $\ms{S}_1'(\ol{k_F})$. Using the smoothness of the connected Neron model $\mf{S}_1^\circ$ of $S_1$, we can lift this image to an element of $\mf{S}_1^\circ(O_{F^u}) = S_1(F)_{x,0} \subset G(F)_{x,0}$ and use it to modify $g$ so that the image of $g$ in $\ms{G}_x^\circ(\ol{k_F})$ is now trivial. Thus, $g \in G(F^u)_{x,0+}$.

Now consider $g^{-1}\sigma(g)$, where $\sigma$ is the Frobenius. This element belongs to the intersection $N(S_1,G)(F^u) \cap G(F^u)_{x,0+} = S_1(F^u) \cap G(F^u)_{x,0+} = S_1(F^u)_{0+}$.
According to \cite[Prop 2.2]{Yu01}, $H^1(F^u/F,S_1(F^u)_{0+})=0$, thus we can modify $g$ so that it becomes an element of $G(F)_{x,0+}$.
\end{proof}

\begin{lem} \label{lem:weyl} Let $S \subset G$ be a maximally unramified elliptic maximal torus with associated point $x \in \mc{B}^\tx{red}(G,F)$. The natural maps
\[ N(S,G(F)_{x,0})/S(F)_0 \to N(S,G)(F)/S(F) \]
and
\[ N(S,G(F)_{x,0})/S(F)_0 \to N(\ms{S}',\ms{G}_x^\circ)(k_F)/\ms{S}'(k_F) \]
are bijective.
\end{lem}
\begin{proof}
Consider the first map. Its injectivity follows from the equality $S(F) \cap G(F)_{x,0} = S(F)_0$, which can be seen in different ways, one being the Haines-Rapoport characterization of parahoric subgroups \cite[Appendix]{PR08}. To see surjectivity, let $n \in N(S,G)(F)$ and let $w$ be its image in $[N(S,G)/S](F)$. Write $S_\tx{sc}$ for the preimage of $S$ in $G_\tx{sc}$. The natural map $N(S_\tx{sc},G_\tx{sc})/S_\tx{sc} \to N(S,G)/S$ is an isomorphism of finite algebraic groups and allows us to view $w$ as an element of $[N(S_\tx{sc},G_\tx{sc})/S_\tx{sc}](F)$. Steinberg's theorem \cite[Theorem 1.9]{Ste65reg} asserts that $H^1(I_F,S_\tx{sc}(F^u))=\{0\}$ and lets us lift $w$ to an element $\dot n \in N(S_\tx{sc},G_\tx{sc})(F^u)$. The action of this element on $G(F^u)$ by conjugation preserves $S(F^u)$ and acts on it in the same way as $n$, thus preserving the $F$-structure. It follows that the action of $\dot n$ on $\mc{B}^\tx{red}(G,F^u)$ preserves the apartment $\mc{A}^\tx{red}(S,F^u)$ and induces an action on it that commutes with the action of $\tx{Gal}(F^u/F)$. We conclude that $\dot n$ fixes the point $x$ and hence belongs to $G_\tx{sc}(F^u)_{x,0}$. Let $n' \in G(F^u)_{x,0}$ be the image of $\dot n$. We have
\[ n'^{-1}\tx{Fr}(n') = (n^{-1}n')^{-1}\tx{Fr}(n^{-1}n')\]
and the left hand side belongs to $G(F^u)_{x,0}$, while the right hand side belongs to $S(F^u)$. We conclude that both sides belong to $S(F^u)_0$. From $H^1(\tx{Fr},S(F^u)_0)=\{0\}$ we find $s \in S(F^u)_0$ so that $n's \in G(F^u)_{x,0}^\tx{Fr}=G(F)_{x,0}$. By construction $n's$ is a lift of $w$ and this completes the proof of surjectivity.

Consider now the second map. If we replace $F$ by $F^u$ then its bijectivity is the content of \cite[4.6.12]{BT2}. Moreover, this bijective map is Frobenius-equivariant. The claim now follows from the fact that both $S(F^u)_0$ and $\ms{S}'(\ol{k_F})$ have trivial $H^1(\tx{Fr},-)$.
\end{proof}

We shall call a vertex $x \in \mc{B}^\tx{red}(G,F)$
 \emph{superspecial} if it is a (necessarily special) vertex that is special in $\mc{B}^\tx{red}(G,E)$, where $E$ is any finite Galois extension of $F$ splitting $G$. If $G$ is quasi-split then superspecial vertices exist: The Chevalley valuation \cite[\S4.2.1]{BT2} associated to any $F$-pinning of $G$ is such a vertex.

\begin{lem} \label{lem:weyl2} Let $S \subset G$ be maximally unramified elliptic and assume that its associated point $x \in \mc{B}^\tx{red}(G,F)$ is superspecial. Then the natural inclusion
\[ N(S,G)(F)/S(F) \to \Omega(S,G)(F) \]
is an isomorphism.
\end{lem}
\begin{proof}
By Steinberg's theorem \cite[Theorem 1.9]{Ste65reg} we have $\Omega(S,G)(F^u)=N(S,G)(F^u)/S(F^u)$. The point $x$ remains special over $F^u$ and according to \cite[6.2.19]{BT1} the natural map $N(S,G(F^u)_{x,0})/S(F^u)_0 \to N(S,G)(F^u)/S(F^u)$ is an isomorphism (note that $S$ is not elliptic over $F^u$ and Lemma \ref{lem:weyl} doesn't apply in this setting). Applying again $H^1(\tx{Fr},S(F^u)_0)=\{0\}$ we obtain the lemma.
\end{proof}

\begin{lem} \label{lem:super} Let $G$ be quasi-split, $x \in \mc{B}^\tx{red}(G,F)$ superspecial, and $S \subset G$ a maximally unramified elliptic maximal torus. Then $x$ is the point associated to some maximal torus that is stably conjugate to $S$.
\end{lem}
\begin{proof}
Fix a minimal Levi subgroup $T \subset G$ whose apartment contains $x$.

Since both $S$ and $T$ become minimal Levi subgroups over $F^u$ we see that the class $\tx{cls}(S)$ defined in Subsection \ref{sub:reviewtrans} is inflated from $H^1(\tx{Fr},N(T,G)(F^u))$. Projecting to the Weyl group we obtain an element of $H^1(\tx{Fr},\Omega(T,G)(F^u))$. We have $\Omega(T,G)(F^u)=N(T,G)(F^u)/T(F^u)$ by Steinberg's theorem \cite[Theorem 1.9]{Ste65reg}. Since the point $x$ remains special over $F^u$ the natural map $N(T,G(F^u)_{x,0})/T(F^u)_0 \to N(T,G)(F^u)/T(F^u)$ is an isomorphism according to \cite[6.2.19]{BT1}. Via this isomorphism we obtain from $\tx{cls}(S)$ and element of $H^1(\tx{Fr},N(T,G(F^u)_{x,0})/T(F^u)_0)$ which, according to Lemma \ref{lem:h2van} lifts to an element of $H^1(\tx{Fr},N(T,G(F^u)_{x,0}))$. This element can be represented by $\sigma \mapsto g^{-1}\sigma(g)$ for some $g \in G(F^u)_{x,0}$, because $H^1(\tx{Fr},G(F^u)_{x,0})$ is trivial by \cite[Lemma 2.3.1]{DR09}. By construction the images in $\Omega(T,F)(F^u)$ of $g^{-1}\sigma(g)$ and $\tx{cls}(S)$ coincide. We conclude that $\tx{Ad}(g)T$ is stably conjugate to $S$. The point $x=gx$ belongs to $\mc{A}^\tx{red}(\tx{Ad}(g)T,F^u)$ and being Frbenius-fixed it must be the unique Frobenius-fixed element of $\mc{A}^\tx{red}(\tx{Ad}(g)T,F^u)$ and hence the point associated to $\tx{Ad}(g)T$.
\end{proof}

\subsubsection{Depth zero characters}

Let $S \subset G$ be a maximally unramified elliptic maximal torus and let $\theta : S(F) \to \C^\times$ be a character of depth zero. The restriction $\theta|_{S(F)_0}$ is thus a character $\bar\theta$ of $S(F)_0/S(F)_{0+}$, which by Lemma \ref{lem:sred} is the same as $S'(F)_0/S'(F)_{0+}$. This is the group of $k_F$-rational points of the maximal torus $\ms{S}'$ of the reductive group $\ms{G}_x^\circ$, where $x \in \mc{B}^\tx{red}(G,F)$ is the point associated to $S$. In \cite[Definition 5.15]{DL76}, Deligne and Lusztig define two regularity conditions for a character $\bar\theta : \ms{S}'(k_F) \to \C^\times$, which we shall now recall. For this, let us temporarily forget how $\ms{S}'$ and $\ms{G}_x^\circ$ came about and treat them as an arbitrary rational maximal torus $\ms{S}'$ in an arbitrary connected reductive group $\ms{G}_x^\circ$ defined over a finite field $k_F$. They say that $\bar\theta$ is in general position, if its stabilizer in $\Omega(\ms{S}',\ms{G}_x^\circ)(k_F)$ is trivial. They say that $\bar\theta$ is non-singular, if the pull-back of $\bar\theta$ to $X_*(\ms{S}')$ is non-trivial on each coroot $\alpha^\vee \in R^\vee(\ms{S}',\ms{G}_x^\circ) \subset X_*(\ms{S}')$. We need to recall the map $X_*(\ms{S}') \to \ms{S}'(k_F)$, which is given by \eqref{eq:dl1} below. We choose an embedding $\ol{k_F}^\times \to \Q/\Z$. Its image is $(\Q/\Z)_{p'}$, the subgroup of elements whose order is prime to $p$, and this embedding allows us to identify $\ms{S}'(\ol{k_F})$ with $X_*(\ms{S}') \otimes (\Q/\Z)_{p'}$. We have the commutative diagram
\[ \xymatrix{
0\ar[r]&X_*(\ms{S}')\ar[r]\ar[d]^{F-1}&X_*(\ms{S}')\otimes\Q\ar[r]\ar[d]^{F-1}&X_*(\ms{S}')\otimes \Q/\Z\ar[r]\ar[d]^{F-1}&0\\
0\ar[r]&X_*(\ms{S}')\ar[r]&X_*(\ms{S}')\otimes\Q\ar[r]&X_*(\ms{S}')\otimes \Q/\Z\ar[r]&0
}\]
with exact rows, where $F$ is the endomorphism of $X_*(\ms{S}')$ obtained functorially from the Frobenius endomorphism of $\ms{S}'$ (bear in mind that $F$ is not of finite order, but rather $F^n=q^n$ for some natural number $n$). Applying the kernel-cokernel lemma to this diagram and noting the middle vertical arrow is an isomorphism we obtain the exact sequence
\begin{equation} \label{eq:dl1} 0 \to X_*(\ms{S}') \to X_*(\ms{S}') \to \ms{S}'(k_F) \to 1. \end{equation}
We will now reinterpret the notion of non-singular in a way that does not involve the choice of an isomorphism $\ol{k_F}^\times \to (\Q/\Z)_{p'}$ and is closer to the $p$-adic torus $S$.

\begin{fct} \label{fct:dl1} Let $k'$ be a finite extension of $k_F$. The exact sequence \eqref{eq:dl1} fits into the commutative diagram
\[ \xymatrix{
0\ar[r]&X_*(\ms{S}')\ar[r]^{^{F^n-1}}&X_*(\ms{S}')\ar[r]\ar[d]^{\tx{id}}&\ms{S}'(k')\ar[r]\ar[d]^N&1\\
0\ar[r]&X_*(\ms{S}')\ar[r]^{F-1}&X_*(\ms{S}')\ar[r]&\ms{S}'(k_F)\ar[r]&1\\
}\]
where $n=[k':k_F]$ and $N : \ms{S}'(k') \to \ms{S}'(k_F)$ is the norm map.
\end{fct}
\begin{proof}
This is a direct computation.
\end{proof}

\begin{lem} \label{lem:dl2} Let $k'$ be a finite extension of $k_F$ splitting $\ms{S}'$. A character $\bar\theta : \ms{S}'(k_F) \to \C^\times$ is non-singular if and only if for each $\alpha^\vee \in R^\vee(\ms{S}',\ms{G}_x^\circ)$ the character $\bar\theta\circ N \circ \alpha^\vee : k'^\times \to \C^\times$ is non-trivial.
\end{lem}
\begin{proof}
According to Fact \ref{fct:dl1} we may reduce the proof to the case where $\ms{S}'$ is split. In that case the Frobenius endomorphism $F$ of $X_*(\ms{S}')$ is simply given by multiplication by $q$ and the map $X_*(\ms{S}') \to \ms{S}'(k_F)$ sends $\lambda \in X_*(\ms{S}')$ to $\lambda(\zeta)$, where $\zeta \in k_F^\times$ is the generator whose image under the chosen isomorphism $\ol{k_F}^\times \to (\Q/\Z)_{p'}$ is $1/(q-1) \in (\Q/\Z)_{p'}$. By definition, $\bar\theta$ is non-singular if and only if $\bar\theta(\alpha^\vee(\zeta)) \neq 1$ for all $\alpha^\vee \in R^\vee(\ms{S}',\ms{G}_x^\circ)$. Since $\zeta$ is a generator of $k_F^\times$ this is equivalent to requiring that the character $\bar\theta\circ\alpha^\vee$ be non-trivial.
\end{proof}

We will now define a third regularity condition on $\bar\theta$. We say that $\bar\theta$ is absolutely regular, if for some (hence any) finite extension $k'$ of $k_F$ splitting $\ms{S}'$ the character $\bar\theta\circ N$ has trivial stabilizer in $\Omega(\ms{S}',\ms{G}_x^\circ)$. It is clear that absolutely regular implies general position. By \cite[Corollary 5.18]{DL76} general position implies non-singular.

\begin{lem} If the center of $\ms{G}_x^\circ$ is connected, then the notions of non-singular, general position, and absolutely regular, are equivalent.
\end{lem}
\begin{proof}
According to \cite[Proposition 5.16]{DL76}, the notions of non-singular and general position are equivalent. By Fact \ref{fct:dl1} $\bar\theta$ is non-singular if and only if $\bar\theta \circ N$ is non-singular. But for $\bar\theta \circ N$ the notions of general position and absolute regularity coincide.
\end{proof}

We now return to the maximally unramified elliptic maximal torus $S$. We shall call the depth-zero character $\theta : S(F) \to \C^\times$ \emph{regular} if the stabilizer of $\theta|_{S(F)_0}$ in $N(S(F),G(F))/S(F)$ is trivial. We will call $\theta$ \emph{extra regular} if the stabilizer of $\theta|_{S(F)_0}$ in $\Omega(S,G)(F)$ is trivial.

\begin{fct} \label{fct:rdzc} $\theta$ is regular if and only if $\bar\theta$ is in general position. If the point of $S$ is superspecial, then $\theta$ is regular if and only if it is extra regular.
\end{fct}
\begin{proof} This follows from Lemmas \ref{lem:weyl} and \ref{lem:weyl2}.
\end{proof}

\subsubsection{Definition and construction}

We now come to the definition and construction of regular depth-zero supercuspidal representations. Let $\pi$ be an irreducible supercuspidal representation of $G(F)$ of depth zero. According to \cite[Proposition 6.8]{MP96} there exists a vertex $x \in \mc{B}^\tx{red}(G,F)$ such that the restriction $\pi|_{G(F)_{x,0}}$ contains the inflation to $G(F)_{x,0}$ of an irreducible cuspidal representation $\kappa$ of $G(F)_{x,0:0+}$. We shall call $\pi$ \emph{regular} if $\kappa$ is a Deligne-Lusztig cuspidal representation $\pm R_{\ms{S}',\bar\theta}$ associated to an elliptic maximal torus $\ms{S}'$ of $\ms{G}_x^\circ$ and a character $\bar\theta : \ms{S}'(k_F) \rw \C^\times$ in general position.

Regular depth-zero supercuspidal representations of $G(F)$ are constructed as follows. Let $S$ be a maximally unramified elliptic maximal torus of $G$ and let $\theta : S(F) \rw \C^\times$ be a regular depth-zero character. The restriction $\theta|_{S(F)_0}$ factors through a character $\bar\theta$ of $S(F)_{0:0+}$ that is in general position according to Fact \ref{fct:rdzc}. Let $x \in \mc{B}^\tx{red}(G,F)$ be the vertex (by Lemma \ref{lem:vertex}) associated to $S$. Let $\kappa_{(S,\theta)} = \pm R_{\ms{S}',\bar\theta}$ be the irreducible cuspidal representation arising from the Deligne-Lusztig construction applied to the reductive quotient $\ms{S}'$ of the special fiber of the connected Neron model of $S$ and the character $\bar\theta$. Identify $\kappa$ with its inflation to $G(F)_{x,0}$.

Note that $S(F)$ normalizes $G(F)_{x,0}$.
It is easy to check, and we shall do so soon, that the normalizer in $G(F)_x$ of $\kappa_{(S,\theta)}$ is equal to $S(F) \cdot G(F)_{x,0}$. This means that, in order to obtain an irreducible representation of $G(F)$, we need to extend $\kappa_{(S,\theta)}$ to $S(F) \cdot G(F)_{x,0}$ before inducing it. In \cite[\S4.4]{DR09} this is done using the fact that $S(F)=Z(F) \cdot S(F)_0$ (see e.g. \cite[Lemma 7.1.1]{KalECI}), which implies $S(F) \cdot G(F)_{x,0} = Z(F) \cdot G(F)_{x,0}$, and it is clear how to extend $\kappa_{(S,\theta)}$ to $Z(F)$. The same is also true for ramified unitary groups \cite[Theorem 3.4.1, Proposition 3.4.2, Proposition 5.2.3]{Roe11}. However, for general tamely ramified groups this is not true and a counterexample was shown to us by Cheng-Chiang Tsai, which we have included at the end of this subsubsection.

The extension of $\kappa_{(S,\theta)}$ to $S(F) \cdot G(F)_{x,0}$ must thus be obtained differently. We shall do this in the next subsubsection, along with the study of its character. For now we just assume that an extension $\tilde\kappa_{(S,\theta)}$ of $\kappa_{(S,\theta)}$ to $S(F) \cdot G(F)_{x,0}$ is given.

\begin{lem} \label{lem:rdzconst} The representation $\pi_{(S,\theta)}=\tx{c-Ind}_{S(F)G(F)_{x,0}}^{G(F)} \tilde\kappa_{(S,\theta)}$ is irreducible (and hence supercuspidal).
\end{lem}
\begin{proof}
The proof is the same as for \cite[Lemma 4.5.1]{DR09}.
By \cite[Proposition 6.6]{MP96} it is enough to show that $\tilde\kappa_{(S,\theta)}$ induces irreducibly to the normalizer of $G(F)_{x,0}$ in $G(F)$. Note here that the Levi subgroup $\mc{M}$ in loc. cit. is equal to $G$ in our case since $x$ is a vertex. The normalizer of $G(F)_{x,0}$ is equal to $G(F)_x$, the stabilizer of the vertex $x$ for the action of $G(F)$ on $\mc{B}^\tx{red}(G,F)$. It is enough to show that the normalizer of $\tilde\kappa_{(S,\theta)}$ in $G(F)_x$ is equal to $S(F)\cdot G(F)_{x,0}$. For this, let $h \in G(F)_x$ normalize $\tilde\kappa_{(S,\theta)}$. Then it in particular normalizes $\kappa_{(S,\theta)}$, so by \cite[Theorem 6.8]{DL76} there is $g \in G(F)_{x,0}$ so that $\tx{Ad}(gh)(\ms{S}',\bar\theta)=(\ms{S}',\bar\theta)$. By Lemma \ref{lem:rut2} there is $l \in G(F)_{x,0+}$ so that $\tx{Ad}(lgh)(S,\theta|_{S(F)_0})=(S,\theta|_{S(F)_0})$. Thus $lgh \in N(S,G)(F)$ and then the regularity of $\theta$ implies that $lgh \in S(F)$.
\end{proof}

As we shall see in Lemma \ref{lem:rdzclass}, every regular depth-zero supercuspidal representation is of the form $\pi_{(S,\theta)}$. We will call $\pi_{(S,\theta)}$ \emph{extra regular} if $\theta$ is extra regular. By Fact \ref{fct:rdzc} a regular $\pi_{(S,\theta)}$ is automatically extra regular if the vertex of $S$ is superspecial.

\textbf{Example:} We come to Cheng-Chiang's example of the failure of $Z(F)\cdot S(F)_0 = S(F)$. Consider $F$ of odd residual characteristic and not containing a fourth root of unity, for example $F=\Q_p$ with $4|(p-3)$, and a ramified quadratic extension $E/F$. Let $\bar G = \tx{Res}_{E/F} \tx{PGL}_4$ and $\bar T = \tx{Res}_{E/F} \bar T_0$, where $\bar T_0 = \mb{G}_m^4/\mb{G}_m$ is the standard maximal torus of $\tx{PGL}_4$. We have $X_*(\bar T)=\Z^4/\Z \oplus \Z^4/\Z$. Let $L$ be the kernel of the addition map $\Z^4/\Z \oplus \Z^4/\Z \to \Z/4\Z$. Let $T \to \bar T$ be the isogeny of tori specified by $X_*(T)=L$ and let $G \to \bar G$ be the corresponding isogeny of connected reductive groups. Then $G$ is semi-simple, quasi-split, and its center is $\mu_4$. A direct computation shows $X_*(T)_I=\Z^3 \oplus \Z/2\Z$. Let $S$ be an anisotropic maximal torus of $G$ that is conjugate to $T$ over $F^u$ (such an $S$ does exist). Then $X_*(S)_I=\Z^3 \oplus \Z/2\Z$ with an action of Frobenius. This action must be trivial on the $\Z/2\Z$ factor, which is the torsion subgroup of $X_*(S)_I$. Since $S$ is anisotropic, we conclude that $[X_*(S)_I]^\tx{Fr}=\Z/2\Z$. Since this is the quotient $S(F)/S(F)_0$, in order for $Z(F)\cdot S(F)_0 = S(F)$ to hold the composed map $Z(F) \to S(F) \to [X_*(S)_I]^\tx{Fr}$ must be surjective. However, $Z(F^u)=\mu_4(F^u)$ has order $4$, while $Z(F)=\mu_4(F)$ has order 2, by our assumption on $F$. We see that whatever the map $Z(F^u) \to S(F^u) \to X_*(S)_I$ might be, its restriction to $Z(F) \to S(F) \to [X_*(S)_I]^\tx{Fr}$ is the zero map.

\subsubsection{An extension and its character} \label{subsub:ext}

We shall now construct the representation $\tilde\kappa_{(S,\theta)}$ of $S(F) \cdot G(F)_{x,0}$ that extends $\kappa_{(S,\theta)}$. For this we must first recall the construction of $\kappa_{(S,\theta)}$. Let $\ms{U} \subset \ms{G}_x^\circ$ be the unipotent radical of a Borel subgroup defined over $\ol{k_F}$ and containing the maximal torus $\ms{S}'$. Let
\[ X = \{g \in \ms{G}_x^\circ(\ol{k_F})| g^{-1}\tx{Fr}(g) \in \ms{U} \}\]
be the corresponding Deligne-Lusztig variety, where $\tx{Fr}$ stands for the Frobenius automorphism of $\ms{G}_x^\circ$. By construction, $\ms{G}_x^\circ(k_F)$ acts on this variety by multiplication from the left, and $\ms{S}'(k_F)$ acts by multiplication on the right. The $l$-adic cohomology with compact support $H^i_c(X,\ol{\Q_l})$ is thus a $(\ms{G}_x^\circ(k_F),\ms{S}'(k_F))$-bimodule. It is shown in \cite[Corollary 9.9]{DL76} that if $X$ is assume affine and $\bar\theta : \ms{S}'(k_F) \to \ol{\Q_l}^\times$ is in general position, then the $\bar\theta$-isotypic subspace $H^i_c(X,\ol{\Q_l})_{\bar\theta}$ is non-zero for exactly one value of $i$, namely $i=l(w)$, where $w$ is the Weyl element determined by the maximal torus $\ms{S}'$. In fact, according to \cite[Remark 9.15.1]{DL76} the affineness assumption on $X$ can be relaxed to the assumption that some $X(w')$ is affine, where $w'$ is an element of the Weyl group that is Frobenius-conjugate to $w$. The latter assumption has been proved to always hold \cite[Theorem 1.3]{He08}. We let $V=H^{l(w)}_c(X,\ol{\Q_l})$ and
\[ V_{\bar\theta} = \{v \in V| vt=\bar\theta(t)v,\quad\forall t \in \ms{S}'(k_F)\}. \]
The $\ms{G}_x^\circ(k_F)$-module $V_{\bar\theta}$, inflated to $G(F)_{x,0}$, is the representation $\kappa_{(S,\theta)}$.

The extension of $\kappa_{(S,\theta)}$ to $S(F) \cdot G(F)_{x,0}$ begins with the observation that besides a left $\ms{G}_x^\circ(k_F)$-action and a right $\ms{S}'(k_F)$-action, the variety $X$ also carries an $\ms{S}'_\tx{ad}(k_F)$-action by conjugation, where $\ms{S}'_\tx{ad}$ is the image of $\ms{S}'$ in the adjoint group of $\ms{G}_x^\circ$. This action -- a special case of \cite[1.21]{DL76} -- is simply the restriction to $\ms{S}'_\tx{ad}(k_F)$ of the action of $\ms{S}'_\tx{ad}$ on $\ms{G}_x^\circ$ by conjugation, and is readily seen to preserve $X$. We shall write $\tx{Ad}(\bar s)$ for the action of $\bar s \in \ms{S}'_\tx{ad}(k_F)$ by conjugation on $\ms{G}_x^\circ$ as well as on the subvariety $X$ of $\ms{G}_x^\circ$. The three actions on $X$ have the following compatibility relation:
\[ \tx{Ad}(\bar s)[g\!\cdot\!x\!\cdot\!s'] = [\tx{Ad}(\bar s)g]\!\cdot\![\tx{Ad}(\bar s)x]\!\cdot\!s', \bar s\!\in\!\ms{S}'_\tx{ad}(k_F),g\!\in\!\ms{G}_x^\circ(k_F), s'\!\in\!\ms{S}'(k_F), x\!\in\!X.\]
In particular, the action of $\ms{S}'_\tx{ad}(k_F)$ by conjugation commutes with the action of $\ms{S}'(k_F)$ on the right. Furthermore, the action of $s' \in \ms{S}'(k_F)$ on the right can be recovered as $vs' = \tx{Ad}(\bar s'^{-1})s'v$, where $\bar s' \in \ms{S}'_\tx{ad}(k_F)$ is the image of $s'$.

There is a natural map $S(F) \to \ms{S}'_\tx{ad}(k_F)$ given as follows. To avoid confusion, let $H$ be the adjoint group of $G$. Then the natural map $\ms{G}_x^\circ \to [\ms{G}_x^\circ]_\tx{ad}$ factors as the composition $\ms{G}_x^\circ \to \ms{H}_x^\circ \to [\ms{G}_x^\circ]_\tx{ad}$. In particular, we have the map $\ms{S}_H' \to \ms{S}'_\tx{ad}$, where $S_H$ is the image of $S$ in $H$. Since $S$ is maximally unramified, $X_*(S_H)$ is an induced $I$-module and the norm map for the $I$-action provides an embedding $X_*(S_H)_I \to X_*(S_H)^I$. Since $S$ is elliptic, we see that $[X_*(S_H)_I]^\tx{Fr} \subset X_*(S_H)^\Gamma=\{0\}$. It follows that $S_H(F)_0 = S_H(F)$. The map $S(F) \to \ms{S}'_\tx{ad}(k_F)$ is then obtained as the composition $S(F) \to S_H(F)=S_H(F)_0 \to S_H(F)_{0:0+}=\ms{S}'_H(k_F) \to \ms{S}'_\tx{ad}(k_F)$.

Via the map $S(F) \to \ms{S}'_\tx{ad}(k_F)$ we can let $S(F)$ act on $\ms{G}_x^\circ$ and $X$ by conjugation, and the action on $X$ gives an action on $V$, which preserves the subspace $V_{\bar\theta}$. This action factors through $S(F)/S(F)_{0+}$. The subgroup $\ms{S}'(k_F)$ of $\ms{S}(k_F) = S(F)/S(F)_{0+}$ embeds via $s \mapsto (s,s^{-1})$ into the center of $\ms{G}_x^\circ(k_F) \rtimes \ms{S}(k_F)$ and we let $\tilde{\ms{G}}(k_F)$ be the quotient. We extend the representation of $\ms{G}_x^\circ(k_F)$ on $V_{\bar\theta}$ to $\tilde{\ms{G}}(k_F)$ via the formula
\begin{equation} \label{eq:tkd} (g,s)\cdot v = \theta(s)g[\tx{Ad}(s)v]. \end{equation}
It is immediate to check that this formula gives an action of $\ms{G}_x^\circ(k_F) \rtimes \ms{S}(k_F)$ which descends to $\tilde{\ms{G}}(k_F)$. The maps $G(F)_{x,0} \to \ms{G}_x^\circ(k_F)$ and $S(F) \to \ms{S}(k_F)$ splice together to a map $G(F)_{x,0} \cdot S(F) \to \tilde{\ms{G}}(k_F)$ via which we obtain an action of $G(F)_{x,0} \cdot S(F)$ on $V_{\bar\theta}$. This is the extension  $\tilde\kappa_{(S,\theta)}$ of $\kappa_{(S,\theta)}$ that we were seeking.

We will now compute the character of the representation $\tilde\kappa_{(S,\theta)}$ of the group $G(F)_{x,0}S(F)$. The resulting formula will be used in the classification of regular depth-zero supercuspidal representations in the next subsubsection. We begin with a discussion of a variant of the topological Jordan decomposition. When $S$ is unramified, the equation $G(F)_{x,0}S(F)=G(F)_{x,0}Z(F)$ allows one to use the usual topological Jordan decomposition of elements of $G(F)_{x,0}$ for the computation of the character of $\tilde\kappa_{(S,\theta)}$. The failure of this equation when $S$ is ramified precludes this, because elements of $G(F)_{x,0}S(F)$ are in general not compact and hence do not have a usual topological Jordan decomposition. However, if we let $A_G$ be the maximal split central torus of $G$, and $\bar G=G/A_G$, then the image of $G(F)_{x,0}S(F)$ in $\bar G(F)$ is contained in $\bar G(F)_x$, whose elements are compact and have a topological Jordan decomposition. We recall from \cite[Definition 2.15]{Spice08} that if $T \subset G$ is a maximal torus, an element $\gamma \in T(F)$ is called \emph{absolutely semi-simple modulo $A_G$} if $\chi(\gamma) \in F^{s,\times}$ is a topologically semi-simple element (i.e. a Teichm\"uller representative) for each $\chi \in X^*(T/A_G)$. Recall further that $\gamma \in T(F)$ is called topologically unipotent if it lies in the pro-$p$-Sylow subgroup of $T(F)_b$.

\begin{lem} Let $\gamma \in G(F)_{x,0} \cdot S(F)$ be a strongly regular semi-simple element and let $T$ be its centralizer. Then $\gamma=\gamma_s\gamma_u$, where $\gamma_s \in T(F) \cap G(F)_x$ is absolutely semi-simple modulo $A_G$ and $\gamma_u \in T(F) \cap G(F)_{x,0}$ is topologically unipotent. Both $\gamma_s$ and $\gamma_u$ are unique up to multiplication by elements of $A_G(F)_{0+}$. The image of the decomposition $\gamma=\gamma_s\gamma_u$ in $\mathsf{\bar G}_x(k_F)$ is the usual Jordan decomposition in the (possibly disconnected) finite group of Lie type $\mathsf{\bar G}_x$.
\end{lem}
\begin{proof} Let $\bar \gamma \in \bar G(F)$ be the image of $\gamma$ and let $\bar T=T/A_G$. Since $\bar G$ has anisotropic center, the group $\bar G(F)_x$ is compact, so $\bar\gamma \in \bar T(F) \cap \bar G(F)_x \subset \bar T(F)_b$. The latter is profinite and decomposes as the product of its pro-$p$-Sylow subgroup and a unique (automatically finite) complement, see for example \cite[Theorem 2.3.15]{RZ10}. Write correspondingly $\bar\gamma = \bar\gamma_s\bar\gamma_u$, where $\bar\gamma_u$ belongs to the pro-$p$-Sylow subgroup and $\bar\gamma_s$ has finite order prime to $p$. Thus $\bar\gamma=\bar\gamma_s \cdot \bar\gamma_u$ is a topological Jordan decomposition  and according to \cite[Lemma 2.21]{Spice08} we have $\bar\gamma_s,\bar\gamma_u \in \bar G(F)_x$. Moreover, the orders of these elements being prime-to-$p$ and pro-$p$ we see that their images in $\mathsf{\bar G}_x$ form the usual Jordan decomposition of the image of $\bar\gamma$ there.

We will now show that $\bar\gamma_u$ lifts to a topologically unipotent element $\gamma_u \in T(F) \cap G(F)_{x,0}$. If $T$ is tame this is immediate, because then $\bar\gamma_u \in \bar T(F)_{0+}$ and we can apply Lemma \ref{lem:mpex1} to lift it to $T(F)_{0+}$, which lies in $T(F) \cap G(F)_{x,0}$ by Corollary \ref{cor:par}.
If $T$ is not tame we need a more roundabout argument, relying on the tameness of $S$. Namely, map the decomposition $\bar\gamma=\bar\gamma_s\bar\gamma_u$ to $\bar G(F)_x/\bar G(F)_{x,0}$. In this abelian group, the order of $\bar\gamma_u$ is again equal to the $p$-part of the order of $\bar\gamma$. But $\bar\gamma$ lies in $\bar G(F)_{x,0} \bar S(F)$ which modulo $\bar G(F)_{x,0}$ equals $\bar S(F)/\bar S(F)_0 \cong X_*(\bar S)_I^\tx{Fr}$. Since $\bar S$ is tame, this group has no $p$-torsion. We conclude that the order of $\bar\gamma$ in the quotient $\bar G(F)_x/\bar G(F)_{x,0}$ has no $p$-part, so the image of $\bar\gamma_u$ in that quotient is trivial. In other words, $\bar\gamma_u \in \bar G(F)_{x,0}$.

Using the surjectivity of $G(F)_{x,0} \to \bar G(F)_{x,0}$ guaranteed by Lemma \ref{lem:par} we lift $\bar\gamma_u$ to an element of $G(F)_{x,0}$. This element automatically belongs to $T(F)$ and hence also to the intersection $T(F) \cap G(F)_{x,0} \subset T(F)_b$. It may not be topologically unipotent. But we can use again the product decomposition of $T(F)_b$ into its pro-$p$-Sylow part and prime-to-$p$ part to write this lift as a product $\gamma_u \cdot \delta$. The images of $\gamma_u$ and $\delta$ in $\bar T(F)$ have orders pro-$p$
 and prime-to-$p$ respectively, but their product is $\bar\gamma_u$ whose order is pro-$p$, so we conclude that the image of $\delta$ in $\bar T(F)$ is trivial. Thus $\delta \in T(F)_b \cap A_G(F)=A_G(F)_0$. Consequently $\gamma_u$ lifts $\bar\gamma_u$ and belongs to $G(F)_{x,0}$. This concludes the proof that $\bar\gamma_u$ lifts to a topologically unipotent element of $T(F) \cap G(F)_{x,0}$.

Set now $\gamma_s = \gamma \gamma_u^{-1}$. Then $\gamma_s$ is a lift of $\bar\gamma_s$, hence lies in $T(F) \cap G(F)_x$ and is absolutely semi-simple modulo $A_G$. It is clear that the ambiguity in the choice of $\gamma_u$ lies in the pro-$p$-Sylow subgroup of $A_G(F)_0$, i.e. $A_G(F)_{0+}$.
\end{proof}

\begin{pro} \label{pro:tkchar} The character of $\tilde\kappa_{(S,\theta)}$ at a strongly regular semi-simple element $\gamma \in G(F)_{x,0}\cdot S(F)$ is given by the formula
\[ (-1)^{r_G-r_S}|C(\gamma_s)^\circ(k_F)|^{-1}\sum_{\substack{k \in \ms{G}_x^\circ(k_F)\\ k^{-1} \gamma_s k \in \ms{S}(k_F)}} \theta(k^{-1}\gamma_sk) Q_{k\ms{S}'k^{-1}}^{C(\gamma_s)^\circ}(\gamma_u), \]
where $C(\gamma_s) \subset \ts{G}_x^\circ$ is the subgroup whose action on $G(F)_{x,0}S(F)/G(F)_{x,0+}$ fixes the image of $\gamma_s$.
\end{pro}

\begin{proof}
We first need to replace the decomposition $\gamma=\gamma_s\gamma_u$ by a different one, which appears more cumbersome but is better adapted to computations over the finite field. Write $\gamma=gr$, where $r \in S(F)$ and $g \in G(F)_{x,0}$. We map $r$ to $\bar r \in \ms{S}'_\tx{ad}(k_F)$ via the natural map $S(F) \to \ts{S}'_\tx{ad}(k_F)$ described above and then lift $\bar r$ arbitrarily to $\dot r \in \ms{S}'(\ol{k_F})$. We also map $g$ to $\ms{G}_x^\circ(k_F)$. Let $g\dot r = su$ be the Jordan decomposition of $g\dot r \in \ms{G}_x^\circ$. Let $z \in Z(\ms{G}_x^\circ)$ be such that $\tx{Fr}(\dot r)=z\dot r$. Then $\tx{Fr}(g\dot r)=zg\dot r$ and the uniqueness of the Jordan decomposition implies $u \in \ms{G}_x^\circ(k_F)$ and $\tx{Fr}(s)=zs$. We see that $s\dot r^{-1} \in \ts{G}_x^\circ(k_F)$, and moreover the centralizer $C(s)$ of $s$ in $\ms{G}_x^\circ$ has a $k_F$-structure. Since the action of $S(F^u)/S(F^u)_{0+}$ on $G(F^u)_{x,0}/G(F^u)_{x,0+}=\ts{G}_x^\circ(\ol{k_F})$ by conjugation factors through the natural map $S(F^u) \to \ts{S}'_\tx{ad}(\ol{k_F})$ we see that $\dot r^{-1}r$ acts trivially on $\ts{G}_x^\circ(\ol{k_F})$.

We claim that in the group $G(F)_{x,0}S(F)/G(F)_{x,0+}$ we have the identities $\gamma_s = s\dot r^{-1}r$ and $\gamma_u=u$. Indeed, recalling the notation $H=G_\tx{ad}$, we map the decomposition $\gamma=\gamma_s\gamma_u$ to $H(F)$. The element $\gamma$ then belongs to $H(F)_{x,0}$ and its image in $\ms{H}_x^\circ(k_F)$ has Jordan decomposition given by the images of $\gamma_s$ and $\gamma_u$. Since the images of $\gamma$ and $g\dot r$ in the adjoint group of $\ms{G}_x^\circ$ agree, the images of $\gamma_u$ and $u$ there also agree. Both of these elements being unipotent elements of $\ms{G}_x^\circ$ we conclude that they are equal in $\ms{G}_x^\circ$. Now $\gamma_s=s\dot r^{-1}r$ follows from $s\dot r^{-1}ru=su\dot r^{-1}r=gr=\gamma$, which uses the fact that $\dot r^{-1}r$ commutes with $\ts{G}_x^\circ$.

We will now compute the character of $\tilde\kappa_{(S,\theta)}$ at $\gamma=gr$ using the Jordan decomposition $g\dot r=su$ following the arguments in \cite[\S7]{Carter93}. The virtual $(\ms{G}_x^\circ(k_F),\ms{S}'(k_F))$-bimodule $\sum_i (-1)^i H^i_c(X,\ol{\Q_l})$ will be denoted by $W$ and its $\bar\theta$-isotypic component for the right action of $\ms{S}'(k_F)$ will be denoted by $W_{\bar\theta}$. Thus $W_{\bar\theta}$ is a virtual $\ms{G}_x^\circ(k_F)$-module. By the above mentioned vanishing result we have $W_{\bar\theta}=(-1)^{l(w)}V_{\bar\theta}$, so it will be enough to compute the character of the action of $S(F)\cdot G(F)_{x,0}$ on $W_{\bar\theta}$, noting that $(-1)^{l(w)}=(-1)^{r_{\ms{G}_x^\circ}-r_{\ms{S}'}}$, $r_S=r_{\ms{S}'}$,  and $r_G=r_{\ms{G}_x^\circ}$, the latter according to \cite[Corollary 5.1.11]{BT2}.

We now use all three actions we have on $W$, i.e. the fact that it is a $([\ms{G}_x^\circ(k_F) \rtimes \ms{S}(k_F)],\ms{S}'(k_F))$-bimodule. According to \eqref{eq:tkd}, the action of $gr$ on $W_{\bar\theta}$ is given by $\theta(r)$ times the action of $g \rtimes \bar r \in \ms{G}_x^\circ(k_F) \rtimes \ms{S}'_\tx{ad}(k_F)$ on $W_{\bar\theta}$. The element in the group algebra of $\ms{S}'(k_F)$ given by $e = |\ms{S}'(k_F)|^{-1}\sum_{t \in \ms{S}'(k_F)} \bar\theta(t^{-1})t$ projects $W$ to $W_{\bar\theta}$. Then the trace of $g \rtimes \bar r \in \ms{G}_x^\circ(k_F) \rtimes \ms{S}'_\tx{ad}(k_F)$ on $W_{\bar\theta}$ is equal to the trace of $(g \rtimes \bar r,e)$ on $W$. The element $(g \rtimes \bar r,e)$ of the group algebra of $[\ms{G}_x^\circ(k_F) \rtimes \ms{S}'_\tx{ad}(k_F)] \times \ms{S}'(k_F)$ has the same action on $W$ as the element
\[ e' = |\ms{S}'(k_F)|^{-1} \sum_{t \in \ms{S}'(k_F)} \bar\theta(t^{-1}) [gt \rtimes \bar t^{-1}\bar r] \]
of the group algebra of $\ms{G}_x^\circ(k_F) \rtimes \ms{S}'_\tx{ad}(k_F)$. The trace of $gr$ on $W_{\bar\theta}$ is thus equal to $\theta(r)$ times the trace of $e'$, i.e. to the expression
\[ |\ms{S}'(k_F)|^{-1} \sum_{t \in \ms{S}'(k_F)} \theta(r)\bar\theta(t^{-1}) \mc{L}(gt \rtimes \bar t^{-1}\bar r,X), \]
where $\mc{L}$ denotes the Lefschetz number. For the computation of the Lefschetz number, we use \cite[Property 7.1.10]{Carter93}, which involves the Jordan decomposition in the algebraic group $\ms{G}_x^\circ \rtimes \ms{S}'_\tx{ad}$. This decomposition is computed as follows: Given $g \rtimes \bar t$, lift $\bar t \in \ms{S}'_\tx{ad}(k_F)$ to $t \in \ms{S}'(\ol{k_F})$ and let $gt=su$ be the Jordan decomposition of $gt$ in the algebraic group $\ms{G}_x^\circ$. Then $[g \rtimes \bar t]=[st^{-1} \rtimes \bar t] \cdot [u \rtimes 1]$ is the Jordan decomposition of $g \rtimes \bar t$ in the algebraic group $\ms{G}_x^\circ \rtimes \ms{S}'_\tx{ad}$. Note that $s$ depends on the choice of lift $t$ of $\bar t$ and that $st^{-1}$ is independent of this choice and is a $k_F$-point, even though neither $s$ nor $t$ is a $k_F$-point in general.

Applying this to the element $gt \rtimes \bar t^{-1} \bar r$, we must decompose $g\dot r=su$ in $\ms{G}_x^\circ$ and then obtain $[gt \rtimes \bar t^{-1} \bar r] = [st\dot r^{-1} \rtimes \bar t^{-1}\bar r] \cdot [u \rtimes 1]$ as the Jordan decomposition in $\ms{G}_x^\circ \rtimes \ms{S}'_\tx{ad}$. The subvariety of $X$ fixed by the action of the semi-simple part $[st \dot r^{-1} \rtimes \bar t^{-1}\bar r]$ is $X^{s,t\dot r^{-1}} = \{ x \in X|x^{-1}sx = (t\dot r^{-1})^{-1} \}$. We are following here the notation of \cite[Proposition 7.2.5]{Carter93}, but need to keep in mind that $s$ and $\dot r$ are not Frobenius-fixed, but rather satisfy the relation $F(\dot r)=\dot rz$ and $F(s)=sz$, for some $z$ in the center of $\ms{G}_x^\circ$. Nonetheless, the conclusions of Propositions 7.2.6 and Propositions 7.2.7 remain valid with the same proofs, and the arguments in the proof of Theorem 7.2.8 carry over as well. We give a brief sketch.

The trace of $gr$ on $W_{\bar\theta}$ is now seen to equal
\[ |\ms{S}'(k_F)|^{-1} \sum_{t \in \ms{S}'(k_F)} \theta(r)\bar\theta(t^{-1}) \mc{L}(u,X^{s,t\dot r^{-1}}). \]
One checks that the centralizer $C(t\dot r^{-1})$ of $t\dot r^{-1}$ in $\ms{G}_x^\circ$ is defined over $k_F$, even though $t\dot r^{-1}$ is not. Let $[\ms{G}_x^\circ(k_F)]^{s,t\dot r^{-1}}$ denote the subset $\{g \in \ms{G}_x^\circ(k_F)|g^{-1}sg=(t\dot r^{-1})^{-1}\}$ and let $Y_{t\dot r^{-1}} = X \cap C(t\dot r^{-1})^\circ$. Just as in the proof of Proposition 7.2.6 we see that the morphism
\[ [\ms{G}_x^\circ(k_F)]^{s,t\dot r^{-1}} \times Y_{t\dot r^{-1}} \to X^{s,t\dot r^{-1}},\quad (g,y) \mapsto gy \]
is a surjective and its fibers are the orbits for the action of $C(t\dot r^{-1})^\circ(k_F)$ on $[\ms{G}_x^\circ(k_F)]^{s,t\dot r^{-1}} \times Y_{t\dot r^{-1}}$ given by $c(g,y)=(gc^{-1},cy)$. This implies that $X^{s,t\dot r^{-1}}$ is the disjoint union of closed subsets
\[ X^{s,t\dot r^{-1}} = \bigsqcup_{k \in [\ms{G}_x^\circ(k_F)]^{s,t\dot r^{-1}}/C(t\dot r^{-1})^\circ(k_F)} kY_{t\dot r^{-1}}, \]
each of which is invariant under left multiplication by $u$. Plugging this into the Lefschetz number we obtain the trace of $gr$ on $W_{\bar\theta}$ as
\[ |\ms{S}'(k_F)|^{-1} \sum_{t \in \ms{S}'(k_F)} |C(t\dot r^{-1})^\circ(k_F)|^{-1} \sum_{k \in [\ms{G}_x^\circ(k_F)]^{s,t\dot r^{-1}}} \theta(r)\bar\theta(t^{-1}) \mc{L}(u,kY_{t \dot r^{-1}}). \]
Now $\mc{L}(u,kY_{t\dot r^{-1}})=\mc{L}(k^{-1}uk,Y_{t \dot r^{-1}})=\mc{L}(u,Y_s)=|\ms{S}'(k_F)|Q_{k\ms{S}'k^{-1}}^{C(s)^\circ}(u)$. Combining the two sums and re-indexing we arrive at the formula
\[ (-1)^{r_G-r_S}|C(s)^\circ(k_F)|^{-1}\sum_{\substack{k \in \ms{G}_x^\circ(k_F)\\ k^{-1} s k \in \ms{S}'}} \bar\theta(k^{-1}sk\dot r^{-1})\theta(r) Q_{k\ms{S}'k^{-1}}^{C(s)^\circ}(u). \]
To convert this to the formula in the statement of the proposition, recall that $\dot r^{-1}r$ commutes with every element of $\ts{G}_x^\circ$. This implies that $C(\gamma_s)=C(s)$, and moreover $k^{-1}sk\dot r^{-1}r=k^{-1}(s\dot r^{-1}r)k^{-1}=k^{-1}\gamma_sk$, and $k^{-1}sk \in \ts{S}'$ is equivalent to $k^{-1}\gamma_sk \in \ts{S}$.
\end{proof}

\begin{cor} \label{cor:tktwist}
Let $\phi : G(F) \to \C^\times$ be a character of depth zero. Then $\tilde\kappa_{(S,\phi\cdot\theta)}=\phi|_{S(F)\cdot G(F)_{x,0}} \otimes \tilde\kappa_{(S,\theta)}$.
\end{cor}
\begin{proof}
It is enough to show that the characters of both sides are equal. This reduces to showing that $\phi(k^{-1}\gamma_sk)=1$, which is clear since both $k$ and $\gamma_s$ lift to elements of $G(F)$.
\end{proof}

\begin{cor} \label{cor:tkcharss}
The character of $\tilde\kappa_{(S,\theta)}$ at a strongly regular semi-simple element $\gamma \in S(F)$ that is absolutely semi-simple modulo $A_G$  is given by the formula
\[ (-1)^{r_G-r_S}\sum_{w \in N(S,G)(F)/S(F)} \theta(\gamma^w). \]
\end{cor}
\begin{proof}
We apply Proposition \ref{pro:tkchar} and use that $\gamma=\gamma_s$. The summation index $k$ then runs over $N(\ms{S},\ms{G}_x^\circ)(k_F)=N(\ms{S}',\ms{G}_x^\circ)(k_F)$. Taking into account the normalizing factor $|C(\gamma_s)^\circ(k_F)|^{-1}$ and the fact that $k \in \ms{S}'(k_F)$ commutes with $\gamma_s$, we see that the formula becomes
\[ (-1)^{r_G-r_S}\sum_{w \in N(\ms{S}',\ms{G}_x^\circ)(k_F)/\ms{S}'(k_F)} \theta(\gamma^w). \]
According Lemma \ref{lem:weyl} the indexing set of this sum is $N(S,G)(F)/S(F)$ and the proof is complete.
\end{proof}

\subsubsection{Classification} \label{subsub:dzclass}

\begin{lem} \label{lem:rdztwist}
Let $S \subset G$ be maximally unramified elliptic, $\theta : S(F) \to \C^\times$ a regular character, and $\phi : G(F) \to \C^\times$ a character. If the depth of $\phi\otimes\pi_{(S,\theta)}$ is zero, then the depth of $\phi$ is zero. If the depth of $\phi$ is zero, then $\phi\otimes\pi_{(S,\theta)}=\pi_{(S,\phi\cdot\theta)}$.
\end{lem}
\begin{proof}
The representation $\phi \otimes\pi_{(S,\theta)}$ is supercuspidal and if its depth is zero, then by \cite[Proposition 6.8]{MP96} it is of the form $\tx{c-Ind}_{G(F)_y}^{G(F)}\tau$ for an irreducible representation $\tau$ of $G(F)_y$, whose restriction to $G(F)_{y,0}$ contains a cuspidal representation $\sigma$. At the same time, writing $\dot\kappa_{(S,\theta)}=\tx{Ind}_{S(F)G(F)_{x,0}}^{G(F)_x}\tilde\kappa_{(S,\theta)}$ we see $\phi\otimes\pi_{(S,\theta)}=\tx{c-Ind}_{G(F)_x}^{G(F)}(\phi \otimes\dot\kappa_{(S,\theta)})$. Applying Kutzko's Mackey formula \cite{Ku77} we see that
\begin{eqnarray*}
\tx{End}_{G(F)}(\phi\otimes\pi_{(S,\theta)})&=&
\tx{Hom}_{G(F)}(\tx{c-Ind}_{G(F)_y}^{G(F)}\tau,\tx{c-Ind}_{G(F)_x}^{G(F)}(\phi \otimes\dot\kappa_{(S,\theta)}))\\
&=&\bigoplus_g \tx{Hom}_{G(F)_y \cap G(F)_{gx}}(\tau,{^g[\phi\otimes\dot\kappa_{(S,\theta)}]}),
\end{eqnarray*}
where $g$ runs over $G(F)_y\lmod G(F)/G(F)_x$. Since the left hand side is non-zero there must exist $g$ for which the corresponding summand on the right is non-zero. This summand is a subgroup of $\tx{Hom}_{G(F)_{y,0+} \cap G(F)_{gx,0+}}(\tau,{^g[\phi\otimes\dot\kappa_{(S,\theta)}]})$. Since both $\tau$ and $^g\dot\kappa_{(S,\theta)}$ are 1-isotypic upon restriction to $G(F)_{y,0+} \cap G(F)_{gx,0+}$ we see that $\phi$ must be trivial upon restriction to this group. By \cite[Lemma 2.45, Definition 2.46]{HM08} this implies that $\phi$ has depth zero.

The equality $\phi\otimes\pi_{(S,\theta)}=\pi_{(S,\phi\cdot\theta)}$ now follows from Corollary \ref{cor:tktwist} and the obvious equality $\phi\otimes\pi_{(S,\theta)}=\tx{c-Ind}_{S(F)G(F)_{x,0}}^{G(F)}(\phi\otimes\tilde\kappa_{(S,\theta)})$.
\end{proof}

\begin{lem} \label{lem:rdzclass} Every regular depth-zero supercuspidal representation of $G(F)$ is of the form $\pi_{(S,\theta)}$ for some maximally unramified elliptic maximal torus $S$ and regular depth-zero character $\theta : S(F) \to \C^\times$. Two representations $\pi_{(S_1,\theta_1)}$ and $\pi_{(S_2,\theta_2)}$ are isomorphic if and only if the pairs $(S_1,\theta_1)$ and $(S_2,\theta_2)$ are $G(F)$-conjugate.
\end{lem}
\begin{proof}
The first statement follows from \cite[Proposition 6.6]{MP96} and Lemmas \ref{lem:rut1} and \ref{lem:weyl}. We come to the second statement. The first half of the proof is the same as for \cite[Lemma 3.1.1]{KalIso}. It is clear that conjugate pairs lead to isomorphic representations, so we need to prove the opposite implication. Let $x_i \in \mc{B}^\tx{red}(G,F)$ be the point for $S_i$ and $\tilde \kappa_i$ and $\kappa_i$ the representations of $S_i(F)G(F)_{x_i,0}$ and $G(F)_{x_i,0}$ respectively. By \cite[Theorem 3.5]{MP96} the unrefined minimal $K$-types of depth zero $(G(F)_{x_1,0},\kappa_1)$ and $(G(F)_{x_2,0},\kappa_2)$ are associate, so there exists $g \in G(F)$ with $gx_1=x_2$ and $\tx{Ad}(g)\kappa_1=\kappa_2$. Conjugating $(S_1,\theta_1)$ by $g$ we may assume $g=1$, so $x_1=x_2=:x$ and $\kappa_1=\kappa_2$. By \cite[Theorem 6.8]{DL76} we can find $g \in G(F)_{x,0}$ such that $\tx{Ad}(g)(\ms{S}_1,\bar\theta_1)=(\ms{S}_2,\bar\theta_2)$. By Lemma \ref{lem:rut2} there is $l \in G(F)_{x,0+}$ such that $\tx{Ad}(lg)(S_1,\bar\theta_1)=(S_2,\bar\theta_2)$. We again conjugate $(S_1,\theta_1)$ by $lg$ and assume that $S_1=S_2$ and $\bar\theta_1=\bar\theta_2$. Let's write $S=S_1=S_2$ and $\bar\theta_1=\bar\theta_2=\bar\theta$.

In the unramified case the proof is now complete, because $S(F)=S(F)_0 \cdot Z(F)$, which implies $\theta_1=\theta_2$, because the central character of $\pi_{(S,\theta_i)}$ is $\theta_i|_{Z(F)}$. Since this fails in the ramified case, we need an additional argument. It follows the proof of \cite[Proposition 4.2]{Mor89}. The argument given there, combined with \cite[Proposition 5.2]{Mor89}, shows that for any $g \notin G(F)_x$ the group $\tx{Hom}_{G(F)_x \cap {^gG(F)_x}}(\dot\kappa_1,\dot\kappa_2)$ vanishes, where $\dot\kappa_i = \tx{Ind}_{S(F)G(F)_{x,0}}^{G(F)_x} \tilde\kappa_i$.
We conclude
\[ \tx{Hom}_{G(F)}(\pi_{(S,\theta_1)},\pi_{(S,\theta_2)}) = \tx{Hom}_{G(F)_x}(\dot\kappa_1,\dot\kappa_2), \]
using Kutzko's Mackey formula. We use again the (this time ordinary) Mackey formula to the subgroup $S(F) \cdot G(F)_{x,0}$ of finite index of $G(F)_x$.
Then we get
\[ \tx{Hom}_{G(F)_x}(\dot\kappa_1,\dot\kappa_2) = \bigoplus_{g \in G(F)_x/S(F)G(F)_{x,0}} \tx{Hom}_{S(F)G(F)_{x,0}}(\tilde\kappa_1,{^g\tilde\kappa_2}). \]
For any coset $g$ the corresponding summand on the right is a subgroup of $\tx{Hom}_{G(F)_{x,0}}(\kappa_1,{^g\kappa_2})$. But we already know $\kappa_1 \cong \kappa_2$, so by the same argument as above there exist $h \in G(F)_{x,0}$ and $l \in G(F)_{x,0+}$ such that $\tx{Ad}(g)(S,\bar\theta)=\tx{Ad}(lh)(S,\bar\theta)$. It follows as above that $g^{-1}lh \in S(F)$, which means that $g$ must represent the trivial coset in $G(F)_x/S(F)G(F)_{x,0}$. This implies
\[ \tx{Hom}_{G(F)}(\pi_{(S,\theta_1)},\pi_{(S,\theta_2)}) = \tx{Hom}_{S(F)G(F)_{x,0}}(\tilde\kappa_1,\tilde\kappa_2). \]
From \eqref{eq:tkd} we see that both $\tilde\kappa_1$ and $\tilde\kappa_2$ act on the same vector space $V_{\bar\theta}$ and $\tilde\kappa_2 = \theta_2\theta_1^{-1} \otimes \tilde\kappa_1$, where $\theta_2\theta_1^{-1}$ is a character of $S(F)/S(F)_0 = S(F)G(F)_{x,0}/G(F)_{x,0}$. Since $\kappa_1=\kappa_2$ is already an irreducible representation of $G(F)_{x,0}$, any element of $\tx{Hom}_{S(F)G(F)_{x,0}}(\tilde\kappa_1,\tilde\kappa_2)$ is a scalar multiple of the identity, which forces $\theta_2=\theta_1$.
\end{proof}

\subsection{Review of the work of Hakim and Murnaghan} \label{sub:reviewhm}

In accordance with \cite{HM08} we now assume that the residual characteristic of $F$ is not $2$. Let $((G^0 \subset G^1 \dots \subset G^d),\pi_{-1},(\phi_0,\phi_1,\dots,\phi_d))$ be a reduced generic cuspidal $G$-datum, in the sense of \cite[Definition 3.11]{HM08}. We recall that each $G^i$ is a tame twisted Levi subgroup of $G$, i.e. a connected reductive subgroup of $G$ that is defined over $F$ and becomes a Levi subgroup of $G$ over a tame extension of $F$, $G^d=G$, $\pi_{-1}$ is a depth-zero supercuspidal representation of $G^0(F)$, and $\phi_i : G^i(F) \to \C^\times$ is a smooth character of depth $r_i>0$, which is $G^{i+1}$-generic when $i \neq d$. We refer the reader to \cite[\S3.1]{HM08} for the notion of a generic character, as well as for the precise list of conditions imposed on this data.

From a reduced generic cuspidal $G$-datum, the construction of \cite{Yu01} produces an irreducible supercuspidal representation of $G(F)$. We can think of Yu's construction as a map from the set of reduced generic cuspidal $G$-data to the set of isomorphism classes of irreducible supercuspidal representations of $G(F)$. One of the main results of \cite{HM08} is the description of the fibers of this map. Hakim and Murnaghan introduce three operations on the set of reduced generic cuspidal $G$-data: elementary transformation, $G$-conjugation, and refactorization. According to \cite[Theorem 6.6]{HM08}, the equivalence relation generated by these operations, called $G$-equivalence in \cite{HM08}, places two reduced generic cuspidal $G$-data in the same equivalence class precisely when they lead to isomorphic supercuspidal representations. This theorem is valid under a certain technical hypothesis, called $C(\vec G)$.

Our goal in this subsection is to recall the notion of $G$-equivalence and Hypothesis $C(\vec G)$, and then show that \cite[Theorem 6.6]{HM08} is valid even without assuming $C(\vec G)$. For this, we will first prove that Hypothesis $C(\vec G)$ holds for all $G$ for which the fundamental group of $G_\tx{der}$ has order prime to $p$. In particular, it holds when $G_\tx{der}$ is simply connected. We will then use this to treat the case of general $G$.

We now recall Hypothesis $C(\vec G)$ from \cite[\S2.6]{HM08}. Given a tower $\vec G=(G^0 \subset G^1 \dots \subset G^d)$ of twisted Levi subgroups of $G$, Hypothesis $C(\vec G)$ is the concatenation of hypotheses $C(G^i)$. In turn, Hypothesis $C(G)$ stipulates that whenever $\phi : G(F) \to \C^\times$ is a character of positive depth $r>0$ and $x \in \mc{B}(G,F)$, the restriction $\phi|_{G(F)_{x,(r/2)+}}$ is realized by an element of $\tx{Lie}^*(Z(G)^\circ)(F)_{-r}$. This means the following. Let $\mf{g}=\tx{Lie}(G)$ and let $\Lambda : F \to \C^\times$ be an additive character of depth zero. For any $r > s$ the Pontryagin dual of the abelian group $\mf{g}(F)_{x,s+}/\mf{g}(F)_{x,r+}$ is identified with the abelian group $\mf{g}^*(F)_{x,-r}/\mf{g}^*(F)_{x,-s}$, via the pairing
\[ (Y,X^*) \mapsto \Lambda\<X^*,Y\>,\qquad Y \in \mf{g}(F)_{x,s+},  X^* \in \mf{g}^*(F)_{x,-r}. \]
Whenever $r>s \geq r/2$ we have the Moy-Prasad isomorphism (see \cite[Corollary 2.4]{Yu01} and the discussion following \cite[Definition 2.46]{HM08})
\[ \tx{MP}_x : \mf{g}(F)_{x,s+}/\mf{g}(F)_{x,r+} \to G(F)_{x,s+}/G(F)_{x,r+}, \]
via which $\mf{g}^*(F)_{x,-r}/\mf{g}^*(F)_{x,-s}$ is identified with the Pontryagin dual of the abelian group
$G(F)_{x,s+}/G(F)_{x,r+}$. An element $X^* \in \mf{g}^*(F)_{x,-r}/\mf{g}^*(F)_{x,-s}$ is said to realize the character of $G(F)_{x,s+}/G(F)_{x,r+}$ that it corresponds to under this identification. Now let $\mf{z}=\tx{Lie}(Z(G)^\circ)$. As discussed in \cite[\S8]{Yu01}, there is a natural way to view $\mf{z}^*$ as a subspace of $\mf{g}^*$. Namely, the natural projection $\mf{g}^* \to \mf{z}^*$ that is dual to the inclusion $\mf{z} \to \mf{g}$ becomes an isomorphism upon restriction to the subspace of $\mf{g}^*$ that is 1-isotypic for the adjoint action of $G$.

\begin{lem} \label{lem:charsc}
If the fundamental group of $G_\tx{der}$ has order prime to $p$, then every character of $G_\tx{der}(F)$ has depth at most zero.
\end{lem}
\begin{proof}
Assume first that $G_\tx{der}$ is simply connected and write $G_\tx{der}=G_\tx{sc}$ as the product $G_\tx{sc} = G_{\tx{sc},1} \times \dots \times G_{\tx{sc},n}$, with each $G_{\tx{sc},i}$ simple over $F$. Then $G_{\tx{sc},i}$ is either isotropic, or else by \cite[\S4.5,\S4.6]{BT3} isomorphic to $\tx{Res}_{E/F}\tx{SL}_1(D)$, where $E/F$ is a finite extension and $D/E$ is a division algebra. In the first case, $G_{\tx{sc},i}$ satisfies the Kneser-Tits conjecture \cite[\S1.2]{TitsWhitehead} and hence $G_{\tx{sc},i}(F)$ has no characters. In the second case $G_{\tx{sc},i}(F)$ is isomorphic to the group $D^{(1)}$ of elements of $D$ whose reduced norm is equal to $1$. According to \cite[\S5 Corollary]{Riehm70}, the derived subgroup of $D^{(1)}$ is equal to $(1+\mf{p}_D) \cap D^{(1)}$, where $\mf{p}_D$ is the maximal ideal of $D$. In terms of Moy-Prasad filtrations this means that the derived subgroup of $G_{\tx{sc},i}(F)$ is $G_{\tx{sc},i}(F)_{x,0+}$, where $x$ is the unique point in the reduced building of $G_{\tx{sc},i}(F)$. We conclude that every character of $G_\tx{sc}(F)$ is trivial on $G_\tx{sc}(F)_{x,0+}$ for any $x \in \mc{B}(G,F)$.

For the general case, let $x \in \mc{B}(G,F)$. Let $A \subset G$ be a maximal split torus such that $x$ belongs to the apartment of $A$. According to \cite[Corollaire 5.1.12]{BT2} there exists a maximal torus $T \subset G$ defined over $F$, containing $A$, and maximally split over $F^u$. Since $G$ is tame, so is $T$. Let $T_\tx{der}$ and $T_\tx{sc}$ be the corresponding maximal tori of $G_\tx{der}$ and $G_\tx{sc}$. Applying Lemma \ref{lem:mpex1} to the isogeny $T_\tx{sc} \to T_\tx{der}$ we see that the natural map $G_\tx{sc}(F)_{x,0+} \to G_\tx{der}(F)_{x,0+}$ is bijective.
\end{proof}

\begin{lem}\label{lem:hypcg}
Assume the fundamental group of $G_\tx{der}$ has order prime to $p$. Then Hypothesis $C(G)$ holds, and moreover  Hypothesis $C(\vec G)$ holds for any tower of twisted Levi subgroups of $G$.
\end{lem}
\begin{proof}
The fundamental group of the derived subgroup of any twisted Levi subgroup of $G$ is a subgroup of the fundamental group of the derived subgroup of $G$. It is therefore enough to establish Hypothesis $C(G)$. When $G$ is a torus the statement is clear, because then $\mf{g}^*=\mf{z}^*$. For the general case, let $D=G/G_\tx{der}$ and let $\phi : G(F) \to \C^\times$ be a character of depth $r>0$. The restriction of $\phi$ to $G_\tx{der}(F)_{x,(r/2)+}$ is trivial by Lemma \ref{lem:charsc}, hence its restriction to $G(F)_{x,(r/2)+}$ factors through a character of $D(F)_{(r/2)+}$ (the map $G(F)_{x,(r/2)+} \to D(F)_{(r/2)+}$ is surjective by Lemma \ref{lem:mpex1}). This character is represented by an element $X^* \in \tx{Lie}^*(D)(F)_{-r}$. Under the exact sequence of dual Lie algebras
\[ 1 \rw \tx{Lie}^*(D) \rw \tx{Lie}^*(G) \rw \tx{Lie}^*(G_\tx{der}) \rw 1 \]
the image of $\tx{Lie}^*(D)$ in $\tx{Lie}^*(G)$ is precisely the subspace that Yu identifies with $\tx{Lie}^*(Z(G)^\circ)$ in \cite[\S8]{Yu01}. Thus the image of $X^*$ in $\tx{Lie}^*(G)_{x,-r}$ realizes $\chi|_{G(F)_{x,(r/2)+}}$.
\end{proof}

At the moment we do not know if Hypothesis $C(G)$ holds without the assumption on the fundamental group of $G_\tx{der}$. However, we can still prove that \cite[Theorem 6.6]{HM08} is valid without this assumption, by reducing to the case where $G_\tx{der}$ is simply connected. The main tool that we exploit for that is $z$-extensions, introduced by Langlands and Kottwitz. Recall from \cite[\S1]{Kot82} that a $z$-extension of $G$ is an exact sequence $1 \to K \to \tilde G \to G \to 1$ of connected reductive groups defined over $F$, the derived subgroup of $\tilde G$ is simply connected, and $K$ is an induced torus. In particular, the map $\tilde G(F) \to G(F)$ is surjective. Such a $z$-extension always exists.

\begin{lem} \label{lem:mpex3} Let $x \in \mc{B}(G,F)$. For any $r \geq 0$ the sequence
\[ 1 \to K(F)_r \to \tilde G(F)_{x,r} \to G(F)_{x,r} \to 1 \]
is exact. For any $r>0$ the sequence
\[ 1 \to G_\tx{sc}(F)_{x,r} \to \tilde G(F)_{x,r} \to D(F)_r \to 1 \]
is exact, where $D=\tilde G/G_\tx{sc}$.
\end{lem}
\begin{proof}
Let $A \subset G$ be a maximal split torus such that $x$ belongs to the apartment of $A$. According to \cite[Corollaire 5.1.12]{BT2} there exists a maximal torus $T \subset G$ defined over $F$, containing $A$, and maximally split over $F^u$. Since $G$ is tame, so is $T$. Let $\tilde T$ be the preimage of $T$ in $\tilde G$ and let $T_\tx{sc} = \tilde T \cap G_\tx{sc}$. The claims now follow from Lemmas \ref{lem:mpex1} and \ref{lem:mpex2} applied to the sequences $1 \to K \to \tilde T \to T \to 1$ and $1 \to T_\tx{sc} \to \tilde T \to D \to 1$.
\end{proof}

Given such a $z$-extension, we can pull-back the reduced generic cuspidal $G$-datum to $\tilde G$ to obtain $((\tilde G^0 \subset \tilde G^1 \dots \subset \tilde G^d),\tilde\pi_{-1},(\tilde\phi_0,\tilde\phi_1,\dots,\tilde\phi_d))$. Here $\tilde G^i$ is the preimage of $G^i$ in $\tilde G$ and is a twisted Levi subgroup of $G$, $\tilde \pi_{-1}$ is the composition of $\pi_{-1}$ with the surjective homomorphism $\tilde G^0(F) \to G^0(F)$ and is an irreducible supercuspidal representation of depth-zero, and $\phi_i$ is the composition of $\phi_i$ with the surjection $\tilde G^i(F) \to G^i(F)$ and is a character of the same depth as $\phi_i$, generic when $i \neq d$. The result of this procedure is a reduced cuspidal generic datum for $\tilde G$. The irreducible supercuspidal representation of $\tilde G(F)$ associated to this datum by Yu's construction is the pull-back of the irreducible supercuspidal representation of $G(F)$ associated to the reduced cuspidal $G$-datum we started with.

We now recall the notion of $G$-equivalence of reduced generic cuspidal $G$-data. It is the equivalence relation generated by three operations: $G$-conjugation, elementary transformation, and refactorization. The operation of $G$-conjugation is obvious from its name -- one replaces each member of the $G$-datum by its conjugate under a given $g \in G(F)$. An elementary transformation consists of replacing $\pi_{-1}$ by an isomorphic representation. If we are already thinking of $\pi_{-1}$ as an isomorphism class of representations, then this operation is vacuous. Finally, a datum $((G'^0 \subset G'^1 \dots \subset G'^d),\pi'_{-1},(\phi'_0,\phi'_1,\dots,\phi'_d))$ is a refactorization of $((G^0 \subset G^1 \dots \subset G^d),\pi_{-1},(\phi_0,\phi_1,\dots,\phi_d))$ if the following conditions involving
\[ \chi_i : G^i(F) \to \C^\times,\qquad \chi_i(g) := \prod_{j=i}^d\phi_j(g)\phi_j'(g)^{-1}, \]
are satisfied:
\begin{enumerate}
	\item[F0.] If $\phi_d=1$ then $\phi'_d=1$;
	\item[F1.] $\chi_i$ is of depth at most $r_{i-1}$ for all $i = 0, \dots, d$, where $r_{-1}=0$;
	\item[F2.] $\pi_{-1}'=\pi_{-1}\otimes\chi_0$.
\end{enumerate}

Note that the three operations of $G$-conjugation, elementary transformation, and refactorization, commute.

\begin{lem} Let $1 \to K \to \tilde G \to G \to 1$ be a $z$-extension. Two reduced generic cuspidal $G$-data are $G$-equivalent if and only if their pull-backs to $\tilde G$ are $\tilde G$-equivalent.
\end{lem}
\begin{proof}
Let $((G^0 \subset G^1 \dots \subset G^d),\pi_{-1},(\phi_0,\phi_1,\dots,\phi_d))$ and $((G'^0 \subset G'^1 \dots \subset G'^d),\pi'_{-1},(\phi'_0,\phi'_1,\dots,\phi'_d))$ be the two reduced data for $G$. It is enough to check the three relations that generate $G$-equivalence: $G$-conjugacy, elementary transformation, and refactorization. For these, the statement follows immediately from the surjectivity of the maps $\tilde G^i(F) \to G^i(F)$ and $\tilde G^i(F)_{x,r} \to G^i(F)_{x,r}$ for any $x \in \mc{B}(G^i,F)$ and $r \geq 0$.
\end{proof}

\begin{cor} \label{cor:hm} Let $\Psi$ and $\dot\Psi$ be two reduced generic cuspidal $G$-data, and let $\pi(\Psi)$ and $\pi(\dot \Psi)$ be the corresponding irreducible supercuspidal representations of $G(F)$. Then $\pi(\Psi)$ and $\pi(\dot\Psi)$ are isomorphic if and only if $\Psi$ and $\dot\Psi$ are $G$-equivalent (without assuming Hypotheses $C(\vec G)$ and $C(\vec {\dot G})$).
\end{cor}
\begin{proof}
This follows immediately from the preceding discussion and \cite[Theorem 6.6]{HM08}.
\end{proof}

\subsection{Tame regular elliptic pairs} \label{sub:rs}

\begin{dfn} \label{dfn:normalized} We shall call a reduced generic cuspidal $G$-datum $((G^0 \subset G^1 \dots \subset G^d),\pi_{-1},(\phi_0,\phi_1,\dots,\phi_d))$ \emph{normalized} if the pull-back of $\phi_i$ to $ G^i_\tx{sc}(F)$ is trivial, for all $0 \leq i \leq d$.
\end{dfn}

\begin{lem} If $p$ does not divide the order of the fundamental group of $G_\tx{der}$, then every reduced generic cuspidal $G$-datum is $G$-equivalent to a normalized one. \end{lem}
\begin{proof}
Let $\phi : G(F) \to \C^\times$ be a character of depth $r>0$. Put $D=G/G_\tx{der}$. Let $x \in \mc{B}(G,F)$. By Lemma \ref{lem:charsc} $\phi|_{G_\tx{der}(F)_{x,0+}}$ is trivial, while Lemma \ref{lem:mpex1} implies that $G(F)_{x,0+} \to D(F)_{0+}$ is surjective. Thus $\phi$ induces a smooth character of $D(F)_{0+}$. It is of finite order and by Pontryagin duality extends to a character $\phi'$ of $D(F)$, which we may pull back to $G(F)$. Then $\phi\cdot(\phi')^{-1}$ is a character trivial on $G(F)_{x,0+}$.

If $(\phi_0,\dots,\phi_d)$ is the sequence of characters in a reduced generic cuspidal $G$-datum, we can use this procedure to replace $(\phi_{d-1},\phi_d)$ by $(\phi_{d-1}\phi_d(\phi_d')^{-1},\phi_d')$, thereby obtaining a refactorization for which $\phi_d'$ is trivial on $G_\tx{der}(F)$. Doing this inductively leads to the desired refactorization.
\end{proof}

\begin{dfn} \label{dfn:regdat} Let $((G^0 \subset G^1 \dots \subset G^d),\pi_{-1},(\phi_0,\phi_1,\dots,\phi_d))$ be a reduced generic cuspidal $G$-datum. We shall call it
\begin{enumerate}
\item \emph{regular}, if $\pi_{-1}$ is a regular depth-zero supercuspidal representation of $G^0(F)$ in the sense of Subsection \ref{sub:rdz};
\item \emph{extra regular}, if it is normalized and $\pi_{-1}$ is an extra regular depth-zero supercuspidal representation of $G^0(F)$ in the sense of Subsection \ref{sub:rdz}.
\end{enumerate}
\end{dfn}

The regularity of $\pi_{-1}$ is a non-trivial restriction. Already in the case of $\tx{SL}_2$, four of the supercuspidal representations of this group are not regular. Nonetheless, most depth-zero supercuspidal representations are regular.

According to Lemma \ref{lem:rdzclass}, we can replace $\pi_{-1}$ by a maximally unramified elliptic maximal torus $S$ of $G^0$ and a depth-zero character $\phi_{-1} : S(F) \rw \C^\times$ that is (extra) regular with respect to $G^0$. It will sometimes be convenient to write $S=G^{-1}$ and $r_{-1}=0$. Note that it may happen that $G^{-1}=G^0$.

Using Lemma \ref{lem:rdztwist} one sees that the regularity of a $G$-datum is an invariant of its $G$-equivalence class. One also sees that the extra regularity of a normalized $G$-datum is an invariant of its $G$-equivalence class, provided we only consider normalized $G$-data, for then we are only allowed to replace $\pi_{-1}$ by $\chi_0 \otimes \pi_{-1}$, where $\chi_0 : G^0(F) \to \C^\times$ is trivial on $G^0_\tx{sc}(F)$, but given $w \in \Omega(S,G^0)(F)$ and $s \in S(F)$ the element $wsw^{-1}s^{-1} \in S(F)$ lifts to $G^0_\tx{sc}(F)$. On the other hand, the $\pi_{-1}$ component of a $G$-datum that is not normalized, but equivalent to a normalized extra regular datum, need not be an extra regular depth-zero supercuspidal representation of $G^0(F)$. As an example we can take $G^0(F)$ to be the norm-1 elements in a division algebra $D/F$ of degree $d$ and $S(F)$ to be the norm-1 elements in the unramified extension $F_d$ of $F$ of degree $d$, embedded into $D$. Then $\Omega(S,G)(F)$ is isomorphic to $\tx{Gal}(F_d/F)$ and acts on $S(F)=F_d^1$ in the natural way. The Frobenius element of $\tx{Gal}(F_d/F)$ can be represented by an element $g \in D^\times$ normalizing $F_d^\times$ and whose valuation is $a/d$ for some $a \in \Z$ comprime to $d$. Thus $N(S,G)(F)=S(F)$ and we see that all characters of $S(F)$ are regular, even the trivial character. On the other hand, extra regular is a non-trivial condition. We have $S(F)_{0:0+}=k_{F_d}^1$, the elements of $k_{F_d}^\times$ whose $k_{F_d}/k_F$-norm is trivial. It is easy to see that extra regular characters exist. Furthermore, by \cite[Theorem 1.8, Proposition 1.8]{PR94} the inclusion $S(F) \to G^0(F)$ induces a bijection $S(F)_{0:0+} \to G^0(F)_{x,0:0+}$, where $x$ is the unique point in the building of $G^0$. Thus every extra regular character extends (even uniquely) to a character of $G^0(F)$.

\begin{dfn} \label{dfn:regrep}
Assume $p$ does not divide the order of the fundamental group of $G_\tx{der}$. We shall call a supercuspidal representation of $G(F)$ \emph{(extra) regular} if it arises via Yu's construction from an (extra) regular (reduced generic cuspidal) Yu-datum.
\end{dfn}
We will drop the condition on $p$ in Subsection \ref{sub:nsc}.

Our goal for the reminder of this Section will be to show that the (extra) regular supercuspidal representations can be described by data that is much simpler than (extra) regular Yu-data, namely tame elliptic (extra) regular pairs.

\begin{dfn} \label{dfn:tre} Let $S \subset G$ be a maximal torus and $\theta : S(F) \rw \C^\times$ a character. We shall call the pair $(S,\theta)$ tame elliptic regular (resp. extra regular) if
\begin{enumerate}
\item $S$ is elliptic and split over a tame extension;
\item the action of inertia on the root subsystem
\[ R_{0+} = \{\alpha \in R(S,G)| \theta(N_{E/F}(\alpha^\vee(E^\times_{0+})))=1 \} \]
preserves a positive set of roots, where $E/F$ is any tame Galois extension splitting $S$;
\item the character $\theta|_{S(F)_0}$ has trivial stabilizer for the action of $N(S,G^0)(F)/S(F)$ (resp. $\Omega(S,G^0)(F)$), where $G^0 \subset G$ is the reductive subgroup with maximal torus $S$ and root system $R_{0+}$.
\end{enumerate}
\end{dfn}

We recall from Fact \ref{fct:relur} that the second condition is equivalent to saying that $S$ is a maximally unramified maximal torus of $G^0$. Under a mild assumption on $p$, one can replace $G^0$ by $G$ in the third condition, due to the following lemma.

\begin{lem} \label{lem:wr}
Assume that $p$ is not a bad prime for $G$ and does not divide the order of the fundamental group of $G_\tx{der}$. If a pair $(S,\theta)$ satisfies conditions 1 and 2 of Definition \ref{dfn:tre} and furthermore $\theta|_{S(F)_0} \neq 1$, then any element of $\Omega(S,G)(F)$ stabilizing $\theta|_{S(F)_0}$ belongs to $\Omega(S,G^0)(F)$.
\end{lem}
We postpone the proof of this lemma to the next subsection.

\begin{fct} \label{fct:srtwist} If $(S,\theta)$ is a regular (resp. extra regular) tame elliptic pair and $\delta_0 : S(F) \to \C^\times$ is a character of depth zero that is invariant under $N(S,G^0)(F)/S(F)$ (resp. in $\Omega(S,G^0)(F)$), then $(S,\theta\delta_0)$ is regular (resp. extra regular).
\end{fct}
\begin{proof} This follows from the fact that neither $R_{0+}$ nor the appropriate stabilizer of $\theta|_{S(F)_0}$ changes when we pass from $\theta$ to $\theta\delta_0$.
\end{proof}

Recall that when $p \nmid N$ the supercuspidal representations of $\tx{GL}_N$ are classified by admissible characters. The notion of admissible character and the construction of a supercuspidal representation from an admissible character appears in \cite{Howe77}, while the exhaustion is proved in \cite{Moy86} (under the assumption that $F$ has characteristic zero). We will now argue that the notion of a tame elliptic (extra) regular pair is a generalization of the notion of an admissible character to an arbitrary tamely ramified reductive $p$-adic group. An admissible character is really a pair $(K^\times,\theta)$, where $K/F$ is a field extension of degree $N$ and $\theta : K^\times \to \C^\times$ is a character satisfying certain axioms, listed on the first page of \cite{Howe77}, see also \cite[Definition 3.29]{HM08}. Since the equation $S(F)=K^\times$ provides a bijection between the conjugacy classes of elliptic maximal tori $S$ of $\tx{GL}_N$ and the isomorphism classes of field extensions $K/F$ of degree $N$, admissibility can be seen as a property of a pair $(S,\theta)$, where $S$ is an elliptic maximal torus of $\tx{GL}_N$ and $\theta : S(F) \to \C^\times$ is a character. Note that the splitting extension $E/F$ of $S$ is the normal closure of $K$, and in particular $E/K$ is unramified.

\begin{lem} If $G=\tx{GL}_N$ then the notions of extra regular, regular, and admissible, pairs coincide.

\end{lem}

\begin{proof}
The notions of extra regular and regular coincide, because $H^1(F,S)=0$ for every maximal torus $S$, hence $\Omega(S,G^0)(F)=N(S,G^0)(F)/S(F)$.

To show that a regular elliptic pair $(S,\theta)$ is admissible, let $K/F$ be the degree-$N$ extension such that $S(F)=K^\times$. If there exists an intermediate extension $K/L/F$ such that $\theta=\theta_L \circ N_{K/L}$ for some $\theta_L : L^\times \to \C^\times$, then we can consider the twisted Levi subgroup $M=\tx{Res}_{L/F}\tx{GL}_{N'}$, where $N'=[K:L]$, and realize $S$ as a maximal torus inside of it. Then $\theta$ is the restriction to $S(F)$ of the character $\theta_L \circ \tx{Res}_{L/F}(\tx{det})$ of $M(F)$. For every $\alpha \in R(S,M)$ the character $\theta \circ N_{E/F} \circ \alpha^\vee$ of $E^\times$ is trivial, in particular $R(S,M) \subset R_{0+}$. But $\theta$ is obviously invariant under $\Omega(S,M)(F)$. This contradicts Definition \ref{dfn:tre} unless $S=M$, i.e. $L=K$.

Now assume that $\theta|_{K^\times_{0+}}=\theta_L \circ N_{K/L}$ for an intermediate extension $K/L/F$ and a character $\theta_L : L^\times_{0+} \to \C^\times$. With $M$ as above we see again that $R(S,M) \subset R_{0+}$. By Definition \ref{dfn:tre} the action of inertia preserves a positive subsystem of $R_{0+}$ and hence of $R(S,M)$. This means that over $F^u$ the torus $S=\tx{Res}_{K/F}\mb{G}_m$ becomes a minimal Levi subgroup of $M=\tx{Res}_{L/F}\tx{GL}_{N'}$, which implies that $K/L$ is unramified.

Conversely let $(S,\theta)$ be admissible. Since $S=\tx{Res}_{K/F}\mb{G}_m$ for some field extension $K/F$ of degree $N$, it is automatically a tame elliptic maximal torus. Consider the root system $R_{0+}$. By Lemma \ref{lem:levisystem} below, it spans a twisted Levi subgroup $M \subset G$ whose center is contained in $S$ and hence anisotropic mod $Z(G)$. Thus $M=\tx{Res}_{L/F}\tx{GL}_{N'}$ for some intermediate extension $K/L/F$. Let $S_{M_\tx{der}}$ be the intersection of $S$ with the derived subgroup of $M$. The group $S_{M_\tx{der}}(F)_{0+}$ is generated by its subgroups $N_{E/F}\alpha^\vee(E_{0+}^\times)$ for $\alpha \in R_{0+}$ and thus $\theta|_{S(F)_{0+}}$ factors through the exact sequence
\[ 1 \to S_{M_\tx{der}}(F)_{0+} \to S(F)_{0+} \to [M/M_\tx{der}](F)_{0+} \to 1 \]
of Lemma \ref{lem:mpex1}. We have the isomorphism $\tx{Res}_{L/F}(\tx{det}) : M/M_\tx{der} \to \tx{Res}_{L/F}\mb{G}_m$, which restricted to $S$ becomes the map $N_{K/L} : S(F)=K^\times \to L^\times$. Thus $\theta|_{S(F)_{0+}}$ factors through $N_{K/L}$ and the admissibility of $\theta$ implies that $K/L$ is unramified. This means that $\tx{Res}_{K/L}\mb{G}_m$ splits over $L^u$, or equivalently the maximal torus $S=\tx{Res}_{K/F}\mb{G}_m$ of $M=\tx{Res}_{L/F}\tx{GL}_{N'}$ becomes a minimal Levi over $F^u$. This implies that inertia preserves a set of positive roots in $R_{0+}$.

Now consider the stabilizer of $\theta|_{S(F)_0}$ in $\Omega(S,G^0)(F)$. We may as well remove the restriction of scalars $L/F$ and consider $S=\tx{Res}_{K/L}\mb{G}_m$ as a maximal torus of $G^0=\tx{GL}_{N'}/L$ and the stabilizer in $\Omega(S,G^0)(L)$ of the character $\theta : S(L)_0 \to \C^\times$. Now $S(L)=K^\times = O_K^\times \cdot L^\times = S(L)_0 \cdot Z(G^0)(L)$. It follows that the stabilizer in $\Omega(S,G^0)(L)$ of $\theta|_{S(L)_0}$ is the same as the stabilizer of $\theta$. Since $K/L$ is unramified, the splitting field $E$ of $S$ is equal to $K$. The ellipticity of $S$ then implies that the cyclic group $\tx{Gal}(K/L)$ acts on $X^*(S)$ is via a Coxeter element (i.e. an $N'$-cycle) of $\Omega(S,G^0)=S_{N'}$, which in turn implies that $\Omega(S,G^0)(L)$ is cyclic and generated by that Coxeter element. That is, the action of $\Omega(S,G^0)(L)$ on $S(L)$ is translated via $S(L)=K^\times$ to the action of $\tx{Gal}(K/L)$ on $K^\times$. If some element $\sigma \in \tx{Gal}(K/L)$ leaves $\theta$ invariant, then $\theta$ will factor through the group of co-invariants $K^\times_\sigma$. Letting $K_1$ be the subfield of $K$ fixed by $\sigma$, the vanishing of $H^{-1}(\Gamma_{K/K_1},K^\times)=0$ implies that the norm map $K^\times \to K_1^\times$ descends to an isomorphism $K_\sigma^\times \to K_1^\times$. Thus $\theta$ factors through this norm map and its admissibility implies $K_1=K$, i.e. $\sigma=1$.
\end{proof}

We will now begin to connect tame elliptic (extra) regular pairs to (extra) regular reduced generic cuspidal $G$-data. The following Lemma establishes one of the two directions of this connection.

\begin{lem} \label{lem:yutopair} Let $((G^0 \subset G^1 \dots \subset G^d),\pi_{(S,\phi_{-1})},(\phi_0,\phi_1,\dots,\phi_d))$ be a (extra) regular reduced generic cuspidal $G$-datum. Let $\theta = \prod_{i=-1}^{d} \phi_i|_{S(F)}$. Then $(S,\theta)$ is a tame elliptic (extra) regular pair.
\end{lem}
\begin{proof}
Note that at the moment we have two objects labelled as $G^0$ -- the component of the twisted Levi sequence from the statement of this lemma, and the reductive group with root system $R_{0+}$ from Definition \ref{dfn:tre}. We begin by showing that they coincide, i.e. $R(S,G^0)=R_{0+}$, where for now $G^0$ stands for the component of the twisted Levi sequence. For this, note first that if $\alpha \in R(S,G^{i+1})$ and $j>i$ then
\[ \phi_j \circ N_{E/F} \circ \alpha^\vee|_{E^\times_{0+}} = 1, \]
because $N_{E/F} \circ \alpha^\vee$ takes values in $G^{i+1}_{\tx{sc}}(F) \to G^{j}_{\tx{sc}}(F)$ and $\phi_j$ is trivial on $G^j_\tx{sc}(F)_{0+}$ by Lemma \ref{lem:charsc}. From this it follows that for all $\alpha \in R(S,G^0)$ we have $\theta(N_{E/F}(\alpha^\vee(E^\times_{0+})))=\phi_{-1}(N_{E/F}(\alpha^\vee(E^\times_{0+})))=1$, which implies $R(S,G^0) \subset R_{0+}$. To show $R_{0+} \subset R(S,G^0)$, let $\alpha \in R(S,G) \sm R(S,G^0)$ and let $i \geq 0$ be so that $\alpha \in R(S,G^{i+1}) \sm R(S,G^i)$. For $j>i$ we have $\phi_j \circ N_{E/F} \circ \alpha^\vee(E^\times_{0+}) = 1$ as noted above. At the same time, for $j<i$, $\phi_j \circ N_{E/F}\circ \alpha^\vee$ is trivial on $E^\times_{r_i}$, due to $r_j<r_i$. We will show that $\phi_i \circ N_{E/F}\circ \alpha^\vee$ is non-trivial on $E^\times_{r_i}$. A direct computation shows
\[ \phi_i(N_{E/F}(\alpha^\vee(1+x))) = \Lambda\circ\tx{tr}_{E/F}(x\<X^*_i,H_\alpha\>), \]
where $X^*_i \in \tx{Lie}^*(Z(G^i))(F)_{-r_i}$ represents $\phi_i$. By assumption $\tx{ord}(\<X^*_i,H_\alpha\>)=-r_i$, so every element of $O_E^\times$ can be written as $x\<X^*_i,H_\alpha\>$ for some $x \in E_{r_i}$. The character $\Lambda \circ \tx{tr}_{E/F}$ is non-trivial on $O_E$ and trivial on $\mf{p}_E$, so the left-hand side is non-trivial for some $x \in E_{r_i}$. We conclude that $\phi_i\circ N_{E/F}\circ \alpha^\vee$, and hence also $\theta\circ N_{E/F} \circ \alpha^\vee$, is non-trivial on $E_{r_i}^\times \subset E_{0+}^\times$, showing $\alpha \notin R_{0+}$.

Returning to the proof of the lemma, by definition, $S$ is elliptic in $G^0$ and $Z(G^0)/Z(G)$ is anisotropic, so $S$ is elliptic. Also by definition, $S$ is maximally unramified in $G^0$ and $G^0$ is tame, so $S$ is tame. The product $\prod_{i=0}^d\phi_i|_{S(F)}$ is fixed under $N(S,G^0)(F)$, so the regularity of $\pi_{-1}$ implies the regularity of $\phi_{-1}$ and hence that of $\theta$.

Now assume the datum is extra regular. Let $w \in \Omega(S,G^0)(F)$. We may assume $G^0 \neq S$, for otherwise $w=1$. For $i \geq 0$ let $S^i_\tx{sc}$ be the preimage of $S$ in $G^i_\tx{sc}$. The character $\phi_i$ has trivial restriction to $S^i_\tx{sc}(F)$ and in particular to $S^0_\tx{sc}(F)$. For $s \in S(F)$ the element $s^w \cdot s^{-1}$ lifts to $S^0_\tx{sc}(F)$ and we conclude that $[\theta\circ w] \cdot \theta^{-1}=[\phi_{-1}\circ w]\cdot\phi_{-1}^{-1}$.
\end{proof}

This lemma establishes a map from the set of (extra) regular reduced generic cuspidal $G$-data to the set of tame (extra) regular elliptic pairs. In the next subsection we will see that this map is surjective and that each of its fibers lies in a single $G$-equivalence class.

\subsection{Howe factorization} \label{sub:howe}

In this subsection we assume that $p$ is not a bad prime (in particular also not a torsion prime) for $G$. We further assume that $p$ does not divide the order of the fundamental group of $G_\tx{der}$.

Let $(S,\theta)$ be a pair consisting of a tame maximal torus $S \subset G$ and a character $\theta : S(F) \to \C^\times$ of zero or positive depth. For each positive real number $r$ consider the set of roots
\[ R_r = \{ \alpha \in R(S,G)| \theta(N_{E/F}(\alpha^\vee(E_r^\times)))=1\}. \]
Then $r \mapsto R_r$ is a filtration of $R(S,G)$. We have $R_s \subset R_r$ for $s<r$ and define $R_{r+}=
\bigcap_{s>r} R_s$. Let $r_{d-1} > r_{d-2} > \dots > r_0 > 0$ be the breaks of this filtration, that is, the positive real numbers $r$ with $R_{r+} \neq R_r$. We allow here $d=0$, which signifies that there are no breaks, i.e. $R_{0+}=R(S,G)$. We set in addition $r_{-1}=0$ and $r_d=\tx{depth}(\theta)$, so that $r_d \geq r_{d-1} > \dots > r_0 > r_{-1}=0$ if $d>0$ and $r_0 \geq r_{-1}=0$ if $d=0$. By Lemma \ref{lem:levisystem} each $R_r$ is a Levi subsystem of $R(S,G)$. It is moreover invariant under $\Gamma$. For each $d \geq i \geq 0$ we can thus define a twisted Levi subgroup $G^i$ by taking the reductive subgroup of $G$ with maximal torus $S$ and root system $R_{r_{i-1}+}$. By definition the root system of $G^d$ is $R(S,G)$, so $G^d=G$. Moreover, the root system of $G^0$ is $R_{0+}$, which may or may not be empty. If it is empty, then $G^0=S$. Recall that by convention $G^{-1}=S$.

\begin{dfn} \label{dfn:howe} A Howe factorization of $(S,\theta)$ is a sequence of characters $\phi_i : G^i(F) \to \C^\times$ for $i=-1,\dots,d$ with the following properties.
\begin{enumerate}
	\item \begin{equation} \label{eq:thetaprod} \theta = \prod_{i=-1}^{d}\phi_i|_{S(F)}. \end{equation}
	\item For all $0 \leq i \leq d$ the character $\phi_i$ is trivial on $G^i_{\tx{sc}}(F)$.
	\item For all $0 \leq i < d$, $\phi_i$ has depth $r_i$ and is $G^{i+1}$-generic. For $i=d$, $\phi_d$ is trivial if $r_{d}=r_{d-1}$ and has depth $r_{d}$ otherwise. For $i=-1$, $\phi_{-1}$ is trivial if $G^0=S$ and otherwise satisfies $\phi_{-1}|_{S(F)_{0+}}=1$.
\end{enumerate}

\end{dfn}

The discussion of \cite[\S3.5]{HM08} makes clear the direct parallel between this definition and the original notion of Howe factorization for the group $\tx{GL}_N$, as formulated for example in \cite[Defintion 3.33]{HM08}. Before we discuss the existence of Howe factorizations, let us first collect some of their properties.

\begin{fct} \label{fct:fri} Let $S^i_{\tx{sc}}$ be the preimage of $S$ in $G^i_{\tx{sc}}$. Then the restrictions of $\phi_i$ and $\theta$ to $S^{i+1}_{\tx{sc}}(F)_{r_i}$ agree.
\end{fct}
\begin{proof}
This follows from \eqref{eq:thetaprod}, since $\phi_{i+1},\dots,\phi_{d}$ restrict trivially to $G^{i+1}_{\tx{sc}}(F)$ and $\phi_{i-1},\dots,\phi_{-1}$ restrict trivially to $S(F)_{r_i}$.
\end{proof}

\begin{lem} \label{lem:refactor} If $(\phi_{-1},\dots,\phi_d)$ and $(\dot\phi_{-1},\dots,\dot\phi_d)$ are two Howe factorizations of the same pair $(S,\theta)$, then they are refactorizations of each other in the sense of \cite[Definition 4.19]{HM08}
\end{lem}

\begin{proof}
We need to check the three properties F0,F1,F2 in \cite[Definition 4.19]{HM08}. As there, we set
\[ \chi_i = \prod_{j=i}^d \phi_j \cdot \dot\phi_j^{-1}, \]
for $0 \leq i \leq d$. For F0, the definition of Howe factorization implies that $\phi_d = 1$ if and only if $r_d=r_{d-1}$ if and only if $\dot\phi_d=1$. For F1, we observe that by \eqref{eq:thetaprod} we have $\chi_i=\prod_{j=-1}^{i-1}\phi_j^{-1}\dot\phi_j$, which implies $\tx{depth}(\chi_i) \leq r_{i-1}$. For F2, we have again by \eqref{eq:thetaprod} the equality $\dot\phi_{-1}=\phi_{-1}\chi_0$ and the statement follows from Lemma \ref{lem:rdztwist}.
\end{proof}

We will now show that Howe factorizations exist and give an explicit recursive construction for them.

We first deal with the following two trivial cases: If $d=0$ and $r_0=r_{-1}=0$ then the twisted Levi sequence we have is $G=G^0 \supset S$ and we set $\phi_0=1$ and $\phi_{-1}=\theta$ and we are done.
If $d=1$, $r_1=r_0>r_{-1}=0$, and $R_{0+}=\emptyset$, then the twisted Levi sequence we have is $G=G^1 \supset G^0 = S$ and $\theta$ is a $G$-generic character of $S(F)$ of depth $r_1=r_0$ according to Lemma \ref{lem:genchar}, and we set $\phi_1=1$, $\phi_0=\theta$, and $\phi_{-1}=1$.%

Assume now $r_{d}>0$. To begin the recursion: If $r_d>r_{d-1}$ let $i=d$ and $\theta_d=\theta$, and if $r_{d}=r_{d-1}$ let $i={d-1}$, $\phi_d=1$, and $\theta_{d-1}=\theta_d=\theta$.

The recursion step assumes that we are in the following situation (which is tautologically true in the beginning): $\theta_i : S(F) \to \C^\times$ is a character of depth $r_i>0$, $R(S,G^i) \neq \emptyset$, and that for any $r \geq 0$ and $\alpha \in R(S,G^i)$, we have
\[ \theta_i(N_{E/F}(\alpha^\vee(E_{r}^\times)))=\theta(N_{E/F}(\alpha^\vee(E_{r}^\times))). \]
Given that, we apply Corollary \ref{cor:factstep} to $G^i$ and $\theta_i$ and obtain a character $\phi_i : G^i(F) \rw \C^\times$ of depth $r_i$ and set $\theta_{i-1}=\theta_i \cdot \phi_i^{-1}|_{S(F)}$. Note that the character $\phi_i$ is trivial on $G^i_{\tx{sc}}(F)$ and hence for all $r \geq 0$ and $\alpha \in R(S,G^i)$ we have the following strengthening of the second recursion hypothesis
\begin{equation} \label{eq:recstr} \theta_{i-1}(N_{E/F}(\alpha^\vee(E_{r}^\times)))=\theta_i(N_{E/F}(\alpha^\vee(E_{r}^\times)))=\theta(N_{E/F}(\alpha^\vee(E_{r}^\times))). \end{equation}
We claim that $r':=\tx{depth}(\theta_{i-1})$ is equal to $r_{i-1}$ if $i>0$, and $r' \leq r_{i-1}=r_{-1}=0$ if $i=0$. If we assume $r'>r_{i-1}$, then $r'>0$ and according to Corollary \ref{cor:factstep} we would have a root $\alpha \in R(S,G^i)$ satisfying $1 \neq \theta_{i-1}(N_{E/F}(\alpha^\vee(E_{r'}^\times))) = \theta(N_{E/F}(\alpha^\vee(E_{r'}^\times)))$, contradicting the definition of $G^i$. Thus $r' \leq r_{i-1}$. If we in addition assume $i>0$, then $r_{i-1}$ is a jump of the filtration $R_r$ and so there exists $\alpha \in R(S,G^i)$ such that $1 \neq \theta(N_{E/F}(\alpha^\vee(E_{r_{i-1}}^\times)))=\theta_{i-1}(N_{E/F}(\alpha^\vee(E_{r_{i-1}}^\times)))$, showing $r' \geq r_{i-1}$.

If $i=0$ the recursion stops and we set $\phi_{-1}=\theta_{-1}$. If $i=1$ but $G^0=S$ the recursion also stops and we set $\phi_0=\theta_0$ and $\phi_{-1}=1$. Otherwise, we have just checked that $(G^{i-1},\theta_{i-1})$ meets the requirements of the recursion step, and we continue with it.

Let us now show that the characters $\phi_i$ obtained in this way are a Howe factorization of $(S,\theta)$. The first two parts of the definition of Howe factorization are immediate from the construction, as well as the claims about $\phi_d$ and $\phi_{-1}$. Let now $d>i>-1$. According to Corollary \ref{cor:factstep} we have $\tx{depth}(\phi_i)=\tx{depth}(\theta_i)=r_i$. Moreover, for $\alpha \in R(S,G^{i+1}) \sm R(S,G^i)$
\[ 1 \neq \theta(N_{E/F}(\alpha^\vee(E_{r_i}^\times)))=\theta_i(N_{E/F}(\alpha^\vee(E_{r_i}^\times)))=\phi_i(N_{E/F}(\alpha^\vee(E_{r_i}^\times))),\]
the first (non)equality holding by definition of $R(S,G^{i+1})\sm R(S,G^i)$, the second by \eqref{eq:recstr}, and the third by the fact that $\tx{depth}(\theta_{i-1})<r_i=\tx{depth}(\theta_i)=\tx{depth}(\phi_i)$. According to Lemma \ref{lem:genchar} $\phi_i$ is $G^{i+1}$-generic.

Let us now discuss how additional properties of the pair $(S,\theta)$ are inherited by a Howe factorization. Since $S$ is split over a tamely ramified extension, so is $G^i$ for all $i$. If $S$ is elliptic, then $Z(G^0)/Z(G)$ is anisotropic and we have a canonical point $x \in \mc{B}(G^0,F) \subset \mc{B}(G,F)$, namely the point associated to $S$ by \cite{Pr01}.

In the case $G^0 \neq S$ we have defined $\phi_{-1}=\theta_{-1}$, so that $\phi_{-1}=\theta\cdot (\phi_{\geq 0}|_{S(F)})^{-1}$, where $\phi_{\geq 0}=\prod_{i=0}^d\phi_i|_{G^0(F)}$. Since $\phi_{\geq 0}$ restricts trivially to $G^0_\tx{sc}(F)$ and for any $w \in \Omega(S,G^0)(F)$ and $s \in S(F)$ the element $wsw^{-1}s^{-1}$ of $S(F)$ lifts to $G^0_\tx{sc}(F)$, we see that the stabilizer of $\phi_{-1}|_{S(F)_0}$ in $\Omega(S,G^0)(F)$ is equal to the stabilizer of $\theta|_{S(F)_0}$ in $\Omega(S,G^0)(F)$. The (extra) regularity of the tame elliptic pair $(S,\theta)$ then implies the (extra) regularity of the depth-zero character $\phi_{-1}$ with respect to $G^0$.

Finally, if $G^0 \neq S$ and the action of inertia leaves a basis of $R_{0+}$ invariant, then $S$ is a maximally unramified maximal torus of $G^0$.

We formulate this discussion as follows.

\begin{pro} \label{pro:howefacex} Any pair $(S,\theta)$ consisting of a tame maximal torus $S \subset G$ and a character $\theta : S(F) \to \C^\times$ of zero or positive depth has a Howe factorization $(\phi_{-1},\dots,\phi_d)$. If the pair $(S,\theta)$ is (extra) regular tame elliptic, then $S$ is a maximally unramified elliptic maximal torus of $G^0$, $\phi_{-1}$ has depth zero and is (extra) regular with respect to $G^0$, and $((G^0\subset \dots \subset G^d),\pi_{S,\phi_{-1}}^{G^0},(\phi_0,\dots,\phi_d))$ is a normalized (extra) regular reduced cuspidal generic $G$-datum.
\end{pro}

\begin{proof}[Proof of Lemma \ref{lem:wr}]
A pair $(S,\theta)$ as in the statement of the lemma has a Howe factorization $(\phi_{-1},\dots,\phi_d)$. The genericity of $\phi_i$ implies that any element of $\Omega(S,G^{i+1})(F)$ fixing $\phi_i|_{S(F)_{r_i}}$ must lie in $\Omega(S,G^i)(F)$. At the same time, $\phi_i|_{S(F)}$ is invariant under $\Omega(S,G^i)(F)$, because for $w \in \Omega(S,G^i)(F)$ and $s \in S(F)$ the element $wsw^{-1}$ lifts to $S^i_\tx{sc}(F)$ (notation as in the proof of Lemma \ref{lem:yutopair}) and is thus killed by $\phi_i$. The claim now follows by induction.
\end{proof}

The remainder of this subsection consists of the proofs of the claims that were used in the construction of the factorization of $\theta$.

\begin{lem} \label{lem:genchar} Let $H \subset G$ be a twisted Levi subgroup containing $S$ and let $\phi : H(F) \to \C^\times$ a character of positive depth $r$, trivial on $H_\tx{sc}(F)$. If for all $\alpha \in R(S,G) \sm R(S,H)$ we have $\phi(N_{E/F}(\alpha^\vee(E_r^\times)))\neq 1$ then $\phi$ is generic.
\end{lem}
\begin{proof}
Assume first that the derived subgroup of $G$, and hence also of $H$, is simply connected. Put $D=H/H_\tx{sc}$. Fix a point $x \in \mc{A}(S,E) \cap \mc{B}(H,F)$. By Kneser's theorem \cite[\S4.7]{BT3} $H(F) \to D(F)$ is surjective and by Lemma \ref{lem:mpex1} $H(F)_{x,s} \to D(F)_s$ remains surjective for all $s>0$, so $\phi$ is inflated from a character of the torus $D(F)$ of depth $r$. Let $\Lambda : F \to \C^\times$ be a character of depth zero, i.e. trivial on $F_1$ but not on $F_0$. Recall the Moy-Prasad isomorphism
\[ \tx{MP}_{H,x,r} : \tx{Lie}(H)(F)_{x,r}/\tx{Lie}(H)(F)_{x,r+} \to H(F)_{x,r}/H(F)_{x,r+}. \]
Let $X^* \in \tx{Lie}^*(D)(F)_{-r}/\tx{Lie}^*(D)(F)_{-r+}$ be the element satisfying
\[ \phi(\tx{MP}_{D,r}(Y)) = \Lambda\<X^*,Y\>, \quad \forall Y \in \tx{Lie}(D)(F)_r. \]
The surjection $H \to D$ leads to an injection $\tx{Lie}^*(D) \to \tx{Lie}^*(H)$ whose image is precisely the subspace that Yu identifies with $\tx{Lie}^*(Z(H)^\circ)$ in \cite[\S8]{Yu01}. Letting $X^*$ stand also for its image in $\tx{Lie}^*(H)$ we then obtain
\[ \phi(\tx{MP}_{H,x,r}(Y)) = \Lambda\<X^*,Y\>, \quad \forall Y \in \tx{Lie}(H)(F)_r. \]
We can restrict this equation to $Y \in \tx{Lie}(S)(F)_r$. Using the fact that $\tx{MP}_{S,r}$ is $\Gamma$-equivariant we see that for all $Y \in \tx{Lie}(S)(E)_r$ we have
\[ \phi(N_{E/F}(\tx{MP}_{S,r}(Y)))\!\!=\!\! \phi(\tx{MP}_{S,r}(\tx{tr}_{E/F}(Y)))\!\!=\!\!\Lambda\<X^*,\tx{tr}_{E/F}Y\>\!\!=\!\!\Lambda\circ\tx{tr}_{E/F}\<X^*,Y\>. \]
From the functoriality of the Moy-Prasad isomorphism in the case of the split torus $S \times E$ we have for all $y \in E_r/E_{r+}$
\[ \tx{MP}_{S,r}(d\alpha^\vee(y))=\alpha^\vee(y+1). \]
Combining the last two equations we obtain
\[ \phi(N_{E/F}(\alpha^\vee(y+1)))=\Lambda\circ\tx{tr}_{E/F}(y\<X^*,H_\alpha\>). \]
The character $\Lambda\circ\tx{tr}_{E/F}$ is of depth zero and the assumption of the lemma implies that $y\<X^*,H_\alpha\>$ doesn't vanish for some $y \in E_r$, which is only possible if $\tx{ord}\<X^*,H_\alpha\>=-r$. Thus $\phi$ satisfies condition GE1 of \cite[\S8]{Yu01} and by \cite[Lemma 8.1]{Yu01} it also satisfies condition GE2.

Now drop the assumption that the derived subgroup of $G$ is simply connected. Let $1 \to K \to \tilde G \to G \to 1$ be a z-extension. Pull-back along the inclusion $H \to G$ gives a z-extension $1 \to K \to \tilde H \to H \to 1$. Let $\tilde\phi : \tilde H(F) \to \C^\times$ be the pull-back of $\phi$. If $X^* \in \tx{Lie}^*(H)(F)_{-r}$ represents $\phi|_{H(F)_r}$ as above, then its image $\tilde X^* \in \tx{Lie}^*(\tilde H)(F)_{-r}$ under the natural inclusion represents $\tilde\phi|_{\tilde H(F)_r}$. Since $H_\alpha$ is naturally an element of $\tx{Lie}(\tilde H)$ the above argument shows $\tx{ord}\<X^*,H_\alpha\> = \tx{ord}\<\tilde X^*,H_\alpha\>=-r$.

\end{proof}

\begin{lem} \label{lem:factstep} Assume that the depth $r$ of $\theta$ is positive. If for all $\alpha \in R(S,G)$ we have $\theta \circ N_{E/F} \circ \alpha^\vee|_{E^\times_r}=1$, then there exists a character $\phi : G(F) \rw \C^\times$ of depth $r$ that is trivial on $G_\tx{sc}(F)$ and satisfies $\phi|_{S(F)_r}=\theta|_{S(F)_r}$.
\end{lem}
\begin{proof}
Let $S_\tx{der}$ and $S_\tx{sc}$ be the preimages of $S$ in $G_\tx{der}$ and $G_\tx{sc}$. We claim that $\theta|_{S_\tx{der}(F)_r}=1$. Let $R = \tx{Res}_{E/F}(S_\tx{sc} \times E)$. We have the norm homomorphism $N_{E/F} : R \rw S_\tx{sc}$. It is surjective and we call its kernel $R^1$. We have the exact sequence
\[ 1 \rw R^1 \rw R \rw S_\tx{sc} \rw 1 \]
of tori defined over $F$ and split over $E$. According to Lemma \ref{lem:mpex1} the homomorphism $S_\tx{sc}(E)_r = R(F)_r \rw S_\tx{sc}(F)_r \to S_\tx{der}(F)_r$ is surjective. The claim would thus follow from $\theta\circ N_{E/F}(S_\tx{sc}(E)_r)=1$. However, $X_*(S_\tx{sc})$ is generated by $R(S,G)^\vee$ over $\Z$ so the latter follows immediately from the assumption of the lemma.

With the claim proved, we turn to the proof of the lemma. Let $D=G/G_\tx{der}$. From Lemma \ref{lem:mpex1} we have the equality $D(F)_r=S(F)_r/S_\tx{der}(F)_r$ and the claim we just proved tells us that $\theta$ descends to a non-trivial character of $D(F)_r$ that is trivial on $D(F)_{r+}$. This finite order character of $D(F)_r$ can be extended by Pontryagin duality to a character $\phi : D(F) \rw \C^\times$, trivial on $D(F)_{r+}$. This character pulls back to a character of $G(F)$ whose restriction to $S(F)_r$ is equal to that of $\theta$.
\end{proof}

\begin{cor} \label{cor:factstep} Let $\theta : S(F) \rw \C^\times$ be a character of positive depth $r$. Assume that $\theta \circ N_{E/F} \circ \alpha^\vee|_{E^\times_r}=1$ for all $\alpha \in R(S,G)$. Then there exists a character $\phi : G(F) \rw \C^\times$ of depth $r$, trivial on $G_\tx{sc}(F)$, such that $\theta' = \theta \cdot \phi^{-1}|_{S(F)}$ has depth $r'<r$. Moreover, if $r'>0$ there exists a root $\alpha \in R(S,G)$ such that $\theta'(N_{E/F}(\alpha^\vee(E^\times_{r'}))) \neq 1$.
\end{cor}
\begin{proof}
We work recursively on the depth of $\theta$. Put $\theta_0=\theta$ and $r_0=r$ and apply Lemma \ref{lem:factstep} and obtain $\phi_0 : G(F) \rw \C^\times$ of depth $r$ such that the depth $r_1$ of $\theta_1 := \theta_0 \cdot \phi^{-1}|_{S(F)}$ is strictly less than $r_0$. If $\theta_1(N_{E/F}(\alpha^\vee(E_{r_1}^\times))) = 1$ for all $\alpha \in R(S,G)$ we apply the recursion step again with $\theta_1$ and $r_1$. The recursion eventually stops because the set of positive numbers that are depths of characters of $S(F)$ (i.e. the breaks of the Moy-Prasad filtration of $S(F)$) has no accumulation points. Let $\phi$ be the product of all $\phi_i$. Its depth is equal to that of $\phi_0$, which is $r_0$.
\end{proof}

\begin{lem} \label{lem:levisystem}
Let $\theta : S(F) \rw \C^\times$ be a character. Then
\[ R_\theta := \{ \alpha \in R(S,G)| \theta(N_{E/F}(\alpha^\vee(E^\times_r)))=1 \} \]
is a Levi subsystem of $R(S,G)$ for any $r>0$.
\end{lem}
\begin{proof}
Let $\varphi \in Z^1(W_F,\hat S)$ be the Langlands parameter of $\theta$. The Langlands parameter of the  character $\theta\circ N_{E/F}\circ\alpha^\vee : E^\times \to \C^\times$ is $\hat\alpha\circ\varphi|_{W_E}$. By local class field theory, $\theta\circ N_{E/F}\circ\alpha^\vee$ has non-zero restriction to $E_r^\times$ if and only if $\hat\alpha\circ\varphi|_{W_E}$ has non-zero restriction to $I_E^r=I_F^r$. Thus
\[ R_\theta = \{\alpha \in R(S,G)| \hat\alpha(\varphi(I^r)) = 1\}. \]
But $\varphi(I^r)$ is a finite subgroup of $\hat S$ and then $R^\vee_\theta=\{\hat\alpha\in R(\hat S,\hat G)|\hat\alpha(\varphi(I^r)) = 1\}$ is the root system of the connected centralizer in $\hat G$ of this finite group. This connected centralizer is a Levi subgroup, as one sees by applying \cite[Proposition A.7]{AS08} repeatedly to the elements in the image of $\varphi(I^r)$. Thus $R_\theta^\vee$ is a Levi subsystem of $R(S,G)^\vee$.
\end{proof}

\subsection{Classification of regular supercuspidal representations} \label{sub:regclass}

In this subsection we are still assuming that $p$ is not a bad prime for $G$ and does not divide the order of the fundamental group of the derived subgroup of $G$.

In the previous subsection we obtained from a tame (extra) regular elliptic pair $(S,\theta)$ a sequence of tame twisted Levi subgroups $G^0 \subset G^1 \subset \dots \subset G^d=G$ such that $S \subset G^0$ is maximally unramified, as well as a Howe-factorization $(\phi_{-1},\dots,\phi_d)$. Then $((G^0 \subset G^1 \dots \subset G^d),\pi_{(S,\phi_{-1})}^{G^0},(\phi_0,\phi_1,\dots,\phi_d))$ is a reduced generic cuspidal $G$-datum in the sense of \cite[Definition 3.11]{HM08}. It is moreover normalized in the sense of Definition \ref{dfn:normalized} and (extra) regular in the sense of Definition \ref{dfn:regdat}. This datum in turn leads to a (extra) regular supercuspidal representation $\pi_{(S,\theta)}$ of $G(F)$ via Yu's construction \cite{Yu01}. The characters $(\phi_{-1},\dots,\phi_d)$ of the Howe factorization were chosen with some ambiguity. Nonetheless, the work of Hakim and Murnaghan allows us to conclude that $(S,\theta) \leftrightarrow \pi_{(S,\theta)}$ is a 1-1 correspondence between the set of $G(F)$-conjugacy classes of tame (extra) regular elliptic pairs and the set of isomorphism classes of (extra) regular supercuspidal representations, as follows.

\begin{cor} \label{cor:rsexist} The representation $\pi_{(S,\theta)}$ depends only on the pair $(S,\theta)$, and not on the Howe factorization. Moreover, the representations $\pi_{(S,\theta)}$ obtained from all (extra) regular tame elliptic pairs $(S,\theta)$ are precisely the (extra) regular supercuspidal representations.
\end{cor}
\begin{proof} That $\pi_{(S,\theta)}$ is independent of the Howe factorization follows from \cite[Theorem 6.6]{HM08} together with Lemma \ref{lem:refactor} and Lemma \ref{lem:hypcg}. By definition it is a (extra) regular supercuspidal representation. To see that any (extra) regular supercuspidal representation $\pi$ arizes this way, we see from the definition that $\pi$ arises from a (extra) regular reduced generic cuspidal $G$-datum. This datum leads via Lemma \ref{lem:yutopair} to a (extra) regular tame elliptic pair $(S,\theta)$ as well as to a pre-chosen Howe factorization $(\phi_{-1},\dots,\phi_d)$ thereof. The representation obtained from this factorization is tautologically isomorphic to $\pi$, thus $\pi=\pi_{(S,\theta)}$.
\end{proof}

\begin{lem} \label{lem:rsuniq} Let $(S_i,\theta_i)$, $i=1,2$, be two tame regular elliptic pairs. The representations $\pi_{(S_i,\theta_i)}$ are isomorphic if and only if the pairs $(S_i,\theta_i)$ are $G(F)$-conjugate.
\end{lem}
\begin{proof} Assume first  the pairs $(S_i,\theta_i)$
are conjugate. If $(S_1,\theta_1)$ leads to the twisted Levi tower $\vec G_1$ and has a Howe factorization $\vec \phi_1$, and if $g \in G(F)$ conjugates $(S_1,\theta_1)$ to $(S_2,\theta_2)$, then $\tx{Ad}(g)(\vec G_1,\vec \phi_1)$ will be the twisted Levi tower and a Howe factorization of $(S_2,\theta_2)$, and the claim follows from \cite[Theorem 6.6]{HM08}.

Conversely, assume that $\pi_{(S_i,\theta_i)}$ are isomorphic. Let $(\vec G_1,\vec \phi_1)$ and $(\vec G_2,\vec \phi_2)$ be pairs consisting of the twisted Levi tower and a Howe factorization of $(S_1,\theta_1)$ and $(S_2,\theta_2)$ respectively. Applying \cite[Theorem 6.6]{HM08} we find $g \in G(F)$ such that $\tx{Ad}(g)(\vec G_2)=\vec G_1$ and $\tx{Ad}(g)[\pi_{(S_2,\phi_{2,-1})}^{G_2^0} \otimes \phi_2] \cong [\pi_{(S_1,\phi_{1,-1})}^{G_1^0} \otimes \phi_1]$, where $\phi_i$ is the product of the restrictions to $G_i^0$ of $\phi_{i,0},\dots,\phi_{i,d}$ (recall condition F2 in the definition of refactorization). We may as well conjugate $(S_2,\theta_2)$ by $g$ and assume $g=1$. If $G^0 = S$ then we are done. Otherwise we apply Lemma \ref{lem:rdztwist} and see that $\phi_2 \cdot \phi_1^{-1}$ is a depth-zero character. We then apply Lemma \ref{lem:rdzclass} and see that $(S_1,\phi_{1,-1} \cdot \phi_1\cdot\phi_2^{-1})$ and $(S_2,\phi_{2,-1})$ are $G^0(F)$-conjugate. Conjugating by an element of $G^0(F)$ we may thus assume $S_1 = S_2$ and $\phi_{1,-1} \cdot \phi_1\cdot \phi_2^{-1} = \phi_{2,-1}$. But then \eqref{eq:thetaprod} implies $\theta_1=\theta_2$.
\end{proof}

\begin{fct} \label{fct:centchar} The central character of $\pi_{(S,\theta)}$ is the restriction of $\theta$ to $Z(G)(F)$.
\end{fct}
\begin{proof}
Let $(\phi_{-1},\dots,\phi_d)$ be a Howe factorization. We examine Yu's construction, following the exposition in \cite[\S3.4]{HM08}. There is a sequence of compact modulo center subgroups of $G(F)$
\[ G^0(F)_x = K^0 \subset K^1 \subset \dots \subset K^d=K \]
On each $K^i$ there is an irreducible representation $\kappa_i$. There is a natural inflation process that makes $\kappa_i$ into a representation of $K$. The tensor product $\kappa_{-1} \otimes \dots \otimes \kappa_d$ is irreducible and its compact induction from $K$ to $G$ is $\pi_{(S,\theta)}$.

Since $Z(G)(F) \subset K^0$, the inflation process doesn't disturb the central character of $\kappa_i$. For $i=0$ the representation $\kappa_i$ is $\phi_0 \otimes \tx{Ind}_{S(F)G^0(F)_{x,0}}^{G^0(F)_x} \tilde\kappa_{(S,\phi_{-1})}$, where $\tilde\kappa$ is as in Lemma \ref{lem:rdzconst}. The central character of this representation is $\phi_0 \cdot \phi_{-1}$. For the intermediate indices $i$ the representation $\kappa_i$ is described after \cite[Remark 3.25]{HM08} and its construction may or may not involve the Weil representation. In either case, its restriction to $Z(G)(F)$ is seen to act via the character $\phi_i$.
\end{proof}

\subsection{When $p$ divides $|\pi_1(G_\tx{der})|$} \label{sub:nsc}
In the previous two subsections we assumed that $p$ is not a bad prime for $G$ and does not divide the order of $\pi_1(G_\tx{der})$. In this subsection we will remove the condition that $p$ does not divide the order of $\pi_1(G_\tx{der})$. This is only an issue for Dynkin type $A_n$, for which there are no bad primes while $\pi_1(G_\tx{der})$ can be any divisor of $n$. For all other Dynkin types a prime that divides $\pi_1(G_\tx{der})$ is automatically bad for $G$, and equals either $2$ or $3$.

Fix a $z$-extension $1 \to K \to G_1 \to G \to 1$.

\begin{lem} \label{lem:updown} Let $(S,\theta)$ be a tame elliptic (extra) regular pair for $G$. Let $S_1$ be the preimage of $S$ in $G_1$ and let $\theta_1$ be the inflation of $\theta$ to $S_1(F)$. The map $(S,\theta) \mapsto (S_1,\theta_1)$ is a depth-preserving bijection between the $G(F)$-conjugacy classes of tame elliptic (extra) regular pairs of $G$ and those tame elliptic (extra) regular pairs of $G_1$ for which $\theta_1|_{K(F)}=1$.
\end{lem}
\begin{proof}
This follows at once from Lemma \ref{lem:mpex3}.
\end{proof}

A supercuspidal representation $\pi$ of $G(F)$ will be called (extra) regular if its inflation $\pi_1$ to $G_1(F)$ is so. This definition together with Corollary \ref{cor:rsexist}, Lemma \ref{lem:rsuniq}, Fact \ref{fct:centchar}, and Lemma \ref{lem:updown}, show that there is a bijection between the $G(F)$-conjugacy classes of tame elliptic (extra) regular pairs of $G(F)$ and (extra) regular supercuspidal representations of $G(F)$.

We claim that the notion of (extra) regularity and this bijection are independent of the choice of $G_1$. Choose another $z$-extension $G_2$ of $G$ and consider the fiber product $G_3$ of $G_1$ and $G_2$ over $G$. Then $G_3$ is a $z$-extension of $G$, of $G_1$, and of $G_2$. If the inflation $\pi_1$ of $\pi$ is (extra) regular then $\pi_1=\pi_{(S_1,\theta_1)}$ for some (extra) regular tame elliptic pair $(S_1,\theta_1)$. Let $(S,\theta)$ be the pair of $G$ corresponding to $(S_1,\theta_1)$. Let $(S_2,\theta_2)$ and $(S_3,\theta_3)$ be the pairs on $G_2$ and $G_3$ corresponding to $(S,\theta)$. According to Lemma \ref{lem:updown}, $(S_3,\theta_3)$ is tame elliptic (extra) regular, and hence $(S_2,\theta_2)$ is tame elliptic (extra) regular. According to Lemmas \ref{lem:mpex1} and \ref{lem:mpex2} the pull-back to $G_3$ of a Howe factorization for $\theta_1$ is a Howe factorization for $\theta_3$. The same is true for $\theta_2$ in place of $\theta_1$. Thus the representation $\pi_{(S_3,\theta_3)}$ is the pull-back to $G_3(F)$ of the representation $\pi_{(S_1,\theta_1)}$ and also of the representation $\pi_{(S_2,\theta_2)}$. But then $\pi_{(S_3,\theta_3)}$ is the pull-back of $\pi$ to $G_3(F)$, which then implies that $\pi_{(S_2,\theta_2)}$ is the pull-back of $\pi$ to $G_2(F)$. We conclude that the pull-back of $\pi$ to $G_2(F)$ is (extra) regular.

\section{The character formula} \label{sec:char}

Let $F$ be a non-archimedean local field and let $G$ be a connected reductive group defined over $F$ and splitting over a tame extension of $F$. Let $((G^0 \subset G^1 \dots \subset G^d),\pi_{-1},(\phi_0,\phi_1,\dots,\phi_d))$ be a reduced generic cuspidal $G$-datum in the sense of \cite[Definition 3.11]{HM08}. From this datum Yu's construction produces not just a single supercuspidal representation $\pi$ of depth $r=r_d$, but in fact a supercuspidal representation $\pi_i$ of the group $G^i(F)$ of depth $r_i$, for each $0 \leq i \leq d$, and $\pi=\pi_d$.

The Harish-Chandra character $\Theta_\pi$ of $\pi$ has been computed in the work of Adler and Spice \cite{AS08,AS09} and later reinterpreted in the work of DeBacker and Spice \cite{DS}. At the moment this work has the additional technical assumption that $G^{d-1}/Z_G$ is anisotropic, but we are hopeful that this condition will be eliminated in the future. The resulting character formula involves various roots of unity. The main purpose of this section is to provide an alternative description of these roots of unity. As we shall see, they can be interpreted in a way that ties them closely to the Langlands-Shelstad transfer factors from the theory of endoscopy \cite{LS87}. More precisely, we shall define certain terms $\epsilon$ and $\Delta_{II}^\tx{abs}$, that can be seen as absolute versions of the corresponding pieces of the transfer factor, and will show that they describe the roots of unity occurring in the character formula. These terms are absolute in the sense that they are associated to the group $G$, a maximal torus of it, and some auxiliary data. In the presence of an endoscopic group $H$, the quotient of either term for $G$ by the corresponding term for $H$ will be equal to the analogous term occurring in the Langlands-Shelstad transfer factor.

Once the roots of unity in the character formula have been reinterpreted in this way, we will show that the resulting expression for the character of a regular supercuspidal representation evaluated at a sufficiently shallow element is precisely analogous to the character formula for discrete series representations of real reductive groups.

\subsection{Hypotheses}

The papers \cite{AS08} and \cite{AS09} impose various hypotheses on the group $G$ under which the character formula is obtained. Besides the assumption that $G^{d-1}/Z_G$ is compact that we already mentioned, these are Hypotheses (A)-(D) of \cite[\S2]{AS08}, Hypothesis 2.3 of \cite{AS09}, and the assumption \cite[\S1.1]{AS09} that the residual characteristic of $F$ is not $2$. As remarked in \cite[\S1.2]{AS09}, Hypotheses (A) and (D) are implied by the tameness of $G$. Hypothsis (C) is satisfied when $G$ has simply connected derived subgroup. According to Lemma \ref{lem:hypcg} the same is also true for Hypothesis 2.3. However, the character function of a representation $\pi$ of $G(F)$ can be computed by first taking a $z$-extension $\tilde G$ of $G$, pulling $\pi$ back to a representation $\tilde\pi$ of $\tilde G(F)$, and then computing the character function of $\tilde\pi$. This means that Hypotheses (C) and 2.3 are in fact superfluous. Finally, Hypothesis (B) is satisfied when the residual characteristic of $F$ is not a bad prime for $G$. Thus we shall make the assumption that the residual characteristic of $F$ is not a bad prime for $G$ whenever we apply the character formula of \cite{AS09}, in addition to our standing assumption that it is not equal to $2$. On the other hand, some results of \cite{AS08} are valid and can be used without this assumption.

\subsection{Review of orbital integrals} \label{sub:orb}

Let $\Lambda : F \rw \C^\times$ be a non-trivial character. Let $S \subset G$ be a maximal torus defined over $F$. We can view $\mf{s}^*=\tx{Lie}^*(S)$ as a subspace of $\mf{g}^*=\tx{Lie}^*(G)$ as explained in \cite[\S8]{Yu01}, namely as the trivial-weight space for the coadjoint action of $G$. Let $X^* \in \mf{s}^*(F) \subset \mf{g}^*(F)$ be an element whose stabilizer for the coadjoint action of $G$ is $S$. For any function $f^* \in \mc{C}^\infty_c(\mf{g}^*(F))$ we have the orbital integral
\[ O_{X^*}(f^*) = \int_{G(F)/S(F)} f^*(\tx{Ad}^*(g)X^*)dg. \]
The measure used for integration is the quotient of a measure on $G(F)$ by a measure on $S(F)$, and on both groups we take the canonical measure introduced by Waldspurger in \cite[\S I.4]{Wal01}, as is done in \cite[Definition 4.1.6]{DS}. For a function $f \in \mc{C}^\infty_c(\mf{g}(F))$ we define its Fourier-transform $\hat f_{\Lambda,dY} \in \mc{C}^\infty_c(\mf{g}^*(F))$ by
\[ \hat f_{\Lambda,dY}(Y^*) = \int_{\mf{g}(F)} f(Y)\Lambda\<Y,Y^*\> dY, \]
where we have indicated as subscripts the dependence on the choices of the character $\Lambda$ and the measure $dY$. A fundamental result of Harish-Chandra is that the distribution $f \mapsto O_{X^*}(\hat f_{\Lambda,dY})$ is represented by a function, i.e. there exists a function $\hat\mu_{X^*,\Lambda}$ on $\mf{g}(F)$ such that
\[ O_{X^*}(\hat f_{\Lambda,dY}) = \int_{\mf{g}(F)} \hat\mu_{X^*,\Lambda}(Y)f(Y) dY \]
for all $f \in \mc{C}^\infty_c(\mf{g}(F))$. The function $\hat\mu_{X^*,\Lambda}$ does not depend on the choice of measure $dY$. We can renormalize it using the usual Weyl discriminants \cite[Definition 2.2.8]{DS} and obtain
\[ \hat\iota_{X^*,\Lambda}(Y) = |D(X^*)|^\frac{1}{2} |D(Y)|^\frac{1}{2}\hat\mu_{X^*,\Lambda}(Y). \]
The function $\hat\mu_{X^*,\Lambda}$ depends on $\Lambda$ via the equation
\[ \hat\iota_{X^*,\Lambda \cdot c} = \hat\iota_{cX^*,\Lambda}, \]
where $c \in F^\times$ and $[\Lambda \cdot c](x)=\Lambda(cx)$. The same is true for $\hat\iota_{X^*,\Lambda}$ provided $c \in O_F^\times$.

\subsection{Review of the work of Adler-Spice and DeBacker-Spice} \label{sub:asds}

In this subsection $F$ is a local field of odd residual characteristic that is not a bad prime for $G$. Set $r=r_{d-1}$ and $\pi=\pi_d$. Let $x$ be the unique point in the building $\mc{B}(G^{d-1},F)$. The formula of Adler-Spice gives the value of the function $\Theta_\pi$ at any regular semi-simple element $\gamma \in G(F)$ that has an $r$-approximation $\gamma = \gamma_{<r} \cdot \gamma_{\geq r}$ in terms of the value of $\Theta_{\pi_{d-1}}$. It is more convenient to replace $\Theta_\pi(\gamma)$ by its renormalization $\Phi_{\pi}(\gamma)=|D_G(\gamma)|^\frac{1}{2}\Theta_\pi(\gamma)$. In the form presented in \cite[Theorem 4.6.2]{DS} the formula for $\Phi_\pi(\gamma)$ is
\begin{equation} \label{eq:char}
\phi_d(\gamma)\ssum{\substack{g \in J^d(F) \lmod G^d(F) / G^{d-1}(F)\\ \gamma_{<r}^g \in G^{d-1}(F)}}\epsilon_{s,r}(\gamma_{<r}^g) \epsilon^r(\gamma_{<r}^g) \tilde e(\gamma_{<r}^g) \cdot \Phi_{\pi_{d-1}}(\gamma_{<r}^g) \hat\iota_{{\mf{j}^d}, {^gX_{d-1}^*}}(\log(\gamma_{\geq r}))
\end{equation}
We need to explain the notation. Fix again a non-trivial character $\Lambda : F \rw \C^\times$, with the additional assumption that $\Lambda$ is trivial on $\mf{p}_F$ but non-trivial on $O_F$. Let $X_{d-1}^* \in \tx{Lie}^*(Z(G^{d-1}))(F)_{-r}$ be a generic element (in the sense of \cite[\S8]{Yu01}) that realizes (in the sense of \cite[\S5]{Yu01}) the character $\phi_{d-1}$. We abbreviate $\gamma_{<r}^g = g^{-1}\gamma_{<r}g$ and $^gX_{d-1}^*=\tx{Ad}^*(g)X_{d-1}^*$. Setting $J^d=\tx{Cent}(\gamma_{<r},G^d)^\circ$, the condition $\gamma_{<r}^g$ on the summation index $g$ implies that $\tx{Ad}(g)Z(G^{d-1})$ is a subgroup of $J^d$ and in particular $^gX_{d-1}^* \in \mf{j}^{d,*}(F)$. Therefore, the function $\hat\iota_{{\mf{j}^d},{^gX_{d-1}^*}}$ that represents the normalized Fourier-transform of the integral along the coadjoint orbit of $^gX_{d-1}^*$ in $\mf{j}^{d,*}(F)$ (as recalled in the Subsection \ref{sub:orb}) makes sense. Moreover, since both the function itself and the element $X_{d-1}^*$ now depend on the choice of $\Lambda$ in a parallel way, the entire expression $\hat\iota_{{\mf{j}^d},{^gX_{d-1}^*}}$ is independent of $\Lambda$. The map $\tx{log}$ is either the true logarithm function, provided it converges at $\gamma_{\geq r}$, or else the inverse of a mock-exponential map \cite[A1]{AS09}.

One place where the technical assumption on the compactness of $G^{d-1}(F)$ modulo $Z_G(F)$ enters is the evaluation of $\Phi_{\pi_{d-1}}(\gamma_{<r}^g)$, because the semi-simple element $\gamma_{<r}^g \in G^{d-1}(F)$ need not be regular. When $G^{d-1}(F)/Z_G(F)$ is compact, the character $\Theta_{\pi_{d-1}}$ is defined on all (semi-simple) elements of $G^{d-1}(F)$, not just the regular elements. Thus the function $\Phi_{\pi_{d-1}}=|D_{G^{d-1}}^\tx{red}|^\frac{1}{2}\Theta_{\pi_{d-1}}$ is also defined on all of $G^{d-1}(F)$. When $\gamma_{<r}$ is itself regular one can hope that the above formula applies even without the compactness assumption and we shall prove in the next subsection that it does, at least in the case when $\gamma=\gamma_{<r}$ is topologically semi-simple modulo $Z(G)^\circ$.

The remaining objects in the formula: $\epsilon_{s,r}$, $\epsilon^r$, and $\tilde e$, are all complex roots of unity of order dividing $4$ and will be the focus of our study. We shall now give their definition following \cite[\S4.3]{DS}. Let $T$ be a maximal torus of $G^{d-1}$ containing $\gamma_{<r}^g$ and such that $x \in \mc{A}(T,E)^\Gamma$ for some finite Galois extension $E/F$ splitting $T$. We consider the following subset of the real numbers, defined for each $\alpha \in R(T,G)$ by
\[ \tx{ord}_x(\alpha) = \{ r \in \R| \mf{g}_\alpha(F_\alpha)_{x,r+} \neq \mf{g}_{\alpha}(F_\alpha)_{x,r} \}, \]
where we have abbreviated by $\mf{g}_\alpha(F_\alpha)_{x,r}$ the intersection $\mf{g}_\alpha(F_\alpha) \cap \mf{g}(F_\alpha)_{x,r}$. Based on this set we define the following subsets of the root system $R(T,G)$

\begin{eqnarray*}
	R_{\gamma_{<r}^g}&=&\{ \alpha \in R(T,G) \sm R(T,G^{d-1})| \alpha(\gamma_{<r}^g) \neq 1 \},\\
	R_{r/2}&=&\{ \alpha \in R_{\gamma_{<r}^g}| r \in 2\tx{ord}_x(\alpha)\},\\
	R_{(r-\tx{ord}_{\gamma_{<r}^g})/2}&=&\{ \alpha \in R_{\gamma_{<r}^g}| r-\tx{ord}(\alpha(\gamma_{<r}^g)-1) \in 2\tx{ord}_x(\alpha)\}.\\
\end{eqnarray*}
For $\alpha \in R_{(r-\tx{ord}_{\gamma_{<r}^g})/2}$ symmetric and ramified we define
\[ t_{\alpha}=\frac{1}{2}e_\alpha N_{F_\alpha/F_{\pm\alpha}}(w_\alpha)\<d\alpha^\vee(1),X_{d-1}^*\>(\alpha(\gamma_{<r}^g)-1) \in O_{F_\alpha}^\times. \]
Here $e_\alpha$ is the ramification degree of $F_\alpha/F$ and $w_\alpha \in F_\alpha^\times$ is any element of valuation $(\tx{ord}(\alpha(\gamma_{<r})-1)-r)/2$. The existence of $w_\alpha$ is argued in the proof of \cite[Proposition 5.2.13]{AS09}. It also follows from Proposition \ref{pro:ordx} below. Finally, we introduce the Gauss sum
\[ \mf{G}=q^{-1/2}\sum_{x\in k_F}\Lambda(x^2) \in \C^\times. \]
With this notation at hand, we come to the definition of the three roots of unity.

\begin{equation} \label{eq:esr}
	\epsilon_{s,r}(\gamma_{<r}^g) = \prod_{\alpha \in \Gamma \lmod (R_{(r-\tx{ord}_{\gamma_{<r}^g})/2})_{\tx{symm},\tx{ram}}} \tx{sgn}_{F_{\pm\alpha}}(G_{\pm\alpha})(-\mf{G})^{f_\alpha}\tx{sgn}_{k_{F_\alpha}^\times}(t_{\alpha}).
\end{equation}
The product here runs over the $\Gamma$-orbits of symmetric ramified roots belonging to $R_{(r-\tx{ord}(\gamma_{<r}^g))/2}$. For each such $\alpha$, let $G_{\pm\alpha}$ be the subgroup of $G$ generated by the root subgroups for the two roots $\alpha$ and $-\alpha$. It is a semi-simple group of rank 1 defined over $F_{\pm\alpha}$, and $\tx{sgn}_{F_{\pm\alpha}}$ denotes its Kottwitz sign \cite{Kot83}, which equals $1$ if the $G_{\pm\alpha}$ is split and $-1$ if it is anisotropic. Furthermore, $f_\alpha$ is the degree of the field extension $k_{F_\alpha}/k_F$, and $\tx{sgn}_{k_{F_\alpha}^\times}$ is the quadratic character of the cyclic group $k_{F_\alpha}^\times$, onto which we can project the element $t_\alpha \in O_{F_\alpha}^\times$. Both $\mf{G}$ and $t_\alpha$ depend on the choice of $\Lambda$ (the latter through $X_{d-1}^*$) and it is easy to check that this dependence cancels out.

\begin{equation} \label{eq:er}
\epsilon^r(\gamma_{<r}^g) =\!\!\!\!\!\! \prod_{\alpha \in \Gamma \times \{\pm 1\} \lmod (R_{r/2})^\tx{sym}}\!\!\!\!\!\!\tx{sgn}_{k_{F_\alpha}^\times}(\alpha(\gamma_{<r}^g)) \cdot\!\!\!\!\!\! \prod_{\alpha \in \Gamma \lmod (R_{r/2})_{\tx{sym},\tx{unram}}}\!\!\!\!\!\! \tx{sgn}_{k_{F_\alpha}^1}(\alpha(\gamma_{<r}^g)).
\end{equation}
Here the superscript sym means that we are taking $\Gamma \times \{\pm 1\}$-orbits of asymmetric roots, while the subscripts sym,unram mean that we are taking $\Gamma$-orbits of roots that are symmetric and unramified. In the first product, we project $\alpha(\gamma_{<r}^g) \in O_{F_\alpha}^\times$ to $k_{F_\alpha}^\times$. In the second product, the $F_\alpha/F_{\pm\alpha}$-norm of the element $\alpha(\gamma_{<r}^g) \in O_{F_\alpha}^\times$ vanishes, because the root $\alpha$ is symmetric. The same is true for the projection of $\alpha(\gamma_{<r}^g)$ to $k_{F_\alpha}^\times$, because the symmetric root $\alpha$ is unramified. The group $k_{F_\alpha}^1$ of elements of $k_{F_\alpha}^\times$ with vanishing $k_{F_\alpha}/k_{F_{\pm\alpha}}$-norm is cyclic and we apply its quadratic character to the projection of $\alpha(\gamma_{<r}^g)$. Finally

\begin{equation} \label{eq:et}
	\tilde e(\gamma_{<r}^g) = \prod_{\alpha \in \Gamma \lmod (R_{(r-\tx{ord}_{\gamma_{<r}^g})/2})_\tx{sym}} (-1).
\end{equation}

Note that while the original definition of $\tilde e$ does not contain the subscript sym, we may restrict the product to symmetric roots by \cite[Remark 4.3.4]{DS}.

Each of these signs implicitly depends on $T$, but their product is independent of $T$. We will soon give an alternative formula for the product $\epsilon_{s,r} \cdot \tilde e$. About $\epsilon^r$ we will only need to know the following:

\begin{fct} \label{fct:eramweyl}
The function $\gamma \mapsto \epsilon^r(\gamma)$ is an $\Omega(T,G)(F)$-invariant character of $T(F)$.
\end{fct}
The observation that this function is a character was already made in \cite{DS} and will be very useful.

\subsection{Character values at shallow elements} \label{sub:charshallow}
In this subsection $F$ is a local field of odd residual characteristic that is not a bad prime for $G$. In the special case $\gamma=\gamma_{<r}$ formula \eqref{eq:char} specializes to
\begin{equation} \label{eq:charshallow1} \Phi_\pi(\gamma)=\phi_d(\gamma)\ssum{\substack{g \in J^d(F) \lmod G^d(F) / G^{d-1}(F)\\ \gamma^g \in G^{d-1}(F)}}\epsilon_{s,r}(\gamma^g) \epsilon^r(\gamma^g) \tilde e(\gamma^g) \cdot \Phi_{\pi_{d-1}}(\gamma^g).
\end{equation}
This formula still requires the assumption that $G^{d-1}$ is anisotropic modulo center. However, we expect that this condition is unnecessary. In this subsection, we will prove that this formula is valid without this condition, but under the stronger assumption on $\gamma$ that it is tame elliptic and topologically semi-simple modulo $Z(G)^\circ$. That is, we are assuming that $\gamma=\gamma_0$, which is stronger than $\gamma=\gamma_{<r}$, and that furthermore $\gamma$ belongs to a tame elliptic maximal torus, which is then equal to $J_d$ by definition. To remind ourselves that this is a tame elliptic maximal torus, let us write $S_\gamma$ for it.

We follow the proof of \cite[Theorem 6.4]{AS09}. We have the point $x \in \mc{B}^\tx{red}(G,F)$, which is called $\bar x$ in loc. cit. Fix a preimage $\dot x \in \mc{B}^\tx{enl}(G,F)$, which will serve as the point $x$ in loc. cit. Recall that the representation $\pi$ is compactly induced from a finite dimensional irreducible representation $\sigma$ of the group $K_\sigma=G^{d-1}(F)_x \cdot G(F)_{x,0+}$. We denote by $\theta_\sigma$ the character of this representation, and by $\dot\theta_\sigma$ the extension by zero of the function $\theta_\sigma$ to all of $G(F)$.

We claim that the function
\[ g \mapsto \dot\theta_\sigma({^g\gamma}) \]
on $G(F)/Z(F)$ is compactly supported. For this, note that because $\gamma$ is topologically semi-simple modulo $Z(G)^\circ$ all of its root values are topologically semi-simple. It follows from \cite[\S3.6]{Tits79} that the set of fixed points of $\gamma$ in $\mc{B}^\tx{red}(G,E)$, where $E$ is the splitting field of $S_\gamma$, is precisely the apartment $\mc{A}^\tx{red}(S_\gamma,E)$. Thus, the set of fixed points of $\gamma$ in $\mc{B}^\tx{red}(G,F)$ is a singleton set $\{x_\gamma\}$. The same is of course true for the element $^g\gamma$, which then has the unique fixed point $gx_\gamma$. Thus, unless $gx_\gamma=x$, the element $^g\gamma$ does not belong to $K_\sigma \subset G(F)_x$ and consequently $\dot\theta_\sigma({^g\gamma})$ vanishes. This function is thus supported on a single coset of $G(F)_x/Z(F)$ in $G(F)/Z(F)$, which is compact.

According to the Harish-Chandra integral character formula, we have
\[ \Theta_\pi(\gamma)=\frac{\tx{deg}(\pi)}{\tx{deg}(\sigma)}\phi_d(\gamma)\int_{G(F)/Z(F)}\int_K\dot\theta_\sigma({^{gc}\gamma})dcdg,\]
where $K$ is any compact open subgroup of $G(F)$ with Haar measure $dc$ normalized so that the volume of $K$ is equal to $1$. We can take for example $K=G(F)_{x_\gamma,0}$. Since the integrand is compactly supported as a function of $g$, we switch the two integrals and then remove the integral over $K$. We arrive at
\[ \Theta_\pi(\gamma)=\frac{\tx{deg}(\pi)}{\tx{deg}(\sigma)}\phi_d(\gamma)\sum_{g \in K_\sigma \lmod G(F)/S_\gamma(F)} \int_{K_\sigma g S_\gamma(F)/Z(F)} \dot\theta_\sigma({^{kgs}\gamma})dkds.\]
We have $\dot\theta_\sigma({^{kgs}\gamma})=\dot\theta({^g\gamma})$ which, as we already discussed, is zero unless $gx_\gamma=x$. Recall the subset $\mc{B}_r(\gamma)$ of \cite[Definition 9.5]{AS08}. By \cite[Lemma 9.6]{AS08} it is equal to $\mc{B}^\tx{enl}(S_\gamma,F)$, which is the preimage of $x_\gamma$ in $\mc{B}^\tx{enl}(G,F)$. Thus, if $gx_\gamma=x$ then $\dot x$ belongs to $\mc{B}_r({^g\gamma})$. Because of this, the rest of the argument in the proof of \cite[Theorem 6.4]{AS09} goes through:  \cite[Corollaries 4.5,4.6]{AS09} can now be used without the assumption that $G^{d-1}(F)/Z(F)$ is anisotropic, whose purpose was to guarantee, via \cite[Lemma 9.13]{AS08}, that $x \in \mc{B}_r({^g\gamma})$.

\subsection{Computation of $\tx{ord}_x(\alpha)$} \label{sub:ordx}

In this subsection $F$ is a local field of odd residual characteristic. The roots of unity occurring in the character formula \eqref{eq:char} of Adler-DeBacker-Spice depend on the sets $\tx{ord}_x(\alpha) \subset \R$ for the various roots $\alpha \in R(T,G)$. According to \cite[Corollary 3.1.9]{DS}, when $\alpha$ is symmetric there are only two possibilities for $\tx{ord}_x(\alpha)$, namely $e_\alpha^{-1}\Z$ and $e_\alpha^{-1}(\Z + \frac{1}{2})$, where $e_\alpha$ is the ramification degree of the extension $F_\alpha/F$. The key to our reinterpretation of these roots of unity is the exact computation of $\tx{ord}_x(\alpha)$, which is as follows.

\begin{pro} \label{pro:ordx} Let $T \subset G$ be a maximal torus with a tamely ramified splitting field $E/F$and let  $x \in \mc{B}(G,F) \cap \mc{A}(T,E)$. For any $\alpha \in R(T,G)_\tx{sym}$ we have
\[ \tx{ord}_x(\alpha) = \begin{cases}
e_\alpha^{-1}\Z,&\tx{if }\alpha\tx{ is ramified}\\
e_\alpha^{-1}\Z,&\tx{if }\alpha\tx{ is unramified and }f_{(G,T)}(\alpha)=+1\\
e_\alpha^{-1}(\Z+\frac{1}{2}),&\tx{if }\alpha\tx{ is unramified and }f_{(G,T)}(\alpha)=-1\\
\end{cases} \]
where $f_{(G,T)}(\alpha)$ is the toral invariant defined in \cite[\S3.1]{KalEpi}.
\end{pro}

The proof will occupy this subsection. The crucial step is the reduction of the proof to the case of semi-simple groups of rank 1.

Let $G_{\pm\alpha}$ be the subgroup of $G$ generated by the root subgroups for the roots $\alpha$ and $-\alpha$. It is a semi-simple group of absolute rank 1 and is defined over $F_{\pm\alpha}$. Its Lie-algebra is $\mf{g}_{\pm\alpha} = \mf{g}_{-\alpha} \oplus \mf{s}_\alpha \oplus \mf{g}_{\alpha}$, where $\mf{s}_\alpha$ is the 1-dimensional subspace of $\mf{g}$ spanned by the coroot $H_\alpha$. Let $S_\alpha \subset G_{\pm\alpha}$ be the maximal torus whose Lie-algebra is $\mf{s}_\alpha$. It is a 1-dimensional anisotropic torus defined over $F_{\pm\alpha}$ and split over $F_\alpha$. Let $x_{\pm\alpha} \in \mc{B}(G_{\pm\alpha},F_{\pm\alpha})$ be the unique point in $\mc{A}(S_\alpha,F_\alpha)^{\Gamma_{\pm\alpha}}$ (we are using here again \cite{Pr01}).

\begin{lem}
The filtration $\mf{g}_\alpha(F_\alpha) \cap \mf{g}(F_\alpha)_{x,r}$ coincides with the filtration $\mf{g}_\alpha(F_\alpha) \cap \mf{g}_{\pm\alpha}(F_\alpha)_{x_{\pm\alpha},r}$.
\end{lem}
\begin{proof}

Since $E/F_\alpha$ is tame we have $\mf{g}(F_\alpha)_{x,r} = \mf{g}(E)_{x,r} \cap \mf{g}(F_\alpha)$ for all $r \in \R$, and the same is true for $\mf{g}_{\pm\alpha}$. We may thus extend scalars to $E$ for the comparison of the filtrations. Consider the root datum $\tx{RD}_G := (T,\{U_{\beta}\}_{\beta \in R(T,G)})$ in the sense of \cite[\S6.1.1]{BT1}, where we have omitted the data $M_\beta$ from the notation because they are redundant in this case, see \cite[\S4.1.19(i)]{BT2}. The point $x \in \mc{B}(G,F) \subset \mc{B}(G,E)$ gives a valuation $\psi_x$ of $\tx{RD}_G$, consisting of functions $\psi_{x,\beta} : U_\beta(E) \rw \R \cup \{\infty\}$, one for each $\beta \in R(T,G)$, satisfying \cite[Definition 6.2.1]{BT1}. On the group $G_{\pm\alpha}$ we have the root datum $\tx{RD}_{G_{\pm\alpha}} := (S_\alpha,\{U_\beta\}_{\beta=\pm\alpha})$ and it is easy to see that the functions $\{\psi_{x,\alpha},\psi_{x,-\alpha}\}$ satisfy the conditions of \cite[Definition 6.2.1]{BT1} and hence form a valuation of $\tx{RD}_{G_{\pm\alpha}}$, which we shall call $\psi_{x,\pm\alpha}$. We claim that this valuation corresponds to a point in $\mc{A}(S_\alpha,E) \subset \mc{B}(G_{\pm\alpha},E)$. For this we must show that $\psi_{x,\pm\alpha}$ is equipollent to a Chevalley valuation of $\tx{RD}_{G_{\pm\alpha}}$ \cite[\S4.2.1]{BT2}. This follows from the fact that $\psi_x$ is equipollent to a Chevalley valuation of $\tx{RD}_G$. Indeed, the latter statement means by definition that there exists a system of isomorphisms $(x_\beta : \mb{G}_a \rw U_\beta)_{\beta \in R(T,G)}$ and an element $v \in X_*(T_\tx{ad}) \otimes \R$ with the following properties:
\begin{enumerate}
\item For all $\beta \in R(T,G)$ we have $[dx_\beta(1),dx_{-\beta}(1)]=H_\beta$ in $\mf{g}$;
\item For all $\beta,\gamma \in R(T,G)$ with $\beta+\gamma \in R(T,G)$ there exists $\epsilon_{\beta,\gamma} \in \{\pm1\}$ with $[dx_\beta(1),dx_\gamma(1)]=\epsilon_{\beta,\gamma}(r_{\beta,\gamma}+1) dx_{\beta+\gamma}(1)$, where $r_{\beta,\gamma}$ is the largest integer such that $\gamma-r\beta \in R(T,G)$;
\item For all $\beta \in R(T,G)$ and $t \in E$ we have $\psi_{x,\beta}(x_\beta(t))=\tx{ord}(t)+\<\beta,v\>$.
\end{enumerate}
Clearly then the system of isomorphisms $(x_\beta)_{\beta=\pm\alpha}$ and the valuation $\psi_{x,\pm\alpha}$ satisfy the same properties (the second being vacuous). We only need to show that in the third property we can replace $v \in X_*(T_\tx{ad}) \otimes \R$ with some $v_{\pm\alpha} \in X_*(S_\alpha) \otimes \R$. For this, we observe that the surjection $X^*(T_\tx{ad}) \otimes \R \rw X^*(S_\alpha) \otimes \R$ induced by the inclusion $S_\alpha \subset T_\tx{ad}$ has a natural section, sending the image of $\alpha$ under this surjection back to $\alpha$. This section is dual to a surjection $X_*(T_\tx{ad})\otimes \R \rw X_*(S_\alpha) \otimes \R$ and we let $v_{\pm\alpha}$ be the image of $v$ under this surjection. Then by definition $\<\alpha,v\>=\<\alpha,v_{\pm\alpha}\>$. This proves that $\psi_{x,\pm\alpha}$ is equipollent to a Chevalley valuation of $\tx{RD}_{G_{\pm\alpha}}$ and thus corresponds to a point in $\mc{A}(S_\alpha,E)$. Finally, because the point $x$, and hence the valuation $\psi_x$, are fixed by $\Gamma$, and in particular by $\Gamma_{\pm\alpha}$, the valuation $\psi_{x,\pm\alpha}$, and hence the corresponding point in $\mc{A}(S_\alpha,E)$, are also fixed by $\Gamma_{\pm\alpha}$. The torus $S_\alpha$ being anisotropic over $F_{\pm\alpha}$, the only point in $\mc{A}(S_\alpha,E)^{\Gamma_{\pm\alpha}}$ is $x_{\pm\alpha}$ and this implies that the point corresponding to $\psi_{x,\pm\alpha}$ is none other than $x_{\pm\alpha}$.
\end{proof}

According to this lemma, we can replace $G$ by $G_{\pm\alpha}$, $T$ by $S_\alpha$, and $x$ by $x_{\pm\alpha}$, in the computation of $\tx{ord}_x(\alpha)$. At the same time, it follows directly from the definition that we can make the same replacement in the computation of $f_{(G,T)}(\alpha)$. This reduces the proof to the case when the group $G$ is semi-simple of absolute rank 1. Such a group is a (necessarily inner) form of either $\tx{SL}_2$ or $\tx{PGL}_2$. Neither $\tx{ord}_x(\alpha)$ nor $f_{(G,T)}(\alpha)$ is affected by passing to an isogenous group, so we may assume that $G$ is an inner form of $\tx{SL}_2$. Then we have to contend with four cases -- $G$ is either split or not, and $S$ is either unramified or not. Rather than going through all four cases by hand, we will use the following lemma, which reduces to the cases where $G=\tx{SL}_2$. We formulate it in general, as we believe this makes the proof more transparent.

\begin{lem}
Let $\xi : G \rw G'$ be an inner twist and $S \subset G$ a tame elliptic maximal torus. Assume that the restriction of $\xi$ to $S$ is defined over $F$, and let $S' := \xi(S)$ and $\alpha' = \xi(\alpha)$. Then Proposition \ref{pro:ordx} is true for $(G,S,\alpha)$ if and only if it is true for $(G',S',\alpha')$.
\end{lem}
\begin{proof}
Let again $E/F$ is the tame finite Galois extension splitting $S$. Let $x$ and $x'$ be the unique $\Gamma$-fixed points in the reduced apartments of $S$ and $S'$ over $E$. Then $\xi : G \rw G'$ is an isomorphism defined over $E$ that restricts to an isomorphism $S \to S'$ defined over $F$. We will need to control three parameters: The failure of $\xi$ to send $x$ to $x'$, the failure of the isomorphism $\mf{g}_\alpha \rw \mf{g'}_{\alpha'}$ induced by $\xi$ to descend to $F_\alpha$, and the possible inequality of $f_{(G,S)}(\alpha)$ and $f_{G',S'}(\alpha')$.

By assumption for any $\sigma \in \Gamma$ there is $t'_\sigma \in S_\tx{ad}$ such that $\xi^{-1}\sigma(\xi)=\tx{Ad}(t'_\sigma)$. Then $t'_\bullet \in Z^1(\Gamma,S_\tx{ad})$. According to \cite[Theorem 1.9]{Ste65reg} the cohomology group $H^1(I,S_\tx{ad})$ vanishes and hence there exist $t_\bullet \in Z^1(\Gamma/I,S_\tx{ad}(F^u))$ and $t \in S_\tx{ad}$ so that $t'_\sigma = t_\sigma \cdot t \cdot \sigma(t)^{-1}$. Replacing $\xi$ by $\xi \circ \tx{Ad}(t)$ we obtain $\xi^{-1}\sigma(\xi)=\tx{Ad}(t_\sigma)$.

We have the isomorphism $\xi : \mc{A}(S,E) \rw \mc{A}(S',E)$. Let $v \in X_*(S_\tx{ad}) \otimes \R$ be the element satisfying $\xi(x+v)=x'$. Then the isomorphism $\mf{g}_\alpha(E) \rw \mf{g'}_{\alpha'}(E)$ induced by $\xi$ restricts for all $r \in \R$ to an isomorphism
\[ \xi: \mf{g}_\alpha(E)_{x+v,r} \rw \mf{g'}_{\alpha'}(E)_{x',r}. \]
This isomorphism is not necessarily equivariant for the action of $\Gamma_\alpha$, but rather satisfies $\sigma(\xi(X)) = \xi(\<\alpha,t_\sigma\> \sigma(X))$ for $X \in \mf{g}_\alpha(E)$ and $\sigma \in \Gamma_\alpha$. Now $\sigma \mapsto \<\alpha,t_\sigma\>$ is the image of $t_\sigma$ under
\[ Z^1(\Gamma/I,S_\tx{ad}(F^u)) \stackrel{\tx{Res}}{\lrw} Z^1(\Gamma_{\alpha}/I_\alpha,S_\tx{ad}(F_\alpha^u)) \stackrel{\alpha}{\lrw} Z^1(\Gamma_{\alpha}/I_\alpha,F_\alpha^{u,\times}). \]
Hilbert's theorem 90 implies that this cocycle takes values not just in $F_\alpha^{u,\times}$, but in $O_{F_\alpha^u}^\times$. The vanishing of $H^1(\Gamma_\alpha/I_\alpha,O_{F_\alpha^u}^\times)$ implies that there exists $u \in O_{F_\alpha^u}^\times$ such that the modified isomorphism
\[u\cdot \xi : \mf{g}_\alpha(E)_{x+v,r} \rw \mf{g'}_{\alpha'}(E)_{x',r} \]
is $\Gamma_\alpha$-equivariant, and hence descends to an isomorphism
\[ \mf{g}_{\alpha}(F_\alpha)_{x,r-\<\alpha,v\>} = \mf{g}_\alpha(F_\alpha)_{x+v,r} \rw \mf{g'}_{\alpha'}(F_\alpha)_{x',r}. \]
This implies $\tx{ord}_x(\alpha)+\<\alpha,v\>=\tx{ord}_{x'}(\alpha')$. In order to prove the lemma we must now compute $\<\alpha,v\>$ and relate it to the invariant $f_{(G,S)}(\alpha)$.

The isomorphism $\xi : \mc{A}(S,E) \rw \mc{A}(S',E)$ is not necessarily $\Gamma$-equivariant. Rather, it satisfies $\xi^{-1}\sigma(\xi)=v(t_\sigma)$, where $v(t_\sigma) \in X_*(S_\tx{ad}) \otimes \R$ is characterized by $\<\beta,v(t_\sigma)\>=-\tx{ord}(\beta(t_\sigma))$ for all $\beta \in R(S,G)$. Applying $\sigma \in \Gamma$ to the equation $\xi(x+v)=x'$ we obtain $v-\sigma(v)=v(t_\sigma)$ and hence $-\tx{ord}(\alpha(t_\sigma))=\<\alpha,v-\sigma(v)\>=\<\alpha-\sigma^{-1}(\alpha),v\>$. Choosing $\sigma \in \Gamma_{\pm\alpha} \sm \Gamma_\alpha$ we then obtain $\<\alpha,v\>=-\frac{1}{2}\tx{ord}(\alpha(t_\sigma))$. We now use that $\sigma \mapsto \alpha(t_\sigma)$ is the image of $t_\bullet$ under
\[ Z^1(\Gamma/I,S_\tx{ad}(F^u)) \stackrel{\tx{Res}}{\lrw} Z^1(\Gamma_{\pm\alpha}/I_{\pm\alpha},S_\tx{ad}(F_{\pm\alpha}^u)) \stackrel{\alpha}{\lrw} Z^1(\Gamma_{\pm\alpha}/I_{\pm\alpha},S_\alpha(F_{\pm\alpha}^u)), \]
where $S_\alpha$ is the 1-dimensional anisotropic torus defined over $F_{\pm\alpha}$ and split over $F_\alpha$ and $F_{\pm\alpha}^u$ denotes fixed subfield in $F^s$ of $I_{\pm\alpha}$. We have $S_\alpha(F_{\pm\alpha}^u)=S_\alpha(F_\alpha^u)^{I_{\pm\alpha}/I_\alpha}=[F_\alpha^{u,\times}]^{I_{\pm\alpha}/I_\alpha}$.

Now we distinguish two cases. If $\alpha$ is ramified, then $I_{\pm\alpha}/I_\alpha$ is of order 2 and its action on $F_\alpha^{u,x}$ is not the standard action, but the twist of the standard action by inversion. In other words, $[F_\alpha^{u,\times}]^{I_{\pm\alpha}/I_\alpha}$ is the kernel of the norm $F_\alpha^{u,\times} \to F_{\pm\alpha}^{u,\times}$. It follows that $\tx{ord}(\alpha(t_\sigma))=0$. If $\alpha$ is unramified, then $I_{\pm\alpha}/I_\alpha=\{1\}$ and $F_\alpha^u=F_{\pm\alpha}^u$. The inflation map $H^1(\Gamma_{\pm\alpha}/\Gamma_\alpha,S_\alpha(F_\alpha)) \to H^1(\Gamma_{\pm\alpha},S_\alpha(F^s))$, which is an isomorphism, factors as
\[ H^1(\Gamma_{\pm\alpha}/\Gamma_\alpha,S_\alpha(F_\alpha)) \to H^1(\Gamma_{\pm\alpha}/I_{\pm\alpha},S_\alpha(F_{\pm\alpha}^u)) \to H^1(\Gamma_{\pm\alpha},S_\alpha(F^s)) \]
and both arrows are isomorphisms. The value at $\sigma \in \Gamma_{\pm\alpha} \sm \Gamma_\alpha$ of any coboundary in the middle term is of the form $x\sigma(x)$ for some $x \in F_{\pm\alpha}^{u,\times}=F_\alpha^{u,\times}$ and its valuation belongs to $2\tx{ord}(F_\alpha^\times)$. This implies that $\tx{ord}(\alpha(t_\sigma)) \in \tx{ord}(c_\sigma)+2\tx{ord}(F_\alpha^\times) $ for any 1-cocycle $c_\bullet \in Z^1(\Gamma_{\pm\alpha}/I_{\pm\alpha},S_\alpha(F_{\pm\alpha}^u))$ that is cohomologous to $\alpha(t_\bullet)$. But if we take $c_\bullet \in Z^1(\Gamma_{\pm\alpha}/\Gamma,S_\alpha(F_\alpha))$, then \cite[Prop. 3.2.2(1)]{KalEpi} implies that $-\frac{1}{2}\tx{ord}(c_\sigma) \in \tx{ord}(F_\alpha^\times)$ if and only if $f_{(G,S)}(\alpha) = f_{G',S'}(\alpha')$. We conclude
\[ \<\alpha,v\> \in
\begin{cases}
\frac{1}{2}\tx{ord}(F_\alpha^\times) \sm \tx{ord}(F_\alpha^\times),&\tx{ if }\alpha\tx{ is unramified and }f_{(G,S)}(\alpha) \neq f_{G',S'}(\alpha')\\
\tx{ord}(F_\alpha^\times),&\tx{otherwise.}
\end{cases} \]
\end{proof}

This lemma reduces the proof to the case $G=\tx{SL}_2$ and $S$ an anisotropic maximal torus. Moreover, we are free to change $S$ within its stable class if we like. This case can be treated by a simple calculation as follows. We have $F_{\pm\alpha}=F$ and $F_\alpha/F_{\pm\alpha}$ is a quadratic extension that may be ramified or unramified. Let $\sigma \in \Gamma_{\pm\alpha}$ denote the non-trivial element and fix an element $\tau \in F_\alpha$ satisfying $\tau+\sigma(\tau)=0$. Set
\[ h = \begin{bmatrix} 1&-\frac{\tau^{-1}}{2}\\ \tau& \frac{1}{2} \end{bmatrix} \in G(F_\alpha). \]
If $T \subset G$ is the split diagonal torus, then $hTh^{-1}$ is stably conjugate to $S$, so we may assume that it is equal to $S$. It will be convenient to change coordinates by $\tx{Ad}(h)$ and represent $S$ as the diagonal torus in $G$. This comes at the expense of replacing the usual action $\sigma_G$ of $\sigma$ on $G(F_\alpha)$ given by applying $\sigma$ to the entries of the matrix representing a given element of $G(F_\alpha)$ by the more complicated action given by $\tx{Ad}(h^{-1}\sigma(h)) \rtimes \sigma_G$. A simple computation reveals
\[ h^{-1}\sigma(h) = \begin{bmatrix}0&\frac{\tau^{-1}}{2}\\ -2\tau&0\end{bmatrix}. \]
According to \cite[Corollary 3.1.8]{DS} we have $\tx{ord}_x(\alpha)=-\tx{ord}_x(-\alpha)$ and so we are free to choose either root of $S$ as the one we study. We take the root $\alpha$ whose root subspace is spanned by the element
\[ X_\alpha = \begin{bmatrix}0&1\\ 0&0 \end{bmatrix}. \]
Then we see $\tx{Ad}(h^{-1}\sigma(h)) \rtimes \sigma_G(X_\alpha) = -4\tau^2X_{-\alpha}$ and this implies $f_{(G,S)}(\alpha)=+1$. In order to understand the filtration $\mf{g}_\alpha(F_\alpha)_{x,r}$ we must compute the point $x \in \mc{A}(S,F_\alpha)^\Gamma$. Let $o \in \mc{A}(S,F_\alpha)$ be the point given by the pinning $X_\alpha$, and let $v \in X_*(S) \otimes \R$ be the element satisfying $o+v=x$. Applying $\tx{Ad}(h^{-1}\sigma(h)) \rtimes \sigma_G$ to this equation we see that $2v$ is equal to the translation on $\mc{A}(S,F_\alpha)$ effected by the action of $\alpha^\vee(\tau^{-1})$, and hence
\[ \<\alpha,v\> = \tx{ord}(\tau). \]
It follows that $\mf{g}_\alpha(F_\alpha)_{x,r}=\mf{g}_{\alpha}(F_\alpha)_{o,r-\tx{ord}(\tau)}$. Since the filtration $\mf{g}_\alpha(F_\alpha)_{o,r}$ has a break at $0$ and $\tx{ord}(\tau) \in \tx{ord}(F_\alpha^\times)$ we conclude that the filtration $\mf{g}_\alpha(F_\alpha)_{x,r}$ also has a break at zero.

\subsection{Definition of $\Delta_{II}^\tx{abs}$} \label{sub:delta2}

In this subsection $F$ is a local field of odd residual characteristic. In \cite[\S3.3]{LS87} Langlands and Shelstad define the term $\Delta_{II}$, which is a component of their transfer factor. It is associated to a connected reductive group defined over a local field, an endoscopic group, a maximal torus that is common to both groups, as well as $a$-data and $\chi$-data. In this subsection, we will introduce a slight variation of $\Delta_{II}$, which we will call $\Delta_{II}^\tx{abs}$. It will be associated to a connected reductive group defined over a local field, a maximal torus thereof, as well as $a$-data and $\chi$-data. We think of $\Delta_{II}^\tx{abs}$ as an absolute version of $\Delta_{II}$, in the precise sense that the original term $\Delta_{II}$ can be written as a quotient with numerator $\Delta_{II}^\tx{abs}$ for the reductive group and denominator $\Delta_{II}^\tx{abs}$ for its endoscopic group.

We begin by recalling the notions of $a$-data and $\chi$-data from \cite[\S2]{LS87}. A set of $a$-data consists of elements $a_\alpha \in F_\alpha^\times$, one for each $\alpha \in R(T,G)$, having the properties $a_{-\alpha}=-a_{\alpha}$ and $a_{\sigma(\alpha)}=\sigma(a_\alpha)$ for $\sigma \in \Gamma$. A set of $\chi$-data consists of characters $\chi_\alpha : F_\alpha^\times \rw \C^\times$, one for each $\alpha \in R(T,G)$, having the properties $\chi_{-\alpha}=\chi_\alpha^{-1}$,  $\chi_{\sigma(\alpha)}=\chi_\alpha \circ \sigma^{-1}$ for each $\sigma \in \Gamma$, and $\chi_\alpha|_{F_{\pm\alpha}^\times}=\kappa_\alpha$ whenever $\alpha \in R(T,G)_\tx{sym}$ and $\kappa_\alpha : F_\alpha^\times \rw \{\pm 1\}$ is the quadratic character associated to the quadratic extension $F_\alpha/F_{\pm\alpha}$.

Sets of $a$-data and $\chi$-data always exist, but there are rarely unique choices for them without further structure. It is clear from the definitions that one can choose $a_\alpha=1$ and $\chi_\alpha=1$ for asymmetric $\alpha \in R(T,G)$, although it is sometimes convenient not to do so. When $\alpha$ is symmetric and unramified, the character $\kappa_\alpha$ is unramified and there is a distinguished choice for $\chi_\alpha$, namely the unramified quadratic character of $F_\alpha^\times$. When $\alpha$ is symmetric and ramified, the character $\kappa_\alpha$ is ramified and the unramified quadratic character of $F_\alpha^\times$ is not a valid choice for $\chi_\alpha$. In this situation, under the assumption $p \neq 2$, there are exactly two tamely-ramified characters of $F_\alpha^\times$ that extend $\kappa_\alpha$. Their quotient (in either order) is the unramified quadratic character of $F_\alpha^\times$, and each of the two tame choices for $\chi_\alpha$ is characterized by the fact that its restriction to $O_{F_\alpha}^\times$ lifts the quadratic character of $k_{F_\alpha}^\times$ and its value on any uniformizer belongs to $\{i,-i\} \subset \C^\times$ if $-1$ is not a square in $F_\alpha$ and to $\{+1,-1\} \subset \C^\times$ otherwise. Regardless of the ramification of $F_\alpha/F_{\pm\alpha}$, it is often useful to allow $\chi_\alpha$ to have arbitrary depth.

We will call a set of $\chi$-data \emph{minimally ramified}, if $\chi_\alpha=1$ for asymmetric $\alpha$, $\chi_\alpha$ is unramified for unramified symmetric $\alpha$, and $\chi_\alpha$ is tamely ramified for ramified symmetric $\alpha$. As just discussed, different choices of minimally ramified $\chi$-data differ only at ramified symmetric roots $\alpha$, and only by the unramified sign character of $F_\alpha^\times$.

Given sets of $a$-data and $\chi$-data, we define
\[ \Delta_{II}^\tx{abs}[a,\chi] : T(F) \rw \C^\times, \qquad \gamma \mapsto \prod_{\substack{\alpha \in \Gamma \lmod R(T,G)\\ \alpha(\gamma) \neq 1}} \chi_\alpha\left(\frac{\alpha(\gamma)-1}{a_\alpha}\right). \]
We will now recall some results from \cite{LS87} about how this term changes when the $a$-data or $\chi$-data are changed. First, the $a$-data $(a_\alpha)_\alpha$ can only be replaced by $(b_\alpha\cdot a_\alpha)_\alpha$, where $b_\alpha \in F_\alpha^\times$ for $\alpha \in R(T,G)$ satisfies $b_{-\alpha}=b_\alpha$ and $\sigma(b_\alpha)=b_{\sigma(\alpha)}$ for all $\sigma \in \Gamma$.

\begin{lem} \label{lem:d2a}
\[ \Delta_{II}^\tx{abs}[ba,\chi](\gamma) = \Delta_{II}^\tx{abs}[a,\chi](\gamma) \cdot \prod_{\substack{\alpha \in \Gamma \lmod R(T,G)_\tx{sym} \\ \alpha(\gamma) \neq 1}} \kappa_\alpha(b_\alpha). \]
\end{lem}
\begin{proof}
Immediate.
\end{proof}
Second, the $\chi$-data $(\chi_\alpha)_\alpha$ can only be replaced by $(\zeta_\alpha \cdot \chi_\alpha)_\alpha$, where $\zeta_\alpha: F_\alpha^\times \rw \C^\times$ is a character for each $\alpha \in R(T,G)$ satisfying $\zeta_{-\alpha}=\zeta_\alpha^{-1}$, $\zeta_{\sigma\alpha} = \zeta_\alpha\circ\sigma^{-1}$ for all $\sigma \in \Gamma$ and, in the case of symmetric $\alpha$, $\zeta_\alpha|_{F_{\pm\alpha}^\times}=1$.

Let $O$ be an orbit of $\Gamma \times \{\pm 1\}$ in $R(T,G)$. We will define a character $\zeta_O : T(F) \rw \C^\times$ as follows. If $O$ consists of two distinct $\Gamma$-orbits, choose $\alpha \in O$ and let $\zeta_O = \zeta_\alpha \circ \alpha$. If $O$ consists of a single $\Gamma$-orbit, choose $\alpha \in O$ and let $\zeta_O$ be the composition
\[ T(F) \stackrel{\alpha}{\lrw} F_\alpha^1 \stackrel{\cong}{\llw} F_\alpha^\times/F_{\pm\alpha}^\times \stackrel{\zeta_\alpha}{\lrw} \C^\times, \]
where $F_\alpha^1$ is the group of elements of $F_\alpha^\times$ with vanishing $F_\alpha/F_{\pm\alpha}$-norm and the middle isomorphism sends $x \in F_\alpha^\times$ to $x/\tau(x) \in F_\alpha^1$, with $\tau \in \Gamma_{\pm\alpha}/\Gamma_\alpha$ being the non-trivial element. In both cases it is straightforward to check that $\zeta_O$ depends only on $O$ and not on the choice of $\alpha$. Taking the product over all orbits $O$ we obtain a character $\zeta_T : T(F) \rw \C^\times$.

\begin{lem} \label{lem:d2c}
\[ \Delta_{II}^\tx{abs}[a,\zeta\cdot\chi](\gamma) = \Delta_{II}^\tx{abs}[a,\chi](\gamma) \cdot \zeta_T(\gamma). \]
\end{lem}
\begin{proof}
The argument for this constitutes the proofs of \cite[Lemma 3.3.A, Lemma 3.3.D]{LS87}.
\end{proof}

\begin{lem} \label{lem:d2d} Let $\gamma \in T(F)_\tx{reg}$ be an element having a decomposition $\gamma = \gamma_{<r} \cdot \gamma_{\geq r}$ with $\gamma_{<r},\gamma_{\geq r} \in T(F)$ satisfying $\tx{ord}(\alpha(\gamma_{<r})-1) < r$ and $\tx{ord}(\alpha(\gamma_{\geq r})-1) \geq r$ for all $\alpha \in R(T,G)$. Assume that the $\chi$-data is tamely ramified, i.e. $\chi_\alpha|_{[F_\alpha^\times]_{0+}}=1$. Then
\[ \Delta_{II}^{\tx{abs},G}[a,\chi](\gamma) = \Delta_{II}^{\tx{abs},G}[a,\chi](\gamma_{<r}) \cdot \Delta_{II}^{\tx{abs},J}[a,\chi](\gamma_{\geq r}), \]
where $J=\tx{Cent}(\gamma_{<r},G)^\circ$ and the superscripts indicate the group relative to which the factor $\Delta_{II}^\tx{abs}$ is taken.
\end{lem}
\begin{proof}
We need to show that
\[ \chi_\alpha\left(\frac{\alpha(\gamma)-1}{a_\alpha}\right) = \begin{cases} \chi_\alpha\left(\frac{\alpha(\gamma_{<r})-1}{a_\alpha}\right),&\tx{ if }\alpha(\gamma_{<r}) \neq 1,\\
\chi_\alpha\left(\frac{\alpha(\gamma_{\geq r})-1}{a_\alpha}\right),&\tx{ if }\alpha(\gamma_{<r}) = 1.\\
\end{cases} \]
The case $\alpha(\gamma_{<r})=1$ is obvious. Assume now $\alpha(\gamma_{<r}) \neq 1$. Write $\alpha(\gamma_{<r})=1+x$ and $\alpha(\gamma_{\geq r})=1+y$ with $\tx{ord}(x) < r$ and $\tx{ord}(y) \geq r$. Then $\alpha(\gamma)-1=(x+1)(y+1)-1=x(1+y+yx^{-1})$. Since $1+y+yx^{-1} \in [F_\alpha^\times]_{0+}$ the proof is complete.
\end{proof}

We now introduce a weaker variant of the notion of $a$-data that can be used in conjunction with tame $\chi$-data and is sometimes more convenient. A \emph{mod-$a$-data} $\{(r_\alpha,\bar a_\alpha)\}$ is an assignment to each $\alpha \in R(T,G)$ of a real number $r_\alpha \in \R$ and a non-zero element $\bar a_\alpha \in [F_\alpha]_{r_\alpha}/[F_\alpha]_{r_\alpha+}$ such that $r_{\sigma\alpha}=r_\alpha$ and such that $\bar a_\alpha$ satisfies the same properties as for usual $a$-data. Given tame $\chi$-data and mod-$a$-data, we can choose an arbitrary lift $a_\alpha \in [F_\alpha]_{r_\alpha}$ for each $\bar a_\alpha$ and consider the function $\Delta_{II}^\tx{abs}[a,\chi]$ (even if the set of lifts $a_\alpha$ does not constitute $a$-data). This function is independent of the chosen lifts, because another choice would have the form $a_\alpha+a_\alpha'$, where $a_\alpha' \in [F_\alpha]_{r_\alpha+}$. Now $a_\alpha+a_\alpha' = a_\alpha b_\alpha$, where $b_\alpha=1+\frac{a_\alpha'}{a_\alpha}$ belongs to $[F_{\alpha}^\times]_{0+}$, and $\chi_\alpha$ restricts trivially to this group. We will denote the resulting function by $\Delta_{II}^\tx{abs}[\bar a,\chi]$.

\subsection{A formula for $\epsilon_{s,r} \cdot \tilde e$} \label{sub:signs}

In this subsection $F$ is a local field of odd residual characteristic. We will use the results of Subsections \ref{sub:ordx} and \ref{sub:delta2} to give a formula for the product of the two roots of unity $\epsilon_{s,r}(\gamma_{<r}^g) \cdot \tilde e(\gamma_{<r}^g)$.

Recall that we have fixed an additive character $\Lambda : F \rw \C^\times$ that is non-trivial on $O_F$ and trivial on $\mf{p}_F$. Recall also that the definition of the roots of unity depends on a tame maximal torus $T$ of $G^{d-1}$ containing $\gamma_{<r}^g$. We now choose $a$-data and $\chi$-data for $R(T,G)$ as follows. If $\alpha \in R(T,G^{d-1})$ or $\alpha(\gamma_{<r}^g)=1$, we leave the choice unspecified, as these roots will not contribute to the formula. For any other $\alpha \in R(T,G)$ we set $a_\alpha = \<H_\alpha,X_{d-1}^*\>$, and we take $\chi_\alpha$ to be the trivial character if $\alpha$ is asymmetric and the unramified quadratic character if $\alpha$ is symmetric and unramified. If $\alpha$ is symmetric and ramified, we choose among the two possible tamely-ramified characters by demanding
\begin{equation} \label{eq:chi} \chi_\alpha(-2a_\alpha^{-1})=f_{(G,T)}(\alpha)\lambda_{F_\alpha/F_{\pm\alpha}}(\Lambda\circ\tx{tr}_{F_{\pm\alpha}/F}),\end{equation}
where $f_{(G,T)}(\alpha)$ is the toral invariant \cite[\S3.1]{KalEpi} and $\lambda_{F_\alpha/F_{\pm\alpha}}$ is Langlands' constant \cite[Theorem 2.1]{LanArt}, \cite[\S1.5]{BH05b}. To see that this specifies a valid $\chi$-data, note that since $a_\alpha \in F_\alpha^\times$ is an element of trace zero, we have $\tx{ord}(a_\alpha) \in \tx{ord}(F_\alpha^\times) \sm \tx{ord}(F_{\pm\alpha}^\times)$. Thus the value of $\chi_\alpha(-2a_\alpha^{-1})$ distinguishes the two possible choices of $\chi_\alpha$. This value has to be a square root of $\kappa_\alpha(-1)$, which is indeed the case here. Finally, it is enough to check that $\chi_{\sigma\alpha}\circ\sigma = \chi_\alpha$ on the element $-2a_{\alpha}^{-1}$, where it is obvious.

Note that the choice of $\chi_\alpha$ depends only on $G$, $T$, and $\phi_{d-1}$, but not on $\Lambda$, because the dependence of the right side of \eqref{eq:chi} on $\Lambda$ is canceled by the dependence of $a_\alpha$ on $\Lambda$.

The main step in our reinterpretation of the character formula is the following expression for the product $\epsilon_{s,r}(\gamma_{<r}^g)\tilde e(\gamma_{<r}^g)$.

\begin{lem} The product $\epsilon_{s,r}(\gamma_{<r}^g)\tilde e(\gamma_{<r}^g)$ is equal to
\[  \prod_{\alpha \in \Gamma \lmod (R_{\gamma_{<r}^g})_\tx{sym}} f_{(G,T)}(\alpha) \lambda_{F_\alpha/F_{\pm\alpha}}(\Lambda\circ\tx{tr}_{F_{\pm\alpha}/F})^{-1} \chi_\alpha\left(\frac{\alpha(\gamma_{<r}^g)-1}{a_\alpha}\right). \]
\end{lem}
\begin{proof}
According to \eqref{eq:esr} and \eqref{eq:et} we can write $\epsilon_{s,r}(\gamma_{<r}^g)\tilde e(\gamma_{<r}^g)$ as the product of
\[ \prod_{\alpha \in \Gamma \lmod (R_{(r-\tx{ord}_{\gamma_{<r}^g})/2})_{\tx{symm},\tx{ram}}} \tx{sgn}_{F_{\pm\alpha}}(G_{\pm\alpha})(-1)^{f_\alpha+1}\mf{G}^{f_\alpha}\tx{sgn}_{k_{F_\alpha}^\times}(t_{\alpha}) \]
with
\[ \prod_{\alpha \in \Gamma \lmod (R_{(r-\tx{ord}_{\gamma_{<r}^g})/2})_{\tx{sym},\tx{unram}}} (-1). \]
We consider the contribution of an individual symmetric $\alpha \in R_{\gamma_{<r}^g}$. If $\alpha$ is unramified, then $\chi_\alpha$ is unramified and $\tx{ord}(\<H_\alpha,X_{d-1}^*\>)=-r \in \tx{ord}(F_\alpha^\times)$, so we have
\[ \chi_\alpha\left(\frac{\alpha(\gamma_{<r}^g)-1}{\<H_\alpha,X_{d-1}^*\>}\right) = (-1)^{e_\alpha(\tx{ord}(\alpha(\gamma_{<r}^g)-1)-r)}, \]
while $\lambda_{F_\alpha/F_{\pm\alpha}}(\Lambda\circ\tx{tr}_{F_{\pm\alpha}/F})=-1$ according to \cite[Lemma 1.5]{BH05b}. The total contribution of $\alpha$ to the right hand side of the equation of the lemma is thus
\[ f_{(G,T)}(\alpha)\cdot (-1)^{e_\alpha(\tx{ord}(\alpha(\gamma_{<r}^g)-1)-r+1)}.\]
According to Proposition \ref{pro:ordx} this expression is equal to $-1$ precisely when $\alpha \in R_{(r-\tx{ord}_{\gamma_{<r}^g})/2}$. The contributions of $\alpha$ to both sides of the equation of the lemma are thus equal.

Now let $\alpha$ be ramified. Since $\alpha(\gamma_{<r}^g) \in F_\alpha^\times$ is an element whose $F_\alpha/F_{\pm\alpha}$-norm vanishes, $\tx{ord}(\alpha(\gamma_{<r}^g)-1)$ is either zero or belongs to $\tx{ord}(F_\alpha^\times) \sm \tx{ord}(F_{\pm\alpha}^\times)$. At the same time, $\<H_\alpha,X_{d-1}^*\> \in F_\alpha^\times$ is an element whose $F_\alpha/F_{\pm\alpha}$-trace vanishes, so $-r = \tx{ord}(\<H_\alpha,X_{d-1}^*\>) \in \tx{ord}(F_\alpha^\times) \sm \tx{ord}(F_{\pm\alpha}^\times)$. It follows from Proposition \ref{pro:ordx} that $\alpha \in R_{(r-\tx{ord}_{\gamma_{<r}^g})/2}$ if and only if $\tx{ord}(\alpha(\gamma_{<r}^g)-1) \neq 0$. Assume first that this is the case. Then $\alpha$ contributes to both sides of the equation of the lemma. For the contribution of the left side, we note that by the theorem of Hasse-Davenport the term $(-1)^{f_\alpha+1}\mf{G}^{f_\alpha}$ is equal to the Gauss sum
\[ q^{f_\alpha/2}\sum_{x\in k_{F_\alpha}}\Lambda(\tx{tr}_{k_{F_\alpha}/k_F}(x^2)). \]
Since the character $\Lambda \circ \tx{tr}_{F_{\pm\alpha}/F} : F_{\pm\alpha} \rw \C^\times$ induces on $k_{F_\alpha}=k_{F_{\pm\alpha}}$ the character $x \mapsto \Lambda(\tx{tr}_{k_{F_\alpha}/k_F}(e_{\pm\alpha}x))$, the latter Gauss sum is equal by \cite[Lemma 1.5]{BH05b} to
\[ \lambda_{F_\alpha/F_{\pm\alpha}}(\Lambda \circ \tx{tr}_{F_{\pm\alpha}/F}) \cdot \kappa_\alpha(e_{\pm\alpha}) = \lambda_{F_\alpha/F_{\pm\alpha}}(\Lambda \circ \tx{tr}_{F_{\pm\alpha}/F})^{-1} \cdot \kappa_\alpha(-e_{\pm\alpha}). \]
Next, using that $\chi_\alpha$ is trivial on $N_{F_\alpha/F_{\pm\alpha}}$-norms and that $\<H_\alpha,X_{d-1}^*\> \in F_\alpha$ is an element whose $F_\alpha/F_{\pm\alpha}$-trace vanishes, we see
\[ \tx{sgn}_{k_{F_\alpha}^\times}(t_\alpha) = \chi_\alpha(t_\alpha) = \kappa_\alpha(e_{\pm\alpha})\kappa_\alpha(-1)\chi_\alpha\left(\frac{\alpha(\gamma_{<r}^g)-1}{\<H_\alpha,X_{d-1}^*\>}\right).\]
These computations and the fact that $f_{(G,T)}(\alpha)=\tx{sgn}_{F_{\pm\alpha}}(G_{\pm\alpha})$ imply that the contributions of $\alpha$ to the both sides of the equation of the lemma agree.

Assume now that $\tx{ord}(\alpha(\gamma_{<r}^g)-1)$ is zero, so that $\alpha \notin R_{(r-\tx{ord}_{\gamma_{<r}^g})/2}$ and thus $\alpha$ does not contribute to the left side of the equation of the lemma. To compute its contribution to the right side, we first notice that $\alpha(\gamma_{<r}^g) \in -1 + \mf{p}_{F_\alpha}$ and hence $\alpha(\gamma_{<r}^g)-1 = (-2)u$ for some $u \in 1+\mf{p}_{F_\alpha}$. Since $u$ is in the kernel of $\chi_\alpha$ we conclude that the contribution of $\alpha$ to the right side of the equation of the lemma is $1$ according to \eqref{eq:chi}.
\end{proof}

Using Kottwitz's result \cite[\S3.5]{KalEpi} on the relationship between $\epsilon$-factors and Weil constants, as well as the factor $\Delta_{II}^\tx{abs}$ defined in Subsection \ref{sub:delta2}, we can restate this lemma as follows. Let $T_{G^d}$ denote the minimal Levi subgroup of the quasi-split inner form of $G^d$, and let $T_{J^d}$ denote the minimal Levi subgroup of the quasi-split inner form of $J^d$.

\begin{cor}
The product $\epsilon_{s,r}(\gamma_{<r}^g)\tilde e(\gamma_{<r}^g)$ is equal to
\[ \frac{e(G^d)e(J^d)}{e(G^{d-1})e(J^{d-1})}\frac{\epsilon_L(X^*(T_{G^d})_\C-X^*(T_{J^d})_\C,\Lambda)}{\epsilon_L(X^*(T_{G^{d-1}})_\C-X^*(T_{J^{d-1}})_\C,\Lambda)}\frac{\Delta_{II}^{\tx{abs},G^d}[a,\chi](\gamma_{<r}^g)}{\Delta_{II}^{\tx{abs},G^{d-1}}[a,\chi](\gamma_{<r}^g)} \]
\end{cor}
\begin{proof}
This follows immediately from \cite[Corollary 3.5.2]{KalEpi} and the additivity of $\epsilon_L$ in degree zero. Note that there is a typo in loc.cit: $\tx{tr}_{F_\alpha/F_{\pm\alpha}}$ should read $\tx{tr}_{F_{\pm\alpha}/F}$. Note also that $J^d \cong \tx{Ad}(g^{-1})J^d$ and $J^{d-1} \cong \tx{Ad}(g^{-1})J^{d-1}$ as reductive groups over $F$.
\end{proof}

We remark here that this expression does not depend on the choice of $\Lambda$, because both $\epsilon_L$ and the $a$-data $a_\alpha=\<H_\alpha,X_{d-1}^*\>$ depend on $\Lambda$ in a parallel way. Thus we may from now on use an arbitrary additive character $\Lambda$, i.e. remove the condition on its depth.

A slight variant of this corollary will also be useful later when we study $L$-packets. It involves the following modified choice of $\chi$-data, where we use
\begin{equation} \label{eq:chi'}
\chi'_\alpha(-2a_\alpha^{-1})=\lambda_{F_\alpha/F_{\pm\alpha}}(\Lambda\circ\tx{tr}_{F_{\pm\alpha}/F})
\end{equation}
instead of \eqref{eq:chi}. Then $\chi'_\alpha$ is a valid set of $\chi$-data for the same reasons that $\chi_\alpha$ was. The usefulness of $\chi_\alpha'$ comes from the fact that it depends only on the torus $T$ and the character $\phi_{d-1}$, but not on the group $G$ in the sense that it is insensitive to replacing $T$ by a stably conjugate torus in an inner form of $G$. The relationship between the two $\chi$-data can be expressed by
\[ \chi'_\alpha(x) = \chi_\alpha(x) \epsilon_\alpha(x), \]
where $\epsilon_\alpha : F_\alpha^\times \rw \C^\times$ is the trivial character unless $\alpha$ is symmetric and ramified, in which case it is given by $\epsilon_\alpha(x)=f_{(G,T)}(\alpha)^{e_\alpha\tx{ord}(x)}$. The collection $\epsilon_\alpha$ satisfies the conditions of \cite[Corollary 2.5.B]{LS87} and thus provides a character $\epsilon_{f,r} : T(F) \rw \C^\times$. This character is similar, but not the same, as the one introduced in \cite[\S3.6]{KalEpi}.

\begin{lem} \label{lem:efr} For all $\gamma \in T(F)$ we have
\[ \epsilon_{f,r}(\gamma) = \prod_{\substack{\alpha \in R(T,G)_\tx{sym,ram}/\Gamma\\ \alpha(\gamma) \neq 1 \\ \tx{ord}(\alpha(\gamma)-1) = 0}} f_{(G,T)}(\alpha). \]
\end{lem}
\begin{proof}
The argument is a slight variant of the proof of \cite[Lemma 3.6.1]{KalEpi}. Following the argument in the proof of \cite[Lemma 3.5.A]{LS87} and using that $\epsilon_\alpha$ is trivial unless $\alpha$ is symmetric and ramified we arrive at
\[ \epsilon_{f,r}(\gamma) = \prod_{\alpha \in R(T,G)_\tx{sym,ram}} \epsilon_\alpha(\delta^\alpha)^{-1}, \]
where $\delta^\alpha \in F_\alpha^\times$ is any element satisfying $\delta^\alpha/\sigma(\delta^\alpha)=\alpha(\gamma)$ for the non-trivial element $\sigma \in \Gamma_{\pm\alpha}/\Gamma_\alpha$. If $\tx{ord}(\alpha(\gamma)-1)>0$ we have $\tx{ord}(\delta^\alpha-1)>0$ and hence $\epsilon_\alpha(\delta^\alpha)=1$. If $\tx{ord}(\alpha(\gamma)-1)=0$ then the fact that $\alpha$ is symmetric and ramified implies $\alpha(\gamma) \in -1 + \mf{p}_{F_\alpha}$. Writing $\alpha(\gamma)=-1 \cdot u$ with $u \in 1+\mf{p}_{F_\alpha}$ we may choose $\delta^\alpha=\omega \cdot v$ with $v \in 1+\mf{p}_{F_\alpha}$ satisfying $v/\sigma(v)=u$ and $\omega \in F_\alpha^\times$ a uniformizer with $\sigma(\omega)=-\omega$. Then $\epsilon_\alpha(\delta^\alpha)=f_{(G,T)}(\alpha)$.
\end{proof}

\begin{fct} \label{fct:eframweyl}
The character $\epsilon_{f,r}$ is $N(T,G)(F)$-invariant.
\end{fct}
\begin{proof}
This follows from the $N(T,G)(F)$-invariance of the function $f_{(G,T)}$.
\end{proof}

\begin{cor} \label{cor:prodeps}
 The product $\epsilon_{s,r}(\gamma_{<r}^g)\tilde e(\gamma_{<r}^g)$ is equal to
\[  \frac{\epsilon_{f,r}^{G^d}(\gamma_{<r}^g)e(G^d)e(J^d)}{\epsilon_{f,r}^{G^{d-1}}(\gamma_{<r}^g)e(G^{d-1})e(J^{d-1})} \frac{\epsilon_L(X^*(T_{G^d})_\C-X^*(T_{J^d})_\C,\Lambda)}{\epsilon_L(X^*(T_{G^{d-1}})_\C-X^*(T_{J^{d-1}})_\C,\Lambda)} \cdot \frac{\Delta_{II}^{\tx{abs},G^d}[a,\chi'](\gamma_{<r}^g)}{\Delta_{II}^{\tx{abs},G^{d-1}}[a,\chi'](\gamma_{<r}^g)}. \]
\end{cor}

Before going further, it will be useful to express the $a$-data $a_\alpha = \<H_\alpha,X_{d-1}^*\>$ in a way that does not reference the structure of the $p$-adic group $G$. In fact, since we are using tame $\chi$-data, it will be enough to specify mod-$a$-data. For this, we consider the character
\[ F_\alpha^\times \rw \C^\times,\quad x \mapsto \phi_{d-1}(N_{F_\alpha/F}(\alpha^\vee(x))), \]
where $N_{F_\alpha/F} : T(F_\alpha) \rw T(F)$ is the norm map. According to \cite[Lemma 2.2.1]{KalEpi}, restriction to $[F_\alpha^\times]_r$ provides a non-trivial character $[F_\alpha^\times]_r/[F_\alpha^\times]_{r+} \rw \C^\times$. At the same time, we have the character
\[ \Lambda \circ \tx{tr}_{F_\alpha/F} : F_\alpha \rw \C^\times,\]
which factors through $[F_\alpha]_0/[F_\alpha]_{0+}$. We have the isomorphism
\[ X \mapsto X+1 : [F_\alpha]_r/[F_\alpha]_{r+} \rw [F_\alpha^\times]_r/[F_\alpha^\times]_{r+}, \]
which is a truncated version of the exponential map. The equation
\begin{equation} \label{eq:adata} \phi_{d-1}(N_{F_\alpha/F}(\alpha^\vee(X+1))) = \Lambda(\tx{tr}_{F_\alpha/F}(\bar a_\alpha X)), \end{equation}
characterizes the image $\bar a_\alpha$ of $\<H_\alpha,X_{d-1}^*\>$ in $[F_\alpha]_{-r}/[F_\alpha]_{-r+}$.

\begin{fct} \label{fct:chiweyl} The $a$-data $a_\alpha=\<H_\alpha,X_{d-1}^*\>$, and hence the $\chi$-data $\chi'$, are invariant under the action of $\Omega(T,G^{d-1})(F)$.
\end{fct}
\begin{proof} For the $a$-data this follows from the fact that $X_{d-1}^*$ belongs to the dual Lie algebra of the center of $G^{d-1}$. For the $\chi$-data this follows from its definition \eqref{eq:chi'}.
\end{proof}

\subsection{The characters of toral supercuspidal representations} \label{sub:toral}

In this subsection $F$ is a local field of odd residual characteristic that is not a bad prime for $G$.

The supercuspidal representation $\pi$ is called toral if it arises from a Yu-datum of the form $((S,G),1,(\phi,1))$, where $G^{d-1}=S$ is an elliptic maximal torus and $\phi=\phi_{d-1}$ is a generic character of $S(F)$ of positive depth. In this special case, the Adler-DeBacker-Spice character formula \eqref{eq:char} applies to all regular semi-simple $\gamma = \gamma_{<r}\cdot \gamma_{\geq r} \in G(F)$, because the compactness assumption is automatically satisfies. Moreover, the formula of Corollary \ref{cor:prodeps} simplifies, because we have $J^{d-1}=G^{d-1}=S$. Thus we obtain

\begin{cor} \label{cor:prodepstoral}
The product $\epsilon_{s,r}(\gamma_{<r}^g)\tilde e(\gamma_{<r}^g)$ is equal to
\[  \epsilon_{f,r}(\gamma_{<r}^g)e(G)e(J) \epsilon_L(X^*(T_{G})_\C-X^*(T_{J})_\C,\Lambda) \cdot \Delta_{II}^{\tx{abs}}[a,\chi'](\gamma_{<r}^g). \]
\end{cor}

Combining this with \eqref{eq:char}, and setting $\theta=\phi_{d-1} : S(F) \rw \C^\times$, we arrive at

\begin{cor} \label{cor:newchartoral} The value of the normalized character $\Phi_\pi$ at the element $\gamma=\gamma_{<r} \cdot \gamma_{\geq r}$ is given as the product
\begin{eqnarray*}
&&e(G)e(J)\epsilon_L(X^*(T_{G})_\C-X^*(T_{J})_\C,\Lambda)\\
&\cdot&\sum_{\substack{g \in J(F) \lmod G(F) / S(F)\\ \gamma_{<r}^g \in S(F)}}  \Delta_{II}^\tx{abs}[a,\chi'](\gamma_{<r}^g)\epsilon_{f,r}(\gamma_{<r}^g) \epsilon^r(\gamma_{<r}^g) \theta(\gamma_{<r}^g) \hat\iota_{\mf{j},{^gX^*}}(\log(\gamma_{\geq r}))
\end{eqnarray*}
\end{cor}

\subsection{Character values of regular supercuspidal representations at shallow elements: depth zero} \label{sub:depthzero}
In this subsection $F$ is a local field of odd residual characteristic.

\begin{lem} \label{lem:dze} Let $S$ be a maximally unramified maximal torus of $G$ and let $T$ be the minimal Levi subgroup of the quasi-split inner form of $G$. If $\Lambda : F \to \C^\times$ is a character of depth zero, then
\[ \epsilon(X^*(S)_\C - X^*(T)_\C,\Lambda) = (-1)^{r_S-r_T}, \]
where $r_S$ and $r_T$ are the split ranks of $S$ and $T$ respectively.
\end{lem}
\begin{proof}
By Lemma \ref{lem:qstrans} the torus $S$ transfers to the quasi-split inner form of $G$. The statement we are proving is invariant under replacing $G$ by its quasi-split inner form, so we may now assume $G$ is quasi-split. We may also assume that $G$ is simply connected. Let $o \in \mc{A}(T,F)$ be the superspecial vertex associated to a $\Gamma$-invariant pinning. By Lemma \ref{lem:super} we may further replace $S$ by a stable conjugate so that $o$ the unique point in $\mc{B}(G,F) \cap \mc{A}(S,F^u)$. In fact, from the proof of that Lemma we see that $S$ and $T$ are conjugate under $G(F^u)_{o,0}$.

We will use Kottwitz's result \cite[Corollary 3.5.2]{KalEpi} to compute the left hand side. It is the formula
\[ \epsilon(X^*(S)_\C - X^*(T)_\C,\Lambda) = \prod_{\alpha \in R(S,G)_\tx{sym}/\Gamma} f_{(G,S)}(\alpha)\lambda_{F_\alpha/F_{\pm\alpha}}(\Lambda\circ\tx{tr}_{F_{\pm\alpha}/F}). \]
Here $f_{(G,S)}$ is the toral invariant of $S \subset G$ \cite[\S3]{KalEpi}, and $\lambda_{F_\alpha/F_{\pm\alpha}}$ is the Langlands constant \cite[\S1.5]{BH05b}. Since $S$ is maximally unramified, each quadratic extension $F_\alpha/F_{\pm\alpha}$ is unramified, so by \cite[Lemma 1.5]{BH05b} the corresponding Langlands constant is $-1$. Moreover, according to Proposition \ref{pro:ordx}, $f_{(G,S)}(\alpha)=+1$ if and only if $0 \in \tx{ord}_o(\alpha)$. According to \cite[Remark 3.1.4]{DS} we may extend scalars to $F^u$ before computing $\tx{ord}_o(\alpha)$. Since $S$ and $T$ are conjugate under $G(F^u)_{o,0}$ we may thus replace $S$ by $T$ in the computation of $\tx{ord}_o(\alpha)$.

We are now interested in the question of whether there is an element of $\mf{g}(E)_\alpha$ whose valuation with respect to $o$ is zero, and which is fixed by the action of $\tx{Gal}(E/F^u_\alpha)$, where $E/F^u$ is the splitting field of $T$. Consider the simple component of the root system $R(T,G)$ to which $\alpha$ belongs. Then $\tx{Gal}(E/F^u_\alpha)$ is a cyclic group preserving this component and acting on it by a pinned automorphism that fixes the root $\alpha$. Let's call this automorphism $\theta$. It preserves the line $\mf{g}_\alpha(E)$. The fixed pinning provides a pair of elements $\{X,-X\} \subset \mf{g}_\alpha(E)$. Both of these elements have $o$-valuation equal to $0$. Now $\theta(X)=\zeta \cdot X$, where $\zeta \in F^{u,\times}$ is a root of unity of order divisible by the order of $\theta$. If $\zeta=1$, then $X \in \mf{g}_\alpha(F^u_\alpha)$ and we conclude that $0 \in \tx{ord}_x(\alpha)$. If $\zeta \neq 1$, then $\varpi X \in \mf{g}_\alpha(F^u_\alpha)$, where $\varpi \in E$ is an element of valuation $(\tx{ord}(\theta)e_\alpha)^{-1}$ and then $0 \notin \tx{ord}_x(\alpha)$. A direct examination of the simple root systems shows that $\zeta=1$ unless $\alpha$ belongs to a simple component of type $A_{2n}$ and $\theta$ is the non-trivial pinned automorphism, in which case $\zeta=-1$.

Returning to the original torus $S$, we can interpret this as follows. Let $S' \subset S$ be the maximal unramified subtorus. Let $R(S',G)$ be the corresponding relative root system. It need not be reduced. The fibers of the map $R(S,G) \to R(S',G)$ induced by the inclusion $S' \to S$ are precisely the inertial orbits in $R(S,G)$. This map then sets up a bijection between the $\Gamma$-orbits in $R(S,G)$ and the Frobenius-orbits in $R(S',G)$ and this bijection restricts to a bijection between the symmetric orbits. A root $\alpha \in R(S,G)$ restricts to a divisible root in $R(S',G)$ if and only if it belongs to a copy of $A_{2n}$ and $\tx{Gal}(E/F^u_\alpha)$ acts non-trivially on that copy. If such a root is symmetric, we have $f_{(G,S)}(\alpha)\lambda_{F_\alpha/F_{\pm\alpha}}(\Lambda\circ\tx{tr}_{F_{\pm\alpha}/F})=1$, because both factors are equal to $-1$. For any other symmetric root, we have $f_{(G,S)}(\alpha)\lambda_{F_\alpha/F_{\pm\alpha}}(\Lambda\circ\tx{tr}_{F_{\pm\alpha}/F})=-1$, because the first factor is equal to $1$ and the second is equal to $-1$. With this, Kottwitz's formula above becomes
\[ \epsilon(X^*(S)_\C - X^*(T)_\C,\Lambda) = (-1)^{\#R(S',G)_\tx{nd,sym}/\tx{Fr}}, \]
where the subscript ``nd'' denotes the set of non-divisible roots. On the other hand, we have
\[ (-1)^{r_S-r_T}=(-1)^{\tx{dim}([X^*(S)_\C]^I)^\tx{Fr}-\tx{dim}([X^*(T)_\C]^I)^\tx{Fr}}=(-1)^{\tx{dim}[X^*(S)_\C]_I^\tx{Fr}-\tx{dim}[X^*(T)_\C]_I^\tx{Fr}}. \]
Since $S$ is maximally unramified, the $I$-modules $X_*(S)_\C$ and $X_*(T)_\C$ are equal. Moreover, the action of Frobenius on $[X_*(S)_\C]_I$ is the twist of the action of Frobenius on $[X_*(T)_\C]_I$ by an unramified 1-cocycle $w_\sigma$. Letting $V=[X_*(T)_\C]_I$, $\phi$ be the automorphism of $V$ by which Frobenius acts, and $w \in \Omega(T,G)$ the value of $w_\sigma$ at the Frobenius element, we get
\[ (-1)^{r_S-r_T}=(-1)^{\tx{dim}V^\phi-\tx{dim}V^{w\phi}}. \]
Both $\phi$ and $w$ are of finite order and preserve a $\Q$-structure on $V$, so their eigenvalues are either $+1$, $-1$, or pairs of conjugate non-real roots of unity. From this we see
\[ (-1)^{r_S-r_T}=\tx{det}(\phi|V)^{-1}\tx{det}(w\phi|V)=\tx{det}(w|V). \]
Note that $V$ is the vector space in which the root system $R(S',G)_\tx{nd}$ resides, and in fact is spanned by that root system, because $X^*(T)_\C$ is spanned by $R(T,G)$. Thus $\tx{det}(w|V)=(-1)^{l(w)}$. According to the argument of \cite[Lemma 4.0.7]{KalECI}, $(-1)^{l(w)}$ is equal to $(-1)^N$, where $N$ is the number of symmetric Frobenius orbits in $R(S',G)_\tx{nd}$.
\end{proof}

For a maximally unramified maximal torus $S \subset G$ all symmetric roots in $R(S,G)$ are unramified. We can thus fix unramified $\chi$-data for $R(S,G)$ and we can fix mod-$a$-data consisting of units, i.e. non-zero elements of $[F_\alpha]_0/[F_\alpha]_{0+}$.

\begin{pro} \label{pro:charshallowdz} Let $S \subset G$ be a maximally unramified maximal torus, $\theta : S(F) \to \C^\times$ a regular depth-zero character, and $\pi_{(S,\theta)}$ the corresponding regular depth-zero supercuspidal representation as in Subsection \ref{sub:rdz}. If $r \in G(F)$ is a regular topologically semi-simple element belonging to an elliptic maximally unramified maximal torus, then the character of $\pi_{(S,\theta)}$ at $r$ is zero, unless $r$ is (conjugate to) an element of $S(F)$, in which case it is given by
\[ e(G)\epsilon(X^*(T)_\C-X^*(S)_\C,\Lambda)\sum_{w \in N(S,G)(F)/S(F)}\Delta_{II}^\tx{abs}[a,\chi](r^w) \theta(r^w), \]
where $\chi$ is any unramified $\chi$-data and $a$ is any mod-$a$-data consisting of units.
\end{pro}
\begin{proof}
Recall from Lemma \ref{lem:rdzconst} that $\pi_{(S,\theta)}=\tx{c-Ind}_{S(F)G(F)_{x,0}}^{G(F)}\tilde\kappa_{(S,\theta)}$. We will perform this induction in stages, where we let $\dot\kappa$ be the induction of $\tilde\kappa$ to $G(F)_x$. Since $S(F)G(F)_{x,0}$ is a subgroup of $G(F)_x$ of finite index, $\dot\kappa$ is still finite dimensional. We compute the character of $\pi_{(S,\theta)}$ in terms of that of $\dot\kappa$ by means of Harish-Chandra's integral character formula \cite[\S9.1]{DR09} and we obtain
\[ \frac{\tx{deg}(\pi;dg)}{\tx{deg}(\dot\kappa)}\int_{G(F)/A(F)}\int_K \dot\chi_{\dot\kappa_{(S,G)}}({^{gk}r})dkdg/dz\]
where $K$ is any compact open subgroup of $G(F)$ with Haar measure $dk$ of normalized volume $1$ and $A$ is the maximal split torus in the center of $G$. Just as in Subsection \ref{sub:charshallow} we can argue that the function $g \mapsto \dot\chi_{\dot\kappa_{(S,G)}}({^{gk}}r)$ is compactly supported modulo center and thus remove the integral over $K$, which leads us to
\[ \frac{\tx{deg}(\pi;dg)}{\tx{deg}(\dot\kappa)}\int_{G(F)/A(F)} \dot\chi_{\kappa_{(S,G)}}({^{g}r})dg/dz\]
The integrand is zero unless $gx_r=x$, where $x_r$ is the unique fixed point of $r$ in $\mc{B}^\tx{red}(G,F)$. Thus if $r$ is not conjugate to an element of $G(F)_x$ the character is zero. Assume now $r \in G(F)_x$. Then the domain of integration reduces to $G(F)_x/A(F)$ and since the integrand is $G(F)_x$-invariant we obtain
\[ \tx{vol}(G(F)_x/A(F);dg/da)\tx{deg}(\pi;dg)\tx{deg}(\dot\kappa)\chi_{\dot\kappa_{(S,\theta)}}(r), \]
which is equal to $\chi_{\dot\kappa_{(S,\theta)}}(r)$. We compute this using the Frobenius formula and obtain
\[ \sum_{[g] \in G(F)_x/S(F)G(F)_{x,0}} \tx{tr}(\tilde\kappa_{(S,\theta)}(g^{-1}rg)). \]
Now $g^{-1}rg \in G(F)_{x,0}$ reduces to a regular semi-simple element of $\ms{G}_x^\circ(k_F)$, so we can take it as the element $s=\dot r$ in Proposition \ref{pro:tkchar}. This proposition tells us that $\tx{tr}(\tilde\kappa_{(S,\theta)}(g^{-1}rg))$ is zero unless the element $g^{-1}rg$ of $\ms{G}_x^\circ(k_F)$ is conjugate to $\ms{S}'$.

According to Lemma \ref{lem:rut2}, such a $g$ exists if and only if $\tx{Cent}(r,G)$ is conjugate to $S$, which is then equivalent to $r$ being conjugate to an element of $S(F)$. This proves the vanishing statement. Assume now that $r \in S(F)$ and consider again the above formula. We claim that element $g^{-1}rg$ of $\ms{G}_x^\circ(k_F)$ is conjugate to $\ms{S}'$, and hence $\tx{tr}(\tilde\kappa_{(S,\theta)}(g^{-1}rg))$ is non-zero, if and only if the coset $[g]$ is trivial. To see that, let $h \in G(F)_{x,0}$ be such that the image of $(gh)^{-1}r(gh)$ belongs to $\ms{S}'$. The torus $(gh)^{-1}S(gh)$ is maximally unramified with associated point $x$
and its image in $\ms{G}_x^\circ$ is equal to $\ms{S}'$. By Lemma \ref{lem:rut2} there is $l \in G(F)_{x,0+}$ such that $ghl \in N(G,S)(F)$. Lemma \ref{lem:weyl} then implies that $ghl \in S(F)G(F)_{x,0}$, hence the claim.

We conclude that the character of $\dot\kappa$ at $r$ is equal to the character of $\tilde\kappa$ at $r$. For the latter we now use Corollary \ref{cor:tkcharss} together with Lemma \ref{lem:dze}. It remains to check that $\Delta_{II}^\tx{abs}[a,\chi](r^w)=1$. The $F_\alpha/F_{\pm\alpha}$-norm of $\alpha(r^w) \in F_\alpha^\times$ is equal to $1$, so $\alpha(r^w) \in O_{F_\alpha}^\times$. Since the mod-$a$-data consists of units and the $\chi$-data is unramified, the claim follows.
\end{proof}

We close this subsection with a remark about the character of extra regular depth-zero supercuspidal representations of groups that split over $F^u$. These are the representations constructed in \cite[\S4.4]{DR09}. DeBacker and Reeder compute in \cite[\S9,10,11,12]{DR09} the character of these representations at arbitrary regular semi-simple elements: For an element $\gamma \in G_\tx{sr}(F)_0$ with topological Jordan decomposition $\gamma=\gamma_s \cdot \gamma_u$ the character of $\pi_{(S,\theta)}$ is given by
\[ (-1)^{r(G)-r(J)} \sum_{\substack{g \in J(F) \lmod G(F) / S(F) \\ \gamma_s^g \in S(F)}} \theta(\gamma_s^g) \hat\mu_{\mf{j},{^gX}}(\log(\gamma_u)), \]
where again $J$ is the connected centralizer of $\gamma_s$ in $G$ and $r(G)$ denotes the split rank of the group $G$. The final paragraph of the preceding proof shows that this formula is the same as the formula of Corollary \ref{cor:newchartoral}. We expect that the same is true for tamely ramified groups as well.

\subsection{Character values of regular supercuspidal representations at shallow elements: general depth} \label{sub:shallow}

In this subsection $F$ is a local field of odd residual characteristic that is not a bad prime for $G$.

Consider a regular supercuspidal representation $\pi_{(S,\theta)}$. Let $\tilde G \to G$ be a $z$-extension and let $\pi_{(\tilde S,\tilde\theta)}$ be the pull-back of $\pi_{(S,\theta)}$ to $\tilde G(F)$. Since the character function of $\pi_{(\tilde S,\tilde\theta)}$ is the pull-back to $\tilde G(F)$ of the character function of $\pi_{(S,\theta)}$, we may assume without loss of generality that $G=\tilde G$.

Let $G^0 \subset \dots \subset G^d$ be the corresponding twisted Levi sequence, $(\phi_{-1},\dots,\phi_d)$ a Howe factorization, and $(r_{-1},r_0,\dots,r_d)$ the sequence of depths of the characters $\phi_i$. Let $\gamma \in G(F)$ be regular semi-simple. If $S\neq G^0$ we will call $\gamma$ ``shallow'' if it is a topologically semi-simple element $\gamma=\gamma_0$. If $S=G^0$ we will call $\gamma$ ``shallow'' if it is topologically semi-simple modulo $Z(G)^\circ$, although we believe that in this case it is enough to require $\gamma=\gamma_{<r_{d-1}}$.

We will now fix mod-$a$-data and $\chi$-data for $R(S,G)$. We do this successively for $R(S,G^i) \sm R(S,G^{i-1})$ as described in Subsection \ref{sub:signs}, where $i$ runs from $d$ to $1$. We also need to handle the step $i=0$ in the case where $G^0 \neq S$. Let us first focus on $i>0$. Fix a character $\Lambda : F \to \C^\times$ of depth zero. According to \eqref{eq:adata} we obtain a set of mod-$a$-data for $R(S,G^i) \sm R(S,G^{i-1})$ by the equation
\[ \phi_{i-1}(N_{F_\alpha/F}(\alpha^\vee(X+1))) = \Lambda(\tx{tr}_{F_\alpha/F}(\bar a_\alpha X)), \]
where $\bar a_\alpha \in [F_\alpha]_{-r_{i-1}}/[F_\alpha]_{-r_{i-1}+}$ and $X \in [F_\alpha]_{r_{i-1}}/[F_\alpha]_{r_{i-1}+}$. Applying Fact \ref{fct:fri} to the Howe factorization $(\phi_{-1},\dots,\phi_d)$ of $\theta$ we see that this equation is equivalent to
\begin{equation} \label{eq:adatahowe} \theta(N_{F_\alpha/F}(\alpha^\vee(X+1))) = \Lambda(\tx{tr}_{F_\alpha/F}(\bar a_\alpha X)), \end{equation}
We take this as the defining equation for the mod-$a$-data, as it clearly demonstrates that this data depends only on $\theta$ and not on the Howe factorization. From the mod-$a$-data we obtain $\chi$-data $\chi'$ via \eqref{eq:chi'}. Now consider the case $i=0$, which is only relevant when $G^0 \neq S$. We can still take \eqref{eq:adatahowe} as the defining equation for mod-$a$-data and we obtain $\bar a_\alpha \in [F_\alpha]_0/[F_\alpha]_{0+}$. On the other hand, we take $\chi'_\alpha$ to be unramified. This is possible because the action of inertia preserves a base in $R(S,G^0)$ so all symmetric roots are unramified. With the mod-$a$-data and $\chi$-data fixed this way, we have the following formula.

\begin{cor} \label{cor:charshallow}
Let $\gamma \in G(F)$ be a shallow regular semi-simple element whose stable class meets $S(F)$. The value of the normalized character $\Phi_\pi$ at a $\gamma$ is zero, unless $\gamma$ is (conjugate to) an element of $S(F)$, in which case is given by
\begin{equation*}
e(G)\epsilon_L(X^*(T)_\C-X^*(S)_\C,\Lambda)\!\!\!\!\!\!\!\! \sum_{w \in N(S,G)(F)/S(F)}\!\!\!\!\!\!\!\!
\Delta_{II}^\tx{abs}[a,\chi'](\gamma^w) \epsilon_{f,r}(\gamma^w) \epsilon^r(\gamma^w)\theta(\gamma^w),
\end{equation*}
where $T$ is the minimal Levi subgroup in the quasi-split inner form of $G$.
\end{cor}
\begin{proof}

We write \eqref{eq:charshallow1} as
\begin{equation*}
\phi_d(\gamma)\sum_{\substack{g \in G_{d,\gamma}(F) \lmod G_d(F) / G_{d-1}(F)\\ \gamma^g \in G_{d-1}(F)}}\epsilon(\pi_{d-1},\gamma^g) \Phi_{\pi_{d-1}}(\gamma^g) ,
\end{equation*}
where we have combined all three roots of unity into the single term $\epsilon(\pi_{d-1},\gamma^g)$. We are now going to unwind the induction inherent in this formula. To see what is going on we substitute the formula for $\Phi_{\pi_{d-1}}$ and obtain
\begin{equation*}
\phi_d(\gamma)\!\!\!\!\!\!\!\!\!\!\!\!\!\!\!\!\!\!\!\sum_{\substack{g \in G^d_{\gamma}(F) \lmod G^d(F) / G^{d-1}(F)\\ \gamma^g \in G^{d-1}(F)}}\!\!\!\!\!\!\!\!\!\!\!\!\!\!\!\!\!\!\!\epsilon(\pi_{d-1},\gamma^g)
\phi_{d-1}(\gamma^g)\!\!\!\!\!\!\!\!\!\!\!\!\!\!\!\!\!\!\!\sum_{\substack{h \in G^{d-1}_{\gamma^g}(F) \lmod G^{d-1}(F) / G^{d-2}(F)\\ \gamma^{gh} \in G^{d-2}(F)}}\!\!\!\!\!\!\!\!\!\!\!\!\!\!\!\!\!\!\!\epsilon(\pi_{d-2},\gamma^{gh}) \Phi_{\pi_{d-2}}(\gamma^{gh}).
\end{equation*}
Recall \cite[Remark 4.3.5]{DS} that the term $\epsilon(\pi_{d-1},\gamma^g)$ remains unchanged if we conjugate both $\pi_{d-1}$ and $\gamma^g$ by an element of $G^d(F)$. If this element happens to belong to $G^{d-1}(F)$, then $\pi_{d-1}$ remains unchanged. With this we obtain
\begin{equation*}
\sum_g \sum_h \phi_{d-1}(\gamma^{gh})\epsilon(\pi_{d-1},\gamma^{gh})\epsilon(\pi_{d-2},\gamma^{gh}) \Phi_{\pi_{d-2}}(\gamma^{gh}),
\end{equation*}
where the summetion indices are as before, and after re-indexing the sum this leads to
\begin{equation*}
\sum_{\substack{g \in G^d_{\gamma}(F) \lmod G^d(F) / G^{d-2}(F)\\ \gamma^g \in G^{d-2}(F)}}
\phi_d(\gamma^g)\phi_{d-1}(\gamma^g)\epsilon(\pi_{d-1},\gamma^g)\epsilon(\pi_{d-2},\gamma^g) \Phi_{\pi_{d-2}}(\gamma^g).
\end{equation*}
We do this inductively, where at the $(-1)$-stage we apply Proposition \ref{pro:charshallowdz} if $G^0 \neq S$, and obtain the formula
\begin{equation*}
\sum_{\substack{g \in G^d_{\gamma}(F) \lmod G^d(F) / S(F)\\ \gamma^g \in S(F)}}
\prod_{i=-1}^{d} \phi_i(\gamma^g)\epsilon(\pi_i,\gamma^g).
\end{equation*}
In particular, we see that the result is zero unless $\gamma$ is $G(F)$-conjugate to an element of $S$.

We now go into the characters $\epsilon(\pi_i,\gamma^g)$. Their definition depends on the choice of a tame maximal torus $T$ containing $\gamma^g$. In the current situation we have a canonical choice for $T$, namely $T=S$. We now apply Corollary \ref{cor:prodeps} using the fixed mod-$a$-data and $\chi$-data. Recalling \eqref{eq:thetaprod} that $\theta$ is the product of all $\phi_i$ restricted to $S(F)$ and letting $\epsilon^r$ be the product of all $\epsilon^r(\pi_i,-)$, the proof is complete.
\end{proof}

It is noted in \cite{DS} that the map $\gamma \mapsto \epsilon^r(\gamma)$ is a character of $S(F)$. If we let $\theta'$ to be the character $\epsilon_{f,r} \cdot \epsilon^r \cdot \theta$, then the character formula takes the form
\begin{equation} \label{eq:charshallow}
e(G)\epsilon_L(X^*(T)_\C-X^*(S)_\C,\Lambda) \sum_{w \in N(S,G)(F)/S(F)}
\Delta_{II}^\tx{abs}[a,\chi'](\gamma^w) \theta'(\gamma^w).
\end{equation}

\subsection{Comparison with the characters of real discrete series representations} \label{sub:realchar}

In this subsection only, we let $G$ be a connected reductive group defined over $\R$ and having a discrete series of representations, or equivalently having elliptic maximal tori. All elliptic maximal tori in $G$ are conjugate under $G(\R)$. Fix one such $S \subset G$.

The discrete series representations of $G(\R)$ are parameterized by pairs $(\theta,\rho)$ (taken up to conjugation by $N(S,G)(\R)/S(\R)$) consisting of a character $\theta : S(\R) \rw \C^\times$ and a choice of positive roots $\rho$ for $R(S,G)$ for which $d\theta$ is dominant (the image of $d\theta$ in $X^*(S_\tx{ad}) \otimes \C$ belongs to $X^*(S_\tx{ad})$, so it makes sense to speak of its dominance). Given such a pair $(\theta,\rho)$, which we shall call a Harish-Chandra parameter, the corresponding representation is characterized by the fact that the value of its character on any regular $\gamma \in S(\R)$ is given by
\begin{equation} \label{eq:chards} (-1)^{q(G)}\sum_{w \in N(S,G)(\R)/S(\R)} \frac{\theta(\gamma^w)}{\prod\limits_{\alpha >0} (1-\alpha(\gamma^w)^{-1})}, \end{equation}
where $q(G)$ is half of the dimension of the symmetric space of $G(\R)$.

We now claim that this formula is the same as \eqref{eq:charshallow} specialized to the case $F=\R$. The latter formula involves a non-trivial character $\Lambda : \R \rw \C^\times$, but is independent of the choice. We choose here the standard character $\Lambda(x)=\exp(2\pi i x)$. It also involves $a$-data, which is to be computed based on $\Lambda$ and $\theta$ according to \eqref{eq:adata}, namely
\[ \theta(N_{\C/\R}(\alpha^\vee(\exp(z)))) = \Lambda(\tx{tr}_{\C/\R}(a_\alpha z)) \]
for $z \in \C$, keeping in mind that all elements of $R(S,G)$ are symmetric, with $F_\alpha=\C$ and $F_{\pm\alpha}=\R$, because complex conjugation acts by $-1$ on the root system of $S$. Evaluating this formula we find
\[ \theta(N_{\C/\R}(\alpha^\vee(\exp(z)))) = \theta(\alpha^\vee(e^{z-\bar z})) = e^{(z-\bar z)\<\alpha^\vee,d\theta\>}, \]
while at the same time $\Lambda(\tx{tr}_{\C/\R}(a_\alpha z)) = e^{2\pi i(a_\alpha z + \bar a_\alpha \bar z)}$.
This implies
\[ a_\alpha = \frac{\<\alpha^\vee,d\theta\>}{2\pi i}. \]
Finally, we need to choose $\chi$-data, and we take $\chi_\alpha(z)=\tx{sgn}_\C(z)$ for $z \in \C^\times$ and $\alpha>0$. We will discuss the significance of this choice at the end of this subsection.

Having made these preparations we now explicate \eqref{eq:charshallow} in this setting. First, we use the real case of \cite[Corollary 3.5.2]{KalEpi}, which gives us
\[ e(G)\epsilon_L(X^*(T)_\C-X^*(S)_\C,\Lambda) = \prod_{\alpha \in R(S,G)/\Gamma} f_{(G,S)}(\alpha)\lambda_{\C/\R}(\Lambda\circ\tx{tr}_{\C/\R})^{-1}. \]
Now $\lambda_{\C/\R}(\Lambda\circ\tx{tr}_{\C/\R})=i$ and $f_{(G,S)}(\alpha)$ equals $-1$ if $\alpha$ is compact and $+1$ if $\alpha$ is non-compact. It follows that
\[ e(G)\epsilon_L(X^*(T)_\C-X^*(S)_\C,\Lambda) = (-1)^{q(G)} \prod_{\alpha \in R(S,G)/\Gamma} i. \]
On the other hand we have
\[ \Delta_{II}^\tx{abs}[a,\chi](\gamma)\!\! =\!\! \prod_{\alpha<0} \tx{sgn}_\C\left(\frac{\alpha(\gamma)-1}{\<\alpha^\vee,d\theta\>(2\pi i)^{-1}}\right)^{-1}\!\!\!\!\!\!\! = \prod_{\alpha<0} (-i) \cdot \prod_{\alpha<0} \tx{sgn}_\C\left(\frac{\alpha(\gamma)-1}{\<\alpha^\vee,d\theta\>}\right)^{-1}\!\!\!\!. \]
Recall that $d\theta$ is dominant for the chosen set of positive roots, so $\<\alpha^\vee,d\theta\> < 0$ whenever $\alpha<0$, leading to
\[ \tx{sgn}_\C\left(\frac{\alpha(\gamma)-1}{\<\alpha^\vee,d\theta\>}\right)^{-1} = \tx{sgn}_\C(1-(-\alpha)(\gamma)^{-1})^{-1}.\]
At the same time
\begin{eqnarray*}
|D(\gamma)|^\frac{1}{2}&=&\prod_{\alpha \in R(S,G)}|\alpha(\gamma)-1|^\frac{1}{2}\\
&=&\left(\prod_{\alpha>0}|\alpha(\gamma)-1||\alpha(\gamma)^{-1}-1|\right)^\frac{1}{2}\\
&=&\left(\prod_{\alpha>0}|\alpha(\gamma)^\frac{1}{2}-\alpha(\gamma)^{-\frac{1}{2}}||\alpha(\gamma)^{-\frac{1}{2}}-\alpha(\gamma)^\frac{1}{2}|\right)^\frac{1}{2}\\
&=&\prod_{\alpha>0}|\alpha(\gamma)^\frac{1}{2}-\alpha(\gamma)^{-\frac{1}{2}}|\\
&=&\prod_{\alpha>0}|1-\alpha(\gamma)^{-1}|\\
\end{eqnarray*}
where the last equality follows form the fact that $\alpha(\gamma)$ is of absolute value $1$.
Combinging these calculations we see that
\[  e(G)\epsilon_L(X^*(T_0)_\C-X^*(S)_\C,\Lambda) |D_G(\gamma)|^{-\frac{1}{2}}\Delta_{II}^\tx{abs}[a,\chi'](\gamma) \]
is equal to
\[ (-1)^{q(G)}\prod_{\alpha>0} (1-\alpha(\gamma)^{-1})^{-1}  \]
and we conclude that the formula \eqref{eq:charshallow}, interpreted for the ground field $\R$, evaluates to
\[ (-1)^{q(G)}\sum_{w \in N(S,G)(\R)/S(\R)} \frac{\theta'(\gamma^w)}{\prod\limits_{\alpha >0} (1-\alpha(\gamma^w)^{-1})}, \]
which is indeed the character formula for a discrete series representation \eqref{eq:chards}.

Finally, we make a comment on the choice of $\chi$-data used here. We do not see a way to choose $\chi$-data uniformly in the real and $p$-adic cases. In the real case, the choice we have just used is well-known by the name of ``based'' $\chi$-data from the work of Shelstad, and is intimately connected with the local Langlands correspondence. In the $p$-adic case, our choice was ad-hoc. This is however of no importance. Indeed, making a different choice of $\chi$-data has the effect of multiplying the term $\Delta_{II}^\tx{abs}$ by a character of $S(F)$. This character can then by absorbed into $\theta$, and so can be the characters $\epsilon_{f,r}$ and $\epsilon^r$. In fact, the particular representation of $G(F)$ that the character $\theta$ of $S(F)$ leads to depends on the details of the construction that is used, and different constructions could lead to slightly different representations (as for example the construction that was recently announced by J.K.Yu in a conference talk, which automatically absorbs $\epsilon^r$ into $\theta$). What is important for us here is that the formula for the character of regular supercuspidal representations at shallow elements has the same structure, including the roots of unity that cannot be absorbed into a character of $S(F)$, as the character formula for real discrete series. This fact will be our guide to the construction and study of $L$-packets in what follows.

\section{Regular supercuspidal $L$-packets} \label{sec:pack}

Let $G$ be a connected reductive group defined and quasi-split over $F$ and split over a tame extension of $F$. Let $\hat G$ be a Langlands dual group for $G$ and $^LG = \hat G \rtimes W_F$ the Weil-form of the corresponding $L$-group.

In this section we are going to construct those $L$-packets of all inner forms of $G$ that consist entirely of regular supercuspidal representations and assign to each such $L$-packet a Langlands parameter. The construction will allow an explicit passage from parameters to representations and conversely. Each of the $L$-packets will contain extra regular supercuspidal representations, which is the reason for their name.

We will eventually assume that the residual characteristic $p$ of $F$ is not a bad prime for $G$ and does not divide $|\pi_0(Z(G))|$. Note that the bad primes for $G$ and $\hat G$ are the same, and that $\pi_0(Z(G))$ has the same order as the fundamental group of the derived subgroup of $\hat G$.

\subsection{Admissible embeddings} \label{sub:recg}

We recall here some basic facts about the relationship between $G$ and $\hat G$.

Let $S$ be a torus defined over $F$ of dimension equal to the rank of $G$, and let $J$ be a $\Gamma$-stable $G$-conjugacy class of embeddings $j : S \rw G$. From $J$ we obtain a $\Gamma$-stable $\hat G$-conjugacy class $\hat J$ of embeddings $\hat j : \hat S \rw \hat G$ as follows.  Fix $\Gamma$-invariant pinnings $(T,B,\{X_\alpha\})$ of $G$ and $(\hat T,\hat B,\{Y_{\hat\alpha}\})$ of $\hat G$. Choose $j \in J$ such that $j(S)=T$ and define $\hat j$ to be the inverse of the isomorphism $\hat T \rw \hat S$ of complex tori induced by $j$. Then the $\hat G$-conjugacy class $\hat J$ of $\hat j$ is $\Gamma$-stable. Indeed, $w : \sigma \mapsto j\circ\sigma(j^{-1})$ is an element of $Z^1(\Gamma,\Omega(T,G))$, which under the isomorphism $\Omega(T,G) \cong \Omega(\hat T,\hat G)$ corresponds to an element $\hat w \in Z^1(\Gamma,\Omega(\hat T,\hat G))$, and we have $\hat j \circ \sigma(\hat j^{-1}) = \hat w_\sigma$. The choice of $j \in J$ can only be altered to $v \circ j$ for some $v \in \Omega(T,G)$, but then $\hat j$ becomes $\hat v \circ \hat j$ and leads to the same $\hat J$. The choices of pinning also have no influence, because any two pinnings of $G$ are conjugate by $G_\tx{ad}(F)$ and any two pinnings of $\hat G$ are conjugate by $\hat G^\Gamma$ \cite[Corollary 1.7]{Kot84}.

The same procedure can be performed in the opposite direction and produces $J$ from $\hat J$. Since $G$ is quasi-split, there exists $\Gamma$-fixed elements $j \in J$ \cite[Corollary 2.2]{Kot82}.

From $J$ we obtain the following structure on $S$.
\begin{itemize}
\item An embedding $Z(G) \rw S$ over $F$, namely by restricting any $j \in J$ to $Z(G)$;
\item A $\Gamma$-invariant subset $R(S,G) \subset X^*(S)$, by choosing $j : S \rw T$ and pulling back $R(T,G)$ along $j$;
\item A $\Gamma$-invariant subgroup $\Omega(S,G) \subset \tx{Aut}_\tx{alg.grp}(S)$, by choosing $j : S \rw T$ and pulling back $\Omega(T,G)$.
\end{itemize}
Again it is clear that this structure depends only on $J$ and not on the choices of pinning of $G$ or $j \in J$. Moreover, if $j \in J$ is $\Gamma$-fixed, then it provides $\Gamma$-equivariant isomorphisms $R(S,G) \to R(jS,G)$ and $\Omega(S,G) \to \Omega(jS,G)$.

We follow standard terminology and call the embeddings belonging to $J$ \emph{admissible}. More generally, if $(G',\xi)$ is an inner twist of $G$, we will call an embedding $j : S \rw G'$ \emph{admissible} if $\xi^{-1}\circ j \in J$. Outside of this subsection, we will not use the symbol $J$ for a conjugacy class of embeddings, but will rather keep it reserved for connected centralizers of semi-simple elements of $G$, referring to the embeddings in the distinguished conjugacy class as \emph{admissible}.

\subsection{Construction of $L$-packets} \label{sub:packconst}

We now introduce the Langlands parameters that correspond to regular supercuspidal $L$-packets. We will give two definitions -- the first one is easier to state and describes most of the parameters we need. It also generalizes many of the parameters that have previously been studied. The second definition is slightly more general and turns out to be the one that we need.

From now on we assume that the residual characteristic of $F$ is not $2$ and is not a bad prime for $G$. We also assume that the characteristic of $F$ is zero. While this latter assumption is not needed for any of the arguments here, it is assumed in \cite{KalRI}, which we will use. We are convinced that the constructions and arguments of \cite{KalRI} are also valid in positive characteristic, so the adventurous reader is encouraged to think about the positive characteristic case as well. Alternatively, if one replaces $H^1(u \to W,-)$ by $B(-)_\tx{bas}$ of \cite{KotBG}, this assumption can be droppped.

\begin{dfn} \label{def:srsp} A strongly regular supercuspidal parameter is a discrete Langlands parameter $\varphi : W_F \rw {^LG}$ such that $\varphi(P_F)$ is contained in a torus of $\hat G$ and $\tx{Cent}(\varphi(I_F),\hat G)$ is abelian.
\end{dfn}

Special cases of such parameters are those discussed in \cite[\S4.1]{DR09} (i.e. regular depth zero supercuspidal parameters of unramified groups), those discussed in \cite{Roe11} (i.e. regular depth zero supercuspidal parameters of ramified unitary groups), those discussed in \cite[\S6.3]{Ree08} (positive depth toral parameters giving rise to an unramified torus), and those discussed in \cite[\S4.1]{KalEpi} (minimal positive depth toral parameters giving rise to a ramified torus). In fact, these examples are special cases of a much smaller class of parameters, which one might call toral. The case of positive depth toral parameters is treated in more detail in Section \ref{sec:toral}, because they are much easier to deal with and because the current state of the Adler-DeBacker-Spice character formula allows us to obtain additional results for them.

Before coming to the second definition, we collect some basic facts.

\begin{lem} \label{lem:levisystem2} Let $\varphi : W_F \to {^LG}$ be a Langlands parameter.
\begin{enumerate}
	\item If $\varphi(P_F)$ is contained in a torus of $\hat G$, then $\hat M = \tx{Cent}(\varphi(P_F),\hat G)^\circ$ is a Levi subgroup of $\hat G$. If $p$ does not divide $|\pi_0(Z(G))|$, then $\tx{Cent}(\varphi(P_F),\hat G)$ is connected.
	\item If $\varphi(P_F)$ is contained in a torus of $\hat G$ and $C=\tx{Cent}(\varphi(I_F),\hat G)^\circ$ is a torus, then $\hat T = \tx{Cent}(C,\hat M)$ is a maximal torus of $\hat G$ normalized by $\varphi(W_F)$ and contained in a Borel subgroup of $\hat M$ normalized by $\varphi(I_F)$. Furthermore, $\hat T$ is normalized by $\tx{Cent}(\phi(I_F),\hat G)$.
\end{enumerate}
\end{lem}
\begin{proof}
By continuity, $\varphi(P_F)$ is a finite $p$-subgroup of $\hat G$, let $x_1,\dots,x_n$ be its elements. By \cite[Proposition A.7]{AS08} $\tx{Cent}(x_1,\hat G)^\circ$ is a Levi subgroup of $\hat G$. Any torus of $\hat G$ containing $x_1,\dots,x_n$ is contained in $\tx{Cent}(x_1,\hat G)^\circ$. If $p$ does not divide $|\pi_0(Z(G))|$, then it does not divide the order of the fundamental group of $\hat G_\tx{der}$ and $\tx{Cent}(x_1,\hat G)$ is connected by \cite[Corollary 4.6]{SS70}. Being a Levi subgroup of $\hat G$, the fundamental group of its derived subgroup is a subgroup of the fundamental group of $\hat G_\tx{der}$. In either case, replace $\hat G$ by $\tx{Cent}(x_1,\hat G)^\circ$ and proceed with $x_2$. This proves the first point.

For the second point, we have $C=\tx{Cent}(\varphi(I_F),\hat G)^\circ = \tx{Cent}(\varphi(I_F),\hat M)^\circ$. The action of $I_F$ on $\hat M$ by $\tx{Ad}(\varphi(-))$ restricts trivially to $P_F$. Since $I_F/P_F$ is pro-cyclic, the centralizer of $\varphi(I_F)$ in $\hat M$ is the fixed-point set of a single automorphism $\theta$ of $\hat M$, namely $\tx{Ad}(\varphi(x))$, where $x \in I_F$ projects onto a topological generator of $I_F/P_F$. The automorphism $\theta$ is semi-simple (in fact of finite order) and by \cite[Theorem 7.5]{Ste68end} it preserves a Borel pair of $\hat M$. Let $\hat T$ be the maximal torus in that Borel pair. From \cite[Theorem 1.1.A]{KS99} we know that $[\hat T \cap C]^\circ$ is a maximal torus of $C$ and hence must equal to $C$, and moreover $\hat T=\tx{Cent}(C,\hat M)$. Since $\hat M$ is a Levi subgroup of $\hat G$, $\hat T$ is also a maximal torus of $\hat G$. Finally, since both $\hat M$ and $C$ are normalized by $\varphi(W_F)$ as well as by $\tx{Cent}(\varphi(I_F),\hat G)$, so is $\hat T$.
\end{proof}

\begin{dfn} \label{def:rsp} A regular supercuspidal parameter is a discrete Langlands parameter $\varphi : W_F \to {^LG}$ satisfying the following:
\begin{enumerate}
\item $\varphi(P_F)$ is contained in a torus of $\hat G$; set $\hat M=\tx{Cent}(\varphi(P_F),\hat G)^\circ$.
\item $C:=\tx{Cent}(\varphi(I_F),\hat G)^\circ$ is a torus; let $\hat S$ be the $\Gamma$-module with underlying abelian group $\hat T := \tx{Cent}(C,\hat M)$ and $\Gamma$-action given by $\tx{Ad}(\varphi(-))$.
\item If $n \in N(\hat T,\hat M)$ projects onto a non-trivial element of $\Omega(\hat S,\hat M)^\Gamma$, then $n \notin \tx{Cent}(\varphi(I_F),\hat G)$.
\end{enumerate}
\end{dfn}

We note that if $\varphi$ is a strongly regular supercuspidal parameter then
\[ N(\hat T,\hat M)\cap \tx{Cent}(\varphi(I_F),\hat G) \subset \tx{Cent}(\varphi(I_F),\hat M) \subset \tx{Cent}(C,\hat M) = \hat T, \]
where the second containement follows from the fact that $\tx{Cent}(\varphi(I_F),\hat M)$ is abelian and contains $C$, thus $\varphi$ is regular. Conversely, almost all regular supercuspidal parameters are strongly regular.

Our task in this subsection is to construct for each such parameter and each inner form of $G$ the corresponding $L$-packet. The first step is to identify the set of equivalence classes of such parameters with a set of equivalence classes of another kind of data. We find it most convenient to organize this new data into a category, which we call the category of regular supercuspidal $L$-packet data. The objects in the category will be tuples $(S,\hat j,\chi,\theta)$, where $S$ is a torus of dimension equal to the absolute rank of $G$, defined over $F$ and split over a tame extension of $F$; $\hat j : \hat S \rw \hat G$ is an embedding of complex reductive groups whose $\hat G$-conjugacy class is $\Gamma$-stable; $\chi$ is minimally ramified $\chi$-data for $R(S,G)$ in the sense of Subsection \ref{sub:delta2}; and $\theta : S(F) \rw \C^\times$ is a character. We require of this data that the $\chi$-data be $\Omega(S,G^0)(F)$-invariant, $S/Z(G)$ be anisotropic, and $(S,\theta)$ be a tame extra regular elliptic pair in the sense of Definition \ref{dfn:tre}. Here we are using the structure on $S$ that is given to us by $\hat j$ as described in Subsection \ref{sub:recg}, and moreover $\Omega(S,G^0)$ is the subgroup of $\Omega(S,G)$ generated by the reflection along the sub-root system $R_{0+} \subset R(S,G)$ as in Definition \ref{dfn:tre}.

A morphism $(S,\hat j,\chi,\theta) \rw (S',\hat j',\chi',\theta')$ is a triple $(\iota,g,\zeta)$, where $\iota : S \rw S'$ is an isomorphism of $F$-tori, $g \in \hat G$, and $\zeta = (\zeta_{\alpha'})_{\alpha'}$ is a collection of characters $\zeta_{\alpha'} : F_{\alpha'}^\times \rw \C^\times$, one for each $\alpha' \in R(S',G)$, satisfying the conditions listed after Lemma \ref{lem:d2a}. We require that $\hat j \circ \hat \iota = \tx{Ad}(g) \circ \hat j'$, that $\chi_{\alpha'\circ \iota} = \chi'_{\alpha'} \cdot \zeta_{\alpha'}$, and that $\zeta_{S'}^{-1} \cdot \theta'\circ\iota=\theta$, where $\zeta_{S'}$ is the character of $S'(F)$ corresponding to $\zeta$ as described in Subsection \ref{sub:delta2}. Composition of morphisms is defined in the obvious way. Note that every morphism is an isomorphism and that the automorphisms of a fixed object $(S,\hat j,\chi,\theta)$ are given by $\Omega(S,G)(F)$.

\begin{pro} \label{pro:pardata} There is a natural 1-1 correspondence between the $\hat G$-conjugacy classes of regular supercuspidal parameters and the isomorphism classes of regular supercuspidal $L$-packet data.
\end{pro}
\begin{proof}
Given a datum $(S,\hat j,\chi,\theta)$ we use the $\chi$-data to extend $\hat j$ to an $L$\-embedding $^Lj : {^LS} \rw {^LG}$ as explained in \cite[\S2.6]{LS87} and let $\varphi_S : W_F \rw {^LS}$ be the parameter for $\theta$. Define $\varphi = {^Lj} \circ \varphi_S$.

Let us first check that the $\hat G$-conjugacy class of $\varphi$ depends only on the isomorphism class of the datum $(S,\hat j,\chi,\theta)$. Keeping this datum fixed, the parameter $\varphi_S$ is determined by $\theta$ up to $\hat S$-conjugacy and the embedding $^Lj$ is determined by $\chi$ and $\hat j$ also up to $\hat S$-conjugacy, so the $\hat G$-conjugacy class of $\varphi$ does not depend on these choices. Now we vary the datum $(S,\hat j,\chi,\theta)$ within its isomorphism class. It is enough to check the three basic cases of an isomorphism: $(\iota,1,1)$, $(1,g,1)$, and $(1,1,\zeta)$. In the first case we have $^Lj\circ{^L\iota} = {^Lj'}$ and ${^L\iota}\circ\varphi_{S'}=\varphi_{S}$, hence ${^Lj}\circ\varphi_S={^Lj'}\circ\varphi_{S'}$. In the second case we have $S=S'$ and ${^Lj}=\tx{Ad}(g)\circ{^Lj'}$. In the third case we have ${^Lj'} = {^Lj} \cdot c^{-1}$, where $c$ is the 1-cocycle of \cite[Corollary 2.5.B]{LS87}, as well as $\theta'=\theta \cdot \zeta_S$. The proof of \cite[Lemma 3.5.A]{LS87} shows that $c$ is the Langlands parameter for the character $\zeta_S$, so that $\varphi_{S'}=\varphi_S \cdot c$.

We now check that $\varphi$ satisfies the conditions of Definition \ref{def:rsp}. Since $S$ is tamely ramified and the $\chi$-data is at most tamely ramified, we have $\varphi|_{P_F} = \hat j \circ \varphi_S|_{P_F}$, so $\varphi(P_F)$ is contained in the image of $\hat j$, which is a maximal torus of $\hat G$. Before we can discuss $\tx{Cent}(\varphi(I_F),\hat G)^\circ$ we need some preparation. Fix a $\Gamma$-invariant pinning $(\hat T,\hat B,\{X_{\hat\alpha}\}_{\hat\alpha \in \Delta^\vee})$ of $\hat G$ and replace $(S,\hat j,\chi,\theta)$ by an isomorphic datum so that $\hat j(\hat S)=\hat T$.

Let $\hat M=\tx{Cent}(\varphi(P_F),\hat G)^\circ$. According to Lemma \ref{lem:levisystem2} this is a Levi subgroup of $\hat G$. It is normalized by the action of $\varphi(W_F)$ and the resulting homomorphism $W_F \to \tx{Aut}(\hat M) \to \tx{Out}(\hat M)$ extends to $\Gamma_F$. We have arranged that $\hat T \subset \hat M$, so the fixed pinning of $\hat G$ induces the pinning $(\hat T,\hat B \cap M,\{X_{\hat\alpha}\}_{\hat\alpha \in \Delta_M^\vee})$ of $\hat M$. This pinning gives a splitting $\tx{Out}(\hat M) \to \tx{Aut}(\hat M)$ of the natural projection, so we obtain an action $\Gamma_F \to \tx{Aut}(\hat M)$ preserving the pinning. Note that the original action of $\Gamma_F$ on $\hat G$ need not preserve $\hat M$, so there is no potential for confusion whenever we speak about ``the'' $\Gamma_F$-action on $\hat M$. The group $\hat M$ endowed with this $\Gamma_F$-action is the dual group of a quasi-split $F$-group $M$.
We now claim that under the identification $R(\hat S,\hat G) = R^\vee(S,G)$ the root system
\[ R(\hat S,\hat M) = \{\hat\alpha \in R(\hat S,\hat G)| \hat\alpha(\varphi(P_F))=1\} \]
becomes identified with the coroot system of the subsystem $R_{0+}$ of Definition \ref{dfn:tre}. For any $\hat\alpha \in R(\hat S,\hat G)$ let $\alpha^\vee \in R^\vee(S,G)$ be the corresponding cocharacter. Letting $E/F$ be the tame Galois extension splitting $S$, the parameter of the character $\theta\circ N_{E/F}\circ \alpha^\vee$ is equal to the restriction to $W_E$ of $\hat\alpha\circ\varphi_S$. Since $P_F=P_E \subset W_E$, we see using \cite[Theorem 7.10]{Yu09} that $R(\hat S,\hat M)$ is the subset of $R^\vee(S,G)$ consisting of those $\alpha^\vee$ for which $\theta\circ N_{E/F}\circ\alpha^\vee$ restricts trivially to $E_{0+}^\times$ and the claim is proved.

Consider the embedding $\hat j : \hat S \to \hat M$. It is trivially $W_F$-equivariant if we endow both sides with the action of $W_F$ given by $\tx{Ad}(\varphi(-))$. It is no longer $W_F$-equivariant if we endow $\hat M$ with the action of $W_F$ via which $\hat M$ becomes the dual group of $M$, but this latter action differs from the previous action only by inner automorphisms of $\hat M$, so the $\hat M$-conjugacy class of $\hat j : \hat S \to \hat M$ is still $\Gamma$-stable. As discussed in Subsection \ref{sub:recg} this gives us the notion of admissible embeddings $S \to M$. Tracking through the definitions we see that the subset $R(S,M)$ of $X^*(S)$ arising from this notion is a subset of $R(S,G)$ and the identification $R(\hat S,\hat G)=R^\vee(S,G)$ identifies $R(\hat S,\hat M)$ with $R^\vee(S,M)$.

By assumption on $\theta$ the action of $I_F$ on $R(\hat S,\hat M)$ leaves a basis invariant.
This implies that all symmetric roots in $R(\hat S,\hat M)$ are unramified and hence there is a canonical (up to $\hat S$-conjugation) extension of the embedding $\hat j:\hat S \to \hat M$ to an $L$-embedding $^Lj_{S,M} : {^LS} \to {^LM}$, namely the one given by the construction \cite[\S2.6]{LS87} for unramified $\chi$-data.

\begin{lem} \label{lem:tamelemb}
The natural inclusion $\hat M \to \hat G$ can be extended (non-canonically) to a tame $L$-embedding ${^LM} \rw {^LG}$.
\end{lem}
\begin{proof}
We have an action of $\Gamma$ on $\hat G$ coming from the fact that $\hat G$ is the dual group of $G$. We also have an action of $\Gamma$ on $\hat M$ coming from viewing $\hat M$ as the dual group of $M$. The inclusion $\hat M \to \hat G$ is not equivariant for these acitons. The restriction to $\hat T$ of the $\Gamma$-action on $\hat M$ differs from the restriction to $\hat T$ of the $\Gamma$-action on $\hat G$ by an element $w_M \in Z^1(\Gamma_{K/F},\Omega(\hat T,\hat G))$, where $K/F$ is a finite Galois extension, tame because the actions of $\Gamma_F$ on $\hat G$ and $\hat M$ are tame. We will find a homomorphism $\xi : W_F/P_F \rw N(\hat T,\hat G) \rtimes W_F/P_F$ such that each $\xi(w)$ preserves the pinning of $\hat M$ and acts on $\hat T$ via $w_M(\sigma_w) \rtimes \sigma_w$, where $\sigma_w \in \Gamma$ is the image of $w$. This $\xi$ will then give us the $L$-embedding
\[ \hat M \rtimes W_F \to \hat G \rtimes W_F,\qquad m \rtimes w \mapsto m\xi(w). \]

For this, let $n_M(\sigma) \in N(\hat T,\hat G)$ be the Tits lift \cite[11.2.9]{Spr81} of $w_M(\sigma)$ relative to the fixed pinning of $\hat G$. For $\hat\alpha \in \Delta_M^\vee \subset \Delta^\vee$ we have $\tx{Ad}(n_M(\sigma))\sigma(X_{\hat\alpha})=X_{\hat\alpha}$ by \cite[11.2.11]{Spr81}. The map $\sigma \mapsto n_M(\sigma) \rtimes \sigma \in N(\hat T,\hat G) \rtimes \Gamma$ is not necessarily a homomorphism. We have by \cite[Lemma 2.1.A]{LS87} that
\[ [n_M(\sigma)\rtimes\sigma] \cdot [n_M(\tau)\rtimes \tau] = t(\sigma,\tau) \cdot n_M(\sigma\tau) \rtimes \sigma\tau,\]
where $t(\sigma,\tau)=\alpha_{\sigma,\tau}(-1)$ and $\alpha_{\sigma,\tau} \in X_*(\hat T)$ is the sum of all members of the set
\begin{equation} \label{eq:x1} \{ \beta \in R(\hat T,\hat G)^\vee| \beta>0, [w_M(\sigma)\sigma]^{-1}\beta<0, [w_M(\sigma)\sigma w_M(\tau)\tau]^{-1}\beta>0 \}. \end{equation}
We claim that $t(\sigma,\tau) \in Z(\hat M)^\circ$. Since $X_*(Z(\hat M)^\circ)$ is the annihilator in $X_*(\hat T)$ of the root lattice $Q(\hat M) \subset X^*(\hat T)$, it will be enough to show that $\alpha_{\sigma,\tau}$ annihilates $Q(\hat M)$. This is equivalent to showing that $\alpha_{\sigma,\tau}$ is fixed by $\Omega(\hat T,\hat M)$, because for any $\hat\beta \in R(\hat T,\hat M)$ we have
\[ \<\alpha_{\sigma,\tau},\beta^\vee\> = 0 \Leftrightarrow \<\alpha_{\sigma,\tau},\hat\beta\> = -\<\alpha_{\sigma,\tau},\hat\beta\> = \<s_{\hat\beta}\alpha_{\sigma,\tau},\hat\beta\> . \]
Now observe that any member $\beta$ of \eqref{eq:x1} must be outside of $R(\hat T,\hat M)^\vee$, because otherwise $[w_M(\sigma)\sigma]^{-1}$ would not make it negative. The action of $\Omega(\hat T,\hat M)$ on $R(\hat T,\hat G)$ preserves the set of positive roots in $R(\hat T,\hat G) \sm R(\hat T,\hat M)$. It follows that if $u \in \Omega(\hat T,\hat M)$, then $u\beta>0$ and for the same reason $[w_M(\sigma)\sigma]^{-1} u\beta = v[w_M(\sigma)\sigma]^{-1}\beta<0$, with $v = [w_M(\sigma)\sigma]^{-1}u[w_M(\sigma)\sigma] \in \Omega(\hat T,\hat M)$. This shows that the set \eqref{eq:x1} is $\Omega(\hat T,\hat M)$-invariant, hence its sum $\alpha_{\sigma,\tau}$ is $\Omega(\hat T,\hat M)$-fixed.

We have thus proved $t(\sigma,\tau) \in Z(\hat M)^\circ$, i.e. $t \in Z^2(\Gamma_{K/F},Z(\hat M)^\circ)$. But then \cite[Lemma 4]{Lan79} implies that there is $r \in C^1(W_{K/F},Z(\hat M)^\circ)$ whose differential is the inflation of $t$. This means that
\[ \xi : w \mapsto r(w) n_M(\sigma_w) \rtimes \sigma_w \]
is a homomorphism $W_{K/F} \rw N(\hat T,\hat G) \rtimes W_F$. Since $r(w) \in Z(\hat M)^\circ$, it acts trivially on the root spaces of $\hat M$ and thus $\xi(w)$ preserves the pinning of $\hat M$.

We claim that after inflating $\xi$ to $W_F$, its restriction to $P_F$ is trivial. Since $K/F$ is tame, the image of $P_F$ in $W_{K/F}$ is equal to $K^\times_{0+}$, so we must check that $r(w)=1$ when $w \in K^\times_{0+}$. This can be extracted from the proof of \cite[Lemma 4]{Lan79}. It proceeds by embedding $Z(\hat M)^\circ$ into an exact sequence
\[ 1 \to Z(\hat M)^\circ \to \hat S_1 \to \hat S_2 \to 1, \]
where $S_1$ and $S_2$ are tori defined over $F$ and split over $K$ and $S_1$ is induced. Then $r(w)$ is expressed as $d^{-1}(w)c(\sigma_w)a^{-1}\sigma_w(a)$, where $c \in C^1(\Gamma_{K/F},\hat S_1)$ is chosen so that its co-boundary is $t$ (it exists because $S_1$ is induced and furthermore $H^2(\Gamma,\C^\times)=H^3(\Gamma,\Z)=0$, the latter because $\Gamma$ has strict cohomological dimension $2$)
and $d \in Z^1(W_{K/F},\hat S_1)$ and $a \in \hat S_1$ are chosen so that the equation $d(w)=c(\sigma_w)a^{-1}\sigma_w(a)$ holds in $\hat S_2$ (they exist because $H^1(W_{K/F},\hat S_1) \to H^1(W_{K/F},\hat S_2)$ is surjective, which follows from the injectivity of $S_2(F) \to S_1(F)$ and the Langlands correspondence for tori).

The tameness of $K/F$ implies that both $c(\sigma_w)$ and $a^{-1}\sigma_w(a)$ are trivial when $w \in K^\times_{0+}$, because then $\sigma_w=1 \in \Gamma_{K/F}$. To show that $d$ can be chosen trivial on $K^\times_{0+}$, we use the fact that $S_2(F)/S_2(F)_{0+} \to S_1(F)/S_1(F)_{0+}$ is injective by Lemma \ref{lem:mpex1}, and hence $H^1(W_{K/F}/K^\times_{0+},\hat S_1) \to H^1(W_{K/F}/K^\times_{0+},\hat S_2)$ is surjective by \cite[Theorem 7.10]{Yu09}.
\end{proof}

Composing the unramified $L$-embedding ${^Lj_{S,M}}:{^LS} \to {^LM}$ with a tame $L$-embedding ${^Lj_{M,G}}: {^LM} \to {^LG}$ supplied by Lemma \ref{lem:tamelemb} we obtain an $L$-embedding $^Lj_1 : {^LS} \to {^LG}$, which extends $\hat j : \hat S \to \hat G$. Then ${^Lj_1} = {^Lj}\cdot b$, for some $b \in Z^1(W_F,\hat S)$. Let $\theta_b : S(F) \to \C^\times$ be the character corresponding to $b$. By \cite[Theorem 7.10]{Yu09} and the fact that $^Lj_1$ and $^Lj$ agree on $P_F$ we see that $\theta_b$ is trivial on $S(F)_{0+}$. We claim that $\theta_b$ is $\Omega(S,M)^\Gamma$-invariant. If $w \in \Omega(S,M)^\Gamma$ then ${^Lj}\circ w = \tx{Ad}(n){^Lj}$ by \cite[(2.6.2)]{LS87} and the fact that $^Lj$ is produced from $\Omega(S,M)^\Gamma$-invariant $\chi$-data, where $n \in N(\hat T,\hat M)$ is a suitable element representing $w$. At the same time ${^Lj_1}\circ w = \tx{Ad}(n')\circ{^Lj_1}$ -- we have ${^Lj_{S,M}}\circ w = \tx{Ad}(n')\circ{^Lj_{S,M}}$, for some possibly different $n' \in N(\hat T,\hat M)$ lifting $w$, for the same reason as for $^Lj$, namely the $w$-invariance of the unramified $\chi$-data, and we have ${^Lj_{M,G}}\circ\tx{Ad}(n')=\tx{Ad}(n')\circ{^Lj_{M,G}}$ tautologically, because $^Lj_{M,G}$ is a group homomorphism extending the identity $\hat M \to \hat G$. We conclude that $b$ and $b \circ w$ are $\hat S$-conjugate, hence $\theta_b=\theta_b \circ w$.

We now have $\varphi={^Lj}\circ\varphi_S = {^Lj_1}\circ b \cdot \varphi_S$ and can proceed with the computation of $\tx{Cent}(\varphi(I_F),\hat G)^\circ$. Clearly $\tx{Cent}(\varphi(I_F),\hat G)^\circ = \tx{Cent}(\varphi(I_F),\hat M)^\circ = \tx{Cent}({^Lj}_1\circ b\cdot \varphi_S(I_F),\hat M)^\circ = \tx{Cent}({^Lj}_{S,M}\circ b\cdot \varphi_S(I_F),\hat M)^\circ$. Since ${^Lj}_{S,M}\circ b\cdot \varphi_S(P_F) \subset Z(\hat M)$ the action of ${^Lj}_{S,M}\circ b\cdot \varphi_S(I_F)$ on $\hat M$ factors through the pro-cyclic quotient $I_F/P_F$. Let $x \in I_F$ be a pre-image of a generator of this quotient and let $\theta$ be the automorphism ${^Lj}_{S,M}\circ b\cdot \varphi_S(x)$ of $\hat M$. It is semi-simple (in fact of finite order)
and by \cite[Theorem 7.5]{Ste68end} it preserves a Borel pair of $\hat M$. Conjugating within $\hat M$ we may assume that $\theta$ preserves the Borel pair belonging to the fixed pinning of $\hat M$. But then it must be given by an element $t \rtimes x \in \hat T \rtimes I_F$. The connected centralizer in $\hat M$ of this automorphism is a reductive subgroup $\hat M^{t \rtimes x,\circ} \subset \hat M$ with maximal torus $\hat T^{x,\circ}$. We recall the description of its root system from \cite[\S1.3]{KS99}. One divides the roots $R(\hat T,\hat M)$ in three types, depending on their image in the relative (and possibly non-reduced) root system $R(\hat T^{x,\circ},\hat M)$. One says that $\hat\alpha \in R(\hat T,\hat M)$ is of type R1/R2/R3, if its image $\hat\alpha_\tx{res} \in R(\hat T^{x,\circ},\hat M)$ is a relative root that: is neither divisible nor multipliable/is multipliable/is divisible. For any $\hat\alpha \in R(\hat T,\hat M)$ we denote by $N\hat\alpha$ the sum of all elements of the orbit of $\hat\alpha$ under the automorphism $x$. Then $\hat\alpha_\tx{res}$ is a root of $\hat M^{t \rtimes x,\circ}$ if and only if it is either of type R1 or R2 and $N\hat\alpha(t)=1$ or if it is of type R3 and $N\hat\alpha(t)=-1$.

Our goal is to show that neither of these cases occurs. For any $\hat\alpha \in R(\hat T,\hat M)$ the homomorphism $N\hat\alpha : \hat T \to \C^\times$ is $I$-invariant and descends to a homomorphism $\hat T_I \to \C^\times$. Note here that we are using the $\Gamma$-action on $\hat T$ inherited from $\hat M$ and not from $\hat G$. In particular, ${^Lj_{S,M}}$ restricts to an isomorphism $\hat S \rtimes I \to \hat T \rtimes I$. Consider the composition
\begin{equation} \label{eq:ic} I_F \stackrel{b\varphi_S}{\lrw} \hat S \rtimes I \stackrel{{^Lj_{S,M}}}{\lrw} \hat T \rtimes I \to \hat T_I \stackrel{N\hat\alpha}{\lrw} \C^\times. \end{equation}
The restriction of this homomorphism to $P_F$ is trivial and the image of $x \in I_F$ under this homomorphism is equal to the value $N\hat\alpha(t)$. We want to show that this image is not equal to $1$ when $\hat\alpha$ is of type R1 or R2 and is not equal to $-1$ when $\hat\alpha$ is of type R3. For a moment let us assume that $\hat\alpha$ is of type R1. Then $N\hat\alpha(t) \neq 1$ is equivalent to \eqref{eq:ic} being non-trivial. We shall interpret the homomorphism \eqref{eq:ic} in terms of the character $\theta_b \cdot \theta$. Since $M$ is quasi-split there exists by Lemma \ref{lem:qstrans} an admissible embedding $S \to M$ defined over $F$. By assumption on $\theta$ the image is a maximally unramified maximal torus and using Lemma \ref{lem:super} we may choose the admissible embedding so that the associated point in $\mc{B}^\tx{red}(M,F)$ is superspecial. Let $F'/F$ be an unramified extension over which $S$ becomes a minimal Levi subgroup of $M$. The point $x$ is still special over $F'$ and the root system of $\ms{M}_x^\circ$ is the subsystem of non-divisible roots in $R(A_S,M)$. The bijection $R(S,M) \leftrightarrow R(\hat S,\hat M)$ sending $\hat\alpha$ to $\alpha=\hat\alpha^\vee$ restricts to a bijection $R(A_S,M) \leftrightarrow R(\hat S^{x,\circ},\hat M)$ that preserves types. Thus $\hat\alpha$ corresponds to a root $\alpha \in R(S,M)$ whose restriction to $A_S$ is neither divisible nor multipliable. This root is then also an element of $R(\ms{S}',\ms{M}_x^\circ)$ and the corresponding coroot is $N\alpha^\vee=N\hat\alpha$. The $L$-embedding $^Lj_{S,M}$ restricts to an isomorphism $\hat S \rtimes W_{F'} \to \hat T \rtimes W_{F'}$ and
\[ W_{F'} \stackrel{b\varphi_S}{\lrw} \hat S \rtimes W_{F'} \stackrel{{^Lj_{S,M}}}{\lrw} \hat T \rtimes W_{F'} \to \hat T_{W_F'} \stackrel{N\hat\alpha}{\lrw} \C^\times \]
is the parameter of the character $[\theta_b \cdot \theta]\circ N_{F'/F}  \circ [N\alpha^\vee]$. The character $[\theta_b \cdot \theta]_{S(F)_0}$ has trivial stabilizer on $\Omega(S,M)^\Gamma$ and thus reduces to a character of $\ms{S}'(k_F)$ in regular position. According to Lemma \ref{lem:dl2} the composition $[\theta_b \cdot \theta]\circ N_{F'/F}  \circ [N\alpha^\vee]$ is a non-trivial character of $k_{F'}^\times$, or, seen as a character of $[F']^\times$, has non-trivial restriction to $O_{F'}^\times$. This in turn is equivalent to the claim that its parameter restricts non-trivially to $I_{F'}=I_F$. But that restriction is exactly \eqref{eq:ic}.

We now turn to the cases where $\hat\alpha$ is of type R2 or R3. These cases are linked together -- if $\hat\alpha$ is a root of type R2 and $l$ is the smallest positive number such that $x^l\hat\alpha=\hat\alpha$, then $l$ is even and $\hat\beta = \hat\alpha+x^{l/2}\hat\alpha$ is a root of type R3. Conversely, every root of type R3 occurs in this way. In this situation, we have $N\hat\beta=N\hat\alpha$. The cases of $\hat\alpha$ and $\hat\beta$ will be handled simultaneously if we can show $(2N)\hat\alpha(t) \neq 1$. But the bijection $R(A_S,M) \leftrightarrow R(\hat S^{x,0},\hat M)$ sends $\hat\alpha$ to a non-divisible relative root $\alpha$ which then occurs in $R(\ms{S}',\ms{M}_x^\circ)$. Its coroot is $2N\alpha^\vee=2N\hat\alpha$. The same argument now shows that the homomorphism \eqref{eq:ic}, where we replace $N\hat\alpha$ by $(2N)\hat\alpha$, is non-trivial.

We have thus shown that $C:=\tx{Cent}(\varphi(I_F),\hat M)^\circ = \hat M^{t \rtimes x,\circ}$ is a reductive group with maximal torus $\hat T^{x,\circ}$ and an empty root system, so it equals $T^{x,\circ}$ and is thus contained in $\hat T$.

It now remains to check the third property in Definition \ref{def:rsp}. Let $n' \in N(\hat T,\hat M)$ project to $w \in \Omega(S,M)^\Gamma$ and centralize $\varphi(I_F)$. Write $n'=s^{-1}n$, where $s \in \hat S$ and $n \in N(\hat T,\hat M)$ satisfies $\tx{Ad}(n)\circ {^Lj} = {^Lj}\circ w$. The latter equality together with $\tx{Ad}(n)\circ \varphi|_{I_F} = \tx{Ad}(s)\circ\varphi|_{I_F}$ implies $w \circ \varphi_S|_{I_F}=\tx{Ad}(s)\circ\varphi_S|_{I_F}$. By Lemma \ref{lem:llcres1} this means that $w$ stabilizes $\theta|_{S(F)_0}$. The regularity of $\theta$ implies $w=1$.

We now give the inverse construction. Let $\varphi : W_F \to \hat G \rtimes W_F$ be a regular supercuspidal parameter. Let $\hat T \subset \hat M \subset \hat G$ be the maximal torus and Levi subgroup from Lemma \ref{lem:levisystem2}, normalized by $\tx{Ad}(\varphi(-))$, and let $\hat S$ be the $\Gamma$-module with underlying abelian group $\hat T$ and $\Gamma$-action given by $\tx{Ad}(\varphi(-))$. By construction $\varphi(P_F) \subset Z(\hat M) \subset \hat T$, so the $\Gamma$-module $\hat S$ is tame. Since $\varphi(I_F)$ preserves a Borel subgroup of $\hat M$ containing $\hat T$, the action of $I_F$ on $R(\hat S,\hat M)$ preserves a positive chamber. We let $\hat j : \hat S \to \hat G$ be the tautological embedding of the abelian group $\hat T$ underlying $\hat S$ into $\hat G$.

\begin{lem} There exists $\chi$-data for $R(S,G)$ that is minimal and $\Omega(S,M)^\Gamma$-invariant.
\end{lem}
\begin{proof}
We mimick some of the arguments of the Howe factorization algorithm of Subsection \ref{sub:howe} and begin by considering the filtration
\[ R_r = \{\alpha \in R(S,G)| \hat\alpha(\varphi(I^r))=1 \}. \]
This is well-defined because $\varphi(P_F) \subset Z(\hat M) \subset \hat T$. Let $r_{d-1} > \dots > r_0 > 0$ be the jumps of this filtration. Set in addition $r_d = \tx{depth}(\varphi)$ and $r_{-1}=0$. Thus $R_{0+}=R(S,M)$ and $R_{r_{d-1}+}=R(S,G)$. Fix an additive character $\Lambda : F \to \C^\times$. For notational convenience, we assume $\Lambda$ is of depth zero, i.e. trivial on $F_{0+}$ but not on $F_0$. Given $\alpha \in R(S,G) \sm R(S,M)$, let $d > i \geq 0$ such that $\alpha \in R_{r_i+} \sm R_{r_i}$.

We define a character $\zeta_\alpha : [F_\alpha^\times]_{r_i}/[F_\alpha^\times]_{r_i+} \to \C^\times$ as follows. The composition $\hat\alpha\circ\varphi$ gives a homomorphism $I^{r_i} \to \C^\times$ that is trivial on $I^{r_i+}$. We claim that this homomorphism can be extended to $W_{F_\alpha}$. Indeed, if we fix arbitrary tame $\chi$-data for $R(S,G)$ we obtain a tame $L$-embedding ${^Lj} : \hat S \rtimes W_F \to \hat G \rtimes W_F$ containing the image of $\varphi$ and hence a factorization $\varphi={^Lj}\circ\varphi_S$ with $\varphi_S : W_F \to \hat S \rtimes W_F$ having the property $\varphi_S|_{P_F}=\varphi|_{P_F}$, and then $\hat\alpha\circ\varphi_S : W_{F_\alpha} \to \hat S \rtimes W_{F_\alpha} \to \hat S_{W_{F_\alpha}} \to \C^\times$ is an extension of $\hat\alpha\circ\varphi|_{P_F}$ to $W_{F_\alpha}$. Thus $\hat\alpha\circ\varphi$ corresponds to a character $\zeta_\alpha : [F_\alpha^\times]_{r_i}/[F_\alpha^\times]_{r_i+} \to \C^\times$, which of course does not depend on the extension of $\hat\alpha\circ\varphi$ to $W_{F_\alpha}$.
The equation
\[ \zeta_\alpha(X+1) = \Lambda(\tx{tr}_{F_\alpha/F}(\bar a_\alpha X)),\qquad X \in [F_\alpha]_{r_i}, \]
specifies an element $\bar a_\alpha \in [F_\alpha]_{-r_i}/[F_\alpha]_{-r_i+}$. We claim that $\{(-r_i,\bar a_\alpha)\}$ is a set of $\Omega(S,M)^\Gamma$-invariant mod-$a$-data for $R(S,G) \sm R(S,M)$. For this we need to show $\tau(\bar a_\alpha)=\bar a_{\tau\alpha}$, $-\bar a_\alpha=\bar a_{-\alpha}$, and $\bar a_{w\alpha}=\bar a_\alpha$, for $\tau \in \Gamma$ and $w \in \Omega(S,M)^\Gamma$. This in turn translates to $\zeta_{\tau\alpha} = \zeta_\alpha\circ\tau^{-1}$, $\zeta_{-\alpha}=\zeta_\alpha^{-1}$ and $\zeta_{w\alpha}=\zeta_\alpha$. Note that $\tau : F_\alpha \to F_{\tau\alpha}$ is an isomorphism of $F$-algebras and $F_{w\alpha}=F_\alpha$, so these formulas make sense.
The claimed properties of the characters $\zeta_\alpha$ are seen as follows: Going from $\alpha$ to $-\alpha$ is trivial, going from $\alpha$ to $w\alpha$ comes from the fact that $w$ can be represented by conjugation by an element of $\hat M$, which centralizes $\varphi(P_F)$, and going from $\alpha$ to $\tau\alpha$ comes from
\[ \tau(\varphi(w))=\tau(\varphi_S(w))=\varphi_S(\tau)\varphi_S(w)\varphi_S(\tau)^{-1}=\varphi_S(\tau w \tau^{-1})=\varphi(\tau w \tau^{-1}), \]
for $w \in P_F$. From the $\Omega(S,M)^\Gamma$-invariant mod-$a$-data we obtain $\Omega(S,M)^\Gamma$-invariant tame $\chi$-data for $R(S,G) \sm R(S,M)$ by \eqref{eq:chi'}. We augment this with unramified $\chi$-data for $R(S,M)$, which suffices since $R(S,M)$ has no ramified symmetric roots, and which is automatically $\Omega(S,M)^\Gamma$-invariant.
\end{proof}

We have thus proved the existence of $\Omega(S,M)^\Gamma$-invariant minimal $\chi$-data for $R(S,G)$. We fix one such $\chi$ and obtain an $\hat G$-conjugacy class of $L$-embedding $^Lj_\chi : \hat S \rtimes W_F \to \hat G \rtimes W_F$. Fix $^Lj$ within its conjugacy class by demanding $^Lj_\chi|_{\hat S}=\hat j$. The image of $^Lj_\chi$ contains the image of $\varphi$ and we obtain the factorization
\[ \varphi = {^Lj_\chi} \circ \varphi_{S,\chi}, \]
for some Langlands parameter $\varphi_{S,\chi} : W_F \rw {^LS}$. Let $\theta=\theta_\chi : S(F) \rw \C^\times$ be the corresponding character. Since any $L$-embedding that is $\hat G$-conjugate to $^Lj_\chi$ and also restricts to $\hat j$ must be conjugate to $^Lj_\chi$ by an element of $\hat T$, the $\hat S$-conjugacy class of $\varphi_{S,\chi}$, and hence the character $\theta_\chi$, are independent of the choice of $^Lj_\chi$. They depend only on the choice of $\chi$.
\begin{lem}
The stabilizer of $\theta|_{S(F)_0}$ in $\Omega(S,M)^\Gamma$ is trivial.
\end{lem}
\begin{proof}
This is equivalent to $[\theta\circ w/\theta]|_{S(F)_0} \neq 1$ for all $w \in \Omega(S,M)^\Gamma$. By Lemma \ref{lem:llcres1} this is equivalent to $w\circ \varphi_S \neq \varphi_S$ in $H^1(I_F,\hat S)$. Explicitly we need to show that for any $w \in \Omega(S,M)^\Gamma$ there does not exist $s_w \in \hat S$ with $w \circ \varphi_S|_{I_F} = \tx{Ad}(s_w)\varphi_S|_{I_F}$. Assume this fails for some $w$ and let $s_w$ be the corresponding element. We have the equality
\[ w\circ \varphi_S|_{I_F} = \tx{Ad}(s_w)\circ\varphi_S|_{I_F}. \]
Composing both sides of the displayed equation with $^Lj_\chi$ we obtain
\[ {^Lj_\chi}\circ w\circ\varphi_S|_{I_F} = \tx{Ad}(s_w)\varphi|_{I_F}. \]
We may now use \cite[(2.6.2)]{LS87} together with the $w$-invariance of $\chi$-data to get ${^Lj}\circ w=\tx{Ad}(n)\circ{^Lj}$, where $n \in N(\hat T,\hat M)$ is a suitable lift of $w$. This leads to
\[ \tx{Ad}(n)\varphi|_{I_F} = \tx{Ad}(s_w)\varphi|_{I_F}, \]
i.e. $s_w^{-1}n \in \tx{Cent}(\varphi(I_F),\hat M)$ contradicting part 3 of Definition \ref{def:rsp}.
\end{proof}
We now form $(S,\hat j,\chi,\theta_\chi)$ and claim that its isomorphism class depends only on the $\hat G$-class of $\varphi$.
Keeping $\varphi$ fixed, recall from Subsection \ref{sub:delta2} that the $\chi$-data can only be changed to $\zeta \cdot \chi$ and then according to \cite[(2.6.3)]{LS87} we have ${^Lj}_{\zeta\chi} = {^Lj}_\chi \cdot c$, where $c$ is the element of $Z^1(W_F,\hat S)$ defined in \cite[Corollary 2.5.B]{LS87}. Thus $\varphi_{S,\zeta\chi} = \varphi_S \cdot c^{-1}$. As we already remarked, $c$ is the Langlands parameter for the character $\zeta_S$. Thus we obtain the isomorphic object $(S,\hat j,\zeta \cdot \chi,\theta_\chi \cdot \zeta_S^{-1})$.

If we replace $\varphi$ by $\varphi'=\tx{Ad}(g)\circ\varphi$ for some $g \in \hat G$ then $\hat T'=\tx{Ad}(g)\hat T$, where $\hat T'$ is the maximal torus analogous to $\hat T$ but corresponding to $\varphi'$, because $\hat T$ can be recovered from $\varphi$ as $\tx{Cent}(C,\hat M)$, with $\hat M=\tx{Cent}(\varphi(P_F),\hat G)^\circ$ and $C=\tx{Cent}(\varphi(I_F),\hat M)^\circ$. If we let $\hat S'$ be the $\Gamma$-module with abelian group $\hat T'$ and $\Gamma$-action given by $\tx{Ad}(\varphi'(-))$, then we see that $\tx{Ad}(g) : \hat S \to \hat S'$ is an isomorphism of $\Gamma$-modules. It gives rise to an isomorphism $\iota : S' \to S$ of algebraic tori. Choose minimal $\Omega(S,\hat M)^\Gamma$-invariant $\chi$-data on $R(S,G)$ and transport it via $\iota$ to $\chi$-data $\chi'$ on $R(S',G)$, which is minimal and $\Omega(S',\hat M)^\Gamma$-invariant. Use it to obtain a character $\theta_{\chi'}  : S'(F) \to \C^\times$. One then checks immediately that the isomorphism $\iota$ identifies the characters $\theta_\chi$ and $\theta_{\chi'}$.
\end{proof}

Let now $(S,\hat j,\chi,\theta)$ be a regular supercuspidal $L$-packet datum. We will define a function $\Theta : S(F)_\tx{reg} \rw \C$ as follows. Choose a non-trivial character $\Lambda : F \rw \C^\times$. For any $\alpha \in R(S,G)$ we have the character $\Lambda \circ \tx{tr}_{F_\alpha/F} : F_\alpha \rw \C^\times$; let $r_{\Lambda,\alpha}$ be its depth. On the other hand, we have the character $\theta \circ N_{F_\alpha/F} \circ \alpha^\vee : F_\alpha^\times \rw \C^\times$; let $r_{\theta,\alpha}$ be its depth. By restriction we obtain a character
\[ [F_\alpha]_{r_{\theta,\alpha}}/[F_\alpha]_{r_{\theta,\alpha}+}  \rw \C^\times, X \mapsto \theta \circ N_{F_\alpha/F} \circ \alpha^\vee(X+1). \]
Let $\bar a_\alpha \in [F_\alpha]_{(r_{\Lambda,\alpha}-r_{\theta,\alpha})}/[F_\alpha]_{(r_{\Lambda,\alpha}-r_{\theta,\alpha})+}$ be the unique element satisfying
\[ \theta \circ N_{F_\alpha/F} \circ \alpha^\vee(X+1) = \Lambda\circ\tx{tr}_{F_\alpha/F}(\bar a_\alpha X). \]
It is immediate to check that $\{(r_{\Lambda,\alpha}-r_{\theta,\alpha},\bar a_\alpha)\}$ is a set of mod-$a$-data.
Define
\begin{equation} \label{eq:charconj} \Theta(\gamma) := \epsilon_L(X^*(T)_\C - X^*(S)_\C,\Lambda)\Delta_{II}^\tx{abs}[\bar a,\chi](\gamma)\theta(\gamma).\end{equation}

\begin{lem} \label{lem:thetaindep} The function $\Theta$ depends only on the datum $(S,\hat j,\chi,\theta)$. Any isomorphism $(S,\hat j,\chi,\theta) \rw (S',\hat j',\chi',\theta')$ carries $\Theta$ over to the corresponding function $\Theta'$ on $S'(F)_\tx{reg}$.
\end{lem}
\begin{proof}
The character $\Lambda$ can be replaced by $\Lambda \cdot c$ for some $c \in F$, where we recall that $[\Lambda \cdot c](x)=\Lambda(cx)$. Then $\bar a_\alpha$ is replaced by $c^{-1}\bar a_\alpha$. Invoking \cite[Corollary 3.5.2]{KalEpi}, using that $\lambda_{F_\alpha/F_{\pm\alpha}}([\Lambda \cdot c] \circ \tx{tr}_{F_{\pm\alpha/F}}) = \kappa_\alpha(c)\lambda_{F_\alpha/F_{\pm\alpha}}(\Lambda \circ \tx{tr}_{F_{\pm\alpha}/F})$, and appealing to Lemma \ref{lem:d2a}, we see that \eqref{eq:charconj} is unchanged.

We now discuss the isomorphisms $(S,\hat j,\chi,\theta) \rw (S',\hat j',\chi',\theta')$. Again we can treat the three basic isomorphism types $(\iota,1,1)$, $(1,g,1)$ and $(1,1,\zeta)$ separately. For the first two the statement is trivial. For the third, we have by definition $\chi=\chi' \cdot\zeta$ and $\theta=\theta' \cdot \zeta_S^{-1}$ and the statement follows from Lemma \ref{lem:d2c}.
\end{proof}

It is clear from Proposition \ref{pro:pardata} that the isomorphism classes of regular supercuspidal $L$-packet data will correspond to $L$-packets. We now introduce another category, which we call the category of regular supercuspidal data, whose isomorphism classes of object will correspond to the individual supercuspidal representations that are to be organized into $L$-packets. The objects in this category are tuples $(S,\hat j,\chi,\theta,(G_1,\xi,z),j)$, where $(S,\hat j,\chi,\theta)$ is a regular supercuspidal $L$-packet datum, $(G',\xi,z)$ is a rigid inner twist of $G$ in the sense of \cite[\S5.1]{KalRI}, and $j : S \rw G'$ is an admissible embedding defined over $F$. A morphism $(S_1,\hat j_1,\chi_1,\theta_1,(G'_1,\xi_1,z_1),j_1) \rw (S_2,\hat j_2,\chi_2,\theta_2,(G'_2,\xi_2,z_2),j_2)$ in this category is given by $(\iota,g,\zeta,f)$, where $(\iota,g,\zeta)$ is an isomorphism of the underlying regular supercuspidal $L$-packet data, $f : (G'_1,\xi_1,z_1) \rw (G_2',\xi_2,z_2)$ is an isomorphism of rigid inner twists, and $j_2\circ \iota = f \circ j_1$. There is an obvious forgetful functor from the category of regular supercuspidal data to the category of regular supercuspidal $L$-packet data. If we fix a regular supercuspidal $L$-packet datum $(S,\hat j,\chi,\theta)$, the set of isomorphism classes of regular supercuspidal data mapping to it is a torsor under $H^1(u \rw W,Z(G) \rw S)$. This torsor is given by the relation
\[ x \cdot (G_1',\xi_1,z_1,j_1) = (G_2',\xi_2,z_2,j_2) \Leftrightarrow x = \tx{inv}(j_1,j_2), \]
see \cite[\S5.1]{KalRI}.

We will now attach to each regular supercuspidal datum $(S,\hat j,\chi,\theta,(G',\xi,z),j)$ a regular supercuspidal representation of $G'(F)$. For this we take our lead from the construction of $L$-packets of real discrete series representations \cite{Lan89}. Ideally we would like to take ``the'' regular supercuspidal representation of $G'(F)$ whose Harish-Chandra character, evaluated at shallow regular elements of $S(F)$, is given by the formula
\begin{equation} \label{eq:charwish} e(G')|D_{G'}(\gamma')|^{-\frac{1}{2}}\sum_{w \in \Omega(jS(F),G'(F))} \Theta(j^{-1}(\gamma'^w)), \end{equation}
where $\Theta$ is the function \eqref{eq:charconj}. We don't know yet quite enough about the Harish-Chandra character of regular supercuspidal representations to know whether this would specify a unique representation. However, we can achieve the same result by the following construction, which, while less elegant, has the virtue of describing explicitly the inducing datum of the representation.

From $\theta$ we construct mod-$a$-data by \eqref{eq:adatahowe} and then $\chi$-data for $R(S,G)$ by \eqref{eq:chi'}. The mod-$a$-data depends on the choice of an additive character $\Lambda : F \to \C^\times$, but the resulting $\chi$-data does not. Replace $(S,\hat j,\chi,\theta,(G',\xi,z),j)$ by an isomorphic object in which the $\chi$-data is the one just constructed. Note that if we constructed mod-$a$-data and $\chi$-data with respect to the new $\theta$, we'd obtain an equivalent result, because the difference between the new and old $\theta$ is tamely ramified, see the next paragraph. Consider the maximal torus $jS \subset G'$ and the character on it given by $j\theta' := \theta\circ j^{-1} \cdot e_{f,r} \cdot \epsilon^r$. Here $e_{f,r}$ is the character of $jS(F)$ constructed prior to the statement of Lemma \ref{lem:efr}, and $\epsilon^r$ is given by \eqref{eq:er}. The character $j\theta'$ is regular according to Facts \ref{fct:eramweyl} and \ref{fct:eframweyl}, but may fail to be extra regular due to the occurrence of $e_{f,r}$. The representation of $G'(F)$ corresponding to the regular supercuspidal datum $(S,\hat j,\chi,\theta,(G',\xi,z),j)$ is then $\pi_{(jS,j\theta')}$. It is regular, but may fail to be extra regular. However, it will be extra regular at least when the point of $\mc{B}([G']^0,F)$ associated to $jS$ is superspecial, by Lemma \ref{lem:weyl2}.

We claim that the character of this representation, evaluated at shallow regular elements of $S(F)$, is given by \eqref{eq:charwish}. This follows at once from Corollary \ref{cor:charshallow} once the following remark has been made: The mod-$a$-data and $\chi$-data occurring in that formula are computed from the character $j\theta'$, while the mod-$a$-data and $\chi$-data we used here were computed from $\theta$. For all $\alpha \in R(S,G) \sm R(S,G^0)$, the mod-$a$-data depends only on $j\theta'|_{jS(F)_{0+}}$ respectively $\theta|_{S(F)_{0+}}$. Since $e_{f,r} \cdot \epsilon^r$ has depth zero, the two restrictions are identified by $j$ and the corresponding mod-$a$-data are equal. The same is then true for the $\chi$-data, which are computed in terms of the mod-$a$-data via \eqref{eq:chi'}. For $\alpha \in R(S,G^0)$, the mod-$a$-data might be different, but they are units in both cases, and since we are taking unramified $\chi$-data in both cases, this difference is irrelevant.

From now on we assume that $p$ does not divide $|\pi_0(Z(G))|$.

We now define the compound $L$-packet $\Pi_\varphi$ to be the following set of equivalence classes of representations of rigid inner twists. Fix a regular supercuspidal $L$-packet datum $(S,\hat j,\chi,\theta)$ corresponding to $\varphi$. For each regular supercuspidal datum $(S,\hat j,\chi,\theta,(G',\xi,z),j)$ let $\pi_j$ be the representation of $G'(F)$ just constructed. Then
\begin{equation} \label{eq:lpack} \Pi_\varphi = \{(G',\xi,z,\pi_j)\}, \end{equation}
where $(S,\hat j,\chi,\theta,(G',\xi,z),j)$ runs over all regular supercuspidal data mapping to the regular supercuspidal $L$-packet datum $(S,\hat j,\chi,\theta)$. By Lemma \ref{lem:super} there exist at least one regular supercuspidal datum for which the point in $\mc{B}([G']^0,F)$ associated to $jS$ is superspecial, which shows that $\Pi_\varphi$ contains extra regular supercuspidal representations.

\subsection{Parameterization of $L$-packets} \label{sub:packpar}

As in the previous subsection, we are assuming that the residual characteristic of $F$ is odd, is not a bad prime for $G$, and does not divide $|\pi_0(Z(G))|$. We also keep the assumption that the characteristic of $F$ is zero due to our useage of \cite{KalRI}.

Let $\varphi : W_F \rw {^LG}$ be a regular supercuspidal parameter and let $(S,\hat j,\chi,\theta)$ be a corresponding regular supercuspidal $L$-packet datum.

\begin{lem} The embedding $\hat j : \hat S \to \hat G$ induces an isomorphism $\hat S^\Gamma \to S_\varphi$.
\end{lem}
\begin{proof}
Recall that $\hat j$ maps $\hat S$ to a maximal torus of $\hat G$ normalized by $\varphi$ and that we have the equation $^Lj_\chi\circ\varphi_{S,\chi} = \varphi$. From this it is immediate that $\hat j(\hat S^\Gamma) = S_\varphi \cap \hat T$. It is thus enough to show $S_\varphi \subset \hat T$. Let $s \in S_\varphi$. Then $s \in \tx{Cent}(\varphi(P_F),\hat G)=\hat M$. Furthermore, $s \in \tx{Cent}(\varphi(I_F),\hat G)$ and thus normalizes $C=\tx{Cent}(\varphi(I_F),\hat G)^\circ$ and then also $\hat T=\tx{Cent}(C,\hat M)$. It follows that $s \in N(\hat T,\hat M)$. Since the projection of $s$ to $\Omega(\hat S,\hat M)$ is $\Gamma$-fixed we conclude that it must be trivial, i.e. $s \in \hat T$.
\end{proof}

For any finite subgroup $Z \subset Z(G)$ this isomorphism extends to an isomorphism $[\hat{\bar S}]^+ \rw S_\varphi^+$, and via \cite[Proposition 5.3]{KalRIBG} to an isomorphism $H^1(u \rw W,Z \rw S) \rw \pi_0(S_\varphi^+)^D$.

The constituents of $\Pi_\varphi$ are in canonical bijection with the set of isomorphism classes of regular supercuspidal data that map to the isomorphism class of $(S,\hat j,\chi,\theta)$ under the forgetful functor. We have already argued that this set is a torsor under $H^1(u \rw W,Z \rw S)$. In this way, we obtain a canonical simply transitive action of $\pi_0(S_\varphi^+)^D$ on $\Pi_\varphi$.

In order to obtain a bijection $\Pi_\varphi \to \pi_0(S_\varphi^+)^D$ from this simply transitive action, we need to fix a base point. Fix a Whittaker datum $\mf{w}$ for $G$. According to the strong tempered $L$-packet conjecture there should exist a unique constituent of $\Pi_\varphi$ that is $\mf{w}$-generic. At the moment we can prove this conjecture only in the case of toral representations, see Lemma \ref{lem:generic}. The same argument should go through without much modification once the character formula for regular supercuspidal representations, which is currently being developed in the work of Spice and others, is known. Granted this result, let $j_\mf{w}$ be the admissible embedding $S \rw G$ so that $\pi_{j_\mf{w}}$ is the generic constituent. Then we obtain the perfect pairing
\[ \<-,-\>_\mf{w} : \Pi_\varphi \times \pi_0(S_\varphi^+) \rw \C \]
by
\[ \<(G',\xi,z,\pi_j),s\>_\mf{w} = \<\tx{inv}(j_\mf{w},j),s\>, \]
where on the right the pairing comes from the isomorphism $H^1(u \rw W,Z \rw S) \rw \pi_0(S_\varphi^+)^D$. Then the map $s \mapsto \<(G',\xi,z,\pi_j),s\>_\mf{w}$ is a character of $\pi_0(S_\varphi^+)$, while the map $(G',\xi,z,\pi_j) \mapsto  \<(G',\xi,z,\pi_j),-\>_\mf{w}$ is a bijection identifying $\Pi_\varphi$ with $\pi_0(S_\varphi^+)^D$.

If $\mf{w'}$ is another Whittaker datum, then we have
\begin{equation} \label{eq:wc} \<(G',\xi,z,\pi_j),s\>_\mf{w'} = \<(G',\xi,z,\pi_j),s\>_\mf{w} \cdot \<\tx{inv}(j_\mf{w'},j_\mf{w}),s\>. \end{equation}

\subsection{Comparison with the case of real groups} \label{sub:realllc}

Continuing the theme of Subsection \ref{sub:realchar} we will now show that the construction of the regular supercuspidal part of the local Langlands correspondence given in Subsections \ref{sub:packconst} and \ref{sub:packpar} is a direct generalization of Langlands' construction \cite{Lan83} of real discrete series $L$-packets and Shelstad's \cite{She82}, \cite{SheTE2}, \cite{SheTE3} parameterization of these.

In this subsection only, let $G$ be a connected reductive group defined and quasi-split over $\R$ and let $\varphi : W_\R \rw {^LG}$ be a discrete Langlands parameter. We briefly recall the construction of the correspondence, following the exposition in \cite[\S5.6]{KalRI}. One chooses a Borel pair $(\hat T,\hat B)$ in $\hat G$ and modifies $\varphi$ in its conjugacy class so that $\varphi(\C^\times) \subset \hat T$. Write $\varphi(z)=z^\mu \bar z^\nu$, with $\mu,\nu \in X_*(\hat T)_\C$, $\mu-\nu \in X_*(\hat T)$. One shows that the image of $\mu$ in $X_*(\hat T_\tx{ad})_\C$ is integral, i.e. belongs to $X_*(\hat T_\tx{ad})$, and moreover regular \cite[Proof of Lemma 3.3]{Lan83}. One then modifies $\varphi$ again within its conjugacy class so that this image is $\hat B$-dominant. The parameter $\varphi$ is now pinned down within its $\hat G$-conjugacy class up to conjugation by $\hat T$. The action of $W_\R$ on $\hat T$ via $\tx{Ad}(\varphi(w))$ factors through $\Gamma_\R$ and gives a twist $\hat S$ of the $\Gamma$-structure on $\hat T$. The real torus $S$ dual to $\hat S$ comes equipped with a stable class of embeddings $S \rw G'$ (note that the images of any two such embeddings are conjugate under $G(\R)$, but the embeddings themselves need not be) into any inner form $G'$ of $G$ (this follows from \cite[Corollary 2.2]{Kot82} in the case of $G$ and from \cite[Lemma 2.8]{She79C} in general). By construction there is a distinguished Weyl-chamber in $X^*(S)$. Using based $\chi$-data \cite[\S9]{SheTE1} for $R(S,G)$ with respect to that chamber, we obtain an $L$-embedding ${^Lj} : {^LS} \rw {^LG}$ whose image contains the image of $\varphi$. We write $\varphi = {^Lj} \circ \varphi_S$, for $\varphi_S : W_\R \rw {^LS}$. The local Langlands correspondence for $S$ produces from $\varphi_S$ a character $\theta : S(\R) \rw \C^\times$. For any embedding $j : S \rw G$ we let $\pi_j$ be the unique discrete series representation of $G(\R)$ whose character evaluates at a strongly regular element $\gamma \in jS(\R)$ to the function \eqref{eq:chards}, where we are to replace $S$ and $\theta$ in this formula with $jS$ and $\theta\circ j^{-1}$. The $L$-packet on any inner form $G'$ of $G$ is defined to be the set $\{\pi_j\}$ where $j$ runs over the rational classes of embeddings $j :S \rw G'$ in the given stable class.

Fixing a Whittaker datum $\mf{w}$, there is a unique embedding $j_\mf{w} : S \rw G$ such that the corresponding representation $\pi_{j_\mf{w}}$ is $\mf{w}$-generic \cite{Kos78}, \cite{Vog78}. For a canonical internal parameterization of the $L$-packets we use rigid inner twists. Fix a finite subgroup $Z \subset G$, a rigid inner twist $(G',\xi,z)$  realized by $Z$, and an admissible rational embedding $j  : S \rw G'$. Given $s \in S_\varphi^+=[\hat{\bar S}]^+$, we define
\[ \<(G',\xi,z,\pi_j,s\>_\mf{w} = \<\tx{inv}(j_\mf{w},j),s\>, \]
where the pairing on the right is the one from \cite[Corollary 5.4]{KalRI}.

This exposition makes the direct analogy with the constructions of Subsections \ref{sub:packconst} and \ref{sub:packpar} almost obvious. In fact, the exposition here is already slightly different from the one presented in \cite[\S5.6]{KalRI} in that it uses $L$-embeddings and factorization of parameters, where in \cite[\S5.6]{KalRI} we kept more closely to the original construction in \cite{Lan83}. That the two presentations are equivalent is explained in \cite[\S7b]{SheTE2}. With Subsection \ref{sub:realchar} in mind, the only point where the construction of regular supercuspidal $L$-packets may seem to differ from that of real discrete series $L$-packets is  that in the real case one chooses a specific parameter within its $\hat G$-conjugacy class based on a pinning of $\hat G$ and the notion of dominance. This choice is important because the $L$-embedding $^Lj$ is constructed from based $\chi$-data with respect to the same Weyl chamber. But if we use the argument of Subsection \ref{sub:realchar} to rewrite the real discrete series character formula \ref{eq:chards} as \eqref{eq:charshallow}, then Lemma \ref{lem:thetaindep} tells us that we can use arbitrary $\chi$-data, at which point the $\hat B$-dominance of $\mu$ becomes irrelevant.

\section{Toral $L$-packets} \label{sec:toral}

In this section we will consider the special case of those regular supercuspidal $L$-packets whose constituents are toral supercuspidal representations. These are the representations arising from Yu-data of the form $(S \subset G,1,(\phi_0,1))$, where $\phi_0 : S(F) \to \C^\times$ is a $G$-generic character of positive depth. These representations were constructed by Adler in the paper \cite{Ad98}, which, as far as we know, was the first construction of supercuspidal representations for general reductive $p$-adic groups, and whose approach formed the basis of Yu's more general construction.

The class of toral supercuspidal representations is general enough to include the epipelagic representations \cite{RY14} when $p$ does not divide the order of the inertial action, and the representations considered by Reeder \cite{Ree08}. It is at the same time special enough so that the construction of $L$-packets simplifies considerably. The biggest advantage of this class of representations is that, from the current standpoint, they are the only ones of the regular supercuspidal representations for which the full character formula is known for all members of the $L$-packet. This will allow us to sharpen and extend our results -- we will prove the existence and uniqueness of a generic constituent in each toral $L$-packet, as well as the stability and endoscopic transfer of these packets.

\subsection{Construction and exhaustion} \label{sub:toralpackconst}
In this subsection we assume that the residual characteristic of $F$ is odd, not a bad prime for $G$, and not a divisor of $|\pi_0(Z(G))|$. We further assume that the characteristic of $F$ is zero due to our use of \cite{KalRI}, but as we already mentioned this assumption is likely unnecessary.

\begin{dfn} \label{def:tp} A toral supercuspidal parameter of generic depth $r>0$ is a discrete Langlands parameter $\varphi : W_F \rw {^LG}$ satisfying the following conditions.
\begin{enumerate}
	\item $\tx{Cent}(\varphi(I^r),\hat G)$ is a maximal torus and contains $\varphi(P_F)$;
	\item $\varphi(I^{r+})$ is trivial.
\end{enumerate}
\end{dfn}
Since $\tx{Cent}(\varphi(I),\hat G) \subset \tx{Cent}(\varphi(I^r),\hat G)$, the toral supercuspidal parameters are a special case of the strongly regular supercuspidal parameters of Definition \ref{def:srsp} and hence their $L$-packets have already been constructed in Subsection \ref{sub:packconst}. However, since the construction in this special case is considerably simpler, we shall examine it in detail, with the hope that it will be more useful to the readers who are only interested in this special case, and will also serve as an introduction to the more general construction.

The first step is to give the corresponding subcategory of the category of regular supercuspidal $L$-packet data. We will call it the category of toral $L$-packet data of generic depth $r$. A regular supercuspidal $L$-packet datum $(S,\hat j,\chi,\theta)$ will belong to this subcategory precisely when $\theta$ is a $G$-generic character of depth $r$.

\begin{pro} \label{pro:torpardata} The construction of Proposition \ref{pro:pardata} restricts to a bijection between the $\hat G$-conjugacy classes of toral supercuspidal parameters of generic depth $r$ and the isomorphism classes of toral $L$-packet data of generic depth $r$.
\end{pro}
\begin{proof}
We fix a $\Gamma$-stable pinning $(\hat T,\hat B,\{X_{\hat\alpha}\})$ of $\hat G$. Let $\varphi : {W_F} \rw {^LG}$ be a toral supercuspidal Langlands parameter of generic depth $r$. We conjugate $\varphi$ so that $\tx{Cent}(\varphi(I^r),\hat G)=\hat T$. The composition
\[ W_F \stackrel{\varphi}{\lrw} N(\hat T,\hat G) \rtimes W_F \rw \Omega(\hat T,\hat G) \rtimes W_F \rw \tx{Aut}_\tx{alg}(\hat T) \]
factors through a finite quotient of $W_F$ and endows $\hat T$ with a new $\Gamma$-module structure, which we will record by using the name $\hat S$. The assumption that $\varphi(P_F) \subset \hat T$ ensures that $P_F$ acts trivially on $\hat S$. The $\Gamma$-module $\hat S$ is the complex dual torus to a torus $S$ defined over $F$. Let $\hat j : \hat S \rw \hat G$ be the embedding coming from the equality $\hat S = \hat T$ of complex tori. The $\hat G$-conjugacy class of the embedding $\hat j$ is $\Gamma$-stable and we obtain a $\Gamma$-stable $G$-conjugacy class of embeddings $S \rw G$ as in Subsection \ref{sub:recg}, which we will call admissible. Choose minimal tame $\chi$-data for $R(S,G)$. Via the construction of \cite[\S2.6]{LS87} it gives a $\hat G$-conjugacy class of $L$-embeddings ${^LS} \rw {^LG}$ extending the $\hat G$-conjugacy class of $\hat j$. We choose one particular $L$-embedding $^Lj_\chi$ within this conjugacy class whose restriction to $\hat S$ is equal to $\hat j$. By definition, the projections of $^Lj_\chi(1 \rtimes w)$ and $\varphi(w)$ in $\Omega(\hat T,\hat G) \rtimes W_F$ are equal for any $w \in W_F$. This implies that the image of $\varphi$ is contained in the image of $^Lj_\chi$, which in turn leads to a factorization
\[ \varphi = {^Lj_\chi} \circ \varphi_{S,\chi}, \]
for some Langlands parameter $\varphi_{S,\chi} : W_F \rw {^LS}$. Let $\theta_\chi : S(F) \rw \C^\times$ be the corresponding character. Since any $L$-embedding that is $\hat G$-conjugate to $^Lj_\chi$ and also restricts to $\hat j$ must be conjugate to $^Lj_\chi$ by an element of $\hat T$, the $\hat S$-conjugacy class of $\varphi_{S,\chi}$, and hence the character $\theta_\chi$, are independent of the choice of $^Lj_\chi$. They depend only on the choice of $\chi$.

We claim that $\theta_\chi$ is generic of depth $r$. Let $E/F$ be the splitting field of $S$. By \cite[Lemma 2.2.1]{KalEpi} we need to check that for each root $\alpha \in R(S,G)$ the character $E_r^\times/E_{r+}^\times \to \C^\times$ given by $\theta_\chi\circ N_{E/F}\circ\alpha^\vee$ is non-trivial and that the stabilizer of $\theta_\chi\circ N_{E/F}|_{S(E)_r}$ in $\Omega(S,G)$ is trivial. For the first point, the parameter of $\theta_\chi\circ N_{E/F} \circ \alpha^\vee$ is the homomorphism $\hat\alpha\circ\varphi_S|_{W_E}$. By \cite[Theorem 7.10]{Yu09} the character restricts non-trivially to $E_r^\times$ if and only if its parameter restricts non-trivially to $I_E^r=I_F^r$. But the restriction of $\hat\alpha\circ\varphi_S$ to $I_F^r$ is equal to the restriction of $\hat\alpha\circ\varphi$, by the tameness of $\chi$-data, and the latter is non-trivial, due to $\tx{Cent}(\varphi(I^r),\hat G)=\hat T$. The second point follows from the same reasoning -- the stabilizer in $\Omega(S,G)$ of $\theta_\chi\circ N_{E/F}|_{S(E)_r}$ is equal to the stabilizer of $\varphi|_{I^r}$, which is trivial by assumption.

The object $(S,\hat j,\chi,\theta_\chi)$ we thus obtain belongs to the category of toral $L$-packet data of generic depth $r$. The proof that its isomorphism class depends only on the $\hat G$-conjugacy class of $\varphi$ is exactly the same as in Proposition \ref{pro:pardata}.

We now give the converse construction. Given a toral $L$-packet datum $(S,\hat j,\chi,\theta)$ of generic depth $r$ we use the $\chi$-data to extend $\hat j$ to an $L$-embedding $^Lj : {^LS} \rw {^LG}$ and let $\varphi_S : W_F \rw {^LS}$ be the parameter for $\theta$. Define $\varphi = {^Lj} \circ \varphi_S$. We claim that $\varphi$ satisfies the conditions of Definition \ref{def:tp}. Since $P_F$ acts trivially on $\hat S$, we can regard $\varphi_S|_{P_F}$ as a homomorphism $P_F \rw \hat S$. We use again \cite[Lemma 2.2.1]{KalEpi} and see that the genericity of $\theta$ implies that the restriction $\varphi_S|_{I^{r+}}$ is trivial; the centralizer of $\varphi_S|_{I^r}$ in $\Omega(\hat S,\hat G)$ is trivial; and for each $\hat\alpha \in R(\hat S,\hat G)$ the composition $\hat\alpha \circ \varphi_S|_{I^r}$ is non-trivial. The tameness of the $\chi$-data and of $G$ implies that we can replace $\varphi_S$ with $\varphi$ in these statements, from which we obtain that the homomorphism $\hat j \circ \varphi|_{P_F}$ is trivial on $I^{r+}$ and $\tx{Cent}(\hat j \circ \varphi(I^r),\hat G)=\hat j(\hat S)$.

Thus $\varphi$ is a toral supercuspidal parameter of generic depth $r$. The proof that its $\hat G$-conjugacy class depends only on the isomorphism class of $(S,\hat j,\chi,\theta)$ is again the same as for Proposition \ref{pro:pardata}.
\end{proof}

We define the category of toral supercuspidal data of generic depth $r$ as a subcategory of the category of regular supercuspidal data in the same way: A regular supercuspidal datum $(S,\hat j,\chi,\theta,(G_1,\xi,z),j)$ belongs to the subcategory precisely when $\theta$ is $G$-generic of depth $r$. To any such datum the representation associated in Subsection \ref{sub:packconst} is a toral representation of generic depth $r$. Indeed, it is by construction the regular supercuspidal representation $\pi_{(jS,j\theta')}$ of $G'(F)$, where we recall that $j\theta' := \theta\circ j^{-1} \cdot e_{f,r} \cdot \epsilon^r$. Since both $e_{f,r}$ and $\epsilon^r$ are of depth zero, $j\theta'$ is still $G$-generic of depth $r$. In the Howe factorization algorithm of Subsection \ref{sub:howe} this is the second ``trivial'' case, i.e. in the notation of that subsection we have $d=1$, $r_1=r_0>r_{-1}=0$, $S=G^0 \subset G^1 = G$, so we obtain the twisted Levi sequence $S \subset G$ and the Howe factorization $(1,\theta,1)$. The resulting Yu-datum then reduces to an Adler datum.

We conclude that the compound $L$-packet $\Pi_\varphi$ consists of toral supercuspidal representations of depth $r$ and moreover every such representation is contained in one of these $L$-packets.

The internal parameterization of $\Pi_\varphi$ is as described in Subsection \ref{sub:packpar}, but with the added precision that we are now in the position to prove the existence and uniqueness of a generic constituent. This will be done in the next subsection.

\subsection{Characters and genericity} \label{sub:chargen}
We keep the assumptions on $F$ from the previous subsection.

In this and the following subsections we will use the character formula for toral supercuspidal representations of Subsection \ref{sub:toral}. For this, fix a character $\Lambda : F \to \C^\times$ of depth zero. We will use the following short-hand notation: $\epsilon(T_G-T_J)=\epsilon_L(X^*(T_G)_\C-X^*(T_J)_\C,\Lambda)$, $\gamma^j=j^{-1}(\gamma)$, and $^jX^*=dj(X^*)$.

\begin{lem} \label{lem:char}
Let $(S,\hat j,\chi,\theta,(G',\xi,z),j)$ be a toral supercuspidal datum of generic depth $r$ and let $\pi$ be the corresponding representation of $G'(F)$. The character of $\pi$ at a regular semi-simple element $\gamma' = \gamma_{<r}' \cdot \gamma_{\geq r}' \in G'(F)$ is given by
\[
\frac{e(G')}{e(J')}\frac{\epsilon_L(T_G-T_J)}{|D_{G'}(\gamma')|^\frac{1}{2}}\sum_{\substack{g \in J'(F) \lmod G'(F) / jS(F)\\ \gamma_{<r}'^g \in jS(F)}}\Delta_{II}^\tx{abs}[a,\chi'](\gamma_{<r}'^{gj}) \theta(\gamma_{<r}'^{gj}) \hat\iota_{\mf{j'},^{jg}X^*}(\log(\gamma_{\geq r}')),
\]
where $J'=\tx{Cent}(\gamma_{<r}',G')^\circ$, and $T_G$ and $T_J$ are the minimal Levi subgroups in the quasi-split inner forms of $G'$ and $J'$.
\end{lem}
\begin{proof} This follows directly from Corollary \ref{cor:newchartoral}.
\end{proof}

Let $(T,B,\{X_\alpha\})$ be an $F$-pinning of $G$. Together with the character $\Lambda$, it determines a Whittaker datum $\mf{w}$ for $G$.

\begin{lem} \label{lem:generic} Let $(S,\hat j,\chi,\theta)$ be a toral $L$-packet datum of generic depth $r$. There exists a unique (up to $G(F)$-conjugacy) admissible rational embedding $j_\mf{w} : S \rw G$ such that the representation corresponding to $(S,\hat j,\chi,\theta,(G,\tx{id},1),j_\mf{w})$ is $\mf{w}$-generic. Moreover, the splitting invariant \cite[\S2.3]{LS87} for the torus $j_\mf{w}S \subset G$ relative to $(T,B,\{X_\alpha\})$ and the mod-$a$-data constructed in \eqref{eq:adata} is trivial.
\end{lem}
\begin{proof}
The statement about genericity is a result of DeBacker and Reeder, \cite[Proposition 4.10]{DR08}. In that reference the statement is formulated only for the case that $S$ is unramified, but the same argument goes through in general. We limit ourselves to a sketch:

Let $j : S \rw G$ be a rational admissible embedding, let $\pi_j$ be the representation of $G(F)$ corresponding to the supercuspidal toral datum $(S,\hat j,\chi,\theta,(G,\tx{id},1),j)$, and let $\Theta_j$ be its character. According to the Harish-Chandra local character expansion, for strongly regular semi-simple elements $\gamma \in G(F)$ that are sufficiently close to the identity we have
\[ \Theta_j(\gamma) = \sum_{\mc{O}} c(\mc{O})\hat\mu_{\mc{O}}(\log(\gamma)), \]
where the sum runs over the set of nilpotent orbits of the adjoint action of $G(F)$ on $\mf{g}(F)$, $c(\mc{O})$ are complex constants, and $\hat\mu_{\mc{O}}$ are the Fourier transforms of the invariant integrals along these orbits.

Fix a $G(F)$-invariant non-degenerate symmetric bilinear form $\beta$ on $\mf{g}(F)$. Define an element $F_\mf{w} \in \mf{u}^-(F)$, where $U^-$ is the unipotent radical of the Borel subgroup of $G$ that is $T$-opposite to $B$, and $\mf{u}^-$ is its Lie-algebra, by $F_\mf{w} = \sum_\alpha \beta(X_\alpha,X_{-\alpha})^{-1} \cdot X_{-\alpha}$, where the sum runs over the $B$-simple roots of $T$. This element has the property that the character of $\mf{u}(F)$ given by $X \mapsto \Lambda(\beta(F_\mf{w},X))$, when composed with $\exp$, is equal to the generic character of $U(F)$ determined by the splitting $(T,B,\{X_\alpha\})$ and the character $\Lambda$. The main result of \cite{MW87} then states that the representation $\pi_j$ is $\mf{w}$-generic if and only if the constant $c(\tx{Ad}(G(F))F_\mf{w})$ is non-zero.

According to Lemma \ref{lem:char}, if $\gamma=\gamma_{\geq r} \in G(F)$ is strongly regular semi-simple element, then
\[ \Theta_j(\gamma)=|D_{G}(\gamma)|^{-\frac{1}{2}}\hat\iota_{{\mf{g}},{^jX^*}}(\log(\gamma)) = |D_{G}(^jX^*)|^\frac{1}{2}\hat\mu_{\mf{g},{^jX^*}}(\log(\gamma)). \]
Equating the last two displayed formulas and using a result of Shelstad \cite{She89}, reinterpreted as \cite[Proposition 4.2]{DR08}, we see that $c(\tx{Ad}(G(F))F_\mf{w})$ is non-zero precisely when the $G(F)$-orbit of $jX^*$ meets the Kostant section $F_\mf{w} + \tx{Cent}(E_\mf{w},\mf{g})$, where $E_\mf{w}=\sum_\alpha \beta(X_\alpha,X_{-\alpha})X_\alpha$, and where we are interpreting $jX^* \in \mf{g}^*(F)$ as an element of $\mf{g}(F)$ via $\beta$. From this, the uniqueness of $j_\mf{w}$ follows.

We turn to the triviality of the splitting invariant. Let $X'_\alpha=\beta(X_\alpha,X_{-\alpha})X_\alpha$. The main result of \cite{Kot99} asserts that the splitting invariant of $j_\mf{w}S$ vanishes, if it is computed with respect to the pinning $(T,B,\{X_\alpha'\})$ and the $a$-data $a_\gamma=d\gamma(j_\mf{w}X^*)$, for $\gamma \in R(S,G)$. Now $d\gamma(j_\mf{w}X^*)=\beta(H_\gamma,j_\mf{w}X^*) \cdot \beta(X_\gamma,X_{-\gamma})^{-1}$. The function $\alpha \mapsto \beta(X_\alpha,X_{-\alpha})$ extends to a $\Omega(T,G) \rtimes \Gamma$-equivariant function and then \cite[Lemma 5.1]{KalGen} implies that the splitting invariant of $j_\mf{w}S$ vanishes, if it is computed with respect to the splitting $(T,B,\{X_\alpha\})$ and the $a$-data $\beta(H_\gamma,j_\mf{w}X^*)$. But the $a$-data $\beta(H_\gamma,j_\mf{w}X^*)=\<H_\gamma,X^*\>$ projects to the mod-$a$-data of \eqref{eq:adata}.
\end{proof}

\subsection{Stability and transfer} \label{sub:toraltrans}

In this subsection we assume that $F$ has characteristic zero and sufficiently large residual characteristic, so that the logarithm map is defined on $G(F)_{0+}$.

We continue with a toral supercuspidal parameter $\varphi$ of generic depth $r$ with associated $L$-packet $\Pi_\varphi$. For any rigid inner twist $(G',\xi,x)$ and any $s \in S_\varphi^+$ define the function
\[ \Theta_{\varphi,\mf{w},x}^s = e(G')\sum_{(G',\xi,x,\pi) \in \Pi_\varphi} \<(G',\xi,x,\pi),s\>_\mf{w} \cdot \Theta_\pi \]
of $G'(F)$. According to \eqref{eq:wc}, when $s=1$ this function does not depend on the choice of $\mf{w}$ and we can denote it by $S\Theta_{\varphi,x} = \Theta_{\varphi,\mf{w},x}^1$.

\begin{lem} \label{lem:schar}
The value of $\Theta_{\varphi,\mf{w},x}^s$ at a regular semi-simple element $\gamma'=\gamma_{<r}' \cdot \gamma_{\geq r}' \in G'(F)$ is given by
\[ e(J')\frac{\epsilon(T_G-T_J)}{|D_{G'}(\gamma')|}\sum_j\Delta_{II}^\tx{abs}[a,\chi'](\gamma_{<r}'^j) \theta(\gamma_{<r}'^j) \sum_{k}\<\tx{inv}(j_\mf{w},k),s\>\hat\iota_{\mf{j'},{^kX^*}}(\log(\gamma_{\geq r}')), \]
where $J'$, $T_G$, and $T_J$ are as in Lemma \ref{lem:char}, $j$ runs over the set of $J'$-stable classes of embeddings $S \rw J'$, whose composition with $J' \subset G'$ is admissible, $k$ runs over the set of $J'(F)$-rational classes inside the stable class $j$, and $j_\mf{w} : S \rw G$ is the admissible embedding given by Lemma \ref{lem:generic}.
\end{lem}

\begin{proof}
Let $(S,\hat j,\chi,\theta)$ be a toral $L$-packet datum of generic depth $r$ in the isomorphism class associated to $\varphi$ by Proposition \ref{pro:torpardata}. According to Lemma \ref{lem:char}, for any admissible embedding $j : S \rw G'$ the value at $\gamma'$ of the character $\Theta_j$ of the corresponding representation at is given by
\[ \frac{e(G')}{e(J)}\frac{\epsilon(T_G-T_J)}{|D_{G'}(\gamma')|}\sum_{k}\Delta_{II}^\tx{abs}[a,\chi'](\gamma_{<r}'^k) \theta(\gamma_{<r}'^k) \hat\iota_{\mf{j'},{^kX^*}}(\log(\gamma_{\geq r})), \]
where $k$ runs over the set of $J'(F)$-conjugacy classes of embeddings $S \rw J'$ that are $G'(F)$-conjugate to $j$. Then we have $\Theta_{\varphi,\mf{w},x}^s = \sum_j \<\tx{inv}(j_\mf{w},j),s\> \Theta_j$, where the sum runs over the $G'(F)$-conjugacy classes of admissible embeddings $j : S \rw G'$ defined over $F$. Putting both sums together and re-indexing we see that $\Theta_{\varphi,\mf{w},x}^s(\gamma')$ is equal to
\[ e(J)\frac{\epsilon(T_G-T_J)}{|D_{G'}(\gamma')|}
\sum_j\sum_{k}\<\tx{inv}(j_\mf{w},k),s\>\Delta_{II}^\tx{abs}[a,\chi'](\gamma_{<r}'^k) \theta(\gamma_{<r}'^k) \hat\iota_{\mf{j'},{^kX^*}}(\log(\gamma_{\geq r}')), \]
where now $j$ runs over the set of $J$-stable conjugacy classes of $G'$-admissible embeddings $S \rw J'$ defined over $F$ and $k$ runs over the set of $J'(F)$-conjugacy classes of embeddings $S \rw J'$ in the $J'$-stable class of $j$. Since $\gamma_{<r}'$ is central in $J'$, this expression is equal to the one in the statement of the lemma.
\end{proof}

Before we begin the study of stability and endoscopic transfer, we make the following convention. Let $T$ be a maximal torus of $G$ and $\gamma \in T(F)$ a strongly regular semi-simple element with a normal $r$-approximation $\gamma = \gamma_{<r} \cdot \gamma_{\geq r}$. If $T'$ is a maximal torus in some inner form of $G$ or in an endoscopic group of $G$ and $f : T \rw T'$ is an admissible isomorphism, then $f(\gamma)=f(\gamma_{<r}) \cdot f(\gamma_{\geq r})$ is a normal $r$-approximation. This is proved in \cite[Lemma 5.2]{DS} for the case of stable conjugacy, but the argument works without change for the case of transfer to an endoscopic group. This fixes the approximations of all stable conjugates and transfers of $\gamma$. It is well-defined, because the only admissible automorphism of $T$ carrying $\gamma$ to itself is the identity.

\begin{thm} The function $S\Theta_{\varphi,\cdot}$ is stable across inner forms. That is, for any two rigid inner twist $(G'_1,\xi_1,x_1)$ and $(G'_2,\xi_2,x_2)$ and stably conjugate strongly regular semi-simple elements $\gamma_1' \in G'_1(F)$ and $\gamma_2' \in G'_2(F)$ we have
\[ S\Theta_{\varphi,x_1}(\gamma_1') = S\Theta_{\varphi,x_2}(\gamma_2'). \]
\end{thm}
\begin{proof}
It is enough to consider the case where one of the two rigid inner twists is trivial. Thus let $(G',\xi,x)$ be a rigid inner twist of $G$, $\gamma=\gamma_{<r} \cdot \gamma_{\geq r}$ a strongly regular semi-simple element of $G(F)$ and $\gamma'=\gamma_{<r}' \cdot \gamma_{\geq r}' \in G'(F)$ stably conjugate to $\gamma$. Let $J=\tx{Cent}(\gamma_{<r},G)$ and $J' = \tx{Cent}(\gamma_{<r}',G')$. The admissible isomorphism $f_{\gamma,\gamma'}$ provides an inner twist $J \rw J'$ which carries $\gamma_{>r}$ to $\gamma_{>r}'$. Moreover, for every $J$-stable class of $G$-admissible rational embeddings $j : S \rw J$, $j' = f_{\gamma,\gamma'} \circ j$ is an $J'$-stable class of $G'$-admissible rational embeddings $S \rw J'$, and $j \leftrightarrow j'$ is a 1-1 correspondence, under which we have $j'\circ j^{-1}(\gamma_{<r})=\gamma_{<r}'$. For each pair $j \leftrightarrow j'$ of corresponding stable classes of embeddings, the result of Waldspurger \cite[Theoreme 1.5]{Wal06ECC} and Kottwitz's computation of $\epsilon$-factors \cite[Theorem 3.5.1]{KalEpi} imply
\[ e(J)\sum_{k}\hat\iota_{\mf{j},{^kX^*}}(\log(\gamma_{\geq r})) = e(J')\sum_{k'}\hat\iota_{\mf{j'},{^{k'}X^*}}(\log(\gamma_{\geq r}')).\]
\end{proof}

Let now $\mf{e}=(H,s,{^L\eta})$ be a tame extended endoscopic triple for $G$. We will prove the endoscopic character identities for toral $L$-packets. The argument for the slightly more general case where the $L$-group of $H$ does not embed into the $L$-group of $G$ is the same, but the notation is more cumbersome, so we leave it to the reader.

For any rigid inner twist $(G',\xi,x)$ we have the normalized transfer factor $\Delta=\Delta_{\mf{w},x}$ defined in \cite[(5.10)]{KalRIBG}. We will drop the prime notation here, but we do alert the reader the this transfer factor is a normalization of the factor $\Delta'$ of \cite[\S5.1]{KS12}, which is slightly different from the factor $\Delta$ of \cite{LS87}. As in \cite[\S5.4]{KalEpi}, we denote by $\mr{\Delta}$ the transfer factor $\Delta$ with its part $\Delta_{IV}$ removed.

\begin{lem} \label{lem:ugly} Let $\gamma^H \in H(F)$ and $\gamma' \in G'(F)$ be strongly regular semi-simple elements. For any sufficiently large natural number $k$ we have
\[ \mathring\Delta(\gamma_{<r}^H \cdot [\gamma_{\geq r}^H]^{p^{2k}},\gamma_{<r}' \cdot [\gamma'_{\geq r}]^{p^{2k}}) = \mathring\Delta(\gamma^H,\gamma'). \]
\end{lem}
\begin{proof}
Since we have arranged that an admissible isomorphism carrying $\gamma^H$ to $\gamma'$ carries $\gamma^H_{<r}$ to $\gamma_{<r}'$ and $\gamma^H_{\geq r}$ to $\gamma_{\geq r}'$, the notion of relatedness is unchanged.

We must compare the terms $\Delta_I$, $\Delta_{II}$, $\Delta_{III_1}$ and $\Delta_{III_2}$ of both sides. For each root $\alpha$ of $T'=\tx{Cent}(\gamma',G')$ we have $\tx{ord}(\alpha(\gamma_{<r}')-1)<r$ and $\tx{ord}(\alpha(\gamma_{\geq r}')-1) \geq r$.
It follows that $\gamma_{<r}' \cdot [\gamma_{\geq r}']^{p^{2k}}$ is still a regular element of $T'$. Thus $\Delta_I$ and $\Delta_{III_1}$ don't change. To treat the other two, we choose tamely ramified $\chi$-data. Then $\Delta_{III_2}$ is a tamely ramified character of $T'(F)$ and thus any power of $\gamma_{\geq r}'$ belongs to its kernel. For $\Delta_{II}$, we apply Lemma \ref{lem:d2d} and see that the contributions of those roots $\alpha$ with $\alpha(\gamma_{<r}') \neq 1$ to both sides are the same. If $\alpha$ is a root with $\alpha(\gamma_{<r}')=1$, let $y=\alpha(\gamma_{\geq r}') \in [F_\alpha^\times]_{r}$. Then the contribution of $\alpha$ to the left-hand side is $\chi_\alpha(a_\alpha^{-1}(y^{p^{2k}}-1))$. According to \cite[Lemma 3.1]{Hal93} and the tameness of $\chi_\alpha$, this is equal to $\chi_\alpha(a_\alpha^{-1}p^{2k}(y-1))$, which is equal to the product of the contribution of $\alpha$ to the right-hand side with $\kappa_\alpha(p)^{2k}=1$.
\end{proof}

\begin{thm} Let $\gamma' \in G'(F)$ be a strongly regular semi-simple element with a normal $r$-approximation $\gamma=\gamma_{<r}' \cdot \gamma_{\geq r}'$. Assume that $\varphi={^L\eta}\circ\varphi^H$ for $\varphi^H : W_F \rw {^LH}$. Then
\[ \Theta_{\varphi,\mf{w},x}^s(\gamma') = \sum_{\gamma^H \in H(F)/\tx{st}} \mr{\Delta}_{\mf{w},x}(\gamma^H,\gamma')\frac{D^H(\gamma^H)}{D^{G'}(\gamma')}S\Theta_{\varphi^H,1}(\gamma^H). \]
\end{thm}
\begin{proof}
Let $T'=\tx{Cent}(\gamma',G')$. We follow the beginning of the proof of \cite[Theorem 5.4.1]{KalEpi}. In doing so, we will make active use of the descent lemmas established in \cite[\S5.3]{KalEpi}. Rather than recalling their fairly technical statements, we refer the reader to the cited exposition, which is self-contained.

Let $Y$ be a set of representatives for the stable classes of preimages in $H(F)$ of $\gamma_{<r}$ chosen so that the connected centralizer $H_y$ is quasi-split for each $y \in Y$. According to \cite[Lemma 5.3.2]{KalEpi} we can write the right hand side as
\[ \sum_{y \in Y} |\pi_0(H^y)(F)|^{-1} \sum_{z \in H_y(F)_1/\tx{st}} \mr{\Delta}_{\mf{w},x}(yz,\gamma')\frac{D^H(yz)}{D^{G'}(\gamma')}S\Theta_{\varphi^H,1}(yz), \]
where $H^y$ denotes the (possibly disconnected) centralizer of $y$ in $H$, and $H_y(F)_1$ is the subset of $H_y(F)$ of those elements $z$ for which $yz$ is strongly regular semi-simple and has normal $r$-approximation with head $y$ and tail $z$. Applying Lemma \ref{lem:ugly}, we can rewrite this as
\begin{equation} \label{eq:ec1} \sum_{y \in Y} |\pi_0(H^y)(F)|^{-1} \ssum{z \in H_y(F)_1/\tx{st}}{\mr{\Delta}_{\mf{w},x}(yz^{p^{2k}},\gamma_{<r}'(\gamma_{\geq r}')^{p^{2k}})\frac{D^H(yz)}{D^{G'}(\gamma')}S\Theta_{\varphi^H,1}(yz)}. \end{equation}
As before let $J'=\tx{Cent}(\gamma_{<r}',G')$. Recall the set $\Xi(H_y,J')$ from \cite[\S5.3]{KalEpi}. It encodes the different inequivalent ways in which $H_y$ can be realized as an endoscopic group of $J'$ via descent. There exists a unique $\xi \in \Xi(H_y,J')$ for which the element $yz^{p^{2k}} \in H_y(F)$ is related to the element $\gamma_{<r}'(\gamma_{\geq r}')^{p^{2k}}$ (for any value of $k$). We apologize here for the double use of $\xi$, but the inner twist $\xi : G \rw G'$ will not be used in this proof. Taking $k$ large enough, we can apply the Langlands-Shelstad descend theorem \cite[Theorem 1.6]{LS90} and conclude that there is a unique normalization $\mathring\Delta_{\mf{w},x}^{\tx{desc},\xi}$ of the transfer factor for the group $J'$ and its endoscopic group $H_y$, realized by descent according to $\xi$, with the property
\[ \mathring\Delta_{\mf{w},x}^{\tx{desc},\xi}(yz^{p^{2k}},\gamma_{<r}'(\gamma_{\geq r}')^{p^{2k}}) = \mr{\Delta}_{\mf{w},x}(yz^{p^{2k}},\gamma_{<r}'(\gamma_{\geq r}')^{p^{2k}}). \]
For any other $\xi$ we take $\mathring\Delta_{\mf{w},x}^{\tx{desc},\xi}$ to be an arbitrary normalization of the transfer factor for $J'$ and $H_y$ and have
\[ \mathring\Delta_{\mf{w},x}^{\tx{desc},\xi}(yz^{p^{2k}},\gamma_{<r}'(\gamma_{\geq r}')^{p^{2k}}) = 0. \]
This discussion allows us to rewrite \eqref{eq:ec1} as
\begin{equation} \label{eq:ec2} \sum_{y \in Y} |\pi_0(H^y)(F)|^{-1} \sum_\xi \sum_{z \in H_y(F)_1/\tx{st}}\!\!\!\! \mr{\Delta}_{\mf{w},x}^{\tx{desc},\xi}(yz^{p^{2k}},\gamma_{<r}'(\gamma_{\geq r}')^{p^{2k}})\frac{D^H(yz)}{D^{G'}(\gamma')}S\Theta_{\varphi^H,1}(yz), \end{equation}
where $\xi$ runs over the set $\Xi(H_y,J')$. The sum over $z$ can be extended to $H_y(F)_\tx{sr}$, since for elements outside of $H_y(F)_1$ the transfer factor will be zero. Furthermore, since $y$ is central in $H_y$ and $\gamma_{<r}'$ is central in $J'$, we may apply \cite[Lemma 3.5.A]{LS90} and obtain
\[ \mr{\Delta}_{\mf{w},x}^{\tx{desc},\xi}(yz^{p^{2k}},\gamma_{<r}'(\gamma_{\geq r}')^{p^{2k}}) = \lambda_{J',\xi}(\gamma_{<r}')\mr{\Delta}_{\mf{w},x}^{\tx{desc},\xi}(z^{p^{2k}},(\gamma_{\geq r}')^{p^{2k}}),\]
where $\lambda_{J',\xi}$ is the character of $Z(J')(F)$ denoted by $\lambda_G$ in \cite{LS90}. Increasing $k$ if necessary we have
\[ \mr{\Delta}_{\mf{w},x}^{\tx{desc},\xi}(z^{p^{2k}},(\gamma_{\geq r}')^{p^{2k}}) = \mr{\Delta}_{\mf{w},x}^{\mf{j'}}(\log(z^{p^{2k}}),\log((\gamma_{\geq r}')^{p^{2k}})), \]
where on the right we have the transfer factor for the Lie-algebra of $J'$ that is compatibly normalized with the one on the left. Since the Lie-algebra transfer factor is invariant under multiplication by $F^{\times,2}$, we can remove the $p^{2k}$-power. Plugging this into \eqref{eq:ec2}, replacing $S\Theta_{\varphi^H,1}$ with the formula from Lemma \ref{lem:schar}, and rearranging terms, we arrive at
\begin{eqnarray} \label{eq:ec3}
\frac{\lambda_{J',\xi}(\gamma_{<r}')}{D^{G'}(\gamma')}\sum_{y \in Y} |\pi_0(H^y)(F)|^{-1} \sum_\xi \sum_{j_H}\Delta_{II}^{\tx{abs},H}[a^H,\chi^H](y^{j_H}) \theta^H(y^{j_H}) \\
\epsilon(T_H-T_{H_y})\sum_{Z \in \mf{h}_y(F)_\tx{rs}/\tx{st}}\!\!\!\! \mr{\Delta}_{\mf{w},x}^{\mf{j'}}(Z,\log((\gamma_{\geq r}')))
  \sum_{k_H}\hat\iota_{{\mf{h}_y},{^{k_H}X^*}}(Z). \nonumber
\end{eqnarray}
Here $j_H$ and $k_H$ run as in Lemma \ref{lem:schar} but with target $H$ instead of $G'$, and we have fixed a toral $L$-packet datum $(S^H,\hat j^H,\chi^H,\theta^H)$ for $\varphi^H$.

Fix a triple $(y,\xi,j_H)$ contributing to the upper line. Via \cite[Lemma 5.3.3]{KalEpi} this triple corresponds to a $J'$-stable class of rational $G'$-admissible embeddings $j : S \rw J'$.

Fix a $J'(F)$-invariant non-degenerate symmetric bilinear form $B$ on the Lie-algebra $\mf{j}'(F)$ and use it to identify this Lie-algebra with its dual. The results of Waldspurger \cite{Wal97}, \cite{Wal06ECC}, and Ng\^o \cite{Ngo10}, imply that then
\[ \sum_{Z \in \mf{h}_y(F)_\tx{rs}/\tx{st}}\!\!\!\! \mr{\Delta}_{\mf{w},x}^{\mf{j}'}(Z,\log((\gamma_{\geq r}')))
  \sum_{k_H}\hat\iota_{{\mf{h}_y},{^{k_H}X^*}}(Z) \]
is equal to
\[ \gamma_\Lambda(\mf{j}',B)\gamma_\Lambda(\mf{h}_y,B)^{-1}\sum_k
\mr{\Delta}_{\mf{w},x}^{\mf{j}'}(j_H X^*,kX^*)\hat\iota_{\mf{j}',^kX^*}(\log(\gamma_{\geq r}')),
\]
where now $k$ runs over the set of $J'(F)$-conjugacy classes in the $J'$-stable class of $j$. According to \cite[Theorem 3.5.1, Lemma 3.4.1]{KalEpi} we have
\[ \gamma_\Lambda(\mf{j}',B)\gamma_\Lambda(\mf{h}_y,B)^{-1} = e(J')\epsilon(T_{H_y}-T_J,\Lambda)\prod_{\alpha \in R(jS,J'-H_y)_\tx{sym}/\Gamma} \kappa_\alpha(B_\alpha), \]
where $\alpha$ runs over the $\Gamma$-orbits of symmetric roots of $jS$ in $J'$ that are outside of $H_y$. Appealing to the bijection $(y,\xi,j_H) \leftrightarrow j$ of \cite[Lemma 5.3.3]{KalEpi} we can thus rewrite \eqref{eq:ec3} as
\begin{eqnarray} \label{eq:ec4}
\frac{\lambda_{J',\xi}(\gamma_{<r}')}{D^{G'}(\gamma')}\sum_j\Delta_{II}^{\tx{abs},H}[a^H,\chi^H](y^{j_H}) \theta^H(y^{j_H})e(J')\epsilon(T_H-T_J)\\
\sprod{\alpha \in R(jS,J'-H_y)_\tx{sym}/\Gamma}{\kappa_\alpha(B_\alpha)}\sum_k
\mr{\Delta}_{\mf{w},x}^{\mf{j}'}(j_H X^*,kX^*)\hat\iota_{\mf{j}',{^kX^*}}(\log(\gamma_{\geq r}')). \nonumber
\end{eqnarray}
Selecting a small $z \in F^\times$ we undo the descent of the transfer factor by
\begin{eqnarray*}
&&\lambda_{J',\xi}(\gamma_{<r}') \mr{\Delta}_{\mf{w},x}^{\mf{j}'}(j_H X^*,kX^*)\\
&=&\mr{\Delta}_{\mf{w},x}(y\exp(z^2j_HX^*),\gamma_{<r}'\exp(z^2kX^*))\\
&=&\mr{\Delta}_{\mf{w},1}(y\exp(z^2j_HX^*),\gamma_{<r}\exp(z^2j_\mf{w}X^*))\<\tx{inv}(j_\mf{w},k),s\>,
\end{eqnarray*}
where $\gamma_{<r}=j_\mf{w}k^{-1}(\gamma_{<r}')$. Then \eqref{eq:ec4} becomes
\begin{eqnarray} \label{eq:ec5}
&&\frac{e(J')}{D^{G'}(\gamma')}\sum_j\Delta_{II}^{\tx{abs},H}[a^H,\chi^H](y^{j_H}) \theta^H(y^{j_H})\epsilon(T_H-T_J)\sprod{\alpha \in R(jS,J'-H_y)_\tx{sym}/\Gamma}{\kappa_\alpha(B_\alpha)}\\
&&\mr{\Delta}_{\mf{w},1}(y\exp(z^2j_HX^*),\gamma_{<r}\exp(z^2j_\mf{w}X^*))\nonumber
\sum_k \<\tx{inv}(j_\mf{w},k),s\>\hat\iota_{\mf{j}',{^kX^*}}(\log(\gamma_{\geq r}')).\nonumber
\end{eqnarray}
Here $j_H$ is still determined by $j$ via the bijection $(y,\xi,j_H) \leftrightarrow j$. Comparing this with Lemma \ref{lem:schar} we see that the theorem will be proved once we prove that $\mr{\Delta}_{\mf{w},1}(y\exp(z^2j_HX^*),\gamma_{<r}\exp(z^2j_\mf{w}X^*))$ is equal to
\[ \epsilon(T_G-T_H)\frac{\Delta_{II}^\tx{abs}(\gamma_{<r}'^j)\theta(\gamma_{<r}'^j)}{\Delta_{II}^{\tx{abs},H}(y^{j_H})\theta^H(y^{j_H})}\prod_{\alpha \in R(jS,J'-H_y)_\tx{sym}/\Gamma}{\kappa_\alpha(B_\alpha)},\]
where $(S,\hat j,\chi,\theta)$ is some toral $L$-packet datum corresponding to $\varphi$.

For this we examine the structure of $\mr{\Delta}_{\mf{w},1}$. Its first argument belongs to the maximal torus $j_HS^H \subset H$ and its second argument belongs to the maximal torus $j_\mf{w}S \subset G$. Modifying $(S^H,\hat j^H,\chi^H,\theta^H)$ within its isomorphism class, we may assume that the isomorphism $\hat j^{-1}\circ \hat \eta \circ \hat j^H : \hat S^H \rw \hat S$ is $\Gamma$-equivariant. Using the dual of this isomorphism we identify $S^H$ and $S$ and also obtain an admissible isomorphism $j_HS \rw j_\mf{w}S$ that we use in the discussion of the transfer factor. We select as $\chi$-data for $j_\mf{w}S$ the transport via $j_\mf{w}$ of the $\chi$-data from the toral $L$-packet datum $(S,\hat j,\chi,\theta)$, and as $a$-data we select the one used in the character formula of Lemma \ref{lem:char}, namely the one from \eqref{eq:adata}. The admissible isomorphism $j_\mf{w}\circ j_H^{-1}$ transports this to $a$-data and $\chi$-data for $S^H$. By modifying the toral $L$-packet datum $(S^H,\hat j^H,\chi^H,\theta^H)$ within its isomorphism class we may assume that the resulting $\chi$-data is the transport via $j_H$ of the $\chi$-data $\chi^H$.

Recall that $\mr{\Delta}=\epsilon(T_0-T_0^H)\Delta_I\Delta_{II}\Delta_{III_2}$, the term $\Delta_{III_1}$ being trivial by our choice of admissible isomorphism. According to Lemma \ref{lem:generic} we have $\Delta_I=1$. By definition, $\Delta_{III_2}$ is the value at $\gamma_{<r}$ of the character of $j_\mf{w}S$ given by $\theta\circ j_\mf{w}^{-1} / \theta^H\circ j_\mf{w}^{-1}$. Taking $z$ small enough and using $j^{-1}(\gamma_{<r}')=j_\mf{w}^{-1}(\gamma_{<r})=j_H^{-1}(y)$ we get $\Delta_{III_2}=\theta(j^{-1}(\gamma_{<r}))/\theta^H(j_H^{-1}(y))$.

To handle the term $\Delta_{II}$ we apply Lemma \ref{lem:d2d}, which reduces the proof to the claim that for small $z$ we have
\[ \prod_{\alpha \in R(jS,J'-H_y)_\tx{sym}/\Gamma}{\kappa_\alpha(B_\alpha)} = \frac{\Delta_{II}^{\tx{abs},J}(\exp(z^2j_\mf{w}X^*))}{\Delta_{II}^{\tx{abs},H_y}(z^2j_HX^*)}. \]
Indeed, we have
\[ \lim_{z \rw 0}\frac{\alpha(\exp(z^2j_\mf{w}X^*))-1}{z^2} = d\alpha(X^*) = \<H_\alpha,X^*\>B_\alpha^{-1} \]
and recalling that $a_\alpha=\<H_\alpha,X^*\>$ we see
\[ \chi_\alpha\left(\frac{\alpha(\exp(z^2j_\mf{w}X^*))-1}{a_\alpha}\right) = \kappa_\alpha(B_\alpha). \]
\end{proof}

\bibliographystyle{amsalpha}
\bibliography{bibliography.bib}

\end{document}

%% file: macros.tex
\def\mf#1{\mathfrak{#1}}
\def\mc#1{\mathcal{#1}}
\def\mb#1{\mathbb{#1}}
\def\tx#1{\textrm{#1}}
\def\tb#1{\textbf{#1}}

\def\ts#1{\textsf{#1}}
\def\ms#1{\mathsf{#1}}

\def\R{\mathbb{R}}

\def\C{\mathbb{C}}
\def\Q{\mathbb{Q}}

\def\Z{\mathbb{Z}}

\def\lmod{\setminus}

\def\ol#1{\overline{#1}}
\def\ul#1{\underline{#1}}

\def\hat{\widehat}

\def\rw{\rightarrow}

\def\lrw{\longrightarrow}
\def\llw{\longleftarrow}

\def\sm{\smallsetminus}

\def\<{\langle}
\def\>{\rangle}

\def\mr#1{\mathring{#1}}

\newenvironment{mytitle}
{\begin{center}\large\sc}
{\end{center}}

%% file: spt.bbl
\newcommand{\etalchar}[1]{$^{#1}$}
\providecommand{\bysame}{\leavevmode\hbox to3em{\hrulefill}\thinspace}
\providecommand{\MR}{\relax\ifhmode\unskip\space\fi MR }
\providecommand{\MRhref}[2]{%
  \href{http://www.ams.org/mathscinet-getitem?mr=#1}{#2}
}
\providecommand{\href}[2]{#2}
\begin{thebibliography}{ABD{\etalchar{+}}64}

\bibitem[ABD{\etalchar{+}}64]{SGA3}
M.~Artin, J.~E. Bertin, M.~Demazure, P.~Gabriel, A.~Grothendieck, M.~Raynaud,
  and J.-P. Serre, \emph{Sch\'emas en groupes.}, S\'eminaire de G\'eom\'etrie
  Alg\'ebrique de l'Institut des Hautes \'Etudes Scientifiques, Institut des
  Hautes \'Etudes Scientifiques, Paris, 1963/1964.

\bibitem[Adl98]{Ad98}
Jeffrey~D. Adler, \emph{Refined anisotropic {$K$}-types and supercuspidal
  representations}, Pacific J. Math. \textbf{185} (1998), no.~1, 1--32.
  \MR{1653184 (2000f:22019)}

\bibitem[Adr10]{Adrian10}
Moshe Adrian, \emph{A new construction of the tame local {L}anglands
  correspondence for {GL}(n,{F}), n a prime}, ProQuest LLC, Ann Arbor, MI,
  2010, Thesis (Ph.D.)--University of Maryland, College Park. \MR{2941466}

\bibitem[Adr13]{Adrian13}
\bysame, \emph{A new realization of the {L}anglands correspondence for {${\rm
  PGL}(2,F)$}}, J. Number Theory \textbf{133} (2013), no.~2, 446--474.
  \MR{2994367}

\bibitem[AS08]{AS08}
Jeffrey~D. Adler and Loren Spice, \emph{Good product expansions for tame
  elements of {$p$}-adic groups}, Int. Math. Res. Pap. IMRP (2008), no.~1, Art.
  ID rpn 003, 95. \MR{2431235 (2009m:22020)}

\bibitem[AS09]{AS09}
\bysame, \emph{Supercuspidal characters of reductive {$p$}-adic groups}, Amer.
  J. Math. \textbf{131} (2009), no.~4, 1137--1210. \MR{2543925 (2011a:22018)}

\bibitem[BH05a]{BH05a}
Colin~J. Bushnell and Guy Henniart, \emph{The essentially tame local
  {L}anglands correspondence. {I}}, J. Amer. Math. Soc. \textbf{18} (2005),
  no.~3, 685--710. \MR{2138141 (2006a:22014)}

\bibitem[BH05b]{BH05b}
\bysame, \emph{The essentially tame local {L}anglands correspondence. {II}.
  {T}otally ramified representations}, Compos. Math. \textbf{141} (2005),
  no.~4, 979--1011. \MR{2148193 (2006c:22017)}

\bibitem[BLR90]{BLR90}
Siegfried Bosch, Werner L{\"u}tkebohmert, and Michel Raynaud, \emph{N\'eron
  models}, Ergebnisse der Mathematik und ihrer Grenzgebiete (3) [Results in
  Mathematics and Related Areas (3)], vol.~21, Springer-Verlag, Berlin, 1990.
  \MR{1045822 (91i:14034)}

\bibitem[Bor91]{Bor91}
Armand Borel, \emph{Linear algebraic groups}, second ed., Graduate Texts in
  Mathematics, vol. 126, Springer-Verlag, New York, 1991. \MR{1102012
  (92d:20001)}

\bibitem[Bor98]{Brv98}
Mikhail Borovoi, \emph{Abelian {G}alois cohomology of reductive groups}, Mem.
  Amer. Math. Soc. \textbf{132} (1998), no.~626, viii+50. \MR{1401491}

\bibitem[BT72]{BT1}
F.~Bruhat and J.~Tits, \emph{Groupes r\'eductifs sur un corps local}, Inst.
  Hautes \'Etudes Sci. Publ. Math. (1972), no.~41, 5--251. \MR{0327923 (48
  \#6265)}

\bibitem[BT84]{BT2}
\bysame, \emph{Groupes r\'eductifs sur un corps local. {II}. {S}ch\'emas en
  groupes. {E}xistence d'une donn\'ee radicielle valu\'ee}, Inst. Hautes
  \'Etudes Sci. Publ. Math. (1984), no.~60, 197--376. \MR{756316 (86c:20042)}

\bibitem[BT87]{BT3}
\bysame, \emph{Groupes alg\'ebriques sur un corps local. {C}hapitre {III}.
  {C}ompl\'ements et applications \`a la cohomologie galoisienne}, J. Fac. Sci.
  Univ. Tokyo Sect. IA Math. \textbf{34} (1987), no.~3, 671--698. \MR{927605
  (89b:20099)}

\bibitem[Car93]{Carter93}
Roger~W. Carter, \emph{Finite groups of {L}ie type}, Wiley Classics Library,
  John Wiley \& Sons, Ltd., Chichester, 1993, Conjugacy classes and complex
  characters, Reprint of the 1985 original, A Wiley-Interscience Publication.
  \MR{1266626 (94k:20020)}

\bibitem[DeB06]{Deb06}
Stephen DeBacker, \emph{Parameterizing conjugacy classes of maximal unramified
  tori via {B}ruhat-{T}its theory}, Michigan Math. J. \textbf{54} (2006),
  no.~1, 157--178. \MR{2214792 (2007d:22012)}

\bibitem[DL76]{DL76}
P.~Deligne and G.~Lusztig, \emph{Representations of reductive groups over
  finite fields}, Ann. of Math. (2) \textbf{103} (1976), no.~1, 103--161.
  \MR{0393266 (52 \#14076)}

\bibitem[DR09]{DR09}
Stephen DeBacker and Mark Reeder, \emph{Depth-zero supercuspidal {$L$}-packets
  and their stability}, Ann. of Math. (2) \textbf{169} (2009), no.~3, 795--901.
  \MR{2480618 (2010d:22023)}

\bibitem[DR10]{DR08}
\bysame, \emph{On some generic very cuspidal representations}, Compos. Math.
  \textbf{146} (2010), no.~4, 1029--1055. \MR{2660683 (2011j:20114)}

\bibitem[DS]{DS}
Stephen DeBacker and Loren Spice, \emph{Stability of character sums for
  positive-depth supercuspidal representations}, arXiv:1310.3306.

\bibitem[Hak16]{Hak16}
Jeffrey Hakim, \emph{Constructing tame supercuspidal representations}, draft,
  2016.

\bibitem[Hal93]{Hal93}
Thomas~C. Hales, \emph{A simple definition of transfer factors for unramified
  groups}, Representation theory of groups and algebras, Contemp. Math., vol.
  145, Amer. Math. Soc., Providence, RI, 1993, pp.~109--134. \MR{1216184
  (94e:22020)}

\bibitem[He08]{He08}
Xuhua He, \emph{On the affineness of {D}eligne-{L}usztig varieties}, J. Algebra
  \textbf{320} (2008), no.~3, 1207--1219. \MR{2427638 (2009c:20085)}

\bibitem[HM08]{HM08}
Jeffrey Hakim and Fiona Murnaghan, \emph{Distinguished tame supercuspidal
  representations}, Int. Math. Res. Pap. IMRP (2008), no.~2, Art. ID rpn005,
  166. \MR{2431732 (2010a:22022)}

\bibitem[How77]{Howe77}
Roger~E. Howe, \emph{Tamely ramified supercuspidal representations of {${\rm
  Gl}_{n}$}}, Pacific J. Math. \textbf{73} (1977), no.~2, 437--460. \MR{0492087
  (58 \#11241)}

\bibitem[Kala]{KalGRI}
Tasho Kaletha, \emph{Global rigid inner forms and multiplicities of discrete
  automorphic representations}, arXiv:1501.01667.

\bibitem[Kalb]{KalRIBG}
\bysame, \emph{Rigid inner forms vs isocrystals}, J. Eur. Math. Soc. (JEMS), to
  appear.

\bibitem[Kal11]{KalECI}
\bysame, \emph{Endoscopic character identities for depth-zero supercuspidal
  {$L$}-packets}, Duke Math. J. \textbf{158} (2011), no.~2, 161--224.
  \MR{2805068 (2012f:22031)}

\bibitem[Kal13]{KalGen}
\bysame, \emph{Genericity and contragredience in the local {L}anglands
  correspondence}, Algebra Number Theory \textbf{7} (2013), no.~10, 2447--2474.
  \MR{3194648}

\bibitem[Kal14]{KalIso}
\bysame, \emph{Supercuspidal {$L$}-packets via isocrystals}, Amer. J. Math.
  \textbf{136} (2014), no.~1, 203--239. \MR{3163358}

\bibitem[Kal15]{KalEpi}
\bysame, \emph{Epipelagic {$L$}-packets and rectifying characters}, Invent.
  Math. \textbf{202} (2015), no.~1, 1--89. \MR{3402796}

\bibitem[Kal16]{KalRI}
\bysame, \emph{Rigid inner forms of real and {$p$}-adic groups}, Ann. of Math.
  (2) \textbf{184} (2016), no.~2, 559--632. \MR{3548533}

\bibitem[Kim07]{Kim07}
Ju-Lee Kim, \emph{Supercuspidal representations: an exhaustion theorem}, J.
  Amer. Math. Soc. \textbf{20} (2007), no.~2, 273--320 (electronic).
  \MR{2276772 (2008c:22014)}

\bibitem[Kos78]{Kos78}
Bertram Kostant, \emph{On {W}hittaker vectors and representation theory},
  Invent. Math. \textbf{48} (1978), no.~2, 101--184. \MR{507800 (80b:22020)}

\bibitem[Kot]{KotBG}
Robert~E. Kottwitz, \emph{{$B(G)$} for all local and global fields},
  arXiv:1401.5728.

\bibitem[Kot82]{Kot82}
\bysame, \emph{Rational conjugacy classes in reductive groups}, Duke Math. J.
  \textbf{49} (1982), no.~4, 785--806. \MR{683003 (84k:20020)}

\bibitem[Kot83]{Kot83}
\bysame, \emph{Sign changes in harmonic analysis on reductive groups}, Trans.
  Amer. Math. Soc. \textbf{278} (1983), no.~1, 289--297. \MR{697075
  (84i:22012)}

\bibitem[Kot84]{Kot84}
\bysame, \emph{Stable trace formula: cuspidal tempered terms}, Duke Math. J.
  \textbf{51} (1984), no.~3, 611--650. \MR{757954 (85m:11080)}

\bibitem[Kot86]{Kot86}
\bysame, \emph{Stable trace formula: elliptic singular terms}, Math. Ann.
  \textbf{275} (1986), no.~3, 365--399. \MR{858284 (88d:22027)}

\bibitem[Kot97]{Kot97}
\bysame, \emph{Isocrystals with additional structure. {II}}, Compositio Math.
  \textbf{109} (1997), no.~3, 255--339. \MR{1485921 (99e:20061)}

\bibitem[Kot99]{Kot99}
\bysame, \emph{Transfer factors for {L}ie algebras}, Represent. Theory
  \textbf{3} (1999), 127--138 (electronic). \MR{1703328 (2000g:22028)}

\bibitem[KS]{KS12}
Robert~E. Kottwitz and Diana Shelstad, \emph{On splitting invariants and sign
  conventions in endoscopic transfer}, arXiv:1201.5658.

\bibitem[KS99]{KS99}
\bysame, \emph{Foundations of twisted endoscopy}, Ast\'erisque (1999), no.~255,
  vi+190. \MR{1687096 (2000k:22024)}

\bibitem[Kut77]{Ku77}
P.~C. Kutzko, \emph{Mackey's theorem for nonunitary representations}, Proc.
  Amer. Math. Soc. \textbf{64} (1977), no.~1, 173--175. \MR{0442145 (56 \#533)}

\bibitem[Lan]{LanArt}
Robert~P. Langlands, \emph{On the functional equation of the {A}rtin
  {$L$}-functions}, \url{http://publications.ias.edu/rpl/paper/61}.

\bibitem[Lan79]{Lan79}
R.~P. Langlands, \emph{Stable conjugacy: definitions and lemmas}, Canad. J.
  Math. \textbf{31} (1979), no.~4, 700--725. \MR{540901 (82j:10054)}

\bibitem[Lan83]{Lan83}
\bysame, \emph{Les d\'ebuts d'une formule des traces stable}, Publications
  Math\'ematiques de l'Universit\'e Paris VII [Mathematical Publications of the
  University of Paris VII], vol.~13, Universit\'e de Paris VII, U.E.R. de
  Math\'ematiques, Paris, 1983. \MR{697567 (85d:11058)}

\bibitem[Lan89]{Lan89}
\bysame, \emph{On the classification of irreducible representations of real
  algebraic groups}, Representation theory and harmonic analysis on semisimple
  {L}ie groups, Math. Surveys Monogr., vol.~31, Amer. Math. Soc., Providence,
  RI, 1989, pp.~101--170. \MR{1011897 (91e:22017)}

\bibitem[Lan13]{Lan13}
Robert~P. Langlands, \emph{Singularit\'es et transfert}, Ann. Math. Qu\'e.
  \textbf{37} (2013), no.~2, 173--253. \MR{3117742}

\bibitem[LS87]{LS87}
R.~P. Langlands and D.~Shelstad, \emph{On the definition of transfer factors},
  Math. Ann. \textbf{278} (1987), no.~1-4, 219--271. \MR{909227 (89c:11172)}

\bibitem[LS90]{LS90}
R.~Langlands and D.~Shelstad, \emph{Descent for transfer factors}, The
  {G}rothendieck {F}estschrift, {V}ol.\ {II}, Progr. Math., vol.~87,
  Birkh\"auser Boston, Boston, MA, 1990, pp.~485--563. \MR{1106907 (92i:22016)}

\bibitem[Mor89]{Mor89}
Lawrence Morris, \emph{{$P$}-cuspidal representations of level one}, Proc.
  London Math. Soc. (3) \textbf{58} (1989), no.~3, 550--558. \MR{988102
  (90c:22056)}

\bibitem[Moy86]{Moy86}
Allen Moy, \emph{Local constants and the tame {L}anglands correspondence},
  Amer. J. Math. \textbf{108} (1986), no.~4, 863--930. \MR{853218 (88b:11081)}

\bibitem[MP94]{MP94}
Allen Moy and Gopal Prasad, \emph{Unrefined minimal {$K$}-types for {$p$}-adic
  groups}, Invent. Math. \textbf{116} (1994), no.~1-3, 393--408. \MR{1253198
  (95f:22023)}

\bibitem[MP96]{MP96}
\bysame, \emph{Jacquet functors and unrefined minimal {$K$}-types}, Comment.
  Math. Helv. \textbf{71} (1996), no.~1, 98--121. \MR{1371680 (97c:22021)}

\bibitem[Mur11]{Mur11}
Fiona Murnaghan, \emph{Parametrization of tame supercuspidal representations},
  On certain {$L$}-functions, Clay Math. Proc., vol.~13, Amer. Math. Soc.,
  Providence, RI, 2011, pp.~439--469. \MR{2767524 (2012e:22024)}

\bibitem[MW87]{MW87}
C.~M{\oe}glin and J.-L. Waldspurger, \emph{Mod\`eles de {W}hittaker
  d\'eg\'en\'er\'es pour des groupes {$p$}-adiques}, Math. Z. \textbf{196}
  (1987), no.~3, 427--452. \MR{913667 (89f:22024)}

\bibitem[Ng{\^o}10]{Ngo10}
Bao~Ch{\^a}u Ng{\^o}, \emph{Le lemme fondamental pour les alg\`ebres de {L}ie},
  Publ. Math. Inst. Hautes \'Etudes Sci. (2010), no.~111, 1--169. \MR{2653248
  (2011h:22011)}

\bibitem[NX91]{NX91}
Enric Nart and Xavier Xarles, \emph{Additive reduction of algebraic tori},
  Arch. Math. (Basel) \textbf{57} (1991), no.~5, 460--466. \MR{1129520
  (92m:14056)}

\bibitem[Pop03]{Popov03}
Sergei~Yu. Popov, \emph{Standard integral models of algebraic tori},
  Preprintreihe des SFB 478 - Geometrische Strukturen der Mathematik
  \textbf{252} (2003), 1--31.

\bibitem[PR94]{PR94}
Vladimir Platonov and Andrei Rapinchuk, \emph{Algebraic groups and number
  theory}, Pure and Applied Mathematics, vol. 139, Academic Press, Inc.,
  Boston, MA, 1994, Translated from the 1991 Russian original by Rachel Rowen.
  \MR{1278263 (95b:11039)}

\bibitem[PR08]{PR08}
G.~Pappas and M.~Rapoport, \emph{Twisted loop groups and their affine flag
  varieties}, Adv. Math. \textbf{219} (2008), no.~1, 118--198, With an appendix
  by T. Haines and Rapoport. \MR{2435422 (2009g:22039)}

\bibitem[Pra01]{Pr01}
Gopal Prasad, \emph{Galois-fixed points in the {B}ruhat-{T}its building of a
  reductive group}, Bull. Soc. Math. France \textbf{129} (2001), no.~2,
  169--174. \MR{1871292 (2002j:20088)}

\bibitem[Rap05]{Rap05}
Michael Rapoport, \emph{A guide to the reduction modulo {$p$} of {S}himura
  varieties}, Ast\'erisque (2005), no.~298, 271--318, Automorphic forms. I.
  \MR{2141705 (2006c:11071)}

\bibitem[Ree08]{Ree08}
Mark Reeder, \emph{Supercuspidal {$L$}-packets of positive depth and twisted
  {C}oxeter elements}, J. Reine Angew. Math. \textbf{620} (2008), 1--33.
  \MR{2427973 (2009e:22019)}

\bibitem[Rie70]{Riehm70}
Carl Riehm, \emph{The norm {$1$} group of a {$\mf{p}$}-adic division algebra},
  Amer. J. Math. \textbf{92} (1970), 499--523. \MR{0262250}

\bibitem[Roe11]{Roe11}
David~Lawrence Roe, \emph{The {L}ocal {L}anglands {C}orrespondence for {T}amely
  {R}amified {G}roups}, ProQuest LLC, Ann Arbor, MI, 2011, Thesis
  (Ph.D.)--Harvard University. \MR{2898606}

\bibitem[RR96]{RR96}
M.~Rapoport and M.~Richartz, \emph{On the classification and specialization of
  {$F$}-isocrystals with additional structure}, Compositio Math. \textbf{103}
  (1996), no.~2, 153--181. \MR{1411570}

\bibitem[RY14]{RY14}
Mark Reeder and Jiu-Kang Yu, \emph{Epipelagic representations and invariant
  theory}, J. Amer. Math. Soc. \textbf{27} (2014), no.~2, 437--477.
  \MR{3164986}

\bibitem[RZ10]{RZ10}
Luis Ribes and Pavel Zalesskii, \emph{Profinite groups}, second ed., Ergebnisse
  der Mathematik und ihrer Grenzgebiete. 3. Folge. A Series of Modern Surveys
  in Mathematics [Results in Mathematics and Related Areas. 3rd Series. A
  Series of Modern Surveys in Mathematics], vol.~40, Springer-Verlag, Berlin,
  2010. \MR{2599132}

\bibitem[Ser79]{SerLF}
Jean-Pierre Serre, \emph{Local fields}, Graduate Texts in Mathematics, vol.~67,
  Springer-Verlag, New York-Berlin, 1979, Translated from the French by Marvin
  Jay Greenberg. \MR{554237 (82e:12016)}

\bibitem[Sha90]{Sha90}
Freydoon Shahidi, \emph{A proof of {L}anglands' conjecture on {P}lancherel
  measures; complementary series for {$p$}-adic groups}, Ann. of Math. (2)
  \textbf{132} (1990), no.~2, 273--330. \MR{1070599 (91m:11095)}

\bibitem[She79]{She79C}
D.~Shelstad, \emph{Characters and inner forms of a quasi-split group over
  {${\bf R}$}}, Compositio Math. \textbf{39} (1979), no.~1, 11--45. \MR{539000
  (80m:22023)}

\bibitem[She82]{She82}
\bysame, \emph{{$L$}-indistinguishability for real groups}, Math. Ann.
  \textbf{259} (1982), no.~3, 385--430. \MR{661206 (84c:22017)}

\bibitem[She89]{She89}
\bysame, \emph{A formula for regular unipotent germs}, Ast\'erisque (1989),
  no.~171-172, 275--277, Orbites unipotentes et repr{\'e}sentations, II.
  \MR{1021506 (91b:22012)}

\bibitem[She08a]{SheTE1}
\bysame, \emph{Tempered endoscopy for real groups. {I}. {G}eometric transfer
  with canonical factors}, Representation theory of real reductive {L}ie
  groups, Contemp. Math., vol. 472, Amer. Math. Soc., Providence, RI, 2008,
  pp.~215--246. \MR{2454336 (2011d:22013)}

\bibitem[She08b]{SheTE3}
\bysame, \emph{Tempered endoscopy for real groups. {III}. {I}nversion of
  transfer and {$L$}-packet structure}, Represent. Theory \textbf{12} (2008),
  369--402. \MR{2448289 (2010c:22016)}

\bibitem[She10]{SheTE2}
\bysame, \emph{Tempered endoscopy for real groups. {II}. {S}pectral transfer
  factors}, Automorphic forms and the {L}anglands program, Adv. Lect. Math.
  (ALM), vol.~9, Int. Press, Somerville, MA, 2010, pp.~236--276. \MR{2581952}

\bibitem[Spi08]{Spice08}
Loren Spice, \emph{Topological {J}ordan decompositions}, J. Algebra
  \textbf{319} (2008), no.~8, 3141--3163. \MR{2408311}

\bibitem[Spr81]{Spr81}
T.~A. Springer, \emph{Linear algebraic groups}, Progress in Mathematics,
  vol.~9, Birkh\"auser, Boston, Mass., 1981. \MR{632835 (84i:20002)}

\bibitem[SS70]{SS70}
T.~A. Springer and R.~Steinberg, \emph{Conjugacy classes}, Seminar on
  {A}lgebraic {G}roups and {R}elated {F}inite {G}roups ({T}he {I}nstitute for
  {A}dvanced {S}tudy, {P}rinceton, {N}.{J}., 1968/69), Lecture Notes in
  Mathematics, Vol. 131, Springer, Berlin, 1970, pp.~167--266. \MR{0268192 (42
  \#3091)}

\bibitem[Ste65]{Ste65reg}
Robert Steinberg, \emph{Regular elements of semisimple algebraic groups}, Inst.
  Hautes \'Etudes Sci. Publ. Math. (1965), no.~25, 49--80. \MR{0180554 (31
  \#4788)}

\bibitem[Ste68]{Ste68end}
\bysame, \emph{Endomorphisms of linear algebraic groups}, Memoirs of the
  American Mathematical Society, No. 80, American Mathematical Society,
  Providence, R.I., 1968. \MR{0230728 (37 \#6288)}

\bibitem[Tit78]{TitsWhitehead}
Jacques Tits, \emph{Groupes de {W}hitehead de groupes alg\'ebriques simples sur
  un corps (d'apr\`es {V}. {P}. {P}latonov et al.)}, S\'eminaire {B}ourbaki,
  29e ann\'ee (1976/77), Lecture Notes in Math., vol. 677, Springer, Berlin,
  1978, pp.~Exp. No. 505, pp. 218--236. \MR{521771}

\bibitem[Tit79]{Tits79}
J.~Tits, \emph{Reductive groups over local fields}, Automorphic forms,
  representations and {$L$}-functions ({P}roc. {S}ympos. {P}ure {M}ath.,
  {O}regon {S}tate {U}niv., {C}orvallis, {O}re., 1977), {P}art 1, Proc. Sympos.
  Pure Math., XXXIII, Amer. Math. Soc., Providence, R.I., 1979, pp.~29--69.
  \MR{546588 (80h:20064)}

\bibitem[Vog78]{Vog78}
David~A. Vogan, Jr., \emph{Gelfand-{K}irillov dimension for {H}arish-{C}handra
  modules}, Invent. Math. \textbf{48} (1978), no.~1, 75--98. \MR{0506503 (58
  \#22205)}

\bibitem[Wal97]{Wal97}
J.-L. Waldspurger, \emph{Le lemme fondamental implique le transfert},
  Compositio Math. \textbf{105} (1997), no.~2, 153--236. \MR{1440722
  (98h:22023)}

\bibitem[Wal01]{Wal01}
Jean-Loup Waldspurger, \emph{Int\'egrales orbitales nilpotentes et endoscopie
  pour les groupes classiques non ramifi\'es}, Ast\'erisque (2001), no.~269,
  vi+449. \MR{1817880 (2002h:22014)}

\bibitem[Wal06]{Wal06ECC}
J.-L. Waldspurger, \emph{Endoscopie et changement de caract\'eristique}, J.
  Inst. Math. Jussieu \textbf{5} (2006), no.~3, 423--525. \MR{2241929
  (2007h:22007)}

\bibitem[Yu]{Yu03}
Jiu-Kang Yu, \emph{Smooth models associated to concave functions in
  {Bruhat}-{Tits} theory}.

\bibitem[Yu01]{Yu01}
\bysame, \emph{Construction of tame supercuspidal representations}, J. Amer.
  Math. Soc. \textbf{14} (2001), no.~3, 579--622 (electronic). \MR{1824988
  (2002f:22033)}

\bibitem[Yu09]{Yu09}
\bysame, \emph{On the local {L}anglands correspondence for tori}, Ottawa
  lectures on admissible representations of reductive {$p$}-adic groups, Fields
  Inst. Monogr., vol.~26, Amer. Math. Soc., Providence, RI, 2009, pp.~177--183.
  \MR{2508725 (2009m:11201)}

\end{thebibliography}
